\documentclass{amsart}
\usepackage{amssymb}
\usepackage{amsbsy, amsthm, amsmath,amstext, amsopn}
\usepackage{tikz-cd}
\usepackage{multicol}
\usepackage{tikz}
\usetikzlibrary{arrows,snakes,backgrounds}
\usetikzlibrary{decorations.markings}
\usetikzlibrary{matrix}
\usetikzlibrary{graphs}

\theoremstyle{definition}
\newtheorem{theorem}{Theorem}[section]
\newtheorem{prop}[theorem]{Proposition}

\theoremstyle{definition}

\newtheorem{definition}[theorem]{Definition}

\newtheorem{observation}[theorem]{Observation}

\theoremstyle{remark}
\newtheorem{rmk}[theorem]{Remark}

\newcommand{\Z}{\mathbb{Z}}
\newcommand{\Q}{\mathbb{Q}}
\newcommand{\R}{\mathbb{R}}

\newcommand{\A}{\mathcal{A}}
\newcommand{\B}{\mathcal{B}}
\newcommand{\X}{\mathcal{X}}

\newcommand{\C}{\mathbb{C}}

\newcommand{\Spec}{\operatorname{Spec}}
\newcommand{\Conf}{\operatorname{Conf}}
\DeclareMathOperator{\End}{End}

\newcommand\ontop[2]{\genfrac{}{}{0pt}{}{#1}{#2}}

\newcommand{\dud}[2]{{\small\begin{pmatrix}  #1 \\   \\ #2 \end{pmatrix}}}

\newcommand{\tcfl}[3]{{\small\begin{pmatrix}  & #1 \\ #2 &  \\ & #3 \end{pmatrix}}}
\newcommand{\tcfr}[3]{{\small\begin{pmatrix}  #1 & \\ & #3  \\ #2 & \end{pmatrix}}}
\newcommand{\tcfu}[3]{{\small\begin{pmatrix}  & #1 & \\ #2 & & #3  \\ & & \end{pmatrix}}}
\newcommand{\tcfd}[3]{{\small\begin{pmatrix}  & & \\ #1 & & #3  \\ & #2 & \end{pmatrix}}}

\newcommand{\qcf}[4]{{\small\begin{pmatrix}  & #1 & \\ #2 & & #4  \\ & #3 & \end{pmatrix}}}
\newcommand{\qcfs}[4]{{\small\begin{pmatrix}  & #1 & \\ #2 & \diagup & #4  \\ & #3 & \end{pmatrix}}}
\newcommand{\qcfb}[4]{{\small\begin{pmatrix}  & #1 & \\ #2 & \diagdown & #4  \\ & #3 & \end{pmatrix}}}

\newcommand{\dul}[2]{{\small\begin{pmatrix}  & #1 \\ #2 &  \\ &  \end{pmatrix}}}
\newcommand{\dld}[2]{{\small\begin{pmatrix}  & \\ #1 &  \\ & #2 \end{pmatrix}}}
\newcommand{\dlr}[2]{{\small\begin{pmatrix}  &  & \\ #1 & & #2  \\ & & \end{pmatrix}}}
\newcommand{\dur}[2]{{\small\begin{pmatrix}  #1 & \\ & #2  \\ & \end{pmatrix}}}
\newcommand{\ddr}[2]{{\small\begin{pmatrix}  & \\ & #2 \\ #1 & \end{pmatrix}}}

\def\DynkinNodeSize{3.5mm}
\def\DynkinArrowLength{3mm}
\tikzset{
  dnode/.style={
    circle,
    inner sep=0pt,
    minimum size=\DynkinNodeSize,
    fill=white,
    draw},
  middlearrow/.style={
    decoration={markings,
      mark=at position 0.6 with
      {\draw (0:0mm) -- +(+135:\DynkinArrowLength); \draw (0:0mm) -- +(-135:\DynkinArrowLength);},
    },
    postaction={decorate}
  },
  leftrightarrow/.style={
    decoration={markings,
      mark=at position 0.999 with
      {
      \draw (0:0mm) -- +(+135:\DynkinArrowLength); \draw (0:0mm) -- +(-135:\DynkinArrowLength);
      },
      mark=at position 0.001 with
      {
      \draw (0:0mm) -- +(+45:\DynkinArrowLength); \draw (0:0mm) -- +(-45:\DynkinArrowLength);
      },
    },
    postaction={decorate}
  },
  sedge/.style={
  },
  dedge/.style={
    middlearrow,
    double distance=0.5mm,
  },
  tedge/.style={
    middlearrow,
    double distance=1.0mm+\pgflinewidth,
    postaction={draw}, 
  },
  infedge/.style={
    leftrightarrow,
    double distance=0.5mm,
  },
}

\newcounter{Qcount}
\begin{document}

\stepcounter{Qcount}

\title{Cluster Structures on Higher Teichmuller Spaces for Classical Groups}

\author{Ian Le}
\address{Department of Mathematics\\ University of Chicago\\Chicago, IL 60637}
\email{ile@math.uchicago.edu}

\begin{abstract} Let $S$ be a surface, $G$ a simply-connected classical group, and $G'$ the associated adjoint form of the group. We show that the spaces of moduli spaces of framed local systems $\X_{G',S}$ and $\A_{G,S}$, which were constructed by Fock and Goncharov (\cite{FG1}), have the structure of cluster varieties, and thus together form a cluster ensemble. This simplifies some of the proofs in \cite{FG1}, and also allows one to quantize higher Teichmuller space following  the formalism of  \cite{FG2}, \cite{FG3}, and \cite{FG5}, which was previously only possible when $G$ had type $A$.

\end{abstract}

\maketitle

\tableofcontents

\section{Introduction}

Let $S$ be a topological surface $S$ of with non-empty boundary, let $G$ be a simply connected semi-simple group, and let $G'$ be the adjoint form of this group. We are interested in the space $\mathcal{L}_{G,S}$, the moduli space of $G$-local systems on the surface $S$, or, equivalently, the space of representations of $\pi_1(S)$ into $G$. There are two closely related spaces $\X_{G',S}$ and $\A_{G,S}$, which were constructed by Fock and Goncharov \cite{FG1}. Both the the spaces $\X_{G',S}$ and $\A_{G,S}$ are variations on the space $\mathcal{L}_{G,S}$; they parameterize local systems with certain types of framing at the boundary of $S$.

One advantage of studying the spaces $\X_{G',S}$ and $\A_{G,S}$ is that they are simpler; for example, they have rational co-ordinate charts, while $\mathcal{L}_{G,S}$ does not, in general. In our view, the spaces $\X_{G',S}$ and $\A_{G,S}$ are perhaps more fundamental, and allow one to understand properties of the space $\mathcal{L}_{G,S}$ that are otherwise hard to see.

In \cite{FG1}, Fock and Goncharov show that each of the spaces $\X_{G',S}$ and $\A_{G,S}$ has an atlas of coordinate charts such that all transition functions involve only addition, multiplication and division. In other words, these spaces each have a {\it positive atlas} and may be called {\it positive varieties}. The main ingredient in the construction of this positive atlas of coordinate charts is Lusztig's theory of total positivity. The fact that $\X_{G',S}$ and $\A_{G,S}$ are positive varieties is quite striking; it is the key result in \cite{FG1}.

For the spaces $\X_{PGL_n,S}$ and $\A_{SL_n,S}$, Fock and Goncharov show more: these spaces have cluster-like structures on their rings of functions, and thus the fact that they are positive varieties is a consequence of the fact that they are {\it cluster varieties}. Moreover, the pair of spaces $(\A_{SL_n,S}, \X_{PGL_n,S})$ are tightly connected, and form what they call a {\it cluster ensemble}.

Let us review some of the consequences of the theory. The moduli space of representations of $\pi_1(S)$ into a split real group $G(\R)$ has several components, one of which is called higher Teichmuller space, because in many ways it behaves like Teichmuller space. Because $\X_{G',S}$ and $\A_{G,S}$ are positive varieties, it makes sense to take the positive $\R_{>0}$ points of these moduli spaces. The resulting spaces $\X_{G',S}(\R_{>0})$ and $\A_{G,S}(\R_{>0})$ parameterize positive representations. With some work, Fock and Goncharov identify positive representations with higher Teichmuller space, thus giving it an algebro-geometric description. The theory gives an explicit parameterization of positive representations, and it becomes manifest that the space of representations is contractible. Moreover, one can show that all positive representations are discrete and faithful \cite{FG1}, \cite{L}.

In this paper, we will construct cluster ensemble structures on $\X_{G',S}$ and $\A_{G,S}$ for $G$ a classical group, i.e. for $G$ of type $B, C$ and $D$. In other words, we show that $(\A_{G,S}, \X_{G',S})$ forms a cluster ensemble. In doing so, we give what we think to be a beautiful set of functions on the spaces $\X_{G',S}$ and $\A_{G,S}$ which generalizes the ``canonical functions'' when $G=SL_n$. The canonical functions for $SL_n$ have a long history, see \cite{RR}, \cite{S}, \cite{HK}, \cite{Hen}. As in the case of $G=SL_n$, our functions will induce positive structures on $\X_{G',S}$ and $\A_{G,S}$, which will agree with the positive structures defined in \cite{FG1}.

The cluster structures on double Bruhat cells and in flag varieties is well known (\cite{BFZ}), and they form a building block for the cluster structures on $\X_{G',S}$ and $\A_{G,S}$. The cluster structures on $\X_{G',S}$ and $\A_{G,S}$ are constructed via ideal triangulations of $S$. The cluster structure on each triangle is close to the cluster structure on the Borel subgroup $B$ for the group $G$. The first main problem is to correctly identify the cluster structure on each triangle and to understand how these structures glue together. We were inspired in our solution to this problem by the idea of ``amalgamation'' developed in \cite{FG4}.

For each triangulation of the surface $S$ and for any ordering of the vertices in each triangle of the triangulation, one can write down an associated seed for the cluster ensemble. (This seed can be constructed out of a reduced word for the longest element $w_0$ of the Weyl group of $G$.) When $G=SL_n$ or $PGL_n$ these functions (for a particular choice of the reduced word for $w_0$) somewhat miraculously exhibit $S_3$ symmetry. This is not the case for general groups $G$. We give the sequence of mutations that relates seeds associated to different orderings of the vertices in a given triangle in the triangulation.

Moreover, we give a sequence of mutations relating seeds that correspond to different triangulations. This problem reduces to the problem of relating seeds coming from triangulations that differ by a ``flip,'' which is the change of triangulation that results from replacing one diagonal of a quadrilateral by another. When $G$ has type $A$, this sequence of mutations comes the octahedron recurrence, which has been very well-studied \cite{RR}, \cite{S}, \cite{HK}. We will give an analogue of the octahedron recurrence for $G$ is a classical group.

Both the sequences of mutations realizing $S_3$ symmetries and flips of triangulations are relatives the octahedron recurrence (\cite{S}) and the tetrahedron recurrence (which is related to the Jeu de Taquin operation on tableaux and the action of the cactus group on crystals \cite{HK}, \cite{Hen}). The computation of these sequences of mutations forms the technical heart of this paper. This has several consequences.

\begin{enumerate}
\item We get an alternative construction of $\X_{G',S}$ and $\A_{G,S}$ as positive varieties. This allows somewhat simpler proofs of many of the results in \cite{FG1}.
\item The system of co-ordinate charts is naturally mapping class group equivariant. We see therefore get an injection of the mapping class group into the cluster modular group for $(\A_{G,S}, \X_{G',S})$.
\item Using the formalism of \cite{FG2}, \cite{FG3}, \cite{FG5}, we can construct the quantization of $\X_{G',S}$ for $G$ any classical group, a result previous only known for $G'=PGL_n$. This quantization if a $*$-algebra, and one also gets quantum representations of this algebra. One can also construct the symplectic double of $\X_{G',S}$.
\item One observes that for classical groups, the cluster ensembles $(\A_{G,S}, \X_{G',S})$ and $(\A_{(G')^{\vee},S}, \X_{G^{\vee},S})$ have Langlands dual seeds. Here $G^{\vee}$ is the Langlands dual group of $G$. This is some manifestation of mirror symmetry, and gives support for conjectures of \cite{GS}.
\item One sees that the deformed algebras $\X_{q,G',S}$ and $\X_{q^{\vee},G^{\vee},S}$ are closely related, as predicted in \cite{FG2},
\item Using \cite{Le}, one can define higher laminations for all classical groups.
\item One can prove sharper versions of the saturation conjecture \cite{Le2}.
\end{enumerate}

In fact, for a general reductive group, our constructions allow us to construct the cluster ensemble structure for $\X_{G',S}$ and $\A_{G,S}$. However, in the exceptional cases, we do not know how to realize the $S_3$ symmetries and the flip. This is the only piece of the puzzle missing that prevents us from getting a mapping class group equivariant quantization of higher Teichmuller space and all the associated results above.

Let us explain the structure of this paper. In Section 2, we review some of the work of Fock and Goncharov. We define the spaces $\X_{G',S}$ and $\A_{G,S}$ and relate them to spaces of configurations of points in $G/B$ and $G/U$: $\Conf_m \A_G$ and $\Conf_m \B_G$. We also recall the necessary facts about cluster algebras.

In Sections 3-5, we give the cluster algebra structure on the spaces $\Conf_m \A_G$ for $G=Sp_{2n}, Spin_{2n+1}, Sp_{2n+2}$. The structure of these three sections is much parallel. This is in part so that the sections may be read independently, and in part to emphasize the similarities (and differences) between these three different cases. A reader only interested in one case may skip the others. These sections each contain the construction of the seed for the cluster algebra structure on $\Conf_m \A_G$, as well the sequences of mutations that realize the $S_3$ symmetries on each triangle and the flip of a triangulation. These sections form the technical heart of the paper. We have included complete formulas in all three cases though in some of the proofs of these formulas, we refer to previous sections if an identity we need has already been proven. Thus in Sections 4 and 5, we emphasize those aspects of the proof of the mutation identities which are novel.

The formulas and mutation sequences can become quite involved. We try to alleviate this in two ways. First of all, much of the important information in our computations is succinctly captured in the diagrams, so that one can extract most of the content of the computations from the diagrams once one understands how to read them. Most of the general phenomena can be seen in small cases, like $Sp_8$, $Spin_9$ or $Spin_{10}$, which are treated in the diagrams.

Second, throughout Sections 3-5, we try to explain the conceptual underpinnings of the calculations. This is done in several ways. We show how the cluster algebra structure is related to reduced words in the Weyl group; we show how the cactus sequence of mutation plays a role in constructing the seeds and computing some of the $S_3$ symmetries; we explain how folding the seed for $SL_{2n}$ gives the seed for $Sp_{2n}$, how Langlands duality relates the seeds for $Sp_{2n}$ and $Spin_{2n+1}$, and how unfolding the seed for $Spin_{2n+1}$ gives the seed for $Spin_{2n+2}$. The latter fact is perhaps most important: all the sequences of mutations that we consider arise from the ones for $Sp_{2n}$ by Langlands duality or unfolding.

The remainder of the paper proceeds as follows. Section 6 treats the cluster-like structure on $\Conf_m \B_G$. Section 7 gives applications to the quantization of higher Teichmuller space. Section 8 is an appendix containing some calculations that are parallel to ones treated in the main body of the text, but seemed unnecessary to include.


\medskip
\noindent{\bf Acknowledgments} 
I thank Joel Kamnitzer for first explaining to me the role that the cactus group plays in the representation theory of $SL_n$. The role of the cactus group and the tetrahedron recurrence only slowly became clear to me in the course of writing this paper. I am also indebted to Inna Zakharevich for showing me how to make the diagrams in this paper.

\section{Background}

\subsection{Setup}

Let $S$ be a compact oriented surface, with or without boundary, and possibly with a finite number of marked points on each boundary component. We will refer to this whole set of data--the surface and the marked points on the boundary--by $S$. We will always take $S$ to be hyperbolic, meaning it either has negative Euler characteristic, or contains enough marked points on the boundary (in other words, we can give it the structure of a hyperbolic surface such that the boundary components that do no contain marked points are cusps, and all the marked points are also cusps).

Let $G$ be a semi-simple algebraic group. When $G$ is adjoint, i.e., has trivial center (for example, when $G=PGL_m$), we can define a higher Teichmuller space $\X_{G,S}$. On the other hand, for $G$ simply-connected  (for example, when $G=SL_m$), we can define the higher Teichmuller space $\A_{G,S}$. They will be the space of local systems of $S$ with structure group $G$ with some extra structure of a framing of the local system at the boundary components of $S$. Alternatively, these spaces describe homomorphisms of $\pi_1(S)$ into $G$ modulo conjugation plus some extra data.

When $S$ does has at least one hole, the spaces $\X_{G,S}$ and $\A_{G,S}$ have a distinguished collection of coordinate systems, equivariant under the action of the mapping class group of $S$. Using an elaboration of Lusztig's work on total positivity, one can show that all the transition functions between these coordinate systems are subtraction-free, and give a {\it positive atlas} on the corresponding moduli space. This positive atlas gives the spaces $\X_{G,S}$ and $\A_{G,S}$ the structure of a {\it positive variety}.

The existence of these extraordinary positive co-ordinate charts depends on G. Lusztig's theory of positivity in semi-simple Lie groups \cite{Lu}, \cite{Lu2}, and is a reflection of the cluster algebra structure of the ring of functions on these spaces.

\subsection{Definition of the spaces $\X_{G,S}$ and $\A_{G,S}$}

The data of a framing of a local system involves the geometry of the flag variety associated to a group. Let $B$ be a Borel subgroup, a maximal solvable subgroup of $G$. Then ${\mathcal B} = G/B$ is the flag variety. Let $U := [B,B]$ be a maximal unipotent subgroup in $G$. Then we will call $\A = G/U$ the ``principal affine space'' (sometimes also referred to as the ``base affine space''). We will refer to elements of $\A$ as ``principal flags.''

Let ${\mathcal L}$ be a $G$-local system on $S$. For any space $X$ on equipped with a $G$-action, we can form the associated bundle ${\mathcal L}_{X}$. For $X=G/B$ we get the associated flag bundle ${\mathcal L}_{\mathcal B}$, and for $X=G/U$, we get the associated principal flag bundle ${\mathcal L}_{\mathcal A}$.

\begin{definition}
A {\rm framed  $G$-local system on $S$} is a pair $({\mathcal L}, \beta)$, where  ${\mathcal L}$ is a  $G$-local system on $S$, and  $\beta$ a flat section of the restriction of ${\mathcal L}_{\mathcal B}$ to the punctured boundary of $S$.

The space ${\mathcal X}_{G', S}$ is the moduli space of framed $G$-local systems on $S$. 
\end{definition}

The definition of the space ${\mathcal A}_{G, S}$ is slightly more complicated. It involves twisted local systems. We shall define this notion.

Let $G$ is simply-connected. The maximal length element $w_0$ of the Weyl group of $G$ has a natural lift to $G$, denoted $\overline w_0$. Let $s_G:= {\overline w}^2_0$. It turns out that $s_G$ is in the center of $G$ and that $s^2_G =e$. Depending on $G$, $s_G$ will have order one or order two. For example, for $G = SL_{2k}$, $s_G$ has order two, while for $G = SL_{2k+1}$, $s_G$ has order one.

The fundamental group $\pi_1(S)$ has a natural central extension by $\Z / 2\Z$. We see this as follows. For a surface $S$, let $T'S$ be the tangent bundle with the zero-section removed. $\pi_1(T'S)$ is a central extension of $\pi_1(S)$ by $\Z$:

$$\Z \rightarrow \pi_1(T'S) \rightarrow \pi_1(S).$$

The quotient of $\pi_1(S)$ by the central subgroup $2 \Z \subset \Z$, gives ${\overline \pi}_1(S)$ which is a central extension of $\pi_1(S)$ by $\Z / 2\Z$:

$$\Z / 2\Z \rightarrow {\overline \pi}_1(S) \rightarrow \pi_1(S).$$

Let $\sigma_S \in {\overline \pi}_1(S)$ denote the non-trivial element of the center.

A {\it twisted $G$-local system} is a representation ${\overline \pi}_1(S)$ in $G$ such that $\sigma_S$ maps to $s_G$. Such a representation gives a local system on $T'S$.

Now we must describe the framing data for a twisted local system. Let ${\mathcal L}$ be a twisted $G$-local system on $S$. Such a twisted local system gives an associated principal affine bundle ${\overline {\mathcal L}}_{\mathcal A}$ on the punctured tangent bundle $T'S$. For any boundary component of $S$, we will construct sections of the punctured tangent bundle above these boundary components. Given any boundary component, consider the outward pointing unit tangent vectors along this component--this gives a section of the punctured tangent bundle above each boundary component of $S$. We get a bunch of loops and arcs in $T'S$ lie over the boundary of $S$. Call this the {\em lifted boundary}.

\begin{definition}

A {\it decorated $G$-local system} on $S$ consists of $({\mathcal L}, \alpha)$, where ${\mathcal L}$ is a twisted local system on $S$ and $\alpha$ is a flat section of ${\overline {\mathcal L}}_{\mathcal A}$ restricted to the lifted boundary.

The space ${\mathcal A}_{G, S}$ is the moduli space of decorated $G$-local systems on $S$.
\end{definition}

Note that in the case where $s_G=e$, a decorated local system is just a local system on $S$ along with a flat section of ${\mathcal L}_{\mathcal A}$ restricted to the boundary. One can generally pretend that this is the case without much danger.

\subsection{Relation to configurations of flags}

The positive co-ordinate systems on ${\mathcal X}_{G', S}$ and ${\mathcal A}_{G, S}$ arise by rationally identifing them with spaces of configurations of flags. For more details, see \cite{FG1}.

Let $S$ be a hyperbolic surface. Give it some hyperbolic structure such that the boundary components that do no contain marked points are cusps, and all the marked points are also cusps. The particular choice of hyperbolic structure will turn out not to matter. Then the universal cover of $S$ will be a subset of the hyperbolic plane, and all these cusps will lie at the boundary at infinity of the hyperbolic plane. These cusps form a set $C$ that has a cyclic ordering. $C$ also carries a natural action of $\pi_1(S)$. The action of $\pi_1(S)$ preserves the cyclic ordering on $C$. The set $C$ with its cyclic order is independent of our choice of hyperbolic structure on $C$. An {\em ideal triangulation} of $S$ consists of a triangulation of $S$ that has vertices at the cusps of $S$. An ideal triangulation of $S$ induces an ideal triangulation of its universal cover. This triangulation of the universal cover will have its vertices at the set $C$.

A $\pi_1(S)$-equivariant configuration of flags (respectively principal affine flags) parameterized by $C$ is a map $\beta: C \rightarrow {\mathcal B}$ (respectively a map $\beta: C \rightarrow {\mathcal A}$) such that there is a map $\rho: \pi_1(S) \rightarrow G$ such that for $\gamma \in \pi_1(S)$,

$$\beta(\gamma \cdot c) = \rho(\gamma) \cdot c$$
for all points $c \in C$.

Starting with any point of ${\mathcal X}_{G', S}$ (respectively ${\mathcal A}_{G, S}$), we may look at the universal cover of $S$. On the universal cover, the local system becomes trivial, and the framing of the local system then gives a flag (repectively a principal affine flag) at each point of the cyclic set $C$. Thus any point in ${\mathcal X}_{G', S}$ (respectively ${\mathcal A}_{G, S}$) gives a $\pi_1(S)$ equivariant configuration of flags (respectively principal affine flags) parameterized by $C$. 

\begin{theorem}
\cite{FG1} The space ${\mathcal X}_{G', S}$ has a positive atlas that comes from identifying a framed local system with a $\pi_1(S)$-equivariant positive configuration of flags parameterized by $C$.

The space ${\mathcal A}_{G, S}$ has a positive atlas that comes from identifying a decorated local system with a $\pi_1(S)$-equivariant twisted positive cyclic configuration of principal affine flags parameterized by $C$.

\end{theorem}

This theorem gives a birational equivalence between ${\mathcal X}_{G', S}$ (respectively ${\mathcal A}_{G, S}$) and the stack of $\pi_1$-equivariant positive configurations of flags (respectively principal affine flags) parameterized by $C$. This birational morphism is an isomorphism on positive points: the positive points of ${\mathcal X}_{G', S}$ (respectively ${\mathcal A}_{G, S}$) correspond to positive representations of $\rho: \pi_1(S) \rightarrow G$ plus decoration. Any point of ${\mathcal X}_{G', S}(\R_{>0})$ (respectively ${\mathcal A}_{G, S}(\R_{>0})$) gives a $\pi_1(S)$-equivariant positive configuration of flags parameterized by $C$. This configuration of flags uniquely determines the representation $\rho$ as well as the decoration on the corresponding local system.

When $S$ is a disk with marked points on the boundary we simply get moduli spaces of configurations of points in the flag variety ${\mathcal B}: = G'/B$ and twisted configurations of points of the principal affine variety ${\mathcal A}:= G/U$. For more details, see \cite{FG1}

\subsection{Positive Structures}

In this section, we explain how to construct the positive structures on ${\mathcal X}_{G', S}$ and ${\mathcal A}_{G, S}$.

Let $T=\Z [x_1^{\pm 1}, x_2^{\pm 1}, \dots, x_n^{\pm 1} ]$ be a split algebraic torus. A positive rational function on $T$ is a nonzero rational function on $T$ which can be written as a ratio of two polynomials in the $x_i$ with positive integral coefficients.

A positive rational morphism $\phi: T_1 \rightarrow T_2$ of two split tori is a morphism such that for any character $\chi$ of $T_2$, $\phi^*(\chi)$ is a positive rational function.

A positive atlas on an irreducible space $\mathcal{Y}$ over $\Q$ is a non-empty collection of split tori $T_{\mathbf{c}}$ along with of birational isomorphisms over $\Q$
$$\alpha_{\mathbf{c}} : T_{\mathbf{c}} \rightarrow \mathcal{Y}$$
such that for any pair $\mathbf{c}, \mathbf{c'}$, we have that the map $\phi_{\mathbf{c}, \mathbf{c'}} :=\alpha_{\mathbf{c}}^{-1} \circ \alpha_{\mathbf{c'}}$  is a positive birational morphism from $T_{\mathbf{c'}}$ to $T_{\mathbf{c}}$. (One can also require that $\alpha_{\mathbf{c}}$ is regular on a complement to a divisor given by positive rational function, but this will not play a role for us.) Each map $\alpha_{\mathbf{c}}$ gives a positive co-ordinate chart on $\mathcal{Y}$.

A positive rational function $F$ on $\mathcal{Y}$ is a rational function given by a subtraction free rational function in one, and hence in all, of the coordinate systems of the positive atlas on  $\mathcal{Y}$.

A positive rational map  $\mathcal{Y} \rightarrow  \mathcal{Z}$ is a rational map given by positive rational functions in one, and hence in all, positive coordinate systems.

When $S$ is a disc with $m$ marked points, the spaces ${\mathcal X}_{G', S}$ and ${\mathcal A}_{G, S}$ become the spaces of configurations of points in the flag variety $\Conf_m \B$ and configurations of points in the affine variety, $\Conf_m \A$. We will first describe the positive structures on these spaces.

We start with some notation. Let $U^+$ and $U^-$ denote the upper and lower unipotent subgroups of $G'$. Also let $B^+$ and $B^-$ denote the upper and lower Borel subgroups of $G'$. We can identify points in $\B$ with Borel subgroups. If $B$ is a Borel subgroup, we let $g \cdot B$ denote the Borel subgroup $gBg^{-1}$. The positive structures on these spaces are defined as follows. 

There is a is a birational map
$$\underbrace{U^- \times \cdots  \times U^-}_{\mbox{$n$ copies}}/H \longrightarrow {\Conf}_{n+2}({\B})$$
that sends
\begin{equation} \label{ConfB}
(u_1, \ldots , u_{n}) \rightarrow (B^-, B^+, u_1\cdot B^+, u_1u_2\cdot B^+, \ldots , u_1 ... u_n\cdot B^+).
\end{equation}

Here $H$ diagonally on the $U^-$ by conjugation. There is a natural positive structure on $U^-$ (\cite{FZ}), and this positive structure is preserved by conjugation. Thus $(U^-)^n/H$ is a positive variety, and birational equivalence induces a positive structure on ${\Conf}_{n+2}({\B})$.

Similarly, there is a birational map 
$$H \times \underbrace{B^- \times \cdots  \times B^-}_{\mbox{$n$ copies}}/ \longrightarrow {\Conf}_{n+2}({\A})$$
that sends
\begin{equation} \label{ConfA}
(h, b_1, \ldots , b_{n}) \rightarrow (U^-, h \cdot \overline{w_0}U^-, b_1\cdot \overline{w_0}U^-, b_1b_2\cdot \overline{w_0}U^-, \ldots , b_1 ... b_n \cdot \overline{w_0}U^-).
\end{equation}

The natural positive structure on $H \times (B^-)^n$ induces a positive structure on ${\Conf}_{n+2}({\A})$. (Note that $\overline{w}$ denotes a particular lift of the element $w$ of the Weyl group to $G$. See \cite{BFZ} or \cite{FG1}. Here, $w_0$ denotes the longest element of the Weyl group.)

The positive structures on ${\Conf}_{n+2}({\B})$ and ${\Conf}_{n+2}({\A})$ are invariant under cyclic shift and twisted cyclic shift, respectively. The cyclic shift on ${\Conf}_{n}({\B})$ maps $(B_1, B_2, \dots, B_n)$ to $(B_n, B_1, \dots, B_{n-1})$. The twisted cyclic shift on ${\Conf}_{n}({\A})$ maps $(U_1, U_2, \dots, U_n)$ to $(s_G \cdot U_n, U_1, \dots, U_{n-1})$.

\subsection{Reduction to the case of $\Conf_m \A$ or $\Conf_m \B$}

We now recall the positive structures on the spaces ${\mathcal X}_{G', S}$ and ${\mathcal A}_{G, S}$ as constructed in \cite{FG1}.

In \cite{FG1}, the authors explain how to construct a positive co-ordinate chart on ${\Conf}_{m}({\B})$ for each triangulation of an $m$-gon. If we place the $m$ flags at the vertices of an $m$-gon, then to each triangle in the triangulation of the $m$-gon, we get a configuration of three flags, and to this configuration of three flags, we attach some {\em face functions}. The face functions give a positive co-ordinate chart on ${\Conf}_{3}({\B})$. Any edge in the triangulation belongs to two triangles. We attach to each edge a set of {\em edge functions}, which depend on the four flags at the corners of the two triangles. For any edge, its edge functions along with the face functions for the triangles sharing that edge form a positive co-ordinate chart on ${\Conf}_{4}({\B})$. Thus the face functions give invariants of three flags, while the edge functions tell us how to glue two configurations of three flags into a configuration of four flags. Gluing along all edges of a triangulation will give a configuration of $m$ flags.

There are also positive co-ordinate charts on ${\Conf}_{m}({\A})$ attached to each triangulation of an $m$-gon. These are constructed slightly differently. To each edge in the triangulation, we attach a set of {\em edge functions} which depend on the two flags at the ends of the edge. To each triangle in the triangulation, we get a configuration of three principal flags, and in addition to the edge functions of each pair of edges in the triangle, we attach some {\em face functions}, which depend on all three of the principal flags (not just two at a time). Thus whereas for ${\Conf}_{m}({\B})$, the edge functions give us the data for gluing configurations of three flags, for ${\Conf}_{m}({\A})$, two configurations of three principal flags can be glued along an edge only if the edge functions along that edge are identical. Thus we exchange gluing data for restrictions on when we can glue.

For a general surface $S$, we use the identification of ${\mathcal X}_{G', S}$ (respectively ${\mathcal A}_{G, S}$) with $\pi_1$-equivariant configurations of flags (respectively, principal flags) parameterized by the cyclic set $C$. Any ideal triangulation of $S$ gives a triangulation of the infinite polygon with vertices the points of the cyclic set $C$. We take the face and edge functions for this triangulation. Because of $\pi_1$-equivariance, we get a finite set of functions coming from the edges and faces of the triangulation of $S$. This set of functions forms a positive co-ordinate chart for ${\mathcal X}_{G', S}$ (respectively ${\mathcal A}_{G, S}$). The co-ordinate charts coming from different triangulations of the surface give a positive atlas on ${\mathcal X}_{G', S}$ (respectively ${\mathcal A}_{G, S}$).

One of the goals of this paper will be to show that these edge and face functions can be realized as part of a cluster ensemble structure on the pair of spaces $({\mathcal X}_{G', S}, {\mathcal A}_{G, S})$.

\subsection{Cluster algebras} \label{cluster}

We review here the basic definitions of cluster algebras, following \cite{W}. Cluster algebras are commutative rings that come equipped with a collection of distinguished sets of generators, called \emph{cluster variables} or $\A$-coordinates. One can obtain one set of generators from another set of generators by a process called mutation.

Each set of generators belongs to a seed, which roughly consists of the set of generators along with a $B$-matrix. The $B$-matrix encodes how one mutates from one seed to any adjacent seed. Starting from any initial seed, the process of mutation gives all the seeds (and all the sets of generators) for the cluster algebra. The cluster variables are coordinates on the $\A$-space.

The same combinatorial data underlying a seed gives rise to a second, related, algebraic structure, called $\X$-coordinates. The $\X$-coordinates are functions on the $\X$ space. The $\A$-coordinates and $\X$-coordinates are related by a canonical monomial transformation, which gives a map from the $\A$-space to the $\X$-space. Together, the data of the $\A$-space and the $\X$-space, along with their distinguished sets of coordinates, is called a cluster ensemble.

Cluster algebras and $\X$-coordinates are defined by seeds. A seed $\Sigma = (I,I_0,B,d)$ consists of the following data: 
\begin{enumerate}
\item An index set $I$ with a subset $I_0 \subset I$ of ``frozen'' indices. 
\item A rational $I \times I$ \emph{exchange matrix} $B$. It should have the property that $b_{ij} \in \Z$ unless both $i$ and $j$ are frozen.  
\item A set $d = \{d_i \}_{i \in I}$ of positive integers that skew-symmetrize $B$; that is, $b_{ij}d_j = -b_{ji}d_i$ for all $i,j \in I$.
\end{enumerate}

For most purposes, the values of $d_i$ are only important up to simultaneous scaling. Also note that the values of $b_{ij}$ where $i$ and $j$ are both frozen will play no role in the cluster algebra, though it is sometimes convenient to assign values to $b_{ij}$ for bookkeeping purposes. These values become important in \emph{amalgamation}, where one unfreezes some of the frozen variables.

Let $k \in I \setminus I_0$ be an unfrozen index of a seed $\Sigma$.  We say another seed $\Sigma' = \mu_k(\Sigma)$ is obtained from $\Sigma$ by mutation at $k$ if we identify the index sets in such a way that the frozen variables and $d_i$ are preserved, and the exchange matrix $B'$ of $\Sigma'$ satisfies
\begin{align}\label{eq:matmut}
b'_{ij} = \begin{cases}
-b_{ij} & i = k \text{ or } j=k \\
b_{ij} & b_{ik}b_{kj} \leq 0 \\
b_{ij} + |b_{ik}|b_{kj} & b_{ik}b_{kj} > 0.
\end{cases}
\end{align}
Two seeds $\Sigma$ and $\Sigma'$ are said to be mutation equivalent if they are related by a finite sequence of mutations.

To a seed $\Sigma$ we associate a collection of \emph{cluster variables} $\{A_i\}_{i \in I}$ and a split algebraic torus $\A_\Sigma := \Spec \Z[A^{\pm 1}_I]$, where $\Z[A^{\pm 1}_I]$ denotes the ring of Laurent polynomials in the cluster variables.  If $\Sigma'$ is obtained from $\Sigma$ by mutation at $k \in I \setminus I_0$, there is a birational \emph{cluster transformation} $\mu_k: \A_\Sigma \to \A_{\Sigma'}$.  This is defined by the \emph{exchange relation}
\begin{align}\label{eq:Atrans}
\mu_k^*(A'_i) = \begin{cases}
A_i & i \neq k \\
A_k^{-1}\biggl(\prod_{b_{kj}>0}A_j^{b_{kj}} + \prod_{b_{kj}<0}A_j^{-b_{kj}}\biggr) & i = k.
\end{cases}
\end{align}
These transformations provide gluing data between any tori $\A_\Sigma$ and $\A_{\Sigma'}$ of mutation equivalent seeds $\Sigma$ and $\Sigma'$.  The $\A$-space $\A_{|\Sigma|}$ is defined as the scheme obtained from gluing together all such tori of seeds mutation equivalent with an initial seed $\Sigma$.  

\begin{definition}
Let $\Sigma$ be a seed.  The cluster algebra $\A(\Sigma)$ is the $\Z$-subalgebra of the function field of $\A_{|\Sigma|}$ generated by the collection of all cluster variables of seeds mutation equivalent to $\Sigma$.  The upper cluster algebra $\overline{\A}(\Sigma)$ is
\[
\overline{\A}(\Sigma) := \Z[\A_{|\Sigma|}] = \bigcap_{\Sigma' \sim \Sigma} \Z[\A_{\Sigma'}] \subset \mathbb{Q}(\A_{|\Sigma|}),
\]
or the intersection of all Laurent polynomial rings in the cluster variables of seeds mutation equivalent to $\Sigma$.  
\end{definition}

Given a seed $\Sigma$ we also associate a second algebraic torus $\X_\Sigma := \Spec \Z[X_I^{\pm 1}]$, where $\Z[X_I^{\pm 1}]$ again denotes the Laurent polynomial ring in the variables $\{X_i\}_{i \in I}$.  If $\Sigma'$ is obtained from $\Sigma$ by mutation at $k \in I \setminus I_0$, we again have a birational map $\mu_k: \X_\Sigma \to \X_{\Sigma'}$.  It is defined by
\begin{align}\label{eq:Xtrans}
\mu_k^*(X'_i) = \begin{cases}
X_iX_k^{[b_{ik}]_+}(1+X_k)^{-b_{ik}} & i \neq k  \\
X_k^{-1} & i = k,
\end{cases}
\end{align}
where $[b_{ik}]_+:=\mathrm{max}(0,b_{ik})$.  The $\X$-space $\X_{|\Sigma|}$ is defined as the scheme obtained from gluing together all such tori of seeds mutation equivalent with an initial seed $\Sigma$.

Since $B$ is skew-symmetrizable, there is a canonical Poisson structure on each $\X_\Sigma$ given by
\[
\{X_i,X_j\} = b_{ij}d_jX_iX_j.
\]
The cluster transformations of equation \ref{eq:Xtrans} intertwine the Poisson brackets on $\X_{\Sigma}$ and $\X_{\Sigma'}$, hence these assemble into a Poisson structure on $\X_{|\Sigma|}$.  

Now we will describe the natural map from $\A_{\Sigma}$ to $\X_{\Sigma}$. Let us assume that the entries of the $B$-matrix are all integers. Then we can define $p: \A_\Sigma \to \X_\Sigma$ by
$$p^*(X_i) = \prod_{j \in I}A_j^{B_{ij}}.$$

This formula appears to depend on the seed, but it actually intertwines the mutation of both the $\A$-coordinates and the $\X$-coordinates. In other words, if $\Sigma'$ is obtained from $\Sigma$ by mutation at $k$, there is a commutative diagram

\vspace{-2mm}
 \[
\begin{tikzcd}
\A_{\Sigma} \arrow[dashed]{r}{\mu_k} \arrow{d}{p} & \A_{\Sigma'} \arrow{d}{p'} \\
\X_{\Sigma} \arrow[dashed]{r}{\mu_k} & \X_{\Sigma'} \\
\end{tikzcd}
\]

If $B$ is not integral, we still have a well-defined map from $\A_{\Sigma}$ to $\X_{\Sigma}^*$, where $\X_{\Sigma}^*$ is obtained from $\X_{\Sigma}$ by projecting to the unfrozen variables.

\section{The cluster algebra structure on $\Conf_m \A$ for $G=Sp_{2n}$}

\subsection{Construction of the seed}

In this section we explain how to construct the seeds for the cluster structure on $\Conf_m \A$ when $G=Sp_{2n}$. Throughout this section, $G=Sp_{2n}$ unless otherwise noted.

Recall that $Sp_{2n}$ is associated to the type $C$ Dynkin diagram:

\begin{multicols}{2}
\begin{center}
\begin{tikzpicture}
    \draw (-1,0) node[anchor=east]  {$C_n$};

    \node[dnode,label=below:$1$] (1) at (0,0) {};
    \node[dnode,label=below:$2$] (2) at (1,0) {};
    \node[dnode,label=below:$n-2$] (3) at (3,0) {};
    \node[dnode,label=below:$n-1$] (4) at (4,0) {};
    \node[dnode,label=below:$n$] (5) at (5,0) {};

    \path (1) edge[sedge] (2)
          (2) edge[sedge,dashed] (3)
          (3) edge[sedge] (4)
          (5) edge[dedge] (4)
          ;

\draw[yshift=-2cm]
  node[below,text width=6cm] 
  {
  Figure 1 $C_n$ Dynkin diagram
  };

\end{tikzpicture}
\end{center}
\end{multicols}

Figure 1 Dynkin diagram

The nodes of the diagram correspond to $n-1$ short roots, numbered $1, 2, \dots, n-1$, and one long root, which is numbered $n$. To describe the cluster structure on $\Conf_3 \A$, we need to give the following data: the set $I$ parameterizing vertices, the functions on $\Conf_3 \A$ corresponding to each vertex, the $B$-matrix for this seed, and the multipliers $d_i$ for each vertex $i$.

The $B$-matrix is encoded via a quiver which  consists of $n^2+2n$ vertices, of which $n+2$ have $d_i=\frac{1}{2}$, while the remaining vertices have $d_i=1$. We color the vertices with $d_i=\frac{1}{2}$ white, while the remaining vertices are black. There are $n$ edge functions for each edge of the triangle, and $n^2-n$ face functions. There is one white vertex for each edge. The $B$-matrix is read off from the quiver by the following rules:

\begin{itemize}
\item An arrow from $i$ to $j$ means that $b_{ij}>0$.
\item $|b_{ij}|=2$ if $d_i=1$ and $d_j=\frac{1}{2}$.
\item $|b_{ij}|=1$ otherwise.
\end{itemize}

In Figure 2, we see the quiver for $Sp_6$. The generalization for other values of $n$ should be clear.

\begin{center}
\begin{tikzpicture}[scale=2]
\begin{scope}[xshift=-1cm]
  \foreach \x in {1, 2}
    \foreach \y in {0, 1, 2, 3}
      \node[] (x\y\x) at (\x,-\y) {\Large $\ontop{x_{\y\x}}{\bullet}$};
  \foreach \x in {3}
    \foreach \y in {0, 1, 2, 3}
      \node[] (x\y\x) at (\x,-\y) {\Large $\ontop{x_{\y\x}}{\circ}$};
  \node[] (y1) at (0,-2)  {\Large $\ontop{y_1}{\bullet}$};
  \node[] (y2) at (0,-1)  {\Large $\ontop{y_2}{\bullet}$};
  \node[] (y3) at (4,-1)  {\Large $\ontop{y_3}{\circ}$};
  \draw [->] (x31) to (x21);
  \draw [->] (x21) to (x11);
  \draw [->] (x11) to (x01);
  \draw [->] (x32) to (x22);
  \draw [->] (x22) to (x12);
  \draw [->] (x12) to (x02);
  \draw [->] (x33) to (x23);
  \draw [->] (x23) to (x13);
  \draw [->] (x13) to (x03);
  \draw [->, dashed] (y2) .. controls +(up:2) and +(up:2) .. (y3);
  \draw [->, dashed] (y1) to (y2);

  \draw [->, dashed] (x03) to (x02);
  \draw [->, dashed] (x02) to (x01);
  \draw [->] (y3) to (x13);
  \draw [->] (x13) to (x12);
  \draw [->] (x12) to (x11);
  \draw [->] (x11) to (y2);
  \draw [->] (x23) to (x22);
  \draw [->] (x22) to (x21);
  \draw [->] (x21) to (y1);
  \draw [->, dashed] (x33) to (x32);
  \draw [->, dashed] (x32) to (x31);

  \draw [->] (x01) to (x12);
  \draw [->] (x02) to (x13);
  \draw [->] (x03) to (y3);
  \draw [->] (y2) to (x21);
  \draw [->] (x11) to (x22);
  \draw [->] (x12) to (x23);
  \draw [->] (y1) to (x31);
  \draw [->] (x21) to (x32);
  \draw [->] (x22) to (x33);

\end{scope}

\draw[yshift=-3.5cm,xshift=1cm]
  node[below,text width=6cm] 
  {
  Figure 2. Quiver encoding the cluster structure for $\Conf_3 \A_{Sp_{6}}$
  };

\end{tikzpicture}
\end{center}

A dotted arrow means that $b_{ij}$ is half the value it would be if the arrow were solid. In other words, $|b_{ij}|=1$ if $d_i=1$ and $d_j=\frac{1}{2}$.
and $|b_{ij}|=\frac{1}{2}$ otherwise.

We will no longer use single letters like $i, j$ to denote vertices of the quiver, because it will be convenient for us to use the pairs $(i, j)$ to parameterize the vertices of the quiver. In the formulas in the remainder of this section, we will not refer to the particular entries of the $B$-matrix, $b_{ij}$. Instead, the values of the entries of the $B$-matrix will be encoded in quivers. This will hopefully avoid any notational confusion.

Label the vertices of the quiver $x_{ij}$ and $y_k$, where $0 \leq i \leq n+1$, $1 \leq j \leq n$, $1 \leq k \leq n$. The white vertices correspond to $x_{in}$ and $y_n$. The vertices $y_k$ and $x_{ij}$ for $i= 0$ or $n$ are frozen. We will sometimes write $x_{i,j}$ for $x_{ij}$ for orthographic reasons. Note that the dotted arrows only go between frozen vertices, thus the entries $b_{ij}$ of the $B$-matrix are integral unless $i$ and $j$ are both frozen, and thus the $B$-matrix defines a cluster algebra.

Our next goal will be to define the functions attached to the vertices in the quiver. We first need to recall some facts about the representation theory of $Sp_{2n}$. The fundamental representations of $Sp_{2n}$ are labelled by the fundamental weights $\omega_1, \dots, \omega_n$. $Sp_{2n}$ has a standard $2n$-dimensional representation $V$. Let $<-,->$ be the symplectic pairing. Then the representation $V_{\omega_i}$ corresponding to $\omega_i$ is a direct summand of $\bigwedge\nolimits^i V$. In fact, it is the kernel of the homomorphism $\bigwedge\nolimits^i V \rightarrow \bigwedge\nolimits^{i-2} V$ that comes from contracting with the symplectic form. Note that the symplectic form also induces an isomorphism between $\bigwedge\nolimits^i V$ and $\bigwedge\nolimits^{2n-i} V$.

Let $\A_{G}$ denote the principal affine space for $G$. We will sometimes drop the subscript ``$G$''if it is clear which group we are referring to. There is a natural embedding of $Sp_{2n} \hookrightarrow SL_{2n}$ given by taking the subgroup of $SL_{2n}$ that fixes the symplectic form $<v_i, v_{2n+1-i}>=(-1)^{i-1}$ (the signs are chosen here to be compatible with the positive structures on $Sp_{2n}$ and $SL_{2n}$). This induces a natural embedding of $\A_{Sp_{2n}}$ inside $\A_{SL_{2n}}$, which we now describe.

The variety $\A_{Sp_{2n}}$ parameterizes chains of isotropic vector spaces 
$$V_1 \subset V_2 \subset \cdots \subset V_n \subset V$$ inside the $2n$-dimensional standard representation $V$, where $\operatorname{dim} V_i= i$, and where each $V_i$ is equipped with a volume form.

Equivalently, a point of $\A_{Sp_{2n}}$ is given by a sequence of vectors 
$$v_1, v_2, \dots, v_n,$$
where $$V_i:=<v_1, \dots, v_i>$$ is isotropic, and where $v_i$ is only determined up to adding linear combinations of $v_j$ for $j < i$.

The volume form on $V_i$ is then $v_1 \wedge \cdots \wedge v_i$.

From the sequence of vectors $v_1, \dots, v_n$, we can complete to a symplectic basis $v_1, v_2, \dots v_{2n}$, where $<v_i, v_{2n+1-i}>=(-1)^{i-1}$, and $<v_i,v_j>=0$ otherwise. Equivalently, the symplectic form induces an isomorphism $<-,-> : V \rightarrow V^*$. At the same time, there are perfect pairings
$$\bigwedge\nolimits^k V \times \bigwedge\nolimits^k V^* \rightarrow F$$
$$\bigwedge\nolimits^{2n-k} V \times \bigwedge\nolimits^k V \rightarrow F$$
that induce an isomorphism
$$\bigwedge\nolimits^{2n-k} V \simeq \bigwedge\nolimits^k V^*.$$ Composing this with the inverse of the isomorphism 
$$<-,-> : \bigwedge\nolimits^k V \rightarrow \bigwedge\nolimits^k V^*$$
gives an isomorphism 
$$\bigwedge\nolimits^{2n-k} V \simeq \bigwedge\nolimits^k V^* \simeq \bigwedge\nolimits^k V.$$
Then $v_{n+1}, \dots, v_{2n}$ are chosen so that this isomorphism takes $v_1 \wedge \cdots v_{i}$ to $v_1 \wedge \cdots v_{2n-i}$.

Then $v_1, v_2, \dots v_{2n}$ determines a point of $\A_{SL_{2n}}$, as $\A_{SL_{2n}}$ parameterizes chains of vector subspaces  
$$V_1 \subset V_2 \subset \cdots \subset V_{2n} \subset V$$ along with volume forms $v_1 \wedge \cdots v_{i}$, $1 \leq i \leq 2n$.

From the embedding 
$$\A_{Sp_{2n}} \hookrightarrow \A_{SL_{2n}},$$
 one naturally gets an embedding $\Conf_m \A_{Sp_{2n}} \hookrightarrow \Conf_m \A_{SL_{2n}}$. We will define the cluster functions on $\Conf_m \A_{Sp_{2n}}$ via certain cluster functions on $\Conf_m \A_{SL_{2n}}$. In fact, we will see later in this paper that the entire cluster algebra structure on $\Conf_m \A_{Sp_{2n}}$ comes from {\em folding} the cluster algebra structure on $\Conf_m \A_{SL_{2n}}$.

It is well-known that the functions on $\A_G$ are naturally isomorphic to 
$$\bigoplus_{\lambda \in \Lambda_+} V_{\lambda}.$$
The functions on $\Conf_3 \A_{SL_{2n}}$ that we will use to define the cluster structure on $\Conf_3 \A_{Sp_{2n}}$ will be invariants of triple tensor products of representations of $SL_{2n}$, i.e., they will lie inside 
$$[V_{\lambda} \otimes V_{\mu} \otimes V_{\nu}]^G.$$
Here the factors $V_{\lambda}, V_{\mu}, V_{\nu}$ correspond to a graded subspace of the vector space of functions on the first, second, and third flags, respectively.

Let $1 \leq a, b, c, d < N$ be four integers satisfying $a > N-c > b$, $a > N-d > b$ and $a+b+c+d=2N$. Then there is a one-dimensional space of invariants inside the representation $$[V_{\omega_a+\omega_b}  \otimes V_{\omega_c} \otimes V_{\omega_d}]^{SL_N}.$$ We pick out the function given by the two equivalent webs in Figure 3.

\vspace{5mm}

\begin{center}
\begin{tikzpicture}[scale=0.4]
\begin{scope}[decoration={
    markings,
    mark=at position 0.5 with {\arrow{>}}},
    xshift=-8cm, yshift=-2cm
    ] 
\draw [postaction={decorate}] (-4,6) -- (-2,3) node [midway,below left] {$a$}; 
\draw [postaction={decorate}] (0,0) -- (-2,3) node [midway,below left] {$N-a$};
\draw [postaction={decorate}] (-1,8) -- (1,5) node [midway,above right] {$b$};
\draw [postaction={decorate}] (3,2) -- (1,5) node [midway,above right] {$N-b$};
\draw [postaction={decorate}] (-4,-6) -- (0,0) node [midway, above left] {$c$};
\draw [postaction={decorate}] (10,2) -- (3,2) node [midway, above] {$d$};
\draw [postaction={decorate}] (0,0) -- (3,2) node [midway, below right] {$a+c-N$};

\draw (-2.3,2.8) -- (-2,3) ;
\draw (0.7,4.8) -- (1,5) ;
\end{scope}

\begin{scope}[decoration={
    markings,
    mark=at position 0.5 with {\arrow{>}}},
    xshift=10cm, yshift=-2cm
    ] 
\draw [postaction={decorate}] (-4,6) -- (-2,3) node [midway,below left] {$b$}; 
\draw [postaction={decorate}] (0,0) -- (-2,3) node [midway,below left] {$N-b$};
\draw [postaction={decorate}] (-1,8) -- (1,5) node [midway,above right] {$a$};
\draw [postaction={decorate}] (3,2) -- (1,5) node [midway,above right] {$N-a$};
\draw [postaction={decorate}] (-4,-6) -- (0,0) node [midway, above left] {$c$};
\draw [postaction={decorate}] (10,2) -- (3,2) node [midway, above] {$d$};
\draw [postaction={decorate}] (0,0) -- (3,2) node [midway, below right] {$b+c-N$};

\draw (-2.3,2.8) -- (-2,3) ;
\draw (0.7,4.8) -- (1,5) ;
\end{scope}

\draw[yshift=-10cm,xshift=2cm]
  node[below,text width=6cm] 
  {
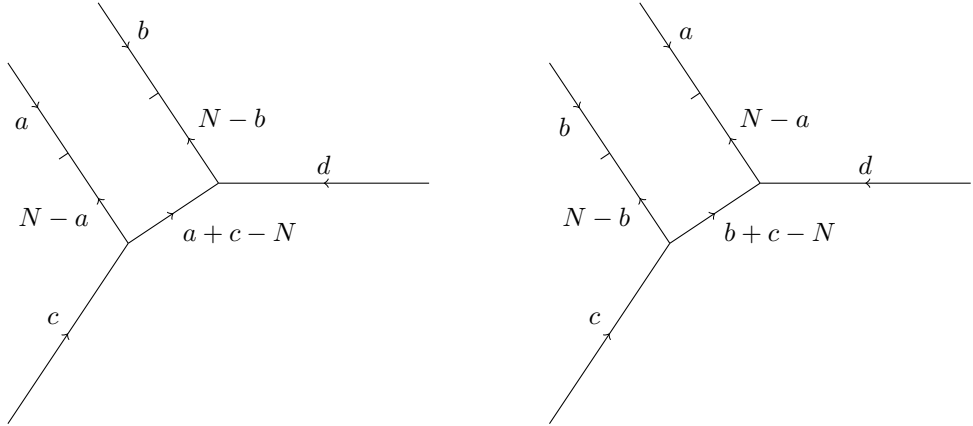
 Figure 3. Two equivalent webs for the function $\tcfr{a,b}{c}{d}$
  };

\end{tikzpicture}
\end{center}

Webs are a convenient tool for encoding tensor product invariants and for computing with them. See \cite{CKM} for their definition and for how one computes with them. The above diagram defines the invariant we are looking for up to a sign, which depends on conventions for webs that vary widely from author to author. For that reason, we give a more concrete description of the function.

Given three flags 
$$u_1, \dots, u_N;$$
$$v_1, \dots, v_N;$$
$$w_1, \dots, w_N;$$
first consider the forms
$$U_a := u_1 \wedge \cdots \wedge u_a,$$
$$U_b := u_1 \wedge \cdots \wedge u_b,$$
$$V_c := v_1 \wedge \cdots \wedge v_c,$$ 
$$W_d := w_1 \wedge \cdots \wedge w_d.$$
There is a natural map 
$$\phi_{a+c-N, N-a}: \bigwedge\nolimits^{c} V \rightarrow \bigwedge\nolimits^{a+c-N} V \otimes \bigwedge\nolimits^{N-a} V.$$
There are also natural maps 
$$U_b \wedge - \wedge W_d :  \bigwedge\nolimits^{a+c-N} V \rightarrow \bigwedge\nolimits^{N} V \simeq F$$ and 
$$U_a \wedge - : \bigwedge\nolimits^{N-a} V \rightarrow \bigwedge\nolimits^{N} V \simeq F.$$
Applying these maps to the first and second factors of $\phi_{a+c-N,N-a}(V_c)$, respectively, and then multiplying, we get get the value of our function. This is a function on $\Conf_3 \A_{SL_{N}}$. Let $N=2n$. They pulling back gives a function on $\Conf_3 \A_{Sp_{2n}}$. We will use the notation $\tcfr{a, b}{c}{d}$ to denote this function on either of those two spaces.

An equivalent way to calculate the function, associated to the second web above, is to use the natural map
$$\phi_{N-a, a+d-N}: \bigwedge\nolimits^{d} V \rightarrow \bigwedge\nolimits^{N-a} V \otimes \bigwedge\nolimits^{a+c-N} V.$$
There are natural maps 
$$U_b \wedge V_c \wedge - :  \bigwedge\nolimits^{a+d-N} V \rightarrow \bigwedge\nolimits^{N} V \simeq F$$ and 
$$U_a \wedge - : \bigwedge\nolimits^{N-a} V \rightarrow \bigwedge\nolimits^{N} V \simeq F.$$
Applying these maps to the second and first factors of $\phi_{N-a, a+c-N}$, respectively, and then multiplying, we get get the value of our function. We will use $\tcfr{b, a}{c}{d}$ to denote this function. Note that $\tcfr{b, a}{c}{d}=\tcfr{a, b}{c}{d}$.

Now let us define another, simpler, set of functions. Let $0 \leq a, b, c <N$ be three integers such that $a+b+c=N$. Then there is a one-dimensional space of invariants inside the representation $V_{\omega_a}  \otimes V_{\omega_b} \otimes V_{\omega_c}]^{SL_N}$. We pick out the function given by the web in Figure 4:

\begin{center}
\begin{tikzpicture}[scale=0.4]
\begin{scope}[decoration={
    markings,
    mark=at position 0.5 with {\arrow{>}}}
    ] 
\draw [postaction={decorate}] (-4,6) -- (-2,3) node [midway,below left] {$a$}; 
\draw [postaction={decorate}] (0,0) -- (-2,3) node [midway,below left] {$N-a$};
\draw [postaction={decorate}] (-4,-6) -- (0,0) node [midway,below right] {$b$};
\draw [postaction={decorate}] (8,0) -- (0,0) node [midway,above] {$c$};

\draw (-2.3,2.7) -- (-2,3) ;
\draw[yshift=-6.85cm]
  node[below,text width=6cm] 
  {
 Figure 4. Web for the function $\tcfr{a}{b}{c}$, where $a+b+c=N$.
  };
\end{scope}
\end{tikzpicture}
\end{center}

The function can be calculated as follows. Given three flags 
$$u_1, \dots, u_N;$$
$$v_1, \dots, v_N;$$
$$w_1, \dots, w_N;$$
first consider the forms
$$U_a := u_1 \wedge \cdots \wedge u_a,$$ 
$$V_b := v_1 \wedge \cdots \wedge v_b,$$
$$W_c := w_1 \wedge \cdots \wedge w_c.$$
Then $U_a \wedge V_n \wedge W_c$ is a multiple of $e_1 \wedge \dots \wedge e_N$, and this multiple is the value of our function. Call this function $\tcfr{a}{b}{c}$, whether it is viewed as a function on $\Conf_3 \A_{SL_{N}}$ or $\Conf_3 \A_{Sp_{2n}}$ when $N=2n$.

We are now ready to define the functions attached to the vertices. Here are the rules:

\begin{enumerate}
\item We assign the function $\dud{k}{2n-k}$ to $y_k$.
\item When $i \geq j$, we assign the function $\tcfr{n-i}{n+i-j}{j}$ to $x_{ij}$.
\item When $i < j$ and $i \neq 0$, we assign the function $\tcfr{n-i, 2n+i-j}{n}{j}$ to $x_{ij}$.
\item When $i=0$, we assign the function $\tcfr{2n-j}{}{j}$ to $x_{ij}$.
\end{enumerate}

This completely describes the cluster structure on $\Conf_3 \A_{Sp_{2n}}$. Note that the cluster structure is not symmetric with respect to the three flags. Performing various $S_3$ symmetries, we obtain six different possible cluster structures on  $\Conf_3 \A_{Sp_{2n}}$. We will later see that these six structures are related by sequences of mutations. In Figure 5, we depict two of the cluster structures for $\Conf_3 \A_{Sp_{8}}$ that are obtained from the original cluster structure by an $S_3$ symmetry.

\begin{center}
\begin{tikzpicture}[scale=2]
  \node (x01) at (-1,0) {$\tcfr{7}{}{1}$};
  \node (x02) at (0,0) {$\tcfr{6}{}{2}$};
  \node (x03) at (1,0) {$\tcfr{5}{}{3}$};
  \node (x04) at (2,0) {$\tcfr{4}{}{4}$};
  \node (x11) at (-1,-1) {$\tcfr{3}{4}{1}$};
  \node (x12) at (0,-1) {$\tcfr{3, 7}{4}{2}$};
  \node (x13) at (1,-1) {$\tcfr{3, 6}{4}{3}$};
  \node (x14) at (2,-1) {$\tcfr{3, 5}{4}{4}$};
  \node (x21) at (-1,-2) {$\tcfr{2}{5}{1}$};
  \node (x22) at (0,-2) {$\tcfr{2}{4}{2}$};
  \node (x23) at (1,-2) {$\tcfr{2, 7}{4}{3}$};
  \node (x24) at (2,-2) {$\tcfr{2, 6}{4}{4}$};
  \node (x31) at (-1,-3) {$\tcfr{1}{6}{1}$};
  \node (x32) at (0,-3) {$\tcfr{1}{5}{2}$};
  \node (x33) at (1,-3) {$\tcfr{1}{4}{3}$};
  \node (x34) at (2,-3) {$\tcfr{1, 7}{4}{4}$};
  \node (x41) at (-1,-4) {$\tcfr{}{7}{1}$};
  \node (x42) at (0,-4) {$\tcfr{}{6}{2}$};
  \node (x43) at (1,-4) {$\tcfr{}{5}{3}$};
  \node (x44) at (2,-4) {$\tcfr{}{4}{4}$};
  \node (y1) at (-2,-3) {$\dud{1}{7}$};
  \node (y2) at (-2,-2) {$\dud{2}{6}$};
  \node (y3) at (-2,-1) {$\dud{3}{5}$};
  \node (y4) at (3,-1) {$\dud{4}{4}$};

  \draw [->] (x41) to (x31);
  \draw [->] (x31) to (x21);
  \draw [->] (x21) to (x11);
  \draw [->] (x11) to (x01);
  \draw [->] (x42) to (x32);
  \draw [->] (x32) to (x22);
  \draw [->] (x22) to (x12);
  \draw [->] (x12) to (x02);
  \draw [->] (x43) to (x33);
  \draw [->] (x33) to (x23);
  \draw [->] (x23) to (x13);
  \draw [->] (x13) to (x03);
  \draw [->] (x44) to (x34);
  \draw [->] (x34) to (x24);
  \draw [->] (x24) to (x14);
  \draw [->] (x14) to (x04);
  \draw [->, dashed] (y3) .. controls +(up:2) and +(up:2) .. (y4);
  \draw [->, dashed] (y2) to (y3);
  \draw [->, dashed] (y1) to (y2);

  \draw [->, dashed] (x04) to (x03);
  \draw [->, dashed] (x03) to (x02);
  \draw [->, dashed] (x02) to (x01);
  \draw [->] (y4) to (x14);
  \draw [->] (x14) to (x13);
  \draw [->] (x13) to (x12);
  \draw [->] (x12) to (x11);
  \draw [->] (x11) to (y3);
  \draw [->] (x24) to (x23);
  \draw [->] (x23) to (x22);
  \draw [->] (x22) to (x21);
  \draw [->] (x21) to (y2);
  \draw [->] (x34) to (x33);
  \draw [->] (x33) to (x32);
  \draw [->] (x32) to (x31);
  \draw [->] (x31) to (y1);
  \draw [->, dashed] (x44) to (x43);
  \draw [->, dashed] (x43) to (x42);
  \draw [->, dashed] (x42) to (x41);

  \draw [->] (x01) to (x12);
  \draw [->] (x02) to (x13);
  \draw [->] (x03) to (x14);
  \draw [->] (x04) to (y4);
  \draw [->] (y3) to (x21);
  \draw [->] (x11) to (x22);
  \draw [->] (x12) to (x23);
  \draw [->] (x13) to (x24);
  \draw [->] (y2) to (x31);
  \draw [->] (x21) to (x32);
  \draw [->] (x22) to (x33);
  \draw [->] (x23) to (x34);
  \draw [->] (y1) to (x41);
  \draw [->] (x31) to (x42);
  \draw [->] (x32) to (x43);
  \draw [->] (x33) to (x44);

\begin{scope}[yshift=-5cm]
  \node (x01) at (-2,-1) {$\dud{7}{1}$};
  \node (x02) at (-2,-2) {$\dud{6}{2}$};
  \node (x03) at (-2,-3) {$\dud{5}{3}$};
  \node (x04) at (-2,-4) {$\dud{4}{4}$};
  \node (x11) at (-1,-1) {$\tcfr{3}{1}{4}$};
  \node (x12) at (-1,-2) {$\tcfr{3, 7}{2}{4}$};
  \node (x13) at (-1,-3) {$\tcfr{3, 6}{3}{4}$};
  \node (x14) at (-1,-4) {$\tcfr{3, 5}{4}{4}$};
  \node (x21) at (0,-1) {$\tcfr{2}{1}{5}$};
  \node (x22) at (0,-2) {$\tcfr{2}{2}{4}$};
  \node (x23) at (0,-3) {$\tcfr{2, 7}{3}{4}$};
  \node (x24) at (0,-4) {$\tcfr{2, 6}{4}{4}$};
  \node (x31) at (1,-1) {$\tcfr{1}{1}{6}$};
  \node (x32) at (1,-2) {$\tcfr{1}{2}{5}$};
  \node (x33) at (1,-3) {$\tcfr{1}{3}{4}$};
  \node (x34) at (1,-4) {$\tcfr{1, 7}{4}{4}$};
  \node (x41) at (2,-1) {$\tcfr{}{1}{7}$};
  \node (x42) at (2,-2) {$\tcfr{}{2}{6}$};
  \node (x43) at (2,-3) {$\tcfr{}{3}{5}$};
  \node (x44) at (2,-4) {$\tcfr{}{4}{4}$};
  \node (y1) at (1,0) {$\dud{7}{1}$};
  \node (y2) at (0,0) {$\dud{6}{2}$};
  \node (y3) at (-1,0) {$\dud{5}{3}$};
  \node (y4) at (-1,-5) {$\dud{4}{4}$};

  \draw [->] (x01) to (x11);
  \draw [->] (x11) to (x21);
  \draw [->] (x21) to (x31);
  \draw [->] (x31) to (x41);
  \draw [->] (x02) to (x12);
  \draw [->] (x12) to (x22);
  \draw [->] (x22) to (x32);
  \draw [->] (x32) to (x42);
  \draw [->] (x03) to (x13);
  \draw [->] (x13) to (x23);
  \draw [->] (x23) to (x33);
  \draw [->] (x33) to (x43);
  \draw [->] (x04) to (x14);
  \draw [->] (x14) to (x24);
  \draw [->] (x24) to (x34);
  \draw [->] (x34) to (x44);
  \draw [->, dashed] (y4) .. controls +(left:2) and +(left:2) .. (y3);
  \draw [->, dashed] (y3) to (y2);
  \draw [->, dashed] (y2) to (y1);

  \draw [->, dashed] (x01) to (x02);
  \draw [->, dashed] (x02) to (x03);
  \draw [->, dashed] (x03) to (x04);
  \draw [->] (y3) to (x11);
  \draw [->] (x11) to (x12);
  \draw [->] (x12) to (x13);
  \draw [->] (x13) to (x14);
  \draw [->] (x14) to (y4);
  \draw [->] (y2) to (x21);
  \draw [->] (x21) to (x22);
  \draw [->] (x22) to (x23);
  \draw [->] (x23) to (x24);
  \draw [->] (y1) to (x31);
  \draw [->] (x31) to (x32);
  \draw [->] (x32) to (x33);
  \draw [->] (x33) to (x34);
  \draw [->, dashed] (x41) to (x42);
  \draw [->, dashed] (x42) to (x43);
  \draw [->, dashed] (x43) to (x44);

  \draw [->] (x12) to (x01);
  \draw [->] (x13) to (x02);
  \draw [->] (x14) to (x03);
  \draw [->] (y4) to (x04);
  \draw [->] (x21) to (y3);
  \draw [->] (x22) to (x11);
  \draw [->] (x23) to (x12);
  \draw [->] (x24) to (x13);
  \draw [->] (x31) to (y2);
  \draw [->] (x32) to (x21);
  \draw [->] (x33) to (x22);
  \draw [->] (x34) to (x23);
  \draw [->] (x41) to (y1);
  \draw [->] (x42) to (x31);
  \draw [->] (x43) to (x32);
  \draw [->] (x44) to (x33);

\draw[yshift=-5.5cm]
  node[below,text width=6cm] 
  {
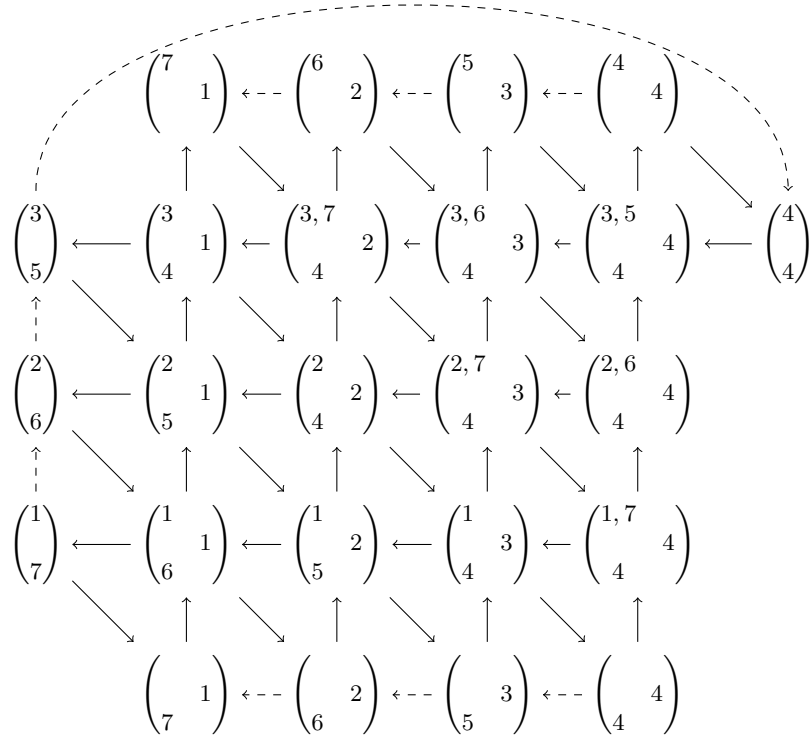
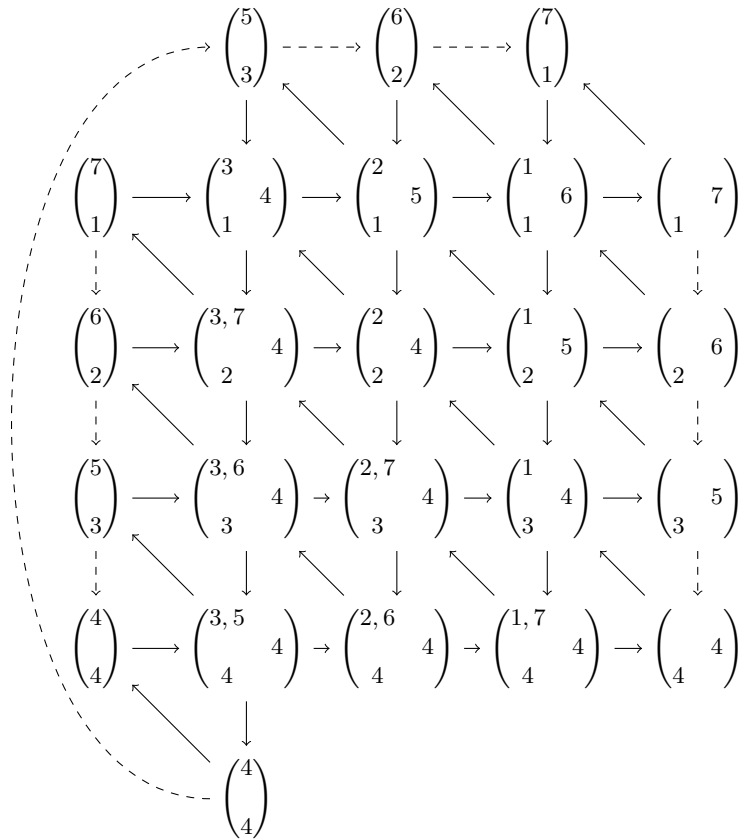
  Figure 5. Two cluster structures for $\Conf_3 \A_{Sp_{8}}$ related by $S_3$ symmetries.
  };

\end{scope}

\end{tikzpicture}
\end{center}

Let us describe in more detail how to obtain these other cluster structures. Let us first describe the quivers. If the $S_3$ symmetry is even, i.e., a rotation, we just rotate the quiver. If the $S_3$ symmetry is odd, i.e., a transposition, we transpose the quiver and also reverse the arrows.

The functions attached to the vertices come from permuting the arguments in our notation for the function. For example, rotating the function $\tcfr{n-i, 2n+i-j}{n}{j}$ gives the function $\tcfr{j}{n-i, 2n+i-j}{n}$, while transposing the first two arguments gives $\tcfr{n}{n-i, 2n+i-j}{j}$. A function of the form $\tcfr{a}{b, c}{d}$ is defined as follows. There is the {\emph twisted cyclic shift} map 
$$T: \Conf_3 \A  \rightarrow \Conf_3 \A$$
which acts on three principal flags $A, B, C$ as follows:
$$T: (A,B,C) \rightarrow (s_G \cdot C, A, B).$$
Then 
$$\tcfr{a}{b, c}{d} := (T^2)^*\tcfr{b, c}{d}{a}.$$
A function of the form ${a; b; c, d}$ is similarly defined using the cyclic shift map.

Another set of functions will be useful to us. Let $0 \leq a, b, c <N$ be three integers such that $a+b+c=2N$. Then there is a one-dimensional space of invariants inside the representation $V_{\omega_a}  \otimes V_{\omega_b} \otimes V_{\omega_c}]^{SL_N}$. We pick out the function given by the web in Figure 6.

\begin{center}
\begin{tikzpicture}[scale=0.4]
\begin{scope}[decoration={
    markings,
    mark=at position 0.5 with {\arrow{>}}}
    ] 
\draw [postaction={decorate}] (-4,6) -- (0,0) node [midway,below left] {$a$}; 
\draw [postaction={decorate}] (0,0) -- (-2,-3) node [midway,below right] {$N-b$};
\draw [postaction={decorate}] (-4,-6) -- (-2,-3) node [midway,below right] {$b$};
\draw [postaction={decorate}] (0,0) -- (4,0) node [midway,above] {$N-c$};
\draw [postaction={decorate}] (8,0) -- (4,0) node [midway,above] {$c$};

\draw (4.0,-0.4) -- (4,0) ;
\draw (-2.3, -2.8) -- (-2,-3) ;

\end{scope}

\draw[yshift=-6cm]
  node[below,text width=6cm] 
  {
  Figure 6. Web for the function $\tcfr{a}{b}{c}$ where $a+b+c=2N$.
  };

\end{tikzpicture}
\end{center}

The function can be calculated as follows. Given three flags $u_1, \dots, u_N$; $v_1, \dots, v_N$; and $w_1, \dots, w_N$; first consider the forms $U_a := u_1 \wedge \cdots \wedge u_a$, $V_b := v_1 \wedge \cdots \wedge v_b$, and $W_c := w_1 \wedge \cdots \wedge w_c$. There is a natural map 
$$\phi_{n-b, n-c}: \bigwedge\nolimits^{a} V \rightarrow \bigwedge\nolimits^{N-b} V \otimes \bigwedge\nolimits^{N-c} V.$$
We can take $\phi_{n-b, n-c}(U_a)$ take the wedge of the first factor with $V_b$ and the second factor with $W_c$ and then multiply to get our function. In both cases, we take the wedge on the left, i.e. 
$$V_b \wedge -: \bigwedge\nolimits^{N-b} V \rightarrow  \bigwedge\nolimits^{N} V \simeq F,$$
$$W_c \wedge -: \bigwedge\nolimits^{N-c} V \rightarrow  \bigwedge\nolimits^{N} V \simeq F.$$

Call this function $\tcfr{a}{b}{c}$, whether it is viewed as a function on $\Conf_3 \A_{SL_{N}}$ or $\Conf_3 \A_{Sp_{2n}}$ when $N=2n$. Note that we have the following equalities of functions:

\begin{equation} \label{dualities}
\begin{split}
\dud{k}{2n-k}&=\dud{2n-k}{k} \\
\tcfr{n-i}{n+i-j}{j}&=\tcfr{n+i}{n-i+j}{2n-j} \\
\tcfr{n-i, 2n+i-j}{n}{j}&=\tcfr{n+i, j-i}{n}{2n- j}
\end{split}
\end{equation}
These equalities arise because the symplectic form induces an isomorphism between $\bigwedge\nolimits^i V$ and $\bigwedge\nolimits^{2n-i} V$.

The cluster structure for $\Conf_m \A$ comes from triangulating an $m$-gon and then attaching the cluster structure on $\Conf_3 \A$ to each triangle. Let us make this more precise. On $\Conf_3 \A_{Sp_{2n}}$, consider the vertices labelled $y_k$, $x_{0j}$ and $x_{nj}$. Attached to these vertices are the functions are $\dud{k}{2n-k}$, $\tcfr{2n-j}{}{j}$, and $\tcfr{}{2n-j}{j}$, respectively. These are exactly the vertices that were frozen. Moreover, all the functions attached to these vertices only depend on two of the three flags. We will call these the {\em edge functions}. Let us call all other functions {\em face functions}.

To form the quiver for $\Conf_m \A_{Sp_{2n}}$, we first take a triangulation of an $m$-gon. On each of the $m-2$ triangles, attach any one of the six quivers formed from performing $S_3$ symmetries on the quiver for $\Conf_3 \A_{Sp_{2n}}$ described above. Each edge of each of these triangles has $n$ frozen vertices. Let us describe how to glue two triangles together.

Let $T_1, T_2$ be two triangles with edges $e_1, e_2, e_3$ and $e_4, e_5, e_6$, respectively. Suppose that we would like to glue the edges $e_1$ and $e_4$. $e_1$ and $e_4$ each have $n$ frozen vertices. We will glue these $2n$ vertices together in pairs to form $n$ vertices. Each frozen vertex is glued to another vertex that shares the same function (using the fact that $\dud{k}{2n-k}=\dud{2n-k}{k}$, etc.). These vertices then become unfrozen. If vertices $i$ and $j$ are glued with $i'$ and $j'$ to get new vertices $i''$ and $j''$, then we declare that $$b_{i''j''}=b_{ij}+b_{i'j'}.$$
In other words, two dotted arrows in the same direction glue to give us a solid arrow, whereas two dotted arrows in the opposite direction cancel to give us no arrow. One can easily check that any gluing will result in no dotted arrows using the unfrozen vertices. The arrows involving vertices that were not previously frozen remain the same. Figure 7 below shows one gluing between two triangles for $Sp_6$.

\begin{center}
\begin{tikzpicture}[scale=2.2]
  \node (x-31) at (-3,0) {$\dul{1}{5}$};
  \node (x-21) at (-2,0) {$\tcfl{1}{2}{3}$};
  \node (x-11) at (-1,0) {$\tcfl{1}{1}{4}$};
  \node (x01) at (0,0) {$\dud{1}{5}$};
  \node (x11) at (1,0) {$\tcfr{2}{5}{5}$};
  \node (x21) at (2,0) {$\tcfr{3}{5}{4}$};
  \node (x31) at (3,0) {$\ddr{5}{1}$};
  \node (x-32) at (-3,-1) {$\dul{2}{4}$};
  \node (x-22) at (-2,-1) {$\tcfl{2}{2, 5}{3}$};
  \node (x-12) at (-1,-1) {$\tcfl{2}{1}{3}$};
  \node (x02) at (0,-1) {$\dud{2}{4}$};
  \node (x12) at (1,-1) {$\tcfr{3}{4}{5}$};
  \node (x22) at (2,-1) {$\tcfr{3}{4}{1, 4}$};
  \node (x32) at (3,-1) {$\ddr{4}{2}$};
  \node (x-33) at (-3,-2) {$\dul{3}{3}$};
  \node (x-23) at (-2,-2) {$\tcfl{3}{2, 4}{3}$};
  \node (x-13) at (-1,-2) {$\tcfl{3}{1, 5}{3}$};
  \node (x03) at (0,-2) {$\dud{3}{3}$};
  \node (x13) at (1,-2) {$\tcfr{3}{3}{1, 5}$};
  \node (x23) at (2,-2) {$\tcfr{3}{3}{2, 4}$};
  \node (x33) at (3,-2) {$\ddr{3}{3}$};
  \node (y-1) at (-0.5,1) {$\dld{1}{5}$};
  \node (y-2) at (-1.5,1) {$\dld{2}{4}$};
  \node (y-3) at (-2.5,-3) {$\dld{3}{3}$};
  \node (y1) at (0.5,1) {$\tcfr{1}{}{5}$};
  \node (y2) at (1.5,1) {$\tcfr{2}{}{4}$};
  \node (y3) at (2.5,-3) {$\tcfr{3}{}{3}$};

  \draw [->] (x01) to (x11);
  \draw [->] (x11) to (x21);
  \draw [->] (x21) to (x31);
  \draw [->] (x02) to (x12);
  \draw [->] (x12) to (x22);
  \draw [->] (x22) to (x32);
  \draw [->] (x03) to (x13);
  \draw [->] (x13) to (x23);
  \draw [->] (x23) to (x33);
  \draw [->, dashed] (y1) to (y2);
  \draw [->, dashed] (y2) .. controls +(right:1) and +(up:1) .. (y3);

  \draw [->] (x03) to (x02);
  \draw [->] (x02) to (x01);

  \draw [->] (x13) to (x12);
  \draw [->] (x12) to (x11);
  \draw [->] (x11) to (y1);
  \draw [->] (y3) to (x23);
  \draw [->] (x23) to (x22);
  \draw [->] (x22) to (x21);
  \draw [->] (x21) to (y2);
  \draw [->, dashed] (x33) to (x32);
  \draw [->, dashed] (x32) to (x31);

  \draw [->] (y1) to (x01);
  \draw [->] (x11) to (x02);
  \draw [->] (x12) to (x03);
 \draw [->] (y2) to (x11);
  \draw [->] (x21) to (x12);
  \draw [->] (x22) to (x13);
  \draw [->] (x31) to (x22);
  \draw [->] (x32) to (x23);
  \draw [->] (x33) to (y3);

  \draw [->] (x01) to (x-11);
  \draw [->] (x-11) to (x-21);
  \draw [->] (x-21) to (x-31);
  \draw [->] (x02) to (x-12);
  \draw [->] (x-12) to (x-22);
  \draw [->] (x-22) to (x-32);
  \draw [->] (x03) to (x-13);
  \draw [->] (x-13) to (x-23);
  \draw [->] (x-23) to (x-33);
  \draw [->, dashed] (y-1) to (y-2);
  \draw [->, dashed] (y-2) .. controls +(left:1) and +(up:1) .. (y-3);

  \draw [->] (x-13) to (x-12);
  \draw [->] (x-12) to (x-11);
  \draw [->] (x-11) to (y-1);
  \draw [->] (y-3) to (x-23);
  \draw [->] (x-23) to (x-22);
  \draw [->] (x-22) to (x-21);
  \draw [->] (x-21) to (y-2);
  \draw [->, dashed] (x-33) to (x-32);
  \draw [->, dashed] (x-32) to (x-31);

  \draw [->] (y-1) to (x01);
  \draw [->] (x-11) to (x02);
  \draw [->] (x-12) to (x03);
  \draw [->] (y-2) to (x-11);
  \draw [->] (x-21) to (x-12);
  \draw [->] (x-22) to (x-13);
  \draw [->] (x-31) to (x-22);
  \draw [->] (x-32) to (x-23);
  \draw [->] (x-33) to (y-3);

\draw[yshift=-3.85cm]
  node[below,text width=6cm] 
  {
  Figure 7. The functions and quiver for the cluster algebra on $\Conf_4 \A_{Sp_{6}}$.
  };

\end{tikzpicture}
\end{center}

Repeat this procedure for each interior edge of the triangulation, and one arrives at the quiver for $\Conf_m \A_{Sp_{2n}}$. The procedure for gluing triangles is very reminiscent of the ``amalgamation'' procedure in \cite{FG4}.

\subsection{Reduced words}

The goal of this section is to relate the cluster structure on $\Conf_3 \A_{Sp_{2n}}$ given in the previous section to Berenstein, Fomin and Zelevinsky's cluster structure on $B$, the Borel in the group $G$ (\cite{BFZ}). This will allow us to see that the cluster structure described above induces a positive structure on $\A_{G,S}$, and moreover that this positive structure on $\A_{G,S}$ is identical to the one given in \cite{FG1}. In fact, it was the comparison between the cluster structures in \cite{BFZ} and \cite{FG1} that was our original motivation for the construction of the previous section.

Moreover, this comparison gives heuristics for how to construct the cluster structure on $\A_{G,S}$ when $G$ is an exceptional group. It will also clarify some computations that we do later in the paper. However, for a reader not interested in the positive structure on $\A_{G,S}$, and more interested in just understanding the cluster structure on $\A_{G,S}$, this section is not logically necessary.

Let us also mention that our cluster structure can also be related to the cluster structure on Bruhat cells in \cite{FG4}. The paper \cite{W} contains the clearest and most definitive statements on the relationship between \cite{FG4} and \cite{BFZ}.

Recall from equation \eqref{ConfA} that the positive structure on $\Conf_3 \A_{Sp_{2n}}$ is given by a map $H \times B^-  \longrightarrow {\Conf}_{3}({\A})$
that sends
$$
(h, b) \rightarrow (U^-, h \cdot \overline{w_0}U^-, b\cdot \overline{w_0}U^-).
$$


Then the natural positive structure on $H \times B$ induces a positive structure on $\Conf_3 \A_{Sp_{2n}}$. This positive structure coincides with the one given in \eqref{ConfA} \cite{FG1}. Let us now restrict our attention to triples of principal flags of the form $(U^-, \overline{w_0}U^-, b\cdot \overline{w_0}U^-).$ We can consider the map

$$i: b \in B^- \rightarrow (U^-, \overline{w_0}U^-, b\cdot \overline{w_0}U^-) \in \Conf_3 \A_{Sp_{2n}}.$$

Now let us recall the constructions of \cite{BFZ}. For $u, v$ elements of the Weyl group $W$ of $G$, we have the double Bruhat cell
$$G^{u,v}=B^+ \cdot u \cdot B^+ \cap B^- \cdot v \cdot B^v.$$
The cell $G^{w_0,e}$ is the on open part of $B^-$.

\begin{prop} The cluster algebra constructed above on $\Conf_3 \A_{Sp_{2n}}$, when restricted to the image of $i$, coincides with the cluster algebra structure given in \cite{BFZ} on $B^-=G^{w_0,e}.$
\end{prop}

\begin{proof} $G^{w_0,e}$ is the on open part of $B^-$. Following \cite{BFZ}, to get a cluster structure on this subset, we must choose a reduced-word for $w_0$. In the numbering of the nodes of the Dynkin diagram given above for $Sp_{2n}$, we choose the reduced word expression $$w_0=(s_n s_{n-1} \cdots s_2 s_1)^n.$$
Here our convention is that the above word corresponds to the string $i_1, i_2, \dots, i_{n-1}, i_n$ repeated $n$ times. 

Now let $G_0 = U^- H U^+ \subset G$ be the open subset of elements of $G$ having Gaussian decomposition $x=[x]_- [x]_0 [x]_+$. Then for any two elements $u, v \in W$, and any fundamental weight $\omega_i$, we can define the {\em generalized minor} $\Delta_{u\omega_i, v\omega_i}(x)$. It is a rational function on $G$ which is given generically by the formula

$$\Delta_{u\omega_i, v\omega_i}(x) := ([\overline{u}^{-1}x \overline{v}]_0)^{\omega_i}.$$

In our situation, we are interested in such minors when $u, v=e$, or when $v=e$ and $u=u_{ij}=(s_n s_{n-1} \cdots s_2s_1)^{i-1} s_{n}s_{n-1}\cdots s_j$ for $1 \leq i \leq n$ and $n \geq j \geq 1$.

Then the cluster functions on $B^-$ given in \cite{BFZ} are $\Delta_{\omega_i, \omega_i}$ for $1 \leq i \leq n$ (these are the functions associated to $u, v=e$), and 
$$\Delta_{u_{ij}\omega_j, \omega_j},$$
which are the functions associated to $v=e$ and $u=u_{ij}=(s_n s_{n-1} \cdots s_2s_1)^i s_{n}s_{n-1}\cdots s_j$. Note that $u_{ij}$ is the subword of $u$ that stops on the $i^{\textrm th}$ iteration of $s_j$.

We have the following claims:
\begin{enumerate}
\item Recall that when $i=0$, we assign the function $\tcfr{2n-j}{}{j}$ to $x_{ij}$. Then 
$$\tcfr{2n-j}{}{j} = \Delta_{\omega_j, \omega_j}$$ 
In other words, the function assigned to $x_{ij}$ is precisely $\Delta_{\omega_j, \omega_j}$.

\item Recall that for $i \geq j$, we assign the function $\tcfr{n-i}{n+i-j}{j}$ to $x_{ij}$. Then 
$$\tcfr{n-i}{n+i-j}{j} = \Delta_{u_{ij}\omega_j, \omega_j}$$ 
In other words, the function assigned to $x_{ij}$ is precisely $\Delta_{u_{ij}\omega_j, \omega_j}$.

\item Recall that when $i < j$, we assign the function $\tcfr{n-i, 2n+i-j}{n}{j}$ to $x_{ij}$. Then 
$$\tcfr{n-i, 2n+i-j}{n}{j} = \Delta_{u_{ij}\omega_j, \omega_j}$$ 
In other words, the function assigned to $x_{ij}$ is precisely $\Delta_{u_{ij}\omega_j, \omega_j}$.
\end{enumerate}

The proof of these claims is a straightforward calculation. Let us carry out this calculation. The calculation is not central to this paper, so can be safely skipped.

It is convenient to choose an embedding of $Sp_{2n}$ into $SL_{2n}$. Moreover, choose the symplectic form so that $<e_i,e_{n+i}>=(-1)^i$ and all other pairings of basis elements are zero. The sign is chosen so that the representation $Sp_{2n} \hookrightarrow SL_{2n}$ preserves positive structure. Now choose a pinning such that under the embedding $Sp_{2n} \hookrightarrow SL_{2n}$,
$$E_{\alpha_n}=E_{n,2n}, F_{\alpha_n}=E_{n,2n},$$
and for $1\leq i<n$,
$$E_{\alpha_i}=E_{i,i+1}+E_{n+i+1,n+i},  F_{\alpha_i}=E_{i+1,i}+E_{n+i,n+i+1}.$$
Here $E_{i,j}$ is the $(i,j)$-elementary matrix, i.e., the matrix with a $1$ in the $(i,j)$ position and $0$ in all other positions.

With respect to this embedding, we can directly calculate $\Delta_{\omega_j, \omega_j}(x)$ where $x \in B^-$. When $x$ is embedded in $SL_{2n}$, $\Delta_{\omega_j, \omega_j}(x)$ is simply the determinant of the minor consisting of the first $j$ rows and the first $j$ columns. Similarly, one can calculate that $\Delta_{u_{ij}\omega_j, \omega_j}(x)$ is precisely the determinant of the minor of $x$ consisting of rows $2n+1-i, 2n+2-i, \dots, 2n+j-i$ (here the indices are taken modulo $2n$) and the first $j$ columns.

We then must calculate the functions $\tcfr{2n-j}{}{j}, \tcfr{n-i}{n+i-j}{j}$ and $\tcfr{n-i, 2n+i-j}{n}{j}$ on the triple of flags $(U^-, \overline{w_0}U^-, b\cdot \overline{w_0}U^-)$. Under the embedding  $Sp_{2n} \hookrightarrow SL_{2n}$, we should choose the flag $U^-$ to be be $-e_{n+1}, e_{n+2}, -e_{n+3}, dots, (-1)^n e_{2n}$, so that $\overline{w_0}U^-$ is given by the flag $e_1, e_2, \dots, e_n$. Direct calculation then shows that $\tcfr{2n-j}{}{j}, \tcfr{n-i}{n+i-j}{j}$ and $\tcfr{n-i, 2n+i-j}{n}{j}$ are given by determinants of the minors consisting of the first $j$ rows and the first $j$ columns and determinants of the minors consisting of rows $2n+1-i, 2n+2-i, \dots, 2n+j-i$ (the indices taken modulo $2n$) and the first $j$ columns.

\end{proof}

\subsection{Cactus transformation}

In this section, first we collect the necessary facts that we need about the ``tetrahedron recurrence'', which can be interpreted as a sequence of mutations on the cluster algebra for $\Conf_3 \A_{SL_{N}}$. Most of what is in this section can be found in \cite{Hen} or \cite{HK}, though our notation is somewhat different. We will relate this to our construction of the cluster structure on $\Conf_3 \A_{Sp_{2n}}$. This will be useful for computations that we do later in the paper.

The tetrahedron recurrence is really just a variation on the octahedron recurrence. It is a sequence of mutations on the cluster algebra for $\Conf_3 \A_{SL_{N}}$ that realizes the operation of replacing three principal flags by in the space $V=\C^n$ with the dual principal flags in $V^*$. Put in another way, there is an outer automorphism of $SL_N$ given by $g \rightarrow ^t g^{-1}$. This automorphism induces an automorphism $\phi$ on $\Conf_3 \A_{SL_{N}}$. It turns out (though this is not obvious), that this is an automorphism of the cluster structure on $\Conf_3 \A_{SL_{N}}$, i.e., if $x_i$ are cluster variables in one seed, then $\phi^*x_i$ will be cluster variables in another seed. The sequence of mutations we are interested in is a sequence of mutations transforming from the initial seed $x_i$ to the seed consisting of the functions $\phi^*x_i$.

The name ``tetrahedron recurrence''comes from the fact that all the cluster variables involved can be put at the integral lattice points of a tetrahedron (\cite{HK}). By performing this sequence of mutations on various triangles in the triangulation of an $m$-gon, we get the action of the {\em cactus group} on $\Conf_3 \A_{SL_{N}}$. For this reason we will call the sequence of mutations realizing the outer automorphism of $SL_N$ the ``cactus seqence.''

First we need to review the cluster algebra structure on $\Conf_3 \A_{SL_{N}}$. The quiver for this cluster algera has a vertex $x_{ijk}$ for all triples $1 \leq i, j, k \leq n$ such that $i+j+k=n$ and no more than one of $i, j, k$ is equal to $0$. The vertex $x_{ijk}$ can be placed at the point $(i,j,k)$ in the plane $i+j+k=n$. Then the quiver looks as in Figure 8 for $SL_4$.

\begin{center}
\begin{tikzpicture}[scale=2]
  \node (x310) at (-2,0) {\Large $\ontop{x_{310}}{\bullet}$};
  \node (x220) at (-2,-1) {\Large $\ontop{x_{220}}{\bullet}$};
  \node (x130) at (-2,-2) {\Large $\ontop{x_{130}}{\bullet}$};

  \node (x301) at (-1,0.5) {\Large $\ontop{x_{301}}{\bullet}$};
  \node (x211) at (-1,-0.5) {\Large $\ontop{x_{211}}{\bullet}$};
  \node (x121) at (-1,-1.5) {\Large $\ontop{x_{121}}{\bullet}$};
  \node (x031) at (-1,-2.5) {\Large $\ontop{x_{031}}{\bullet}$};

  \node (x202) at (0,0) {\Large $\ontop{x_{202}}{\bullet}$};
  \node (x112) at (0,-1) {\Large $\ontop{x_{112}}{\bullet}$};
  \node (x022) at (0,-2) {\Large $\ontop{x_{022}}{\bullet}$};

  \node (x103) at (1,-0.5) {\Large $\ontop{x_{103}}{\bullet}$};
  \node (x013) at (1,-1.5) {\Large $\ontop{x_{013}}{\bullet}$};
 
  \draw [->] (x301) to (x310);
  \draw [->] (x202) to (x211);
  \draw [->] (x211) to (x220);
  \draw [->] (x103) to (x112);
  \draw [->] (x112) to (x121);
  \draw [->] (x121) to (x130);
  \draw [->, dashed] (x013) to (x022);
  \draw [->, dashed] (x022) to (x031);

  \draw [->] (x013) to (x103);
  \draw [->] (x022) to (x112);
  \draw [->] (x112) to (x202);
  \draw [->] (x031) to (x121);
  \draw [->] (x121) to (x211);
  \draw [->] (x211) to (x301);
  \draw [->, dashed] (x130) to (x220);
  \draw [->, dashed] (x220) to (x310);

  \draw [->] (x130) to (x031);
  \draw [->] (x220) to (x121);
  \draw [->] (x121) to (x022);
  \draw [->] (x310) to (x211);
  \draw [->] (x211) to (x112);
  \draw [->] (x112) to (x013);
  \draw [->, dashed] (x301) to (x202);
  \draw [->, dashed] (x202) to (x103);

\draw[yshift=-3cm]
  node[below,text width=6cm] 
  {
  Figure 8. The functions and quiver for the cluster algebra on $\Conf_3 \A_{SL_{4}}$.
  };

\end{tikzpicture}
\end{center}

We have used dotted arrows in line with the conventions above for amalgamation. The function attached to $x_{ijk}$ will then be $\tcfr{i}{j}{k}$ in the notation of section 3.1. The frozen vertices are $x_{i,N-i,0}, x_{0,j,N-j}, x_{i,0,N-i}$.

Now we describe the sequence of mutations. First mutate $x_{N-2,1,1}$. Then mutate $x_{N-3,2,1}$ and $x_{N-3,1,2}$. Then mutate $x_{N-4,3,1}, x_{N-4,2,2}, x_{N-4,1,3}$. Continue in this manner, until we mutate $x_{1,1,N-2}$. Then we start the sequence over again by mutating $x_{N-2,1,1}$, but the second time through, we stop at $x_{2,1,N-3}$. The third time through, we stop at $x_{3,1,N-4}$.
We continue in this manner so that the whole sequence of mutations is
$$x_{N-2,1,1}, x_{N-3,2,1}, x_{N-3,1,2}, x_{N-4,3,1}, x_{N-4,2,2}, x_{N-4,1,3}, \dots x_{1,N-2,1}, \dots x_{1,1,N-2},$$
$$x_{N-2,1,1}, x_{N-3,2,1}, x_{N-3,1,2}, x_{N-4,3,1}, x_{N-4,2,2}, x_{N-4,1,3}, \dots x_{2,N-3,1}, \dots x_{2,1,N-3},$$
$$x_{N-2,1,1}, x_{N-3,2,1}, x_{N-3,1,2}, x_{N-4,3,1}, x_{N-4,2,2}, x_{N-4,1,3}, \dots x_{3,1,N-4},$$
$$\dots$$
$$x_{N-2,1,1},  x_{N-3,2,1}, x_{N-3,1,2},$$
$$x_{N-2,1,1}.$$

Perhaps it is more useful to describe the sequence in another way. Think of the quiver for $\Conf_3 \A_{SL_{N}}$ as consisting of rows, where row $r$ consists of all $x_{ijk}$ where $i=r$. Then for $r=N-2, N-3, \dots 2, 1$, there are $N-1-r$ vertices in that row. We start by mutating the one vertex in row $N-2$, then the two vertices in row $N-3$, then the three vertices in row $N-4$, etc. The sequence of rows that we mutate is

$$N-2, N-3, N-4, \dots, 2, 1, N-2, N-3, \dots, 3, 2, N-2, \dots, 3, \dots, \dots, N-2, N-3, N-2.$$
There are $(N-2)(N-1)/2$ terms in the above list. This gives a total of $N(N-1)(N_2)/6$ mutations. Let us think of the sequence of mutations as happening in $N-2$ stages, where at stage $r$ we mutate all the vertices in rows $N-2, N-3, \dots, r$, in that order. It turns out that because of how the quiver transforms, the mutations in any given row can be performed in any order.

There is actually even more freedom in the order in which we perform mutations. Suppose at some point in the sequence of mutations, we want to mutate row $i$ for the $j^{\textrm th}$ time. It turns out that any vertex in row $i$ can be mutated only after

\begin{itemize}
\item the two vertices in row $i-1$ directly below it have been mutated $j-1$ times and
\item the two vertices in row $i+1$ directly above it have been mutated $j$ times.
\end{itemize}

The combinatorics of this will become clear after we analyze the how the quiver transforms under mutations. Thus we have many equivalent sequences of mutations.

For example, we could mutate the rows
$$N-2, N-3, N-2, N-4, N-3, N-2, N-5, \dots, N-2, N-6, \dots, \dots, 1, 2, 3, \dots, N-2.$$
getting the mutation sequence
$$x_{N-2,1,1}, x_{N-3,2,1}, x_{N-3,1,2}, x_{N-2,1,1}, x_{N-4,3,1}, x_{N-4,2,2}, x_{N-4,1,3}, x_{N-3,2,1}, x_{N-3,1,2}, x_{N-2,1,1}$$
$$\dots, x_{1,N-2,1}, \dots, x_{1,1,N-2}, x_{2,N-3,1}, \dots, x_{2,1,N-3}, x_{3,N-4,1}, \dots, \dots, x_{N-2,1,1}.$$

Let us now analyze how the quiver transforms. In Figure 9 we picture the quiver for $SL_5$, the result after mutating row $3$, the result after mutating rows $3, 2$, the result after and the result after mutating rows $3, 2, 1$. The pattern should be clear for $SL_N$.

\begin{center}
\begin{tikzpicture}[scale=2]
  \node (x410) at (-2,0) {$\dud{4}{1}$};
  \node (x320) at (-2,-1) {$\dud{3}{2}$};
  \node (x230) at (-2,-2) {$\dud{2}{3}$};
  \node (x140) at (-2,-3) {$\dud{1}{4}$};

  \node (x401) at (-1,0.5) {$\tcfr{4}{}{1}$};
  \node (x311) at (-1,-0.5) {$\tcfr{3}{1}{1}$};
  \node (x221) at (-1,-1.5) {$\tcfr{2}{2}{1}$};
  \node (x131) at (-1,-2.5) {$\tcfr{1}{3}{1}$};
  \node (x041) at (-1,-3.5) {$\tcfr{}{4}{1}$};

  \node (x302) at (0,0) {$\tcfr{3}{}{2}$};
  \node (x212) at (0,-1) {$\tcfr{2}{1}{2}$};
  \node (x122) at (0,-2) {$\tcfr{1}{2}{2}$};
  \node (x032) at (0,-3) {$\tcfr{}{3}{2}$};

  \node (x203) at (1,-0.5) {$\tcfr{2}{}{3}$};
  \node (x113) at (1,-1.5) {$\tcfr{1}{1}{3}$};
  \node (x023) at (1,-2.5) {$\tcfr{}{2}{3}$};

  \node (x104) at (2,-1) {$\tcfr{1}{}{4}$};
  \node (x014) at (2,-2) {$\tcfr{}{1}{4}$};

  \draw [->] (x014) to (x104);
  \draw [->] (x023) to (x113);
  \draw [->] (x113) to (x203);
  \draw [->] (x032) to (x122);
  \draw [->] (x122) to (x212);
  \draw [->] (x212) to (x302);
  \draw [->] (x041) to (x131);
  \draw [->] (x131) to (x221);
  \draw [->] (x221) to (x311);
  \draw [->] (x311) to (x401);
  \draw [->, dashed] (x140) to (x230);
  \draw [->, dashed] (x230) to (x320);
  \draw [->, dashed] (x320) to (x410);

  \draw [->] (x401) to (x410);
  \draw [->] (x302) to (x311);
  \draw [->] (x311) to (x320);
  \draw [->] (x203) to (x212);
  \draw [->] (x212) to (x221);
  \draw [->] (x221) to (x230);
  \draw [->] (x104) to (x113);
  \draw [->] (x113) to (x122);
  \draw [->] (x122) to (x131);
  \draw [->] (x131) to (x140);
  \draw [->, dashed] (x014) to (x023);
  \draw [->, dashed] (x023) to (x032);
  \draw [->, dashed] (x032) to (x041);

  \draw [->] (x140) to (x041);
  \draw [->] (x230) to (x131);
  \draw [->] (x131) to (x032);
  \draw [->] (x320) to (x221);
  \draw [->] (x221) to (x122);
  \draw [->] (x122) to (x023);
  \draw [->] (x410) to (x311);
  \draw [->] (x311) to (x212);
  \draw [->] (x212) to (x113);
  \draw [->] (x113) to (x014);
  \draw [->, dashed] (x401) to (x302);
  \draw [->, dashed] (x302) to (x203);
  \draw [->, dashed] (x203) to (x104);

\begin{scope}[yshift=-5cm]
  \node (x410) at (-2,0) {$\dud{3}{2}$};
  \node (x320) at (-2,-1) {$\tcfr{4}{}{1}$};
  \node (x230) at (-2,-2) {$\dud{2}{3}$};
  \node (x140) at (-2,-3) {$\dud{1}{4}$};

  \node (x401) at (-1,0.5) {$\tcfr{3}{}{2}$};
  \node (x311) at (-1,-0.5) {$\boldsymbol{\tcfr{2,4}{2}{2}}$};
  \node (x221) at (-1,-1.5) {$\tcfr{2}{2}{1}$};
  \node (x131) at (-1,-2.5) {$\tcfr{1}{3}{1}$};
  \node (x041) at (-1,-3.5) {$\tcfr{}{4}{1}$};

  \node (x302) at (0,0) {$\dud{4}{1}$};
  \node (x212) at (0,-1) {$\tcfr{2}{1}{2}$};
  \node (x122) at (0,-2) {$\tcfr{1}{2}{2}$};
  \node (x032) at (0,-3) {$\tcfr{}{3}{2}$};

  \node (x203) at (1,-0.5) {$\tcfr{2}{}{3}$};
  \node (x113) at (1,-1.5) {$\tcfr{1}{1}{3}$};
  \node (x023) at (1,-2.5) {$\tcfr{}{2}{3}$};

  \node (x104) at (2,-1) {$\tcfr{1}{}{4}$};
  \node (x014) at (2,-2) {$\tcfr{}{1}{4}$};

  \draw [->] (x014) to (x104);
  \draw [->] (x023) to (x113);
  \draw [->] (x113) to (x203);
  \draw [->] (x032) to (x122);
  \draw [->] (x122) to (x212);
  \draw [->] (x302) to (x212);
  \draw [->] (x041) to (x131);
  \draw [->] (x131) to (x221);
  \draw [->] (x311) to (x221);
  \draw [->] (x311) to (x401);
  \draw [->, dashed] (x140) to (x230);

  \draw [->] (x401) to (x410);
  \draw [->] (x320) to (x311);
  \draw [->] (x311) to (x302);
  \draw [->] (x203) to (x212);
  \draw [->] (x221) to (x230);
  \draw [->] (x104) to (x113);
  \draw [->] (x113) to (x122);
  \draw [->] (x122) to (x131);
  \draw [->] (x131) to (x140);
  \draw [->, dashed] (x014) to (x023);
  \draw [->, dashed] (x023) to (x032);
  \draw [->, dashed] (x032) to (x041);

  \draw [->] (x140) to (x041);
  \draw [->] (x230) to (x131);
  \draw [->] (x131) to (x032);
  \draw [->] (x221) to (x320);
  \draw [->] (x221) to (x122);
  \draw [->] (x122) to (x023);
  \draw [->] (x410) to (x311);
  \draw [->] (x212) to (x311);
  \draw [->] (x212) to (x113);
  \draw [->] (x113) to (x014);
  \draw [->, dashed] (x203) to (x104);

  \draw [->, dashed] (x230) .. controls +(180:1.5) and +(180:1.5) .. (x410);
  \draw [->, dashed] (x401) .. controls +(150:1.5) and +(150:1.5) .. (x320);
  \draw [->, dashed] (x302) .. controls +(up:1.5) and +(up:1.5) .. (x410);
  \draw [->, dashed] (x203) .. controls +(60:1.5) and +(60:1.5) .. (x401);

\draw[yshift=-4cm]
  node[below,text width=6cm] 
  {
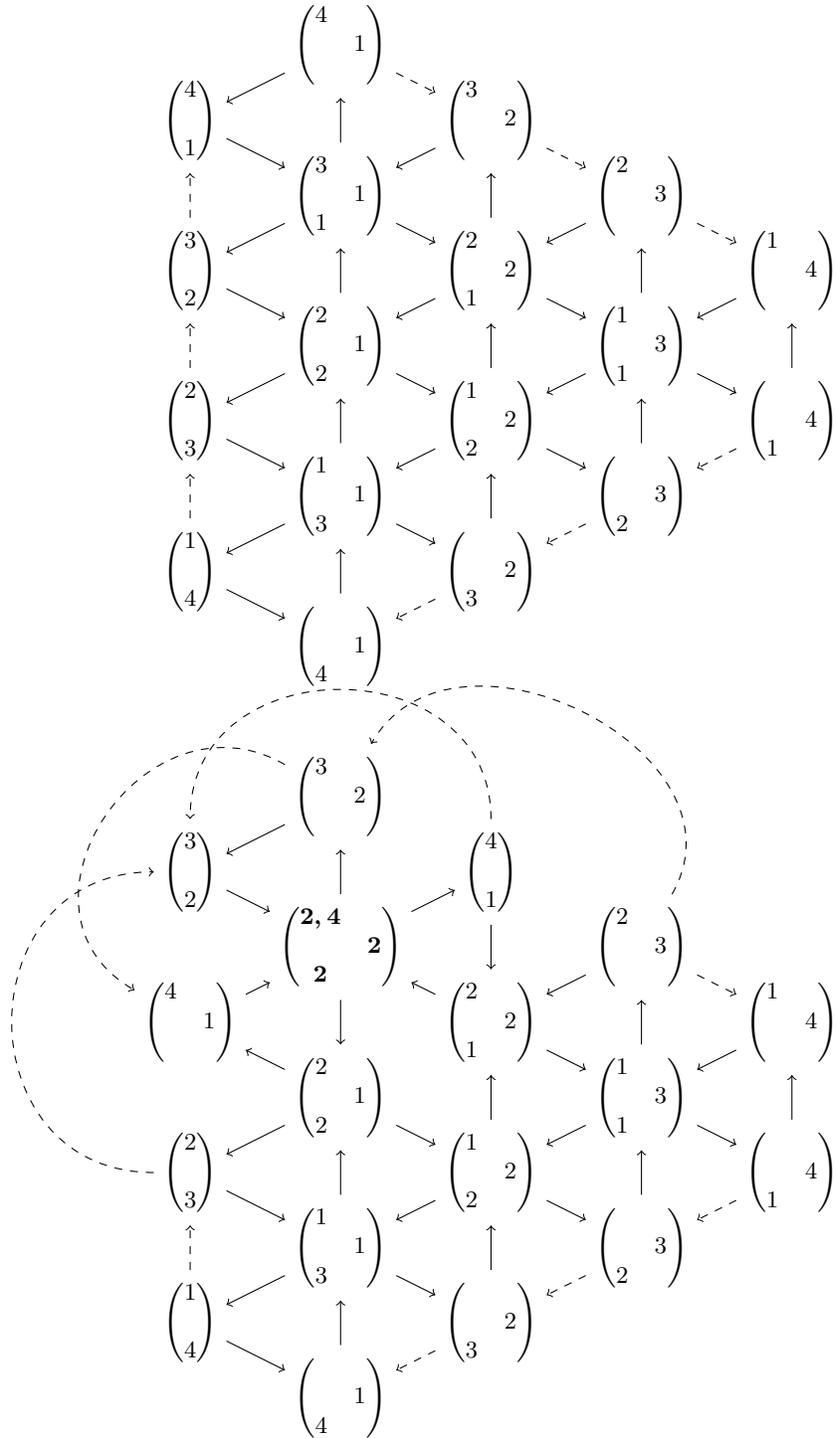
  Figure 9a. The $SL_5$ quiver before mutation and after mutating row $3$.
  };

\end{scope}

\end{tikzpicture}
\end{center}

\begin{center}
\begin{tikzpicture}[scale=2]
  \node (x410) at (-2,0) {$\dud{3}{2}$};
  \node (x320) at (-2,-1) {$\dud{2}{3}$};
  \node (x230) at (-2,-2) {$\tcfr{4}{}{1}$};
  \node (x140) at (-2,-3) {$\dud{1}{4}$};

  \node (x401) at (-1,0.5) {$\tcfr{3}{}{2}$};
  \node (x311) at (-1,-0.5) {$\tcfr{2,4}{2}{2}$};
  \node (x221) at (-1,-1.5) {$\boldsymbol{\tcfr{1,4}{3}{2}}$};
  \node (x131) at (-1,-2.5) {$\tcfr{1}{3}{1}$};
  \node (x041) at (-1,-3.5) {$\tcfr{}{4}{1}$};

  \node (x302) at (0,0) {$\tcfr{2}{}{3}$};
  \node (x212) at (0,-1) {$\boldsymbol{\tcfr{1,4}{2}{3}}$};
  \node (x122) at (0,-2) {$\tcfr{1}{2}{2}$};
  \node (x032) at (0,-3) {$\tcfr{}{3}{2}$};

  \node (x203) at (1,-0.5) {$\dud{4}{1}$};
  \node (x113) at (1,-1.5) {$\tcfr{1}{1}{3}$};
  \node (x023) at (1,-2.5) {$\tcfr{}{2}{3}$};

  \node (x104) at (2,-1) {$\tcfr{1}{}{4}$};
  \node (x014) at (2,-2) {$\tcfr{}{1}{4}$};

  \draw [->] (x014) to (x104);
  \draw [->] (x023) to (x113);
  \draw [->] (x203) to (x113);
  \draw [->] (x032) to (x122);
  \draw [->] (x212) to (x122);
  \draw [->] (x212) to (x302);
  \draw [->] (x041) to (x131);
  \draw [->] (x221) to (x131);
  \draw [->] (x221) to (x311);
  \draw [->] (x311) to (x401);
  \draw [->, dashed] (x320) to (x410);

  \draw [->] (x401) to (x410);
  \draw [->] (x302) to (x311);
  \draw [->] (x311) to (x320);
  \draw [->] (x212) to (x203);
  \draw [->] (x230) to (x221);
  \draw [->] (x104) to (x113);
  \draw [->] (x131) to (x140);
  \draw [->, dashed] (x014) to (x023);
  \draw [->, dashed] (x023) to (x032);
  \draw [->, dashed] (x032) to (x041);

  \draw [->] (x140) to (x041);
  \draw [->] (x131) to (x230);
  \draw [->] (x131) to (x032);
  \draw [->] (x320) to (x221);
  \draw [->] (x122) to (x221);
  \draw [->] (x122) to (x023);
  \draw [->] (x410) to (x311);
  \draw [->] (x311) to (x212);
  \draw [->] (x113) to (x212);
  \draw [->] (x113) to (x014);
  \draw [->, dashed] (x401) to (x302);

  \draw [->, dashed] (x140) .. controls +(180:1.5) and +(180:1.5) .. (x320);
  \draw [->, dashed] (x401) .. controls +(160:1.7) and +(160:1.7) .. (x230);
  \draw [->, dashed] (x203) .. controls +(80:1.7) and +(80:1.7) .. (x410);
  \draw [->, dashed] (x302) .. controls +(60:1.5) and +(60:1.5) .. (x104);

\begin{scope}[yshift=-5cm]
  \node (x410) at (-2,0) {$\dud{3}{2}$};
  \node (x320) at (-2,-1) {$\dud{2}{3}$};
  \node (x230) at (-2,-2) {$\dud{1}{4}$};
  \node (x140) at (-2,-3) {$\tcfr{4}{}{1}$};

  \node (x401) at (-1,0.5) {$\tcfr{3}{}{2}$};
  \node (x311) at (-1,-0.5) {$\tcfr{2,4}{2}{2}$};
  \node (x221) at (-1,-1.5) {$\tcfr{1,4}{3}{2}$};
  \node (x131) at (-1,-2.5) {$\boldsymbol{\tcfr{4}{4}{2}}$};
  \node (x041) at (-1,-3.5) {$\tcfr{}{4}{1}$};

  \node (x302) at (0,0) {$\tcfr{2}{}{3}$};
  \node (x212) at (0,-1) {$\tcfr{1,4}{2}{3}$};
  \node (x122) at (0,-2) {$\boldsymbol{\tcfr{4}{3}{3}}$};
  \node (x032) at (0,-3) {$\tcfr{}{3}{2}$};

  \node (x203) at (1,-0.5) {$\tcfr{1}{}{4}$};
  \node (x113) at (1,-1.5) {$\boldsymbol{\tcfr{4}{2}{4}}$};
  \node (x023) at (1,-2.5) {$\tcfr{}{2}{3}$};

  \node (x104) at (2,-1) {$\dud{4}{1}$};
  \node (x014) at (2,-2) {$\tcfr{}{1}{4}$};

  \draw [->] (x104) to (x014);
  \draw [->] (x113) to (x023);
  \draw [->] (x113) to (x203);
  \draw [->] (x122) to (x032);
  \draw [->] (x122) to (x212);
  \draw [->] (x212) to (x302);
  \draw [->] (x131) to (x041);
  \draw [->] (x131) to (x221);
  \draw [->] (x221) to (x311);
  \draw [->] (x311) to (x401);
  \draw [->, dashed] (x230) to (x320);
  \draw [->, dashed] (x320) to (x410);

  \draw [->] (x401) to (x410);
  \draw [->] (x302) to (x311);
  \draw [->] (x311) to (x320);
  \draw [->] (x203) to (x212);
  \draw [->] (x212) to (x221);
  \draw [->] (x221) to (x230);
  \draw [->] (x113) to (x104);
  \draw [->] (x140) to (x131);
  \draw [->, dashed] (x023) to (x014);
  \draw [->, dashed] (x032) to (x023);
  \draw [->, dashed] (x041) to (x032);

  \draw [->] (x041) to (x140);
  \draw [->] (x230) to (x131);
  \draw [->] (x032) to (x131);
  \draw [->] (x320) to (x221);
  \draw [->] (x221) to (x122);
  \draw [->] (x023) to (x122);
  \draw [->] (x410) to (x311);
  \draw [->] (x311) to (x212);
  \draw [->] (x212) to (x113);
  \draw [->] (x014) to (x113);
  \draw [->, dashed] (x401) to (x302);
  \draw [->, dashed] (x302) to (x203);

  \draw [->, dashed] (x401) .. controls +(170:1.9) and +(170:1.9) .. (x140);
  \draw [->, dashed] (x104) .. controls +(70:1.9) and +(70:1.9) .. (x410);

\draw[yshift=-4cm]
  node[below,text width=6cm] 
  {
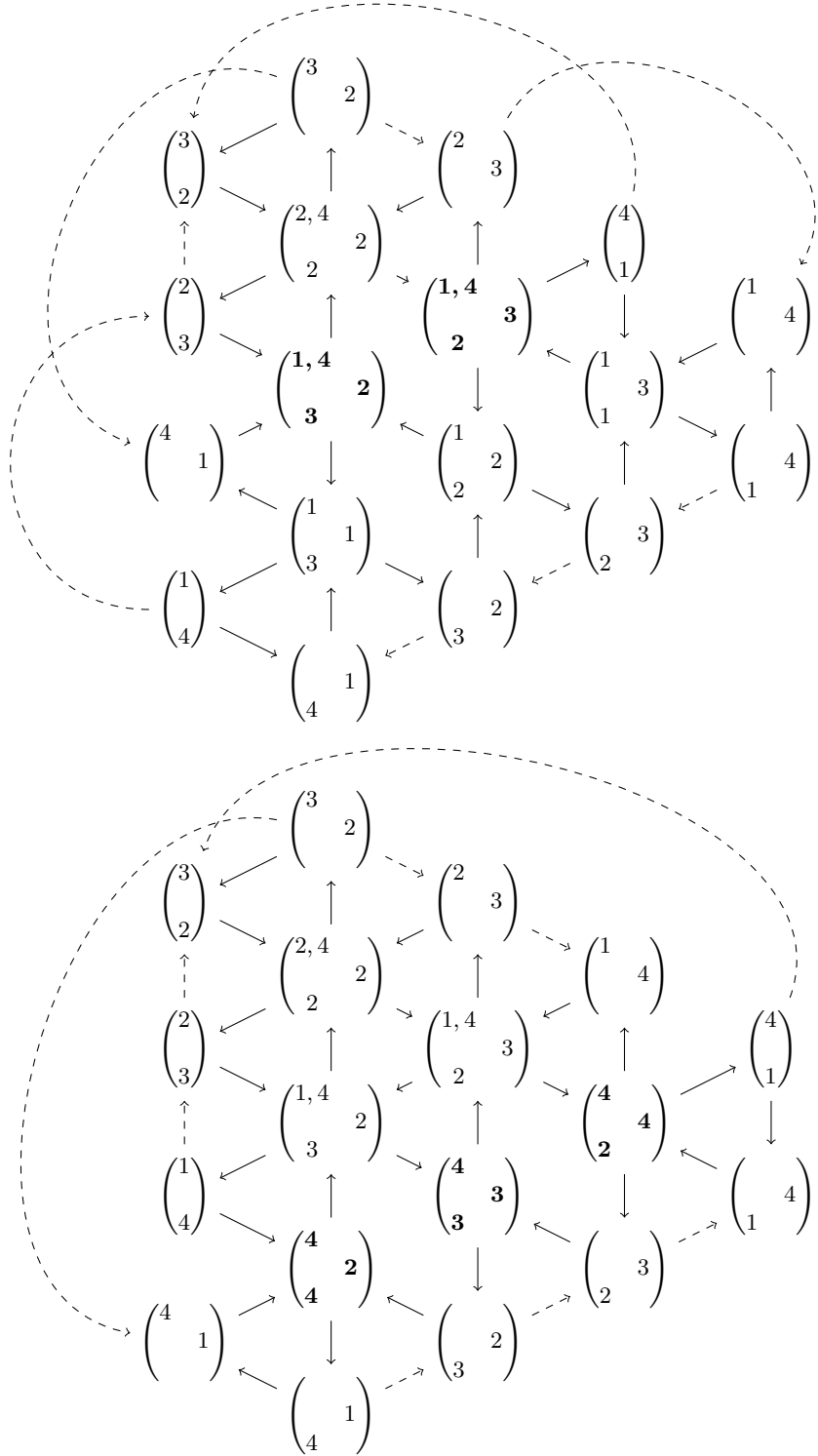
  Figure 9b. The $SL_5$ quiver after mutating rows $3, 2$, and after mutating rows $3, 2, 1$. This gives the first stage of the cactus mutation sequence.
  };

\end{scope}

\end{tikzpicture}
\end{center}

Note that at each stage, we rearrange some of the frozen vertices to make the structure of the quiver more transparent. The non-frozen vertices, however, do not move.

We see that after the mutations in rows $N-2, N-3, \dots 1$ have been performed, we end up with the quiver for $SL_{N-1}$ in the top $N-1$ rows of the diagram. We may then inductively see that after each descending sequence of rows has been mutated, we get the quivers pictured in Figure 10

\begin{center}
\begin{tikzpicture}[scale=2]
  \node (x410) at (-2,0) {$\dud{2}{3}$};
  \node (x320) at (-2,-1) {$\dud{1}{4}$};
  \node (x230) at (-2,-2) {$\tcfr{3}{}{2}$};
  \node (x140) at (-2,-3) {$\tcfr{4}{}{1}$};

  \node (x401) at (-1,0.5) {$\tcfr{2}{}{3}$};
  \node (x311) at (-1,-0.5) {$\boldsymbol{\tcfr{1,3}{3}{3}}$};
  \node (x221) at (-1,-1.5) {$\boldsymbol{\tcfr{3}{4}{3}}$};
  \node (x131) at (-1,-2.5) {$\tcfr{4}{4}{2}$};
  \node (x041) at (-1,-3.5) {$\tcfr{}{4}{1}$};

  \node (x302) at (0,0) {$\tcfr{1}{}{4}$};
  \node (x212) at (0,-1) {$\boldsymbol{\tcfr{3}{3}{4}}$};
  \node (x122) at (0,-2) {$\tcfr{4}{3}{3}$};
  \node (x032) at (0,-3) {$\tcfr{}{3}{2}$};

  \node (x203) at (1,-0.5) {$\dud{3}{2}$};
  \node (x113) at (1,-1.5) {$\tcfr{4}{2}{4}$};
  \node (x023) at (1,-2.5) {$\tcfr{}{2}{3}$};

  \node (x104) at (2,-1) {$\dud{4}{1}$};
  \node (x014) at (2,-2) {$\tcfr{}{1}{4}$};

  \draw [->] (x104) to (x014);
  \draw [->] (x113) to (x023);
  \draw [->] (x203) to (x113);
  \draw [->] (x122) to (x032);
  \draw [->] (x212) to (x122);
  \draw [->] (x212) to (x302);
  \draw [->] (x131) to (x041);
  \draw [->] (x221) to (x131);
  \draw [->] (x221) to (x311);
  \draw [->] (x311) to (x401);
  \draw [->, dashed] (x230) to (x140);
  \draw [->, dashed] (x320) to (x410);

  \draw [->] (x401) to (x410);
  \draw [->] (x302) to (x311);
  \draw [->] (x311) to (x320);
  \draw [->] (x212) to (x203);
  \draw [->] (x230) to (x221);
  \draw [->] (x113) to (x104);
  \draw [->] (x122) to (x113);
  \draw [->] (x131) to (x122);
  \draw [->] (x140) to (x131);
  \draw [->, dashed] (x023) to (x014);
  \draw [->, dashed] (x032) to (x023);
  \draw [->, dashed] (x041) to (x032);

  \draw [->] (x041) to (x140);
  \draw [->] (x131) to (x230);
  \draw [->] (x032) to (x131);
  \draw [->] (x320) to (x221);
  \draw [->] (x122) to (x221);
  \draw [->] (x023) to (x122);
  \draw [->] (x410) to (x311);
  \draw [->] (x311) to (x212);
  \draw [->] (x113) to (x212);
  \draw [->] (x014) to (x113);
  \draw [->, dashed] (x401) to (x302);
  \draw [->, dashed] (x104) to (x203);

  \draw [->, dashed] (x401) .. controls +(160:1.9) and +(160:1.9) .. (x230);
  \draw [->, dashed] (x203) .. controls +(80:1.9) and +(80:1.9) .. (x410);

\begin{scope}[yshift=-5cm]
  \node (x410) at (-2,0) {$\tcfr{1}{}{4}$};
  \node (x320) at (-2,-1) {$\tcfr{2}{}{3}$};
  \node (x230) at (-2,-2) {$\tcfr{3}{}{2}$};
  \node (x140) at (-2,-3) {$\tcfr{4}{}{1}$};

  \node (x401) at (-1,0.5) {$\dud{1}{4}$};
  \node (x311) at (-1,-0.5) {$\boldsymbol{\tcfr{2}{4}{4}}$};
  \node (x221) at (-1,-1.5) {$\tcfr{3}{4}{3}$};
  \node (x131) at (-1,-2.5) {$\tcfr{4}{4}{2}$};
  \node (x041) at (-1,-3.5) {$\tcfr{}{4}{1}$};

  \node (x302) at (0,0) {$\dud{2}{3}$};
  \node (x212) at (0,-1) {$\tcfr{3}{3}{4}$};
  \node (x122) at (0,-2) {$\tcfr{4}{3}{3}$};
  \node (x032) at (0,-3) {$\tcfr{}{3}{2}$};

  \node (x203) at (1,-0.5) {$\dud{3}{2}$};
  \node (x113) at (1,-1.5) {$\tcfr{4}{2}{4}$};
  \node (x023) at (1,-2.5) {$\tcfr{}{2}{3}$};

  \node (x104) at (2,-1) {$\dud{4}{1}$};
  \node (x014) at (2,-2) {$\tcfr{}{1}{4}$};

  \draw [->] (x104) to (x014);
  \draw [->] (x113) to (x023);
  \draw [->] (x203) to (x113);
  \draw [->] (x122) to (x032);
  \draw [->] (x212) to (x122);
  \draw [->] (x302) to (x212);
  \draw [->] (x131) to (x041);
  \draw [->] (x221) to (x131);
  \draw [->] (x311) to (x221);
  \draw [->] (x401) to (x311);
  \draw [->, dashed] (x230) to (x140);
  \draw [->, dashed] (x320) to (x230);
  \draw [->, dashed] (x410) to (x320);

  \draw [->] (x410) to (x401);
  \draw [->] (x311) to (x302);
  \draw [->] (x320) to (x311);
  \draw [->] (x212) to (x203);
  \draw [->] (x221) to (x212);
  \draw [->] (x230) to (x221);
  \draw [->] (x113) to (x104);
  \draw [->] (x122) to (x113);
  \draw [->] (x131) to (x122);
  \draw [->] (x140) to (x131);
  \draw [->, dashed] (x023) to (x014);
  \draw [->, dashed] (x032) to (x023);
  \draw [->, dashed] (x041) to (x032);

  \draw [->] (x041) to (x140);
  \draw [->] (x131) to (x230);
  \draw [->] (x032) to (x131);
  \draw [->] (x221) to (x320);
  \draw [->] (x122) to (x221);
  \draw [->] (x023) to (x122);
  \draw [->] (x311) to (x410);
  \draw [->] (x212) to (x311);
  \draw [->] (x113) to (x212);
  \draw [->] (x014) to (x113);
  \draw [->, dashed] (x302) to (x401);
  \draw [->, dashed] (x203) to (x302);
  \draw [->, dashed] (x104) to (x203);

\draw[yshift=-4cm]
  node[below,text width=6cm] 
  {
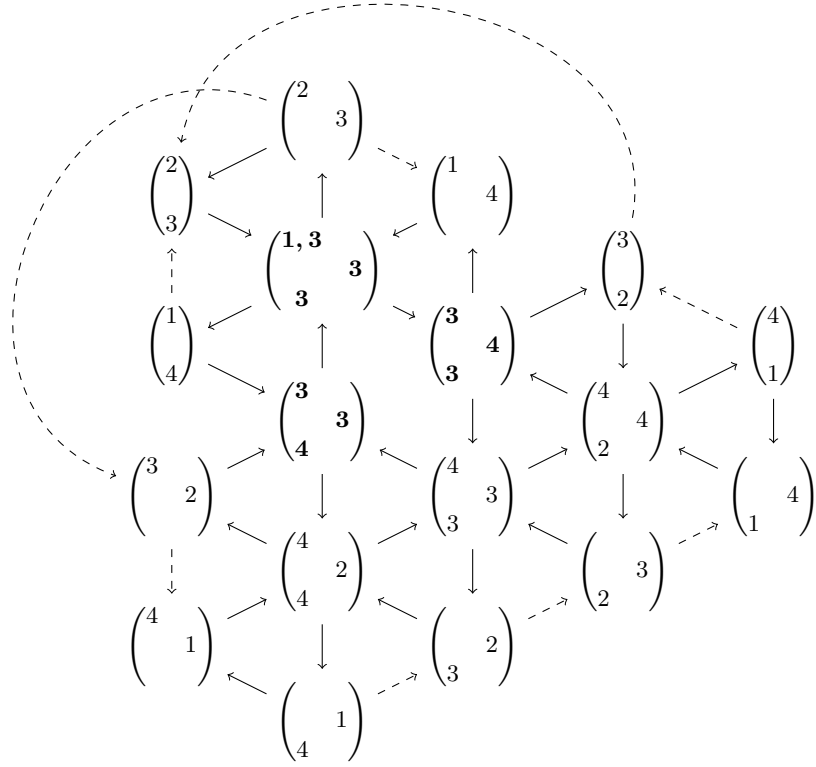
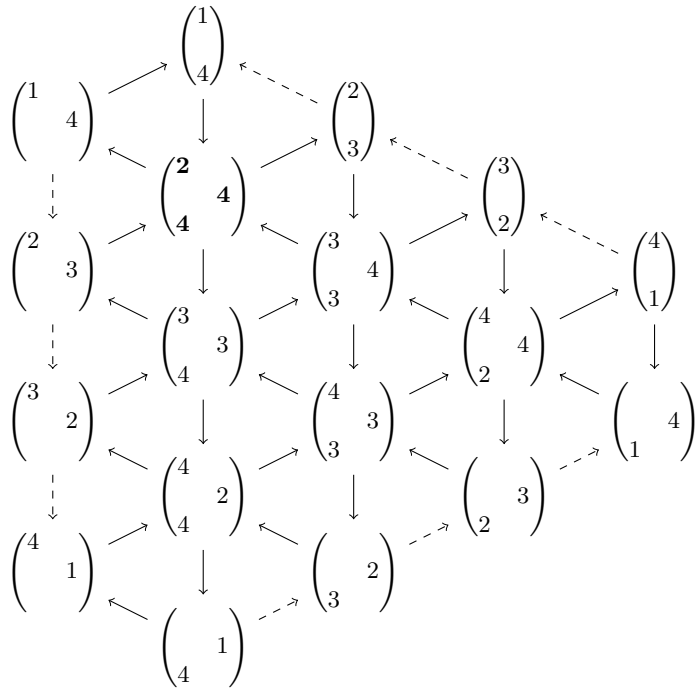
  Figure 10 The $SL_5$ quiver after stages $2$ and $3$ of the cactus sequence
  };

\end{scope}

\end{tikzpicture}
\end{center}

The final quiver we arrive at is essentially the original quiver with arrows reversed (additionally, frozen vertices have moved and some arrows have changed between frozen vertices).

Now we must understand the functions attached to the vertices at the quiver at various stages. Note that $x_{ijk}$ is mutated $i$ times.

\begin{prop} The function associated to $x_{ijk}$ tranforms as  follows:

$$\tcfr{i}{j}{k} \rightarrow \tcfr{n-1, i-1}{j+1}{k+1} \rightarrow \tcfr{n-2, i-2}{j+2}{k+2} \rightarrow \cdots $$
$$\rightarrow \tcfr{n-i+1, 1}{j+i-1}{k+i-1} \rightarrow \tcfr{n-i}{j+i}{k+i}=\tcfr{n-i}{n-k}{n-j}.$$
If we take the convention that $\tcfr{2n, a}{b}{c}=\tcfr{a}{b}{c}=\tcfr{a, 0}{b}{c}$, the pattern above is somewhat clearer.
\end{prop}

\begin{proof}
We have already described the quivers at the various stages of mutation. We must then check that the functions above satisfy the identities of the associated cluster transformations.

We need the following facts:
\begin{itemize}
\item Let $1 \leq a, b, c, d \leq N$, and $a+b+c+d=2N$. 
$$\tcfr{a,b}{c}{d}\tcfr{a-1,b-1}{c+1}{d+1}=$$
$$\tcfr{a-1,b}{c+1}{d}\tcfr{a,b-1}{c}{d+1}+\tcfr{a-1,b}{c}{d+1}\tcfr{a,b-1}{c+1}{d}.$$
\item If $a+c=b+d=N$,
$$\tcfr{a,b}{c}{d}=\dud{a}{c}\tcfr{b}{}{d}.$$
\item If $a+d=b+c=N$,
$$\tcfr{a,b}{c}{d}=\tcfr{a}{}{d}\dud{b}{c}.$$
\item If $b+c+d=N$, then
$$\tcfr{N,b}{c}{d}=\tcfr{b}{c}{d}.$$
\end{itemize}

The first identity is a close relative of the octahedron recurrence. All the identities can be checked using webs and skein relations.

Most of the cluster mutations are precisely the first identity:
$$\tcfr{a,b}{c}{d}\tcfr{a-1,b-1}{c+1}{d+1}=\tcfr{a-1,b}{c+1}{d}\tcfr{a,b-1}{c}{d+1}+\tcfr{a-1,b}{c}{d+1}\tcfr{a,b-1}{c+1}{d}.$$

All other cluster mutations are degenerations of this identity, and are obtained from this one by applying the other three identities. For example, the first mutation is
$$\tcfr{N-2}{1}{1}\tcfr{N-1,N-3}{2}{2}=\tcfr{N,N-2}{1}{1}\tcfr{N-1,N-3}{2}{2}=$$
$$\tcfr{N-1,N-2}{2}{1}\tcfr{N,N-3}{1}{2}+\tcfr{N-1,N-2}{1}{2}\tcfr{N,N-3}{2}{1}=$$
$$\tcfr{N-1}{}{1}\dud{N-2}{2}\tcfr{N-3}{1}{2}+\dud{N-1}{1}\tcfr{N-2}{}{2}\tcfr{N-3}{2}{1}.$$

Locally, all cluster mutations come in the four types depicted in Figure 11.

\begin{center}
\begin{tikzpicture}[scale=2.5]

\begin{scope}[xshift=-1.7cm]
  \node (0) at (0,0) {$\tcfr{a,b}{c}{d}$};
  \node (2) at (60:1) {$\tcfr{a-1,b}{c}{d+1}$};
  \node (3) at (120:1) {$\tcfr{a-1,b}{c+1}{d}$};
  \node (5) at (240:1) {$\tcfr{b-1}{c+1}{d}$};
  \node (6) at (300:1) {$\tcfr{b-1}{c}{d+1}$};

  \draw [->] (2) to (0);
  \draw [->] (0) to (3);
  \draw [->] (5) to (0);
  \draw [->] (0) to (6);
\end{scope}

\begin{scope}[xshift=1.7cm]
  \node (0) at (0,0) {$\tcfr{a,b}{c}{d}$};
  \node (2) at (60:1) {$\tcfr{a-1,b}{c}{d+1}$};
  \node (3) at (120:1) {$\tcfr{a-1}{}{d}$};
  \node (4) at (180:1) {$\dud{b}{c+1}$};
  \node (5) at (240:1) {$\tcfr{b-1}{c+1}{d}$};
  \node (6) at (300:1) {$\tcfr{b-1}{c}{d+1}$};

  \draw [->] (2) to (0);
  \draw [->] (0) to (3);
  \draw [->] (0) to (4);
  \draw [->] (5) to (0);
  \draw [->] (0) to (6);

\end{scope}

\begin{scope}[xshift=-1.7cm,yshift=-2.5cm]
  \node (0) at (0,0) {$\tcfr{a,b}{c}{d}$};
  \node (1) at (0:1) {$\tcfr{b}{}{d+1}$};
  \node (2) at (60:1) {$\dud{a-1}{c}$};
  \node (3) at (120:1) {$\tcfr{a-1,b}{c+1}{d}$};
  \node (5) at (240:1) {$\tcfr{b-1}{c+1}{d}$};
  \node (6) at (300:1) {$\tcfr{b-1}{c}{d+1}$};

  \draw [->] (1) to (0);
  \draw [->] (2) to (0);
  \draw [->] (0) to (3);
  \draw [->] (5) to (0);
  \draw [->] (0) to (6);
\end{scope}

\begin{scope}[xshift=1.7cm,yshift=-2.5cm]
  \node (0) at (0,0) {$\tcfr{a,b}{c}{d}$};
  \node (1) at (0:1) {$\tcfr{b}{}{d+1}$};
  \node (2) at (60:1) {$\dud{a-1}{c}$};
  \node (3) at (120:1) {$\tcfr{a-1}{}{d}$};
  \node (4) at (180:1) {$\dud{b}{c+1}$};
  \node (5) at (240:1) {$\tcfr{b-1}{c+1}{d}$};
  \node (6) at (300:1) {$\tcfr{b-1}{c}{d+1}$};

  \draw [->] (1) to (0);
  \draw [->] (2) to (0);
  \draw [->] (0) to (3);
  \draw [->] (0) to (4);
  \draw [->] (5) to (0);
  \draw [->] (0) to (6);

\end{scope}

\draw[yshift=-3.8cm]
  node[below,text width=6cm] 
  {
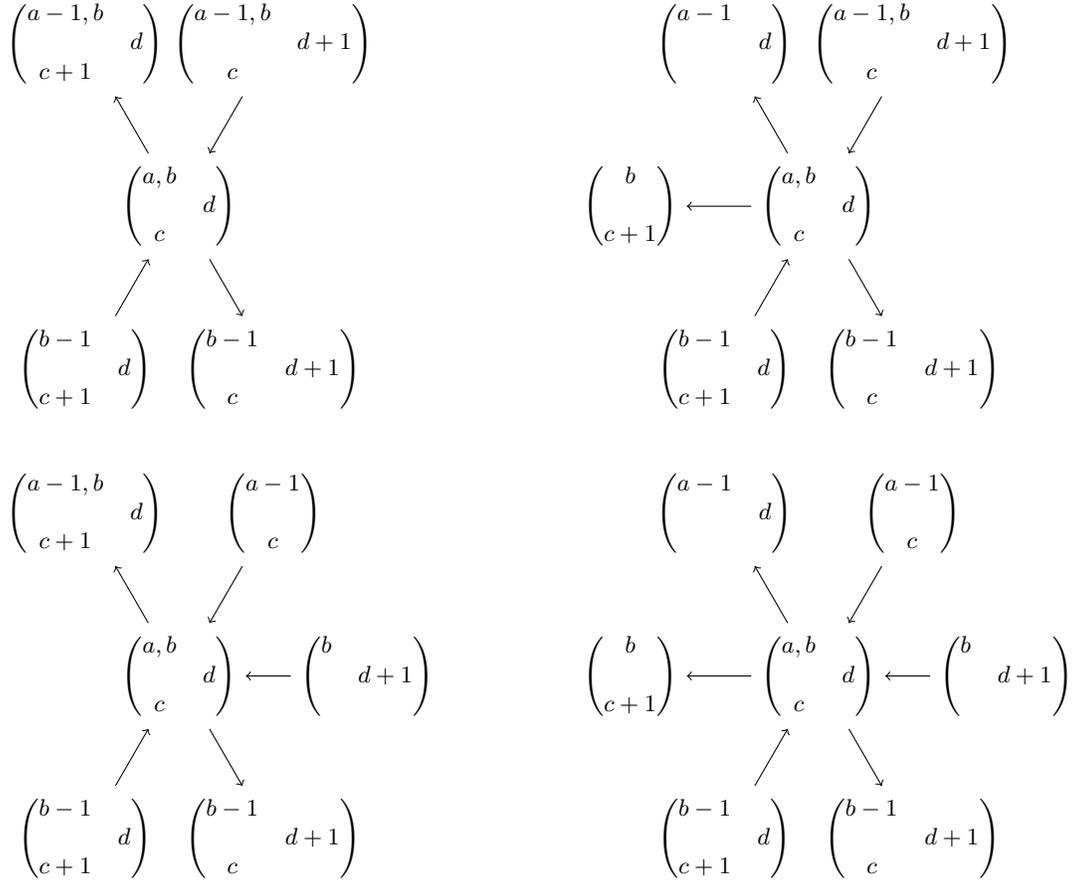
  Figure 11. The four local cluster transformations in the cactus sequence.
  };

\end{tikzpicture}
\end{center}

Note that in line with what was said above, the terms $\tcfr{a-1,b}{c+1}{d}$ and $\tcfr{a-1,b}{c}{d+1}$ will lie directly above $\tcfr{a,b}{c}{d}$, while the terms $\tcfr{a,b-1}{c}{d+1}$ and $\tcfr{a,b-1}{c+1}{d}$ will lie directly below $\tcfr{a,b}{c}{d}$, so that vertex is ready for mutation exactly when the vertices directly above and below have been mutated the appropriate number of times.

\end{proof}

We now specialize to $N=2n$. The cactus sequence will allow us to give an alternative way to construct the cluster algebra structure for $\Conf_3 \A_{Sp_{2n}}$. We perform the following sequence of mutations on $\Conf_3 \A_{SL_{2n}}$:

$$x_{N-2,1,1}, x_{N-3,1,2}, \dots, x_{1,1,N-2}$$
$$x_{N-3,2,1}, x_{N-4,2,2}, \dots, x_{1,2,N-3}$$
$$x_{N-2,1,1}, x_{N-3,1,2}, \dots, x_{2,1,N-3}$$
$$x_{N-4,3,1}, x_{N-5,3,2}, \dots, x_{1,3,N-4}$$
$$x_{N-3,2,1}, x_{N-4,2,2}, \dots, x_{2,2,N-4}$$
$$x_{N-2,1,1}, x_{N-3,1,2}, \dots, x_{3,1,N-4}$$
$$\dots$$
$$x_{n,n-1,1}, x_{n-1,n-1,2}, \dots, x_{1,n-1,n}$$
$$\dots$$
$$x_{N-3,2,1}, x_{N-4,2,2}, \dots, x_{n-2,2,n}$$
$$x_{N-2,1,1}, x_{N-3,1,2}, \dots, x_{n-1,1,n}$$

One can picture the above sequence as consisting of $n-1$ stages, where each stage consists of mutating all vertices lying in a parallelogram. The result will be that we get the quiver pictured in Figure 12 (all frozen vertices have been deleted for simplicity)

\begin{center}
\begin{tikzpicture}[scale=2.4]

  \node (x-31) at (-3,0) {$\tcfr{3}{4}{1}$};
  \node (x-21) at (-2,0) {$\tcfr{3,7}{4}{2}$};
  \node (x-11) at (-1,0) {$\tcfr{3,6}{4}{3}$};
  \node (x01) at (0,0) {$\tcfr{3,5}{4}{4}$};
  \node (x11) at (1,0) {$\tcfr{2,5}{4}{5}$};
  \node (x21) at (2,0) {$\tcfr{1,5}{4}{6}$};
  \node (x31) at (3,0) {$\tcfr{5}{4}{7}$};
  \node (x-32) at (-3,-1) {$\tcfr{2}{5}{1}$};
  \node (x-22) at (-2,-1) {$\tcfr{2}{4}{2}$};
  \node (x-12) at (-1,-1) {$\tcfr{2,7}{4}{3}$};
  \node (x02) at (0,-1) {$\tcfr{2,6}{4}{4}$};
  \node (x12) at (1,-1) {$\tcfr{1,6}{4}{5}$};
  \node (x22) at (2,-1) {$\tcfr{6}{4}{6}$};
  \node (x32) at (3,-1) {$\tcfr{6}{3}{7}$};
  \node (x-33) at (-3,-2) {$\tcfr{1}{6}{1}$};
  \node (x-23) at (-2,-2) {$\tcfr{1}{5}{2}$};
  \node (x-13) at (-1,-2) {$\tcfr{1}{4}{3}$};
  \node (x03) at (0,-2) {$\tcfr{1,7}{4}{4}$};
  \node (x13) at (1,-2) {$\tcfr{7}{4}{5}$};
  \node (x23) at (2,-2) {$\tcfr{7}{3}{6}$};
  \node (x33) at (3,-2) {$\tcfr{7}{2}{7}$};

  \draw [->] (x01) to (x11);
  \draw [->] (x11) to (x21);
  \draw [->] (x21) to (x31);
  \draw [->] (x02) to (x12);
  \draw [->] (x12) to (x22);
  \draw [->] (x22) to (x32);
  \draw [->] (x03) to (x13);
  \draw [->] (x13) to (x23);
  \draw [->] (x23) to (x33);

  \draw [->, dashed] (x33) to (x32);
  \draw [->, dashed] (x32) to (x31);

  \draw [->] (x03) to (x02);
  \draw [->] (x02) to (x01);

  \draw [->] (x13) to (x12);
  \draw [->] (x12) to (x11);
  \draw [->] (x23) to (x22);
  \draw [->] (x22) to (x21);

  \draw [->] (x11) to (x02);
  \draw [->] (x12) to (x03);
 
  \draw [->] (x21) to (x12);
  \draw [->] (x22) to (x13);
  \draw [->] (x31) to (x22);
  \draw [->] (x32) to (x23);

  \draw [->] (x01) to (x-11);
  \draw [->] (x-11) to (x-21);
  \draw [->] (x-21) to (x-31);
  \draw [->] (x02) to (x-12);
  \draw [->] (x-12) to (x-22);
  \draw [->] (x-22) to (x-32);
  \draw [->] (x03) to (x-13);
  \draw [->] (x-13) to (x-23);
  \draw [->] (x-23) to (x-33);

  \draw [->] (x-13) to (x-12);
  \draw [->] (x-12) to (x-11);
  \draw [->] (x-23) to (x-22);
  \draw [->] (x-22) to (x-21);
 
  \draw [->, dashed] (x-33) to (x-32);
  \draw [->, dashed] (x-32) to (x-31);

  \draw [->] (x-11) to (x02);
  \draw [->] (x-12) to (x03);

  \draw [->] (x-21) to (x-12);
  \draw [->] (x-22) to (x-13);
  \draw [->] (x-31) to (x-22);
  \draw [->] (x-32) to (x-23);

\draw[yshift=-2.5cm]
  node[below,text width=6cm] 
  {
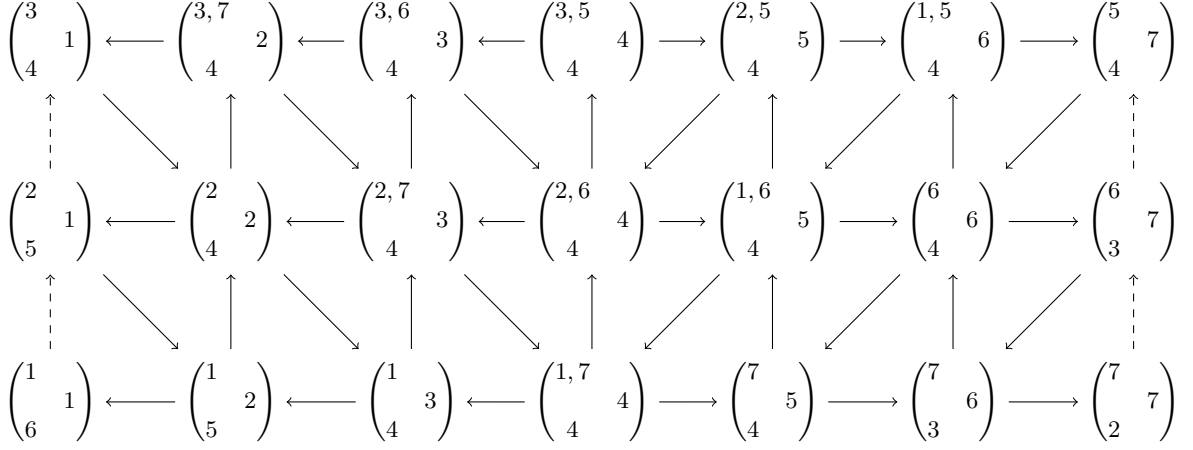
  Figure 12. The functions and quiver for the cluster $\Conf_3 \A_{SL_{2n}}$ which folds to give a cluster for $\Conf_3 \A_{Sp_{2n}}$.
  };

\end{tikzpicture}
\end{center}

The function attached to $x_{ijk}$ (here we take $i,j,k > 0$ to exclude frozen vertices) will be 
\begin{itemize}
\item $\tcfr{i}{j}{k}$ if $j \geq n$,
\item $\tcfr{n+j, i+j-n}{n}{k+n-j}$ if $j < n$ and $k < n$ (in fact for $j \leq n$ and $k \leq n$; equality cases agree with the cases above and below),
\item $(2n-i, i+j; k+i)$ if $k \geq n$.
\end{itemize}

Thus we see that the functions attached to $x_{ijk}$ and $x_{ikj}$ pull back to the same functions on $\Conf_3 \A_{Sp_{2n}}$ via the embedding $\Conf_3 \A_{Sp_{2n}} \hookrightarrow \Conf_3 \A_{SL_{2n}}$. Here we use the identities \eqref{dualities}. In fact, for $j \geq k$ and $j \geq n$, we have that the vertex $x_{ijk}$ corresponds to the vertex $x_{n-i,k}$ in the quiver for $\Conf_3 \A_{Sp_{2n}}$, while for $j \geq k$ and $j < n$, we have that the vertex $x_{ijk}$ corresponds to the vertex $x_{k,n+k-j}$.

The result of this is that we can obtain the cluster structure on $\Conf_3 \A_{Sp_{2n}}$ as follows: Start with the standard cluster algebra on $\Conf_3 \A_{SL_{2n}}$, which has functions attached to vertices $x_{ijk}$. Perform series of mutations given above, which is a subsequence of the cactus sequence. Then fold the resulting cluster structure by identifying the following pairs of vertices: $x_{i,2n-i,0}$ and $x_{2n-i,i,0}$; $x_{i,0,2n-i}$ and $x_{2n-i,0,i}$; $x_{0,j,2n-j}$ and $x_{0,2n-j,j}$; $x_{ijk}$ and $x_{ikj}$ for $i,j,k >0$.

Let us say a few words about folding. Suppose we have a cluster algebra $C$ such that its quiver is preserved under an involution $\sigma$. Note that this forces vertices in the same orbit to have no arrows between them. Then we may form a new cluster algebra $C'$. The $B$-matrix for $C'$ comes from quotienting the quiver for $C$ by the involution $\sigma$ and then coloring the vertices fixed under the involution white and all other vertices black. On the level of algebras, one sets the variables for vertices in the same orbit equal. It is not hard to check that mutating a vertex in $C'$ corresponds to unfolding, mutating all lifts of that vertex in $C$, and then folding again. Thus we constructed above a cluster for $\Conf_3 \A_{SL_{2n}}$ which folds to give a cluster for $\Conf_3 \A_{Sp_{2n}}$. For more details about the procedure of folding for cluster algebras, and for folding under more general automorphisms, see \cite{FZ2}.

More invariantly, there is an outer automorphism of $SL_{2n}$ that has $Sp_{2n}$ as its fixed locus. This gives an involution of $\Conf_3 \A_{SL_{2n}}$ (and more generally $\Conf_m \A_{SL_{2n}}$) that has $\Conf_3 \A_{Sp_{2n}}$ (respectively, $\Conf_m \A_{Sp_{2n}}$) as its fixed locus. It turns out that the cluster algebra structure on $\Conf_3 \A_{SL_{2n}}$ is preserved by this involution, and that, moreover, there is a particular seed, constructed above, that is preserved by this involution. Folding this seed gives the cluster algebra structure on $\Conf_3 \A_{Sp_{2n}}$.

Let us make some observations. Let $\sigma$ be the Dynkin diagram automorphism of $A_{2n-1}$ having quotient $C_{n}$. (We call this automorphism $\sigma$ by abuse of notation: we will see soon that it induces the automorphism of the quiver that we called $\sigma$ above.) Then $\sigma$ induces a map on the root system for $SL_{2n}$, and hence on the fundamental weights and the dominant weights. It also induces an outer automorphism of $SL_{2n}$ having fixed locus $Sp_{2n}$, and an involution on the spaces $\Conf_m \A_{SL_{2n}}$. Let $\pi$ be the map from the vertices of $A_{2n-1}$ to the vertices of $C_n$. This induces a map $\pi$ sending fundamental weights to corresponding fundamental weights, and therefore projects the weight space for $SL_{2n}$ to the weight space for $Sp_{2n}$.

\begin{observation} Let $f$ be a function on $\Conf_3 \A_{SL_{2n}}$ that lies in the invariant space
$$[V_{\lambda} \otimes V_{\mu} \otimes V_{\nu}]^{SL_{2n}}.$$
Then $\sigma^*(f)$ lies in the invariant space
$$[V_{\sigma(\lambda)} \otimes V_{\sigma(\mu)} \otimes V_{\sigma(\nu)}]^{SL_{2n}}.$$
\end{observation}

\begin{observation} Let $f$ be a function on $\Conf_3 \A_{SL_{2n}}$ that lies in the invariant space
$$[V_{\lambda} \otimes V_{\mu} \otimes V_{\nu}]^{SL_{2n}}.$$
Then as a function on $\Conf_3 \A_{Sp_{2n}}$, $f$ lies in the invariant space
$$[V_{\pi(\lambda)} \otimes V_{\pi(\mu)} \otimes V_{\pi(\nu)}]^{Sp_{2n}}.$$
\end{observation}

\begin{observation} Consider a cluster for $\Conf_3 \A_{SL_{2n}}$ which folds to give a cluster for $\Conf_3 \A_{Sp_{2n}}$. Suppose $v$ is a vertex in the quiver for this cluster. Let $f_v$ is the function attached to $v$. Then
$$f_v \in [V_{\lambda} \otimes V_{\mu} \otimes V_{\nu}]^{SL_{2n}},$$
$$f_{\sigma(v)} \in [V_{\sigma(\lambda)} \otimes V_{\sigma(\mu)} \otimes V_{\sigma(\nu)}]^{SL_{2n}}.$$
However, on $\Conf_3 \A_{Sp_{2n}}$, $f_v=f_{\sigma(v)}$. This means that we must have $\pi(\lambda)=\pi(\sigma(\lambda))$, $\pi(\mu)=\pi(\sigma(\mu))$, and $\pi(\nu)=\pi(\sigma(\nu))$.
\end{observation}

It is also clarifying to step back and motivate the sequence of mutations realizing the cactus sequence above. As explained in section 3.2, the cluster structure on $\Conf_3 \A$ comes from a reduced word for the longest element $w_0$ in the Weyl group of $G$. The initial seed for $\Conf_3 \A_{SL_{2n}}$ is built using the reduced word 
$$s_1 s_2 \dots s_{2n-1} s_1 s_2 \dots s_{2n-2} \dots s_1 s_2 s_3 s_1 s_2 s_1.$$
Here $s_1, \dots, s_{2n-1}$ are the generators of the Weyl group for $SL_{2n}$. It is known how to use cluster transformations to pass between the clusters that are associated to different reduced words (\cite{BFZ}). The cactus sequence transforms between the cluster above and the cluster associated to the reduced word
$$s_{2n-1} s_{2n-2} \dots s_{1} s_{2n-1} s_{2n-2} \dots s_{2} \dots s_{2n-1} s_{2n-2} s_{2n-3} s_{2n-1} s_{2n-2} s_{2n-1}.$$
Thus we have the following observation, which we do not believe has been made elsewhere:

\begin{observation} The cactus sequence may be viewed as the sequence of mutations relating the cluster structures on $\Conf_3 \A_{SL_{N}}$ related to the reduced words $$s_1 s_2 \dots s_{N-1} s_1 s_2 \dots s_{N-2} \dots s_1 s_2 s_3 s_1 s_2 s_1$$ and 
$$s_{N-1} s_{N-2} \dots s_{1} s_{N-1} s_{N-2} \dots s_{2} \dots s_{N-1} s_{N-2} s_{N-3} s_{N-1} s_{N-2} s_{N-1}.$$
\end{observation}

The subsequence of the cactus sequence given above transforms the initial cluster into the cluster associated with the reduced word
$$ (s_n s_{n-1} s_{n+1} s_{n-2} s_{n+2} \dots s_{1} s_{2n-1})^n.$$
This reduced word is in some sense ``intermediate'' between the two other reduced words discussed above.

Now let $s'_1, s'_2, \dots s'_n$ be the generators of the Weyl group of $Sp_{2n}$. There is an injection from the Weyl group of $Sp_{2n}$ to the Weyl group of $SL_{2n}$ that takes
$$s'_n \rightarrow s_n,$$
$$s'_i \rightarrow s_i s_{2n-i}.$$
Under this map, 
$$(s'_n s'_{n-1} \dots s'_1)^n \rightarrow (s_n s_{n-1} s_{n+1} s_{n-2} s_{n+2} \dots s_{1} s_{2n-1})^n$$
Therefore the reduced word for the longest element of the Weyl group of $SL_{2n}$ folds to give the reduced word for the longest element of the Weyl group of $Sp_{2n}$. So the folding that gives the cluster structure on $\Conf_3 \A_{Sp_{2n}}$ from the cluster structure on $\Conf_3 \A_{SL_{2n}}$ really takes place on the level of Weyl groups.

\subsection{The sequence of mutations realizing $S_3$ symmetries}

We consider the cluster structure on $\Conf_3 \A_{Sp_{2n}}$ described above. In fact, because the cluster structure we gave was not symmetric, we have described six different cluster structures on $\Conf_3 \A_{Sp_{2n}}$. We would now like to give sequences of mutations relating these six clusters to show that they are actually all clusters in the same cluster algebra.

From the folding construction of the previous section, it is clear that we could unfold the cluster structure of $\Conf_3 \A_{Sp_{2n}}$ to get the cluster structure on $\Conf_3 \A_{SL_{2n}}$, perform a sequence of mutations on $\Conf_3 \A_{SL_{2n}}$, and then refold to obtain one of the other cluster structures on $\Conf_3 \A_{Sp_{2n}}$. The question then becomes whether we can realize this sequence of mutations ``upstairs'' on $\Conf_3 \A_{SL_{2n}}$ ``downstairs''on $\Conf_3 \A_{Sp_{2n}}$. It is by no means obvious that this is possible. Let us give two reasons for this. First, mutating one vertex downstairs generally corresponds to mutating two vertices upstairs (the two that are folded together to give the vertex downstairs), so that this severely restricts the possible sequences of mutations upstairs. Second, note that if we perform the naive procedure of passing to $\Conf_3 \A_{SL_{2n}}$ and performing a sequence of mutations there, when we refold, we are generally no longer folding together the same pairs of vertices.

We will realize the $S_3$ symmetries on $\Conf_3 \A_{Sp_{2n}}$ by exhibiting sequences of mutations that realize two different transpositions in the group $S_3$. As we shall see, one of these transpositions can actually be realized by going going ``upstairs'' and using the calculations of the previous section with some care. The other transposition requires a completely different analysis.

\subsubsection{The first transposition}

Let $(A,B,C) \in \Conf_3 \A_{Sp_{2n}}$ be a triple of flags. Let us first give the sequence of mutations that realizes that $S_3$ symmetry $(A,B,C) \rightarrow (A,C,B).$

The sequence of mutations is as follows:
\begin{equation} \label{23Sp}
\begin{gathered}
x_{11}, x_{21}, x_{22}, x_{12}, x_{13}, x_{23}, x_{33}, x_{32}, x_{33,} \dots, x_{1,n-1}, \dots, x_{n-1, n-1}, \dots, x_{n-1, 1}, \\
x_{11}, x_{21}, x_{22}, x_{12},  \dots x_{1,n-2}, \dots, x_{n-2, n-2}, \dots, x_{n-2, 1}, \\
\dots, \\
x_{11}, x_{21}, x_{22}, x_{12}, \\
x_{11} \\
\end{gathered}
\end{equation}

The sequence can be thought of as follows: At any step of the process, we mutate all $x_{ij}$ such that $\max(i,j)$ is constant. It will not matter in which order we mutate these $x_{ij}$ because the vertices we mutate have no arrows between them. So we first mutate the $x_{ij}$ such that $\max(i,j)=1$, then the $x_{ij}$ such that $\max(i,j)=2$, then the $x_{ij}$ such that $\max(i,j)=3$, etc. The sequence of maximums that we use is 
$$1, 2, 3, \dots n-1, 1, 2, \dots n-2 \dots, 1, 2, 3, 1, 2, 1.$$

In Figure 13, we depict how the quiver and functions for $\Conf_3 \A_{Sp_{8}}$ changes after performing the sequences of mutations of $x_{ij}$ having maximums $1$; $1, 2$; and $1, 2, 3$. The initial quiver is in Figure 5.

\begin{center}
\begin{tikzpicture}[scale=2]
\begin{scope}[xshift=-1.5cm]

  \node (x01) at (1,0) {$\dud{3}{5}$};
  \node (x02) at (2,0) {$\tcfr{6}{}{2}$};
  \node (x03) at (3,0) {$\tcfr{5}{}{3}$};
  \node (x04) at (4,0) {$\tcfr{4}{}{4}$};
  \node (x11) at (1,-1) {$\boldsymbol{\tcfr{2,7}{5}{2}}$};
  \node (x12) at (2,-1) {$\tcfr{3, 7}{4}{2}$};
  \node (x13) at (3,-1) {$\tcfr{3, 6}{4}{3}$};
  \node (x14) at (4,-1) {$\tcfr{3, 5}{4}{4}$};
  \node (x21) at (1,-2) {$\tcfr{2}{5}{1}$};
  \node (x22) at (2,-2) {$\tcfr{2}{4}{2}$};
  \node (x23) at (3,-2) {$\tcfr{2, 7}{4}{3}$};
  \node (x24) at (4,-2) {$\tcfr{2, 6}{4}{4}$};
  \node (x31) at (1,-3) {$\tcfr{1}{6}{1}$};
  \node (x32) at (2,-3) {$\tcfr{1}{5}{2}$};
  \node (x33) at (3,-3) {$\tcfr{1}{4}{3}$};
  \node (x34) at (4,-3) {$\tcfr{1, 7}{4}{4}$};
  \node (x41) at (1,-4) {$\tcfr{}{7}{1}$};
  \node (x42) at (2,-4) {$\tcfr{}{6}{2}$};
  \node (x43) at (3,-4) {$\tcfr{}{5}{3}$};
  \node (x44) at (4,-4) {$\tcfr{}{4}{4}$};
  \node (y1) at (0,-3) {$\dud{1}{7}$};
  \node (y2) at (0,-2) {$\dud{2}{6}$};
  \node (y3) at (0,-1) {$\tcfr{7}{}{1}$};
  \node (y4) at (5,-1) {$\dud{4}{4}$};

  \draw [->] (x41) to (x31);
  \draw [->] (x31) to (x21);
  \draw [->] (x11) to (x21);
  \draw [->] (x01) to (x11);
  \draw [->] (x42) to (x32);
  \draw [->] (x32) to (x22);
  \draw [->] (x12) to (x02);
  \draw [->] (x43) to (x33);
  \draw [->] (x33) to (x23);
  \draw [->] (x23) to (x13);
  \draw [->] (x13) to (x03);
  \draw [->] (x44) to (x34);
  \draw [->] (x34) to (x24);
  \draw [->] (x24) to (x14);
  \draw [->] (x14) to (x04);
  \draw [->, dashed] (x01) .. controls +(45:1) and +(up:2) .. (y4);
  \draw [->, dashed] (y2) .. controls +(135:1) and +(left:2) .. (x01);
  \draw [->, dashed] (y1) to (y2);

  \draw [->, dashed] (x04) to (x03);
  \draw [->, dashed] (x03) to (x02);
  \draw [->, dashed] (x02) .. controls +(135:1) and +(up:1.5) .. (y3);
  \draw [->] (y4) to (x14);
  \draw [->] (x14) to (x13);
  \draw [->] (x13) to (x12);
  \draw [->] (x11) to (x12);
  \draw [->] (y3) to (x11);
  \draw [->] (x24) to (x23);
  \draw [->] (x23) to (x22);
  \draw [->] (x21) to (y2);
  \draw [->] (x34) to (x33);
  \draw [->] (x33) to (x32);
  \draw [->] (x32) to (x31);
  \draw [->] (x31) to (y1);
  \draw [->, dashed] (x44) to (x43);
  \draw [->, dashed] (x43) to (x42);
  \draw [->, dashed] (x42) to (x41);

  \draw [->] (x12) to (x01);
  \draw [->] (x02) to (x13);
  \draw [->] (x03) to (x14);
  \draw [->] (x04) to (y4);
  \draw [->] (x21) to (y3);
  \draw [->] (x22) to (x11);
  \draw [->] (x12) to (x23);
  \draw [->] (x13) to (x24);
  \draw [->] (y2) to (x31);
  \draw [->] (x21) to (x32);
  \draw [->] (x22) to (x33);
  \draw [->] (x23) to (x34);
  \draw [->] (y1) to (x41);
  \draw [->] (x31) to (x42);
  \draw [->] (x32) to (x43);
  \draw [->] (x33) to (x44);

\end{scope}

\draw[yshift=-5cm,xshift=1cm]
  node[below,text width=6cm] 
  {
  Figure 13a The quiver and the functions for $\Conf_3 \A_{Sp_{8}}$ after performing the sequence of mutations of $x_{ij}$ having maximums $1$.
  };

\end{tikzpicture}
\end{center}

\begin{center}
\begin{tikzpicture}[scale=2]

\begin{scope}[xshift=-1.5cm]

  \node (x01) at (1,0) {$\tcfr{6}{}{2}$};
  \node (x02) at (2,0) {$\dud{3}{5}$};
  \node (x03) at (3,0) {$\tcfr{5}{}{3}$};
  \node (x04) at (4,0) {$\tcfr{4}{}{4}$};
  \node (x11) at (1,-1) {$\tcfr{2,7}{5}{2}$};
  \node (x12) at (2,-1) {$\boldsymbol{\tcfr{2, 6}{5}{3}}$};
  \node (x13) at (3,-1) {$\tcfr{3, 6}{4}{3}$};
  \node (x14) at (4,-1) {$\tcfr{3, 5}{4}{4}$};
  \node (x21) at (1,-2) {$\boldsymbol{\tcfr{1,7}{6}{2}}$};
  \node (x22) at (2,-2) {$\boldsymbol{\tcfr{1,7}{5}{3}}$};
  \node (x23) at (3,-2) {$\tcfr{2, 7}{4}{3}$};
  \node (x24) at (4,-2) {$\tcfr{2, 6}{4}{4}$};
  \node (x31) at (1,-3) {$\tcfr{1}{6}{1}$};
  \node (x32) at (2,-3) {$\tcfr{1}{5}{2}$};
  \node (x33) at (3,-3) {$\tcfr{1}{4}{3}$};
  \node (x34) at (4,-3) {$\tcfr{1, 7}{4}{4}$};
  \node (x41) at (1,-4) {$\tcfr{}{7}{1}$};
  \node (x42) at (2,-4) {$\tcfr{}{6}{2}$};
  \node (x43) at (3,-4) {$\tcfr{}{5}{3}$};
  \node (x44) at (4,-4) {$\tcfr{}{4}{4}$};
  \node (y1) at (0,-3) {$\dud{1}{7}$};
  \node (y2) at (0,-2) {$\tcfr{7}{}{1}$};
  \node (y3) at (0,-1) {$\dud{2}{6}$};
  \node (y4) at (5,-1) {$\dud{4}{4}$};

  \draw [->] (x41) to (x31);
  \draw [->] (x21) to (x31);
  \draw [->] (x21) to (x11);
  \draw [->] (x11) to (x01);
  \draw [->] (x42) to (x32);
  \draw [->] (x22) to (x32);
  \draw [->] (x02) to (x12);
  \draw [->] (x43) to (x33);
  \draw [->] (x13) to (x03);
  \draw [->] (x44) to (x34);
  \draw [->] (x34) to (x24);
  \draw [->] (x24) to (x14);
  \draw [->] (x14) to (x04);
  \draw [->, dashed] (x02) .. controls +(up:1) and +(up:2) .. (y4);
  \draw [->, dashed] (y3) .. controls +(up:2) and +(up:1) .. (x02);
  \draw [->, dashed] (y1) .. controls +(left:1) and +(left:1) .. (y3);

  \draw [->, dashed] (x04) to (x03);
  \draw [->, dashed] (x03) .. controls +(up:1) and +(up:1) .. (x01);
  \draw [->, dashed] (x01)  .. controls +(left:2) and +(left:1) .. (y2);
  \draw [->] (y4) to (x14);
  \draw [->] (x14) to (x13);
  \draw [->] (x12) to (x13);
  \draw [->] (x12) to (x11);
  \draw [->] (x11) to (y3);
  \draw [->] (x24) to (x23);
  \draw [->] (x22) to (x23);
  \draw [->] (y2) to (x21);
  \draw [->] (x34) to (x33);
  \draw [->] (x31) to (y1);
  \draw [->, dashed] (x44) to (x43);
  \draw [->, dashed] (x43) to (x42);
  \draw [->, dashed] (x42) to (x41);

  \draw [->] (x01) to (x12);
  \draw [->] (x13) to (x02);
  \draw [->] (x03) to (x14);
  \draw [->] (x04) to (y4);
  \draw [->] (y3) to (x21);
  \draw [->] (x11) to (x22);
  \draw [->] (x23) to (x12);
  \draw [->] (x13) to (x24);
  \draw [->] (x31) to (y2);
  \draw [->] (x32) to (x21);
  \draw [->] (x33) to (x22);
  \draw [->] (x23) to (x34);
  \draw [->] (y1) to (x41);
  \draw [->] (x31) to (x42);
  \draw [->] (x32) to (x43);
  \draw [->] (x33) to (x44);

\end{scope}

\draw[yshift=-5cm,xshift=1cm]
  node[below,text width=6cm] 
  {
  Figure 13b The quiver and the functions for $\Conf_3 \A_{Sp_{8}}$ after performing the sequence of mutations of $x_{ij}$ having maximums $1, 2$.
  };

\end{tikzpicture}
\end{center}

\begin{center}
\begin{tikzpicture}[scale=2]

\begin{scope}[xshift=-1.5cm]

  \node (x01) at (1,0) {$\tcfr{6}{}{2}$};
  \node (x02) at (2,0) {$\tcfr{5}{}{3}$};
  \node (x03) at (3,0) {$\dud{3}{5}$};
  \node (x04) at (4,0) {$\dud{4}{4}$};
  \node (x11) at (1,-1) {$\tcfr{2,7}{5}{2}$};
  \node (x12) at (2,-1) {$\tcfr{2, 6}{5}{3}$};
  \node (x13) at (3,-1) {$\boldsymbol{\tcfr{2, 5}{5}{4}}$};
  \node (x14) at (4,-1) {$\tcfr{3, 5}{4}{4}$};
  \node (x21) at (1,-2) {$\tcfr{1,7}{6}{2}$};
  \node (x22) at (2,-2) {$\tcfr{1,7}{5}{3}$};
  \node (x23) at (3,-2) {$\boldsymbol{\tcfr{1, 6}{5}{4}}$};
  \node (x24) at (4,-2) {$\tcfr{2, 6}{4}{4}$};
  \node (x31) at (1,-3) {$\boldsymbol{\tcfr{7}{7}{2}}$};
  \node (x32) at (2,-3) {$\boldsymbol{\tcfr{7}{6}{3}}$};
  \node (x33) at (3,-3) {$\boldsymbol{\tcfr{7}{5}{4}}$};
  \node (x34) at (4,-3) {$\tcfr{1, 7}{4}{4}$};
  \node (x41) at (1,-4) {$\tcfr{}{7}{1}$};
  \node (x42) at (2,-4) {$\tcfr{}{6}{2}$};
  \node (x43) at (3,-4) {$\tcfr{}{5}{3}$};
  \node (x44) at (4,-4) {$\tcfr{}{4}{4}$};
  \node (y1) at (0,-3) {$\tcfr{7}{}{1}$};
  \node (y2) at (0,-2) {$\dud{1}{7}$};
  \node (y3) at (0,-1) {$\dud{2}{6}$};
  \node (y4) at (5,-1) {$\tcfr{4}{}{4}$};

  \draw [->] (x31) to (x41);
  \draw [->] (x31) to (x21);
  \draw [->] (x21) to (x11);
  \draw [->] (x11) to (x01);
  \draw [->] (x32) to (x42);
  \draw [->] (x32) to (x22);
  \draw [->] (x22) to (x12);
  \draw [->] (x12) to (x02);
  \draw [->] (x33) to (x43);
  \draw [->] (x03) to (x13);
  \draw [->] (x34) to (x44);
  \draw [->] (x24) to (x34);
  \draw [->] (x14) to (x24);
  \draw [->] (x04) to (x14);
  \draw [->, dashed] (x03) to (x04);
  \draw [->, dashed] (y3) .. controls +(up:2) and +(up:1) .. (x03);
  \draw [->, dashed] (y2) to (y3);

  \draw [->, dashed] (y4) .. controls +(up:2) and +(up:1) .. (x02);
  \draw [->, dashed] (x02) to (x01);
  \draw [->, dashed] (x01) .. controls +(left:2) and +(left:1) .. (y1);
  \draw [->] (x14) to (y4);
  \draw [->] (x13) to (x14);
  \draw [->] (x13) to (x12);
  \draw [->] (x12) to (x11);
  \draw [->] (x11) to (y3);
  \draw [->] (x23) to (x24);
  \draw [->] (x23) to (x22);
  \draw [->] (x22) to (x21);
  \draw [->] (x21) to (y2);
  \draw [->] (x33) to (x34);
  \draw [->] (y1) to (x31);
  \draw [->, dashed] (x43) to (x44);
  \draw [->, dashed] (x42) to (x43);
  \draw [->, dashed] (x41) to (x42);

  \draw [->] (x01) to (x12);
  \draw [->] (x02) to (x13);
  \draw [->] (x14) to (x03);
  \draw [->] (x04) to (y4);
  \draw [->] (y3) to (x21);
  \draw [->] (x11) to (x22);
  \draw [->] (x12) to (x23);
  \draw [->] (x24) to (x13);
  \draw [->] (y2) to (x31);
  \draw [->] (x21) to (x32);
  \draw [->] (x22) to (x33);
  \draw [->] (x34) to (x23);
  \draw [->] (x41) to (y1);
  \draw [->] (x42) to (x31);
  \draw [->] (x43) to (x32);
  \draw [->] (x44) to (x33);

\end{scope}

\draw[yshift=-5cm,xshift=1cm]
  node[below,text width=6cm] 
  {
  Figure 13c The quiver and the functions for $\Conf_3 \A_{Sp_{8}}$ after performing the sequence of mutations of $x_{ij}$ having maximums $1,  2, 3$.
  };

\end{tikzpicture}
\end{center}

In Figure 14, we depict the state of the quiver after performing the sequence of mutations of $x_{ij}$ having maximums $1, 2, 3$; $1, 2, 3, 1, 2$; and $1, 2, 3, 1, 2, 1$.

\begin{center}
\begin{tikzpicture}[scale=2]

\begin{scope}[xshift=-1.5cm]

  \node (x01) at (1,0) {$\tcfr{5}{}{3}$};
  \node (x02) at (2,0) {$\dud{2}{6}$};
  \node (x03) at (3,0) {$\dud{3}{5}$};
  \node (x04) at (4,0) {$\dud{4}{4}$};
  \node (x11) at (1,-1) {$\boldsymbol{\tcfr{1,6}{6}{3}}$};
  \node (x12) at (2,-1) {$\boldsymbol{\tcfr{1,5}{6}{4}}$};
  \node (x13) at (3,-1) {$\tcfr{2, 5}{5}{4}$};
  \node (x14) at (4,-1) {$\tcfr{3, 5}{4}{4}$};
  \node (x21) at (1,-2) {$\boldsymbol{\tcfr{6}{7}{3}}$};
  \node (x22) at (2,-2) {$\boldsymbol{\tcfr{6}{6}{4}}$};
  \node (x23) at (3,-2) {$\tcfr{1, 6}{5}{4}$};
  \node (x24) at (4,-2) {$\tcfr{2, 6}{4}{4}$};
  \node (x31) at (1,-3) {$\tcfr{7}{7}{2}$};
  \node (x32) at (2,-3) {$\tcfr{7}{6}{3}$};
  \node (x33) at (3,-3) {$\tcfr{7}{5}{4}$};
  \node (x34) at (4,-3) {$\tcfr{1, 7}{4}{4}$};
  \node (x41) at (1,-4) {$\tcfr{}{7}{1}$};
  \node (x42) at (2,-4) {$\tcfr{}{6}{2}$};
  \node (x43) at (3,-4) {$\tcfr{}{5}{3}$};
  \node (x44) at (4,-4) {$\tcfr{}{4}{4}$};
  \node (y1) at (0,-3) {$\tcfr{7}{}{1}$};
  \node (y2) at (0,-2) {$\tcfr{6}{}{2}$};
  \node (y3) at (0,-1) {$\dud{1}{7}$};
  \node (y4) at (5,-1) {$\tcfr{4}{}{4}$};

  \draw [->] (x11) to (x01);
  \draw [->] (x21) to (x11);
  \draw [->] (x21) to (x31);
  \draw [->] (x31) to (x41);
  \draw [->] (x02) to (x12);
  \draw [->] (x22) to (x32);
  \draw [->] (x32) to (x42);
  \draw [->] (x03) to (x13);
  \draw [->] (x13) to (x23);
  \draw [->] (x23) to (x33);
  \draw [->] (x33) to (x43);
  \draw [->] (x04) to (x14);
  \draw [->] (x14) to (x24);
  \draw [->] (x24) to (x34);
  \draw [->] (x34) to (x44);
  \draw [->, dashed] (y4) .. controls +(up:2) and +(up:1) .. (x01);
  \draw [->, dashed] (x01) .. controls +(left:2) and +(left:1) .. (y2);
  \draw [->, dashed] (y2) to (y1);

  \draw [->, dashed] (y3) .. controls +(up:2) and +(up:1) .. (x02);
  \draw [->, dashed] (x02) to (x03);
  \draw [->, dashed] (x03) to (x04);
  \draw [->] (x11) to (y3);
  \draw [->] (x12) to (x11);
  \draw [->] (x12) to (x13);
  \draw [->] (x13) to (x14);
  \draw [->] (x14) to (y4);
  \draw [->] (y2) to (x21);
  \draw [->] (x22) to (x23);
  \draw [->] (x23) to (x24);
  \draw [->] (y1) to (x31);
  \draw [->] (x31) to (x32);
  \draw [->] (x32) to (x33);
  \draw [->] (x33) to (x34);
  \draw [->, dashed] (x41) to (x42);
  \draw [->, dashed] (x42) to (x43);
  \draw [->, dashed] (x43) to (x44);

  \draw [->] (x01) to (x12);
  \draw [->] (x13) to (x02);
  \draw [->] (x14) to (x03);
  \draw [->] (y4) to (x04);
  \draw [->] (y3) to (x21);
  \draw [->] (x11) to (x22);
  \draw [->] (x23) to (x12);
  \draw [->] (x24) to (x13);
  \draw [->] (x31) to (y2);
  \draw [->] (x32) to (x21);
  \draw [->] (x33) to (x22);
  \draw [->] (x34) to (x23);
  \draw [->] (x41) to (y1);
  \draw [->] (x42) to (x31);
  \draw [->] (x43) to (x32);
  \draw [->] (x44) to (x33);

\end{scope}

\draw[yshift=-5cm,xshift=1cm]
  node[below,text width=6cm] 
  {
  Figure 14a The quiver and the functions for $\Conf_3 \A_{Sp_{8}}$ after performing the sequence of mutations of $x_{ij}$ having maximums $1,  2, 3, 1, 2$.
  };

\end{tikzpicture}
\end{center}

\begin{center}
\begin{tikzpicture}[scale=2]

\begin{scope}[xshift=-1.5cm]

  \node (x01) at (1,0) {$\dud{1}{7}$};
  \node (x02) at (2,0) {$\dud{2}{6}$};
  \node (x03) at (3,0) {$\dud{3}{5}$};
  \node (x04) at (4,0) {$\dud{4}{4}$};
  \node (x11) at (1,-1) {$\boldsymbol{\tcfr{5}{7}{4}}$};
  \node (x12) at (2,-1) {$\tcfr{1,5}{6}{4}$};
  \node (x13) at (3,-1) {$\tcfr{2, 5}{5}{4}$};
  \node (x14) at (4,-1) {$\tcfr{3, 5}{4}{4}$};left
  \node (x21) at (1,-2) {$\tcfr{6}{7}{3}$};
  \node (x22) at (2,-2) {$\tcfr{6}{6}{4}$};
  \node (x23) at (3,-2) {$\tcfr{1, 6}{5}{4}$};
  \node (x24) at (4,-2) {$\tcfr{2, 6}{4}{4}$};
  \node (x31) at (1,-3) {$\tcfr{7}{7}{2}$};
  \node (x32) at (2,-3) {$\tcfr{7}{6}{3}$};
  \node (x33) at (3,-3) {$\tcfr{7}{5}{4}$};
  \node (x34) at (4,-3) {$\tcfr{1, 7}{4}{4}$};
  \node (x41) at (1,-4) {$\tcfr{}{7}{1}$};
  \node (x42) at (2,-4) {$\tcfr{}{6}{2}$};
  \node (x43) at (3,-4) {$\tcfr{}{5}{3}$};
  \node (x44) at (4,-4) {$\tcfr{}{4}{4}$};
  \node (y1) at (0,-3) {$\tcfr{7}{}{1}$};
  \node (y2) at (0,-2) {$\tcfr{6}{}{2}$};
  \node (y3) at (0,-1) {$\tcfr{5}{}{3}$};
  \node (y4) at (5,-1) {$\tcfr{4}{}{4}$};

  \draw [->] (x01) to (x11);
  \draw [->] (x11) to (x21);
  \draw [->] (x21) to (x31);
  \draw [->] (x31) to (x41);
  \draw [->] (x02) to (x12);
  \draw [->] (x12) to (x22);
  \draw [->] (x22) to (x32);
  \draw [->] (x32) to (x42);
  \draw [->] (x03) to (x13);
  \draw [->] (x13) to (x23);
  \draw [->] (x23) to (x33);
  \draw [->] (x33) to (x43);
  \draw [->] (x04) to (x14);
  \draw [->] (x14) to (x24);
  \draw [->] (x24) to (x34);
  \draw [->] (x34) to (x44);
  \draw [->, dashed] (y4) .. controls +(up:2) and +(up:2) .. (y3);
  \draw [->, dashed] (y3) to (y2);
  \draw [->, dashed] (y2) to (y1);

  \draw [->, dashed] (x01) to (x02);
  \draw [->, dashed] (x02) to (x03);
  \draw [->, dashed] (x03) to (x04);
  \draw [->] (y3) to (x11);
  \draw [->] (x11) to (x12);
  \draw [->] (x12) to (x13);
  \draw [->] (x13) to (x14);
  \draw [->] (x14) to (y4);
  \draw [->] (y2) to (x21);
  \draw [->] (x21) to (x22);
  \draw [->] (x22) to (x23);
  \draw [->] (x23) to (x24);
  \draw [->] (y1) to (x31);
  \draw [->] (x31) to (x32);
  \draw [->] (x32) to (x33);
  \draw [->] (x33) to (x34);
  \draw [->, dashed] (x41) to (x42);
  \draw [->, dashed] (x42) to (x43);
  \draw [->, dashed] (x43) to (x44);

  \draw [->] (x12) to (x01);
  \draw [->] (x13) to (x02);
  \draw [->] (x14) to (x03);
  \draw [->] (y4) to (x04);
  \draw [->] (x21) to (y3);
  \draw [->] (x22) to (x11);
  \draw [->] (x23) to (x12);
  \draw [->] (x24) to (x13);
  \draw [->] (x31) to (y2);
  \draw [->] (x32) to (x21);
  \draw [->] (x33) to (x22);
  \draw [->] (x34) to (x23);
  \draw [->] (x41) to (y1);
  \draw [->] (x42) to (x31);
  \draw [->] (x43) to (x32);
  \draw [->] (x44) to (x33);

\end{scope}

\draw[yshift=-5cm,xshift=1cm]
  node[below,text width=6cm] 
  {
  Figure 14b The quiver and the functions for $\Conf_3 \A_{Sp_{8}}$ after performing the sequence of mutations of $x_{ij}$ having maximums $1, 2, 3, 1, 2, 1$.
  };

\end{tikzpicture}
\end{center}

From these diagrams the various quivers in the general case of $\Conf_3 \A_{Sp_{2n}}$ should be clear.

\begin{theorem}
If $\max(i,j)=k$, then $x_{ij}$ is mutated a total of $n-k$ times. Recall that when $i \geq j$, we assign the function $\tcfr{n-i}{n+i-j}{j}$ to $x_{ij}$. Thus the function attached to $x_{ij}$ transforms as follows:
$$\tcfr{n-i}{n+i-j}{j} \rightarrow \tcfr{2n-1, n-i-1}{n+i-j+1}{j+1} \rightarrow \tcfr{2n-2, n-i-2}{n+i-j+2}{j+2} \rightarrow \dots$$ 
$$\rightarrow \tcfr{n+i+1, 1}{2n-j-1}{n-i+j-1} \rightarrow \tcfr{n+i}{2n-j}{n-i+j}=\tcfr{n-i}{j}{n+i-j}$$

When $i < j$ and $i \neq 0$, we assign the function $\tcfr{n-i, 2n+i-j}{n}{j}$ to $x_{ij}$. Thus the function attached to $x_{ij}$ transforms as follows:
$$\tcfr{n-i, 2n+i-j}{n}{j} \rightarrow \tcfr{n-i-1, 2n+i-j-1}{n+1}{j+1} \rightarrow \tcfr{n-i-2, 2n+i-j-2}{n+2}{j+2} \rightarrow \dots$$
$$\rightarrow \tcfr{j-i+1, n+i+1}{2n-j-1}{n-1} \rightarrow \tcfr{j-i, n+i}{2n-j}{n}=\tcfr{n-i, 2n+i-j}{j}{n}$$

\end{theorem}

\begin{proof}

We will give two approaches. One is conceptually transparent, but hard to carry out. The second it more direct.

The first approach would be to use what we know about the cactus sequence. Recall that if we start with the standard cluster structure on $\Conf_3 \A_{SL_{2n}}$ and perform the mutations

$$x_{N-2,1,1}, x_{N-3,1,2}, \dots, x_{1,1,N-2}$$
$$x_{N-3,2,1}, x_{N-4,2,2}, \dots, x_{1,2,N-3}$$
$$x_{N-2,1,1}, x_{N-3,1,2}, \dots, x_{2,1,N-3}$$
$$x_{N-4,3,1}, x_{N-5,3,2}, \dots, x_{1,3,N-4}$$
$$x_{N-3,2,1}, x_{N-4,2,2}, \dots, x_{2,2,N-4}$$
$$x_{N-2,1,1}, x_{N-3,1,2}, \dots, x_{3,1,N-4}$$
$$\dots$$
$$x_{n,n-1,1}, x_{n-1,n-1,2}, \dots, x_{1,n-1,n}$$
$$\dots$$
$$x_{N-3,2,1}, x_{N-4,2,2}, \dots, x_{n-2,2,n}$$
$$x_{N-2,1,1}, x_{N-3,1,2}, \dots, x_{n-1,1,n}$$
then if we identify the vertices $x_{ijk}$ and $x_{ikj}$, we obtain the cluster algebra structure on $\Conf_3 \A_{Sp_{2n}}$. Recall that for $j \leq k$ and $j \geq n$, we have that the vertex $x_{ijk}$ corresponds to the vertex $x_{n-i,k}$ in the quiver for $\Conf_3 \A_{Sp_{2n}}$, while for $j \leq k$ and $j < n$, we have that the vertex $x_{ijk}$ corresponds to the vertex $x_{n-k,n+k-j}$.

If we then apply the sequence of mutations \eqref{23Sp}
$$x_{11}, x_{21}, x_{22}, x_{12}, x_{13}, x_{23}, x_{33}, x_{32}, x_{33,} \dots, x_{1,n-1}, \dots, x_{n-1, n-1}, \dots, x_{n-1, 1},$$
$$x_{11}, x_{21}, x_{22}, x_{12},  \dots x_{1,n-2}, \dots, x_{n-2, n-2}, \dots, x_{n-2, 1},$$
$$\dots,$$
$$x_{11}, x_{21}, x_{22}, x_{12},$$
$$x_{11}$$
on the cluster algebra for $\Conf_3 \A_{Sp_{2n}}$, we can lift this sequence of mutations to the unfolded cluster algebra for $\Conf_3 \A_{SL_{2n}}$. There, once we keep track of indices, we find that the cluster we end up with is exactly the result of applying the sequence of mutations 

$$x_{N-2,1,1}, x_{N-3,2,1}, \dots, x_{1,N-2,1}$$
$$x_{N-3,1,2}, x_{N-4,2,2}, \dots, x_{1,N-3,2}$$
$$x_{N-2,1,1}, x_{N-3,2,1}, \dots, x_{1,N-2,1}$$
$$x_{N-4,1,3}, x_{N-5,2,3}, \dots, x_{1,N-4,3}$$
$$x_{N-3,1,2}, x_{N-4,2,2}, \dots, x_{1,N-3,2}$$
$$x_{N-2,1,1}, x_{N-3,2,1}, \dots, x_{1,N-2,1}$$
$$\dots$$
$$x_{n,1,n-1}, x_{n-1,2,n-1}, \dots, x_{1,n,n-1}$$
$$\dots$$
$$x_{N-3,1,2}, x_{N-4,2,2}, \dots, x_{1,N-3,2}$$
$$x_{N-2,1,1}, x_{N-3,2,1}, \dots, x_{1,N-2,1}$$
to the original seed for $\Conf_3 \A_{SL_{2n}}$. 

Again, we then identify the vertices $x_{ijk}$ and $x_{ikj}$ and look downstairs on $\Conf_3 \A_{Sp_{2n}}$. Then the roles of $j$ and $k$ are exactly reversed: For $k \leq j$ and $k \geq n$, we have that the vertex $x_{ijk}$ corresponds to the vertex $x_{n-i,j}$ in the transposed quiver for $\Conf_3 \A_{Sp_{2n}}$, while for $k \leq j$ and $l < n$, we have that the vertex $x_{ijk}$ corresponds to the vertex $x_{n-j,n+j-k}$ in the transposed quiver.

Alternatively, and perhaps more simply, it is enough to observe that all that mutations in the sequence \eqref{23Sp} reduce to one of the cluster transformations depicted in Figure 11, and that all the identities involved are exactly those listed in our identification of the functions involved in the cactus sequence of mutations.

\end{proof}

Now consider the original seed for the cluster algebra structure on $\Conf_3 \A_{Sp_{2n}}$, where the quiver had vertices $x_{ij}, y_k$ for $0 \leq i \leq n$, $1 \leq j, k \leq n$. Now consider the seed that has quiver with vertices $x'_{ij}, y'_k$ that comes from exchanging the roles of the second and third principal flags. Then the sequence of mutations explained mutates the functions associated to $x_{ij}$ into those associated to $x'_{ij}$ for $i > 0$. The functions $x_{0j}$ are equal to the functions $y'_j$, and the functions $y_{j}$ are equal to the functions $x'_{0j}$. Moreover, the associated quiver on the vertices $x_{ij}, y_k$ mutates into the corresponding quiver on the vertices $x'_{ij}, y'_k$. Thus, this sequence of mutations relates the two seeds that come from the transposition of the second and third flags.

\subsubsection{The second transposition}

Let us now give the sequence of mutations that realizes that $S_3$ symmetry $(A,B,C) \rightarrow (C,B,A).$

The sequence of mutations is as follows:
\begin{equation} \label{13Sp}
\begin{gathered}
x_{n-1,n}, x_{n-2,n-1}, x_{n-2,n}, x_{n-3,n-2}, x_{n-3,n-1}, x_{n-3,n}, \dots, x_{1,2}, \dots, x_{1,n}, \\
x_{n-1,n}, x_{n-2,n-1}, x_{n-2,n}, x_{n-3,n-2}, x_{n-3,n-1}, x_{n-3,n}, \dots, x_{2,3}, \dots, x_{2,n}, \\
x_{n-1,n}, x_{n-2,n-1}, x_{n-2,n}, x_{n-3,n-2}, x_{n-3,n-1}, x_{n-3,n}, \dots, x_{3,4}, \dots, x_{3,n}, \\
\dots
x_{n-1,n}, x_{n-2,n-1}, x_{n-2,n}, x_{n-3,n-2}, x_{n-3,n-1}, x_{n-3,n}, \\
x_{n-1,n}, x_{n-2,n-1}, x_{n-2,n} \\
x_{n-1,n} \\
\end{gathered}
\end{equation}

The sequence can be thought of as follows: We only mutate those $x_{ij}$ with $i < j$. At any step of the process, we mutate all $x_{ij}$ in the $k^{\textrm th}$ row (the $k^{\textrm th}$ row consists of $x_{ij}$ such that $i=k$) such that $i < j$. It will not matter in which order we mutate these $x_{ij}$. The sequence of rows that we mutate is
$$n-1, n-2 \dots, 2,1 n-1, n-2, \dots, 2, n, \dots, 3, \dots, n-1, n-2, n-1.$$ 

In Figure 15, we depict how the quiver for $\Conf_3 \A_{Sp_{10}}$ changes after performing the sequence of mutations of $x_{ij}$ in rows $4$; $4,3$; $4,3,2$; and $4,3,2,1$. We only depice those functions that are affected by this mutation sequence.

\begin{center}

\end{center}

In Figure 16, we depict the state of the quiver after performing the sequence of mutations of $x_{ij}$ in rows$4,3,2,1,4,3,2$; $4,3,2,1,4,3,2,4,3$; and $4,3,2,1,4,3,2,4,3,4$.

\begin{center}
%
\end{center}

The circle on several of the arrows depicts how the quiver for $\Conf_3 \A_{Sp_{10}}$ lifts to $\Conf_3 \A_{SL_{10}}$. We will elaborate upon this further when we analyze how the functions transform. From these diagrams the various quivers in the general case of $\Conf_3 \A_{Sp_{2n}}$ should be clear. 

Now we need to understand the various functions attached to the vertices of the quiver at different stages in the sequence of mutations. In order to do this, we need to define some new functions.

Let $N=2n$. Let $0 \leq a, b, c, d \leq N$ such that $a, c \leq n$, $b,d \geq n$, and $a+b+c+d=4n=2N$. Then we would like to define a function that we will call 
$$\tcfr{a,b}{n, n}{c,d}.$$
This is a function on $\Conf_3 \A_{Sp_{2n}}$ that is pulled back from a function on $\Conf_3 \A_{SL_{2n}}$. The function on $\Conf_3 \A_{SL_{N}}$ is given by an invariant in the space

$$[V_{\omega_a+\omega_b}  \otimes V_{2\omega_n}  \otimes V_{\omega_c+ \omega_d}]^{SL_N}.$$
It turns out that this is a one-dimensional vector space. We pick out the function given by the web in Figure 17:

\begin{center}
\begin{tikzpicture}[scale=0.6]
\begin{scope}[decoration={
    markings,
    mark=at position 0.5 with {\arrow{>}}}
    ] 
\draw [postaction={decorate}] (-8,6) -- (-6,3) node [midway,below left] {$a$}; 
\draw [postaction={decorate}] (-4,0) -- (-6,3) node [midway,below left] {$N-a$}; 
\draw [postaction={decorate}] (-2,9) -- (0,6) node [midway,below left] {$b$};
\draw [postaction={decorate}] (2,3) -- (0,6) node [midway,below left] {$N-b$};
\draw [postaction={decorate}] (-8,-6) -- (-4,0) node [midway,below right] {$n$};
\draw [postaction={decorate}] (-2,-9) -- (2,-3) node [midway,below right] {$n$};
\draw [postaction={decorate}] (0,0) -- (-4,0) node [midway,below] {$n-a$};
\draw [postaction={decorate}] (14,-3) -- (2,-3) node [midway,above] {$c$};
\draw [postaction={decorate}] (2,-3) -- (0,0) node [midway,below left] {$n+c$};
\draw [postaction={decorate}] (0,0) -- (2,3) node [midway,below right] {$a+c$};
\draw [postaction={decorate}] (14,3) -- (8,3) node [midway,above] {$d$};
\draw [postaction={decorate}] (2,3) -- (8,3) node [midway,above] {$N-d$};

\draw (8,2.5) -- (8,3) ;
\draw (-6.3, 2.8) -- (-6,3) ;
\draw (-0.3, 5.8) -- (0,6) ;

\end{scope}

\draw[yshift=-10cm]
  node[below,text width=6cm] 
  {
  Figure 17. Web for the function $\tcfr{a,b}{n,n}{c,d}$ where $a+b+c+d=2N$.
  };

\end{tikzpicture}
\end{center}

Let us give a concrete description of this function.

Given three flags 
$$u_1, \dots, u_N;$$
$$v_1, \dots, v_N;$$
$$w_1, \dots, w_N;$$
first consider the forms
$$U_a := u_1 \wedge \cdots \wedge u_a,$$
$$U_b := u_1 \wedge \cdots \wedge u_b,$$
$$V_n := v_1 \wedge \cdots \wedge v_n,$$ 
$$W_c := w_1 \wedge \cdots \wedge w_c,$$
$$W_d := w_1 \wedge \cdots \wedge w_d.$$

Because $n+c=2N-b-d+n-a$, there is a natural map 
$$\phi_{N-d, N-b,n-a}: \bigwedge\nolimits^{n+c} V \rightarrow \bigwedge\nolimits^{N-d} V \otimes \bigwedge\nolimits^{N-b} V \otimes \bigwedge\nolimits^{n-a} V.$$
Note that $\phi_{N-d, N-b,n-a}$ can be thought of as the composite of $\phi_{2N-d-b,n-a}$ and $\phi_{N-d, N-b}$. There are also natural maps 
$$- \wedge W_d : \bigwedge\nolimits^{N-d} V \rightarrow \bigwedge\nolimits^{N} V \simeq F,$$
$$U_b \wedge - : \bigwedge\nolimits^{N-b} V \rightarrow \bigwedge\nolimits^{N} V \simeq F,$$ and
$$U_a \wedge V_n \wedge - :  \bigwedge\nolimits^{n-a} V \rightarrow \bigwedge\nolimits^{N} V \simeq F.$$ 
Applying these maps to the first, second, and third factors of $\phi_{N-d, N-b,n-a}(V_n \wedge W_c)$, respectively, and then multiplying, we get get the value of our function. This is a function on $\Conf_3 \A_{SL_{N}}$. Pulling back gives a function on $\Conf_3 \A_{Sp_{2n}}$. This defines the function $\tcfr{a,b}{n, n}{c,d}$ on $\Conf_3 \A_{SL_{N}}$ and $\Conf_3 \A_{Sp_{2n}}$.

We now define the function $\tcfr{a}{b}{c, d}$, which is the twisted cyclic shift of the function $\tcfr{c, d}{a}{b}$ defined earlier. Recall the map 
$$T: \Conf_3 \A  \rightarrow \Conf_3 \A$$
which acts on three principal flags $A, B, C$ as follows:
$$(A,B,C) \rightarrow (s_G \cdot C, A, B).$$
Then 
$$\tcfr{a}{b}{c, d} := T^*\tcfr{c, d}{a}{b}.$$

We are now ready to analyze for the functions in our cluster algebra transform under the sequence of mutations \eqref{13Sp}. Note that in the above sequence of mutations \eqref{13Sp}, $x_{ij}$ is mutated $i$ times if $i<j$.

\begin{theorem}
If $i < j$, then $x_{ij}$ is mutated a total of $i$ times. Recall that when $i < j$, we assign the function $\tcfr{n-i, 2n+i-j}{n}{j}$ to $x_{ij}$. The function attached to $x_{ij}$ transforms as follows:
$$\tcfr{n-i, 2n+i-j}{n}{j} \rightarrow \tcfr{n-i+1, 2n+i-j-1}{n, n}{j-1, n+1} \rightarrow $$
$$\tcfr{n-i+2, 2n+i-j-2}{n, n}{j-2, n+2} \rightarrow \dots \rightarrow \tcfr{n-1, 2n-j+1}{n, n}{j-i+1, n+i-1} $$
$$\rightarrow \tcfr{2n-j}{n}{j-i, n+i}=\tcfr{j}{n}{n-i, 2n+i-j}$$

The first transformation can be seen as the composite of two steps,
$$\tcfr{n-i, 2n+i-j}{n}{j} \rightarrow \tcfr{n-i, 2n+i-j}{n, n}{j, n} \rightarrow \tcfr{n-i+1, 2n+i-j-1}{n, n}{j-1, n+1},$$
while the last transformation can also be seen as the composite of two steps, 
$$\tcfr{n-1, 2n-j+1}{n, n}{j-i+1, n+i-1} \rightarrow \tcfr{n, 2n-j}{n, n}{j-i, n+i} \rightarrow \tcfr{2n-j}{n}{j-i, n+i}.$$
Then with each transformation, two of the parameters increase by one, and two decrease by one.

\end{theorem}

\begin{proof}

We have already described the quivers at the various stages of mutation. We must then check that the functions above satisfy the identities of the associated cluster transformations.

This is the first sequence of mutations where we have to mutate the white vertices $x_{in}$ ($y_n$ is also a white vertex, but it does not get mutated). For a black vertex $x_{ij}$, recall that the formula for mutation of says that

$$(\textrm{old function attached to } x_{ij})(\textrm{new function attached to } x_{ij})=$$
$$\prod (\textrm{functions attached to incoming arrows}) + \prod (\textrm{functions attached to outgoing arrows}).$$
However, for black vertices, we have
$$(\textrm{old function attached to } x_{ij})(\textrm{new function attached to } x_{ij})=$$
$$\prod (\textrm{incoming arrows from white vertices})(\textrm{incoming arrows from black vertices})^2 + $$
$$\prod (\textrm{outgoing arrows from white vertices})(\textrm{outgoing arrows from black vertices})^2.$$
In other words, arrows going from black to white vertices are weighted by $2$ when calculating the mutation of a white vertex.

In verifying the cluster identities that we need, we will actually be computing functions on $\Conf_3 \A_{SL_N}$. Thus we will use the arrows with a circle on them as a bookkeeping device. Circled arrows only occur between two black vertices. If a black vertex $x$ has an incoming or outgoing circled arrow from another black vertex $x'$, this means that in computing the mutation of $x$, we should use the dual of the function attached to $x'$.

The reason for this is as follows. We can lift the quiver $Q$ for $\Conf_3 \A_{Sp_N}$ to a quiver $\tilde{Q}$ for $\Conf_3 \A_{SL_N}$ using the following rule. For each vertex black vertex $x$, there are two lifts of this vertex in $\tilde{Q}$, call them $x^*$ and $x^{**}$. Then for every regular arrow between two black vertices $x$ and $x'$ in the quiver $Q$, we lift it to the correponding arrow between $x^*$ and $x'^*$ and between $x^{**}$ and $x'^{**}$. On the other hand, for every arrow with a circle between two black vertices $x$ and $x'$ in the quiver $Q$, we lift it to the correponding arrow between $x^{**}$ and $x'^*$ and between $x^{*}$ and $x'^{**}$.

Now suppose $x$ is a black vertex and $x'$ is a white vertex. $x$ has lifts $x^*$ and $x^{**}$, while $x'$ has just one lift $x'^*$. Then an arrow between $x$ and $x'$ in the quiver $Q$ lifts to the corresponding arrows between $x^*$ and $x'^*$ and between $x^{**}$ and $x'^*$. Hence, when computing the mutation of $x'$, we use the product of the function attached to $x$ and the dual of this function.

With these rules in mind, we are ready to compute mutations.

We need the following facts:
\begin{itemize}
\item Let $1 \leq a, b, c, d \leq N$, and $a+b+c+d=2N$. 
$$\tcfr{a, b}{n, n}{c, d}\tcfr{a+1, b-1}{n, n}{c-1, d+1}=$$
$$\tcfr{a, b}{n, n}{c-1, d+1}\tcfr{a+1, b-1}{n, n}{c, d}+\tcfr{a+1,b}{n, n}{c-1,d}\tcfr{a, b-1}{n, n}{c, d+1}.$$
\item If $a+c=n$ and $b+d=3n$,
$$\tcfr{a, b}{n, n}{c, d}=\tcfr{a}{n}{c}\tcfr{b}{n}{d}.$$
\item If $a=n$,
$$\tcfr{n, b}{n, n}{c, d}=\dud{n}{n}\tcfr{b}{n}{c, d}.$$
Similarly, we have
$$\tcfr{a, n}{n, n}{c, d}=\dud{n}{n}\tcfr{a}{n}{c, d},$$
$$\tcfr{a, b}{n, n}{n, d}=\tcfr{}{n}{n}\tcfr{a, b}{n}{d},$$
$$\tcfr{a, b}{n, n}{c, n}=\tcfr{}{n}{n}\tcfr{a, b}{n}{c}.$$
\item We will need the duality identities of \eqref{dualities}, and also the following duality identity:
$$\tcfr{a, b}{n, n}{c, d}=\tcfr{N-b, N-a}{n, n}{N-d, N-c}$$

\end{itemize}

All these identities can be checked by computing with webs. All identities but the first follow directly from the definitions. Let us give an alternative proof the first (and most important) identity.

Looking carefully at how $\tcfr{a, b}{n, n}{c, d}$ is defined on three flags $U, V, W$, one notices it really depends on four inputs: $U_a \wedge V_n, V_n \wedge W_c, W_d,$ and $U_b$.

The identity we are seeking actually comes from a more general identity on $\Conf_4 \A_{SL_N}$. Let $1 \leq a', b', c', d' \leq N$ such that $a'+b'+c'+d'=3N$. Consider four flags:

$$u_1, \dots, u_N;$$
$$v_1, \dots, v_N;$$
$$w_1, \dots, w_N;$$
$$x_1, \dots, x_N.$$
Also consider the forms
$$U_{a'} := u_1 \wedge \cdots \wedge u_{a'},$$
$$V_{b'} := v_1 \wedge \cdots \wedge v_{b'},$$ 
$$W_{c'} := w_1 \wedge \cdots \wedge w_{c'},$$
$$X_{d'} := x_1 \wedge \cdots \wedge x_{d'}.$$

We would like to define the function $\qcf{a'}{b'}{c'}{d'}$. There is a natural map 
$$\phi_{N-d', N-b',N-a'}: \bigwedge\nolimits^{c'} V \rightarrow \bigwedge\nolimits^{N-d'} V \otimes \bigwedge\nolimits^{N-b'} V \otimes \bigwedge\nolimits^{N-a'} V.$$
There are also natural maps 
$$- \wedge X_{d'} : \bigwedge\nolimits^{N-d'} V \rightarrow \bigwedge\nolimits^{N} V \simeq F,$$
$$V_{b'} \wedge - : \bigwedge\nolimits^{N-b'} V \rightarrow \bigwedge\nolimits^{N} V \simeq F,$$ and
$$U_{a'} \wedge - :  \bigwedge\nolimits^{N-a'} V \rightarrow \bigwedge\nolimits^{N} V \simeq F.$$ 
Applying these maps to the first, second, and third factors of $\phi_{N-d', N-b',N-a'}(W_{c'})$, respectively, and then multiplying, we get get the value of our function $\qcf{a'}{b'}{c'}{d'}$. 

The identity 
$$\qcf{a'}{b'}{c'}{d'}\qcf{a'-1}{b'+1}{c'-1}{d'+1}=$$
$$\qcf{a'}{b'}{c'-1}{d'+1}\tcfr{a'-1, b'+1}{c'}{d'}+\qcf{a'}{b'+1}{c'-1}{d'}\qcf{a'-1}{b'}{c'}{d'+1}$$
is simply the dualization of the octahedron recurrence identity.

If we specialize the above identity to the four flags
$$u_1, \dots, u_N;$$
$$u_1, \dots, u_n, v_1, \dots, v_n;$$
$$v_1, \dots, v_n, w_1, \dots w_n;$$
$$w_1, \dots, w_N;$$
with $a'=b, b'=a+n, c'=n+c, d'=d$, we get exactly the identity that we seek:
$$\tcfr{a, b}{n, n}{c, d}\tcfr{a+1, b-1}{n, n}{c-1, d+1}=$$
$$\tcfr{a, b}{n, n}{c-1, d+1}\tcfr{a+1, b-1}{n, n}{c, d}+$$
$$\tcfr{a+1,b}{n, n}{c-1,d}\tcfr{a, b-1}{n, n}{c, d+1}.$$

Most of the cluster mutations are given by precisely this identity.

All other cluster mutations are degenerations of this identity, and are obtained from this one by applying the other three identities. 

For example, the first three mutations for $n \geq 4$ are
\begin{itemize}
\item $$\tcfr{1, N-1}{n}{n}\tcfr{2, N-2}{n, n}{n-1,n+1}=$$
$$\tcfr{2, N-2}{n}{n}\tcfr{1}{n}{n-1}\tcfr{N-1}{n}{n+1}+\tcfr{}{n}{n}\tcfr{2, N-1}{n}{n-1}\tcfr{1, N-2}{n}{n+1},$$ 
which comes from
$$\tcfr{}{n}{n}\tcfr{1, N-1}{n}{n}\tcfr{2, N-2}{n, n}{n-1,n+1}=$$
$$\tcfr{1, N-1}{n, n}{n, n}\tcfr{2, N-2}{n, n}{n-1,n+1}=$$
$$\tcfr{2, N-2}{n, n}{n, n}\tcfr{1, N-1}{n, n}{n-1, n+1}+\tcfr{2, N-1}{n, n}{n-1, n}\tcfr{1, N-2}{n, n}{n, n+1}=$$
$$\tcfr{}{n}{n}\tcfr{2, N-2}{n}{n}\tcfr{1}{n}{n-1}\tcfr{N-1}{n}{n+1}+$$
$$\tcfr{}{n}{n}^2\tcfr{2, N-1}{n}{n-1}\tcfr{1, N-2}{n}{n+1};$$
\item $$\tcfr{2, N-2}{n}{n}\tcfr{3, N-3}{n, n}{n-1,n+1}=$$
$$\tcfr{3, N-3}{n}{n}\tcfr{2, N-2}{n, n}{n-1,n+1}+\tcfr{}{n}{n}\tcfr{3, N-2}{n}{n-1}\tcfr{2, N-3}{n}{n+1},$$ 
which comes from
$$\tcfr{}{n}{n}\tcfr{2, N-2}{n}{n}\tcfr{3, N-3}{n, n}{n-1,n+1}=$$
$$\tcfr{2, N-2}{n, n}{n, n}\tcfr{3, N-3}{n, n}{n-1,n+1}=$$
$$\tcfr{3, N-3}{n, n}{n, n}\tcfr{2, N-2}{n, n}{n-1,n+1}+\tcfr{3, N-2}{n, n}{n-1,n}\tcfr{2, N-3}{n, n}{n, n+1}=$$
$$\tcfr{}{n}{n}\tcfr{3, N-3}{n}{n}\tcfr{2, N-2}{n, n}{n-1,n+1}+\tcfr{}{n}{n}^2\tcfr{3, N-2}{n}{n-1}\tcfr{2, N-3}{n}{n+1};$$
\item $$\tcfr{2, N-1}{n}{n-1}\tcfr{3, N-2}{n, n}{n-2, n+1}=$$
$$\tcfr{3, N-2}{n}{n-1}\tcfr{N-1}{n}{n+1}(2, n, n-2)+\tcfr{3, N-1}{n}{n-2}\tcfr{2, N-2}{n, n}{n-1,n+1},$$ 
which comes from
$$\tcfr{}{n}{n}\tcfr{2, N-1}{n}{n-1}\tcfr{3, N-2}{n, n}{n-2, n+1}=$$
$$\tcfr{2, N-1}{n, n}{n-1, n}\tcfr{3, N-2}{n, n}{n-2, n+1}=$$
$$\tcfr{3, N-2}{n, n}{n-1, n}\tcfr{2, N-1}{n, n}{n-2, n+1}+\tcfr{3, N-1}{n, n}{n-2, n}\tcfr{2, N-2}{n, n}{n-1,n+1}=$$
$$\tcfr{}{n}{n}\tcfr{3, N-2}{n}{n-1}\tcfr{N-1}{n}{n+1}\tcfr{2}{n}{n-2}+$$
$$\tcfr{}{n}{n}\tcfr{3, N-1}{n}{n-2}\tcfr{2, N-2}{n, n}{n-1,n+1}.$$

\end{itemize}

\end{proof}

\subsection{The sequence of mutations for a flip}

Recall that to construct the cluster structure (a quiver plus the set of functions attached to the vertices) on $\Conf_4 \A_{Sp_{2n}}$, we choose a triangulation of the $4$-gon, as well as one of the six cluster structures on $\Conf_3 \A_{Sp_{2n}}$ for each of the resulting triangles. The previous section showed how to mutate between the six cluster structures on each triangle. In this section, we will give a sequence of mutations that relates two of the clusters coming from different triangulations of the $4$-gon. Combined with the previous section, this allows us to connect by mutations all $72$ different clusters we have constructed for $\Conf_4 \A_{Sp_{2n}}$ (for each of the two triangulations, we have $6 \cdot 6$ different clusters). This will allow us to connect by mutations all the clusters that we have constructed on $\Conf_m \A_{Sp_{2n}}$.

Given a configuration $(A,B,C,D) \in \Conf_4 \A_{Sp_{2n}}$, we will give a sequence of mutations that relates a cluster coming from the triangulation $ABC, ACD$ to a cluster coming from the triangulation $ABD, BCD$.

We will need to relabel the quiver with vertices $x_{ij}$, $y_k$, with $-n \leq i \leq n$, $1 \leq j \leq n$ and $1 \leq |k| \leq n$. The quiver we will start with is as in Figure 18, pictured for $Sp_6$.

\begin{center}
\begin{tikzpicture}[scale=2.4]
  \foreach \x in {-3,-2,-1,0,1,2,3}
    \foreach \y in {1, 2}
      \node[] (x\x\y) at (\x,-\y) {\Large $\ontop{x_{\x\y}}{\bullet}$};
  \foreach \x in {-3,-2,-1,0,1,2,3}
    \foreach \y in {3}
      \node[] (x\x\y) at (\x,-\y) {\Large $\ontop{x_{\x\y}}{\circ}$};
  \node (y-1) at (-0.5,0) {\Large $\ontop{y_{-1}}{\bullet}$};
  \node (y-2) at (-1.5,0) {\Large $\ontop{y_{-2}}{\bullet}$};
  \node (y-3) at (-2.5,-4) {\Large $\ontop{y_{-3}}{\circ}$};
  \node (y1) at (0.5,0) {\Large $\ontop{y_1}{\bullet}$};
  \node (y2) at (1.5,0) {\Large $\ontop{y_2}{\bullet}$};
  \node (y3) at (2.5,-4) {\Large $\ontop{y_{3}}{\circ}$};

  \draw [->] (x01) to (x11);
  \draw [->] (x11) to (x21);
  \draw [->] (x21) to (x31);
  \draw [->] (x02) to (x12);
  \draw [->] (x12) to (x22);
  \draw [->] (x22) to (x32);
  \draw [->] (x03) to (x13);
  \draw [->] (x13) to (x23);
  \draw [->] (x23) to (x33);
  \draw [->, dashed] (y1) to (y2);
  \draw [->, dashed] (y2) .. controls +(right:1) and +(up:1) .. (y3);

  \draw [->] (x03) to (x02);
  \draw [->] (x02) to (x01);

  \draw [->] (x13) to (x12);
  \draw [->] (x12) to (x11);
  \draw [->] (x11) to (y1);
  \draw [->] (y3) to (x23);
  \draw [->] (x23) to (x22);
  \draw [->] (x22) to (x21);
  \draw [->] (x21) to (y2);
  \draw [->, dashed] (x33) to (x32);
  \draw [->, dashed] (x32) to (x31);

  \draw [->] (y1) to (x01);
  \draw [->] (x11) to (x02);
  \draw [->] (x12) to (x03);
 \draw [->] (y2) to (x11);
  \draw [->] (x21) to (x12);
  \draw [->] (x22) to (x13);
  \draw [->] (x31) to (x22);
  \draw [->] (x32) to (x23);
  \draw [->] (x33) to (y3);

  \draw [->] (x01) to (x-11);
  \draw [->] (x-11) to (x-21);
  \draw [->] (x-21) to (x-31);
  \draw [->] (x02) to (x-12);
  \draw [->] (x-12) to (x-22);
  \draw [->] (x-22) to (x-32);
  \draw [->] (x03) to (x-13);
  \draw [->] (x-13) to (x-23);
  \draw [->] (x-23) to (x-33);
  \draw [->, dashed] (y-1) to (y-2);
  \draw [->, dashed] (y-2) .. controls +(left:1) and +(up:1) .. (y-3);

  \draw [->] (x-13) to (x-12);
  \draw [->] (x-12) to (x-11);
  \draw [->] (x-11) to (y-1);
  \draw [->] (y-3) to (x-23);
  \draw [->] (x-23) to (x-22);
  \draw [->] (x-22) to (x-21);
  \draw [->] (x-21) to (y-2);
  \draw [->, dashed] (x-33) to (x-32);
  \draw [->, dashed] (x-32) to (x-31);

  \draw [->] (y-1) to (x01);
  \draw [->] (x-11) to (x02);
  \draw [->] (x-12) to (x03);
  \draw [->] (y-2) to (x-11);
  \draw [->] (x-21) to (x-12);
  \draw [->] (x-22) to (x-13);
  \draw [->] (x-31) to (x-22);
  \draw [->] (x-32) to (x-23);
  \draw [->] (x-33) to (y-3);

\draw[yshift=-3.85cm]
  node[below,text width=6cm] 
  {
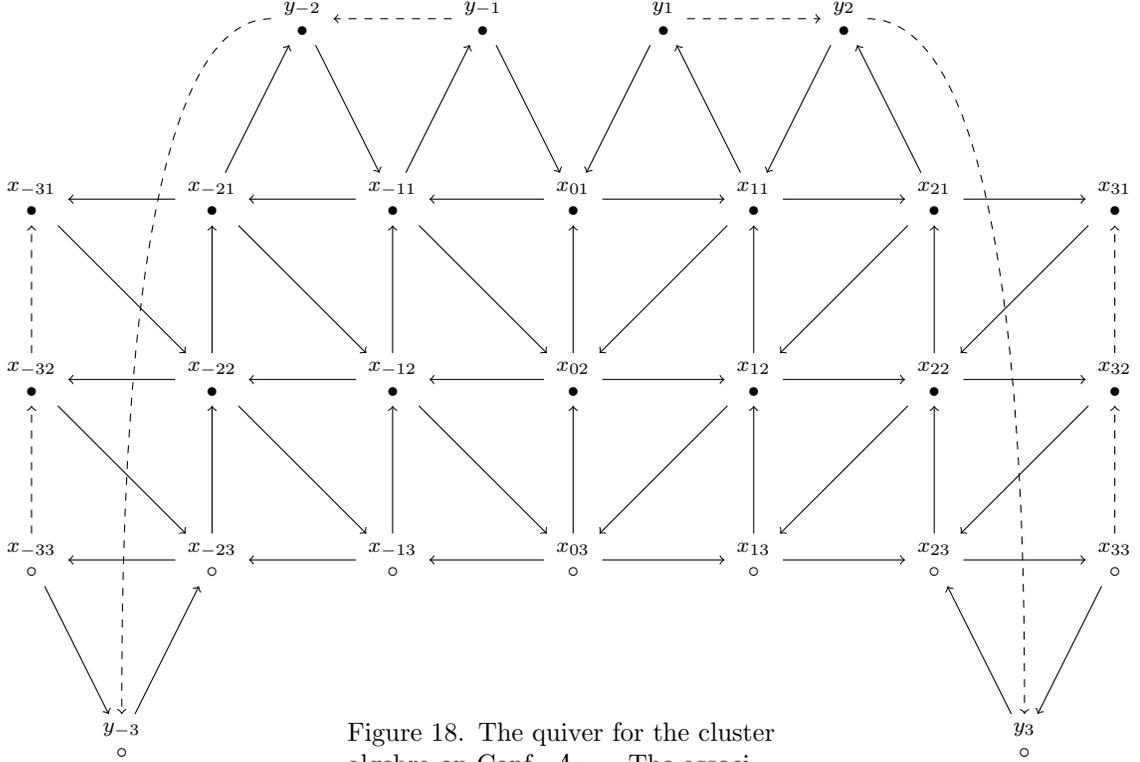
  Figure 18. The quiver for the cluster algebra on $\Conf_4 \A_{Sp_{6}}$. The associated functions are pictured in Figure 7.
  };

\end{tikzpicture}
\end{center}

The functions attached to these vertices are as follows. Let $N=2n$. The edge functions are
\begin{alignat*}{1}
\dur{k}{N-k} &\longleftrightarrow y_k, \textrm{ for } k >0; \\
\dld{|k|}{N-|k|} &\longleftrightarrow y_k, \textrm{ for } k <0; \\
\dul{j}{N-j} &\longleftrightarrow x_{-n,j}; \\
\ddr{N-j}{j} &\longleftrightarrow x_{nj}; \\
\dud{j}{N-j} &\longleftrightarrow x_{0j}. \\
\end{alignat*}
The face functions in the triangle where $i>0$ are
\begin{alignat*}{1}
\tcfr{i+j}{N-j}{N-i} &\longleftrightarrow x_{ij}, \textrm{ for } 0<i<n, i+j \leq n; \\
\tcfr{n}{N-j}{j+i-n, N-i} &\longleftrightarrow x_{ij}, \textrm{ for } 0<i<n, i+j > n; \\
\end{alignat*}
while the face functions in the triangle where $i<0$ are
\begin{alignat*}{1}
\tcfl{j}{|i|}{N-|i|-j} &\longleftrightarrow x_{ij}, \textrm{ for } -n<i<0, |i|+j \leq n; \\
\tcfl{j}{|i|, N+n-|i|-j}{n} &\longleftrightarrow x_{ij}, \textrm{ for } -n<i<0, |i|+j > n.
\end{alignat*}

Figure 7 displays these functions embedded within the quiver for $\Conf_4 \A_{Sp_{6}}$

\begin{rmk} Note that our labelling of the vertices is somewhat different from before. The vertices labelled $x_{ij}$ correspond to the vertices labelled $x_{n-|i|, j}$ in  $\Conf_3 \A_{Sp_{2n}}$.
\end{rmk}

First, let us define the functions above. Note that there are natural maps 
$$p_1, p_2, p_3, p_4: \Conf_4 \A_{Sp_{2n}} \rightarrow \Conf_3 \A_{Sp_{2n}}$$
that map a configuration $(A,B,C,D)$ to $(B, C, D)$, $(A, C, D)$, $(A, B, D)$, $(A,B,C)$, respectively . Pulling back functions from $\Conf_3 \A_{Sp_{2n}}$ allows us to define functions on $\Conf_4 \A_{Sp_{2n}}$. For example,
$$p_3^*\tcfr{n}{j}{n-i, 2n+i-j} =: \tcfu{n}{j}{n-i, 2n+i-j}.$$
Similarly, we can pull back functions from various maps 
$$\Conf_4 \A_{Sp_{2n}} \rightarrow \Conf_2 \A_{Sp_{2n}}$$
to define functions such as 
$$\dld{j}{N-j}.$$

There is also a map
$$T: \Conf_4 \A  \rightarrow \Conf_4 \A$$
which sends
$$(A,B,C,D) \rightarrow (s_G \cdot D, A, B, C)$$
which allows us to define, for example
$$T^*\tcfu{n}{j}{n-i, 2n+i-j} =: \tcfr{j}{n-i, 2n+i-j}{n}.$$ The forgetful maps and twist maps, combined with the functions described below, will furnish all the functions necessary for the computation of the flip mutation sequence.

More interestingly, we will have to use some functions which depend on all four flags. Let $N=2n$. Let $0 \leq a, b, c, d \leq N$ such that $a+b+c+d=4n=2N$ and $b+c \leq N$. Then we would like to define a function that we will call 
$$\qcfs{a}{b}{c}{d}.$$
Note that our notation uses a new symbol, ``:''. This is because the construction does not exhibit cyclic symmetry, i.e.
$$T^*\qcfs{a}{b}{c}{d} \neq \qcfs{b}{c}{d}{a}.$$
Instead, we use the notation
$$T^*\qcfs{a}{b}{c}{d} =: \qcfb{b}{c}{d}{a}.$$
We can also define
$$(T^2)^*\qcfs{a}{b}{c}{d}=:\qcfs{c}{d}{a}{b}.$$

The function $\qcfs{a}{b}{c}{d}$ on $\Conf_4 \A_{Sp_{2n}}$ is pulled back from a function on $\Conf_4 \A_{SL_{N}}$. The function on $\Conf_4 \A_{SL_{N}}$ is given by an invariant in the space
$$[V_{\omega_a}  \otimes V_{\omega_b}  \otimes V_{\omega_c} \otimes V_{\omega_d}]^{SL_N}.$$
It turns out that this is not always a one-dimensional vector space, so we will have to proceed with some care. We pick out the function given by the web in Figure 19:

\begin{center}
\begin{tikzpicture}[scale=1.4]
\begin{scope}[decoration={
    markings,
    mark=at position 0.5 with {\arrow{>}}},
    ] 
\draw [postaction={decorate}] (-4,-1) -- (-1,-1) node [midway,below] {$b$}; 
\draw [postaction={decorate}] (-1,-4) -- (-1,-1) node [midway,right] {$c$}; 
\draw [postaction={decorate}] (-1,-1) -- (0,0) node [midway,below right] {$b+c$}; 
\draw [postaction={decorate}] (1,1) -- (0,0) node [midway,above left] {$N-b-c$}; 
\draw [postaction={decorate}] (1,4) -- (1,1) node [midway,left] {$a$}; 
\draw [postaction={decorate}] (4,1) -- (2.5,1) node [midway,above] {$d$}; 
\draw [postaction={decorate}] (1,1) -- (2.5,1) node [midway,below] {$N-d$};

\draw (-0.1,0.1) -- (0,0);
\draw (2.5,0.9) -- (2.5,1);

\end{scope}

\draw[yshift=-4cm]
  node[below,text width=6cm] 
  {
  Figure 19 Web for the function $\qcfs{a}{b}{c}{d}$, where $a+b+c+d=2N$.
  };

\end{tikzpicture}
\end{center}

Let us give a concrete description of this function.

Given four flags 
$$u_1, \dots, u_N;$$
$$v_1, \dots, v_N;$$
$$w_1, \dots, w_N;$$
$$x_1, \dots, x_N;$$

first consider the forms
$$U_a := u_1 \wedge \cdots \wedge u_a,$$
$$V_b := v_1 \wedge \cdots \wedge v_b,$$
$$W_c := w_1 \wedge \cdots \wedge w_c,$$
$$X_d := x_1 \wedge \cdots \wedge x_d.$$

Because $a=2N-b-c-d$, there is a natural map 
$$\phi_{N-b-c, N-d}: \bigwedge\nolimits^{a} V \rightarrow \bigwedge\nolimits^{N-b-c} V \otimes \bigwedge\nolimits^{N-d} V.$$
There are also natural maps 
$$- \wedge V_b \wedge W_c  :  \bigwedge\nolimits^{N-b-c} V \rightarrow \bigwedge\nolimits^{N} V \simeq F.$$ 
$$- \wedge W_d : \bigwedge\nolimits^{N-d} V \rightarrow \bigwedge\nolimits^{N} V \simeq F,$$
Applying these maps to the first and second factors of $\phi_{N-b-c, N-d}(U_a)$, respectively, and then multiplying, we get get the value of our function. This is a function on $\Conf_4 \A_{SL_{N}}$. Pulling back gives a function on $\Conf_4 \A_{Sp_{2n}}$.

We now define a second type of function on  $\Conf_4 \A_{Sp_{2n}}$. Let $0 \leq a, b, c, d \leq N$ such that $a+b+c+d=5n$, and $c \leq d$. Then we would like to define a function that we will call 
$$\qcfs{n}{a}{b}{c, d}.$$

The function $\qcfs{n}{a}{b}{c, d}$ on $\Conf_4 \A_{Sp_{2n}}$ is pulled back from a function on $\Conf_4 \A_{SL_{2n}}$. It is given by an invariant in the space
$$[V_{\omega_n}  \otimes V_{\omega_a}  \otimes V_{\omega_b} \otimes V_{\omega_c+\omega_d}]^{SL_N}$$
picked out by the web in Figure 20a:

\begin{center}
\begin{tikzpicture}[scale=1.4]
\begin{scope}[decoration={
    markings,
    mark=at position 0.5 with {\arrow{>}}},
    ] 
\draw [postaction={decorate}] (-4,-1) -- (-2.5,-1) node [midway,below] {$a$}; 
\draw [postaction={decorate}] (-1,-1) -- (-2.5,-1) node [midway,below] {$N-a$}; 
\draw [postaction={decorate}] (-1,-4) -- (-1,-1) node [midway,right] {$b$}; 
\draw [postaction={decorate}] (-1,-1) -- (1,1) node [midway,below right] {$a+b-N$}; 

\draw [postaction={decorate}] (4,1) -- (2.5,1) node [midway,above] {$c$}; 
\draw [postaction={decorate}] (1,1) -- (2.5,1) node [midway,below] {$N-c$}; 
\draw [postaction={decorate}] (1,1) -- (1,2) node [midway,left] {$a+b+c-2N$}; 
\draw [postaction={decorate}] (4,2) -- (1,2) node [midway,below] {$d$}; 
\draw [postaction={decorate}] (1,2) -- (1,3) node [midway,left] {$a+b+c+d-2N$}; 
\draw [postaction={decorate}] (1,4) -- (1,3) node [midway,left] {$n$};

\draw (2.5,0.9) -- (2.5,1);
\draw (0.9,3) -- (1,3);
\draw (-2.5,-1.1) -- (-2.5,-1);

\end{scope}

\draw[yshift=-4cm]
  node[below,text width=8cm] 
  {
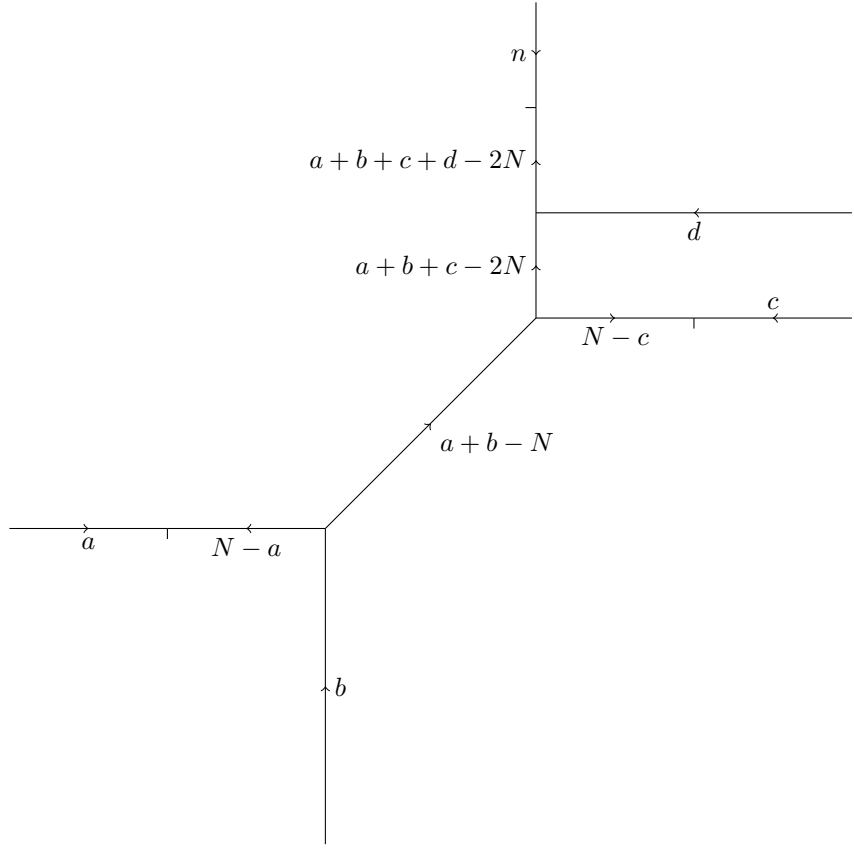
  Figure 20a Web for the function $\qcfs{n}{a}{b}{c, d}$, where $a+b+c+d=2N+n$.
  };

\end{tikzpicture}
\end{center}

Let us give a concrete description of this function.

Given four flags 
$$u_1, \dots, u_N;$$
$$v_1, \dots, v_N;$$
$$w_1, \dots, w_N;$$
$$x_1, \dots, x_N;$$

first consider the forms
$$U_n := u_1 \wedge \cdots \wedge u_n,$$
$$V_a := v_1 \wedge \cdots \wedge v_a,$$
$$W_b := w_1 \wedge \cdots \wedge w_b,$$
$$X_c := x_1 \wedge \cdots \wedge x_c.$$
$$X_d := x_1 \wedge \cdots \wedge x_d.$$

Because $b=2N+n-a-c-d$, there is a natural map 
$$\phi_{N-c, n-d, N-a}: \bigwedge\nolimits^{b} V \rightarrow \bigwedge\nolimits^{N-c} V \otimes \bigwedge\nolimits^{n-d} V \otimes \bigwedge\nolimits^{N-a} V.$$
There are also natural maps 
$$- \wedge X_{c} : \bigwedge\nolimits^{N-c} V \rightarrow \bigwedge\nolimits^{N} V \simeq F,$$
$$U_n \wedge - \wedge X_d: \bigwedge\nolimits^{n-d} V \rightarrow \bigwedge\nolimits^{N} V \simeq F,$$ and
$$V_a \wedge - :  \bigwedge\nolimits^{N-a} V \rightarrow \bigwedge\nolimits^{N} V \simeq F.$$ 
Applying these maps to the first, second, and third factors of $\phi_{N-c, n-d, N-a}(W_b)$, respectively, and then multiplying, we get get the value of our function $\qcfs{n}{a}{b}{c, d}$. Using the twist map $T$, we can also define the functions $\qcfb{a}{b}{c, d}{n}$, $\qcfs{b}{c, d}{n}{a}$, and $\qcfb{c, d}{n}{a}{b}$.

Using duality, there is also a function $\qcfs{n}{a}{b}{c, d}$ on $\Conf_4 \A_{Sp_{2n}}$ for $0 \leq a, b, c, d \leq N$, $a+b+c+d=N+n$, and $c \leq d$. 

The function $\qcfs{n}{a}{b}{c, d}$ on $\Conf_4 \A_{Sp_{2n}}$ is pulled back from a function on $\Conf_4 \A_{SL_{2n}}$. The function on $\Conf_4 \A_{SL_{N}}$ is given by an invariant in the space
$$[V_{\omega_n}  \otimes V_{\omega_a}  \otimes V_{\omega_b} \otimes V_{\omega_c+\omega_d}]^{SL_N}.$$
This vector space is generally multi-dimensional. To pick out the correct invariant, we use the web in Figure 20b:

\begin{center}
\begin{tikzpicture}[scale=1.4]
\begin{scope}[decoration={
    markings,
    mark=at position 0.5 with {\arrow{>}}},
    ] 
\draw [postaction={decorate}] (-4,-1) -- (-1,-1) node [midway,below] {$a$}; 
\draw [postaction={decorate}] (-1,-4) -- (-1,-1) node [midway,right] {$b$}; 
\draw [postaction={decorate}] (-1,-1) -- (1,1) node [midway,below right] {$a+b$}; 

\draw [postaction={decorate}] (4,1) -- (2.5,1) node [midway,above] {$c$}; 
\draw [postaction={decorate}] (1,1) -- (2.5,1) node [midway,below] {$N-c$}; 
\draw [postaction={decorate}] (1,1) -- (1,2) node [midway,left] {$a+b+c-N$}; 
\draw [postaction={decorate}] (4,2) -- (1,2) node [midway,below] {$d$}; 
\draw [postaction={decorate}] (1,2) -- (1,3) node [midway,left] {$a+b+c+d-N$}; 
\draw [postaction={decorate}] (1,4) -- (1,3) node [midway,left] {$n$};

\draw (2.5,0.9) -- (2.5,1);
\draw (0.9,3) -- (1,3);

\end{scope}

\draw[yshift=-4cm]
  node[below,text width=8cm] 
  {
  Figure 20b Web for the function $\qcfs{n}{a}{b}{c, d}$, where $a+b+c+d=N+n$.
  };

\end{tikzpicture}
\end{center}

Let us give a concrete description of this function. Again take

$$u_1, \dots, u_N;$$
$$v_1, \dots, v_N;$$
$$w_1, \dots, w_N;$$
$$x_1, \dots, x_N;$$
$$U_n := u_1 \wedge \cdots \wedge u_n,$$
$$V_a := v_1 \wedge \cdots \wedge v_a,$$
$$W_b := w_1 \wedge \cdots \wedge w_b,$$
$$X_c := x_1 \wedge \cdots \wedge x_c.$$
$$X_d := x_1 \wedge \cdots \wedge x_d.$$

Because $a+b=N+n-c-d$, there is a natural map 
$$\phi_{N-c, n-d}: \bigwedge\nolimits^{a+b} V \rightarrow \bigwedge\nolimits^{N-c} V \otimes \bigwedge\nolimits^{n-d} V.$$
There are also natural maps 
$$- \wedge X_{c} : \bigwedge\nolimits^{N-c} V \rightarrow \bigwedge\nolimits^{N} V \simeq F,$$ and
$$U_n \wedge - \wedge X_d: \bigwedge\nolimits^{n-d} V \rightarrow \bigwedge\nolimits^{N} V \simeq F.$$
Applying these maps to the first and second factors of $\phi_{N-c, n-d}(V_a \wedge W_b)$, respectively, and then multiplying, we get get the value of our function $\qcfs{n}{a}{b}{c, d}$. 

Using the twist map $T$, we can also define the functions $\qcfb{a}{b}{c, d}{n}$, $\qcfs{b}{c, d}{n}{a}$, and $\qcfb{c, d}{n}{a}{b}$.

We will need to define one more type of function to do our calculations. Let $N=2n$. Let $0 \leq a, b, c, d \leq N$ such that $a+b+c+d=4n=2N$, $a \leq n \leq b$ and $c \leq n \leq d$. Then we would like to define a function that we will call 
$$\qcfs{n}{a, b}{n}{c, d}.$$

The function $\qcfs{n}{a, b}{n}{c, d}$ on $\Conf_4 \A_{Sp_{2n}}$ is pulled back from a function on $\Conf_4 \A_{SL_{2n}}$. The function on $\Conf_4 \A_{SL_{N}}$ is given by an invariant in the space
$$[V_{\omega_n}  \otimes V_{\omega_a+\omega_b}  \otimes V_{\omega_n} \otimes V_{\omega_c+\omega_d}]^{SL_N}.$$
This vector space is generally multi-dimensional. To pick out the correct invariant, we use the web in Figure 21:

\begin{center}
\begin{tikzpicture}[scale=1.4]
\begin{scope}[decoration={
    markings,
    mark=at position 0.5 with {\arrow{>}}},
    ] 
\draw [postaction={decorate}] (-4,-1) -- (-2.5,-1) node [midway,below] {$b$}; 
\draw [postaction={decorate}] (-1,-1) -- (-2.5,-1) node [midway,below] {$N-b$}; 
\draw [postaction={decorate}] (-1,-4) -- (-1,-2) node [midway,right] {$n$}; 
\draw [postaction={decorate}] (-4,-2) -- (-1,-2) node [midway,below] {$a$}; 
\draw [postaction={decorate}] (-1,-2) -- (-1,-1) node [midway,right] {$a+n$}; 
\draw [postaction={decorate}] (-1,-1) -- (1,1) node [midway,below right] {$a+b+n-N$}; 

\draw [postaction={decorate}] (4,1) -- (2.5,1) node [midway,above] {$c$}; 
\draw [postaction={decorate}] (1,1) -- (2.5,1) node [midway,below] {$N-c$}; 
\draw [postaction={decorate}] (1,1) -- (1,2) node [midway,left] {$a+b+c+n-2N=n-d$}; 
\draw [postaction={decorate}] (4,2) -- (1,2) node [midway,below] {$d$}; 
\draw [postaction={decorate}] (1,2) -- (1,3) node [midway,left] {$n$}; 
\draw [postaction={decorate}] (1,4) -- (1,3) node [midway,left] {$n$};

\draw (2.5,0.9) -- (2.5,1);
\draw (0.9,3) -- (1,3);
\draw (-2.5,-1.1) -- (-2.5,-1);

\end{scope}

\draw[yshift=-4cm]
  node[below,text width=8cm] 
  {
  Figure 21 Web for the function $\qcfs{n}{a,b}{n}{c, d}$, where $a+b+c+d=2N$.
  };

\end{tikzpicture}
\end{center}

Let us give a concrete description of this function.

Given four flags 
$$u_1, \dots, u_N;$$
$$v_1, \dots, v_N;$$
$$w_1, \dots, w_N;$$
$$x_1, \dots, x_N;$$

first consider the forms
$$U_n := u_1 \wedge \cdots \wedge u_n,$$
$$V_a := v_1 \wedge \cdots \wedge v_a,$$
$$V_b := v_1 \wedge \cdots \wedge v_b,$$
$$W_n := w_1 \wedge \cdots \wedge w_n,$$
$$X_c := x_1 \wedge \cdots \wedge x_c.$$
$$X_d := x_1 \wedge \cdots \wedge x_d.$$

Because $n+a=2N+n-b-c-d$, there is a natural map 
$$\phi_{N-c, n-d,N-b}: \bigwedge\nolimits^{n+a} V \rightarrow \bigwedge\nolimits^{N-c} V \otimes \bigwedge\nolimits^{n-d} V \otimes \bigwedge\nolimits^{N-b} V.$$
There are also natural maps 
$$- \wedge X_{c} : \bigwedge\nolimits^{N-c} V \rightarrow \bigwedge\nolimits^{N} V \simeq F,$$
$$U_n \wedge - \wedge X_d: \bigwedge\nolimits^{n-d} V \rightarrow \bigwedge\nolimits^{N} V \simeq F,$$ and
$$V_b \wedge - :  \bigwedge\nolimits^{N-b} V \rightarrow \bigwedge\nolimits^{N} V \simeq F.$$ 
Applying these maps to the first, second, and third factors of $\phi_{N-c, n-d,N-a}(V_a \wedge W_n)$, respectively, and then multiplying, we get get the value of our function $\qcfs{n}{a, b}{n}{c, d}$. 

Note that when $a=0$, $b=N$, $c=0$ or $d=N$, we have
$$\qcfs{n}{0, b}{n}{c, d}=:\qcfs{n}{b}{n}{c, d},$$
$$\qcfs{n}{a, N}{n}{c, d}=:\qcfs{n}{a}{n}{c, d},$$
$$\qcfs{n}{a, b}{n}{0, d}=:\qcfs{n}{a, b}{n}{d},$$
$$\qcfs{n}{a, b}{n}{c, N}=:\qcfs{n}{a, b}{n}{c}.$$
If $a=0$ and $d=N$, we will have
$$\qcfs{n}{0, b}{n}{c, N}=\qcfs{n}{b}{n}{c},$$
where $\qcfs{n}{b}{n}{c} $ is as defined above. A similar equality holds when $b=N, c=0$. If $a=0$, $b=N$, $c=0$ and $d=N$, we will have that 
$$\qcfs{n}{0, N}{n}{0, N}=\qcfs{n}{0}{n}{0}.$$

Now we give the sequence of mutations realizing the flip of a triangulation. The sequence of mutations leaves $x_{-n,j}, x_{nj}, y_k$ untouched as they are frozen variables. Hence we only mutate $x_{ij}$ for $-n \leq i \leq n$. We now describe the sequence of mutations. The sequence of mutations will have $3n-2$ stages. At the $r^{\textrm{th}}$ step, we mutate all vertices such that 
$$|i|+j \leq r,$$
$$j-|i|+2n -2 \geq r,$$
$$|i|+j \equiv r \mod 2.$$

Note that the first inequality is empty for $r \geq 2n-1$, while the second inequality is empty for $r \leq n$. For example, for $Sp_6$, the sequence of mutations is

\begin{equation}
\begin{gathered}
x_{01}, \\
x_{-1,1}, x_{02}, x_{11},  \\
x_{-2,1}, x_{-1,2}, x_{01}, x_{03}, x_{12}, x_{21},  \\
x_{-2,2}, x_{-1,1}, x_{-1,3}, x_{02}, x_{11}, x_{13}, x_{22},  \\
x_{-2,3}, x_{-1,2}, x_{01}, x_{03}, x_{12}, x_{23},  \\
x_{-1,3}, x_{02}, x_{13},  \\
x_{03}.
\end{gathered}
\end{equation}

In Figure 22, we depict how the quiver for $\Conf_4 \A_{Sp_{6}}$ changes after each of the seven stages of mutation.

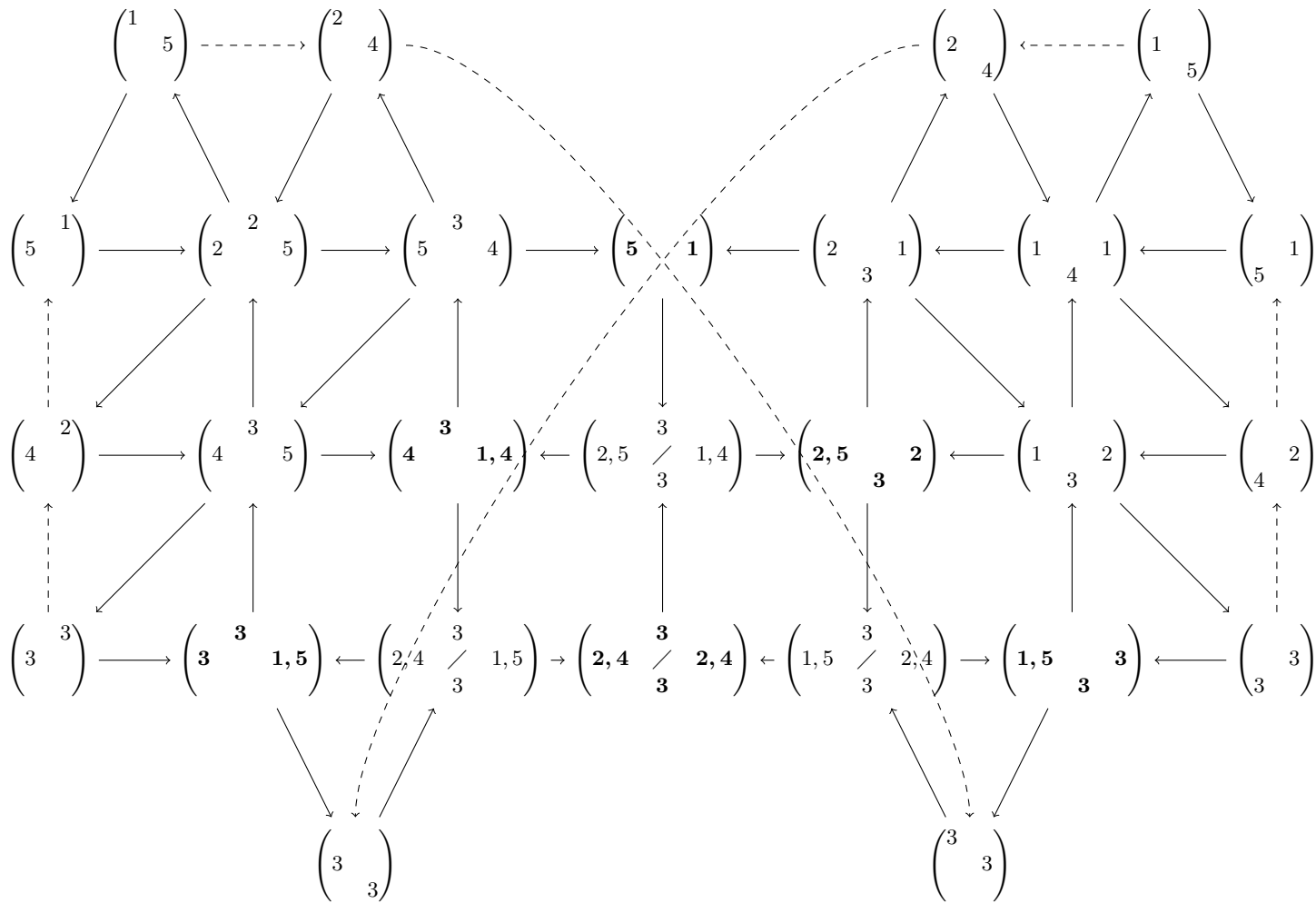
\begin{center}

\end{center}

The analogue $\Conf_4 \A_{Sp_{2n}}$ should be clear. Note that there are no circled arrows in the diagram, so that lifting from this sequence of mutations from $\Conf_4 \A_{Sp_{2n}}$ to $\Conf_4 \A_{SL_{2n}}$ is straightforward.

We now have the main theorem of this section:

\begin{theorem}
We first analyze the situation when $i > 0$. The vertex $x_{ij}$ is mutated a total of $n-i$ times. There are four cases.
\begin{itemize}
\item When $i+j < n$ and $i < j$, the function attached to $x_{ij}$ mutates in three stages, consisting of $n-i-j, i,$ and $j-i$ mutations, respectively:
\begin{enumerate}
\item $$\tcfr{i+j}{N-j}{N-i} \rightarrow \qcfs{i+j+1}{1}{N-j-1}{N-i-1} \rightarrow $$
$$\qcfs{i+j+2}{2}{N-j-2}{N-i-2} \rightarrow $$
$$\dots \rightarrow \qcfs{n}{n-i-j}{n+i}{n+j}$$
\item $$\qcfs{n}{n-i-j}{n+i}{n+j}=\qcfs{n}{n-i-j}{n+i}{0, n+j} \rightarrow $$
$$\qcfs{n}{n-i-j+1}{n+i-1}{1, n+j-1} \rightarrow$$
$$ \qcfs{n}{n-i-j+2}{n+i-2}{2, n+j-2} \rightarrow $$
$$\dots \rightarrow \qcfs{n}{n-j}{n}{i, n+j-i}$$
\item $$\qcfs{n}{n-j}{n}{i, n+j-i}=\qcfs{n}{n-j, N}{n}{i, n+j-i}$$
$$\rightarrow \qcfs{n}{n-j+1, N-1}{n}{i+1, n+j-i-1} \rightarrow $$
$$\qcfs{n}{n-j+2, N-2}{n}{i+2, n+j-i-2} \rightarrow $$
$$\dots \rightarrow [\qcfs{n}{n-i, N-j+i}{n}{j, n}] \textrm{ } \tcfd{n-i, N-j+i}{n}{j}$$
\end{enumerate}

\item When $i+j \geq n$ and $i < j$, the function attached to $x_{ij}$ mutates  in two stages, consisting of $n-j,$ and $j-i$ mutations, respectively:
\begin{enumerate}
\item $$\tcfr{n}{N-j}{j+i-n, N-i} \rightarrow \qcfs{n}{1}{N-j-1}{j+i-n+1, N-i-1} \rightarrow $$
$$\qcfs{n}{2}{N-j-2}{j+i-n+2, N-i-2} \rightarrow$$
$$\dots  \rightarrow \qcfs{n}{n-j}{n}{i, n-i+j}$$
\item $$\qcfs{n}{n-j}{n}{i, n-i+j} [\qcfs{n}{n-j, N}{n}{i, n-i+j}] \rightarrow $$
$$\qcfs{n}{n-j+1, N-1}{n}{i+1, n-i+j-1} \rightarrow $$
$$\qcfs{n}{n-j+2, N-2}{n}{i+2, n-i+j-2} \rightarrow$$
$$\dots  \rightarrow [\qcfs{n}{n-i, N-j+i}{n}{j, n}]  \textrm{ } \tcfd{n-i, N-j+i}{n}{j} $$
\end{enumerate}

\item When $i+j < n$ and $i \geq j$, the function attached to $x_{ij}$ mutates in two stages, consisting of $n-i-j,$ and $j$ mutations, respectively:
\begin{enumerate}
\item $$\tcfr{i+j}{N-j}{N-i} \rightarrow \qcf{i+j+1}{1}{N-j-1}{N-i-1} \rightarrow $$
$$\qcf{i+j+2}{2}{N-j-2}{N-i-2} \rightarrow $$
$$\dots \rightarrow \qcf{n}{n-i-j}{n+i}{n+j}$$
\item $$\qcf{n}{n-i-j}{n+i}{n+j} \rightarrow \qcfs{n}{n-i-j+1}{n+i-1}{1, n+j-1} \rightarrow $$
$$\qcfs{n}{n-i-j+2}{n+i-2}{2, n+j-2} \rightarrow $$
$$\dots \rightarrow [\qcfs{n}{n-i}{n+i-j}{j, n}]  \textrm{ } \tcfd{n-i}{n+i-j}{j}$$
\end{enumerate}

\item When $i+j \geq n$ and $i \geq j$, the function attached to $x_{ij}$ mutates in one stage consisting of $n-i$ mutations:
$$\tcfr{n}{N-j}{j+i-n, N-i} \rightarrow \qcfs{n}{1}{N-j-1}{j+i-n+1, N-i-1} \rightarrow $$
$$\qcfs{n}{2}{N-j-2}{j+i-n+2, N-i-2} \rightarrow$$
$$\dots  \rightarrow [\qcfs{n}{n-i}{n+i-j}{j, n}] \textrm{ } \tcfd{n-i}{n+i-j}{j}$$

\end{itemize}

The mutation sequence when $i \leq 0$ is completely parallel. We include it in an appendix.

In all these sequences, for each mutation, two parameters increase, and two decrease. Within a stage, the same parameters increase or decrease. The only exception is that sometimes after the last mutation, one removes the factor $\dur{n}{n}$ (or $\dld{n}{n}$ when $i \leq 0$). The expressions in square brackets indicate the functions before removing factors of $\dur{n}{n}$ (or $\dld{n}{n}$ when $i \leq 0$).

\end{theorem}

\begin{proof} The proof comes down to a handful of identities used in conjunction, as in previous proofs of this type. Here are the identities we use:

\begin{itemize}
\item Let $0 \leq a, b, c, d \leq N$, and $a+b+c+d=2N$. 
$$\qcfs{a}{b}{c}{d}\qcfs{a+1}{b+1}{c-1}{d-1}=$$
$$\qcfs{a}{b+1}{c-1}{d}\qcfs{a+1}{b}{c}{d-1}+\qcfs{a+1}{b}{c-1}{d}\qcfs{a}{b+1}{c}{d-1}.$$

\item Let $0 \leq a, b, c, d \leq N$, and $a+b+c+d=3n$. 
$$\qcfs{n}{a}{b}{c, d}\qcfs{n}{a+1}{b-1}{c+1, d-1}=$$
$$\qcfs{n}{a+1}{b-1}{c, d}\qcfs{n}{a}{b}{c+1, d-1}+\qcfs{n}{a+1}{b}{c, d-1}\qcfs{n}{a}{b-1}{c+1, d}.$$
There is also a dual identity when $a+b+c+d=5n$ that we use when $i < 0$:
$$\qcfs{n}{a}{b}{c, d}\qcfs{n}{a-1}{b+1}{c+1, d-1}=$$
$$\qcfs{n}{a}{b}{c+1, d-1}\qcfs{n}{a-1}{b+1}{c, d}+\qcfs{n}{a-1}{b}{c+1, d}\qcfs{n}{a}{b+1}{c, d-1}.$$

\item Let $0 \leq a, b, c, d \leq N$, and $a+b+c+d=4n$. 
$$\qcfs{n}{a, b}{n}{c, d}\qcfs{n}{a+1, b-1}{n}{c+1, d-1}=$$
$$\qcfs{n}{a+1, b-1}{n}{c, d}\qcfs{n}{a, b}{n}{c+1, d-1}+\qcfs{n}{a+1, b}{n}{c, d-1}\qcfs{n}{a, b-1}{n}{c+1, d}.$$

\item Let $0 \leq a, b, c, d \leq N$ such that $a+b+c+d=4n=2N$, $a \leq b$ and $c \leq d$. If $a+d=b+c=N$.,
$$\qcfs{a}{b}{c}{d}=\dld{b}{c}\dur{a}{d}$$

\item Let $0 \leq a, b, c, d \leq N$ such that $a+b+c+d=4n=2N$, $a \leq b$ and $c \leq d$. If $a, b, c,$ or $d=n$,
$$\qcfs{n}{n, b}{n}{c, d}=\tcfu{n}{b}{c, d}\dld{n}{n},$$
$$\qcfs{n}{a, n}{n}{c, d}=\tcfu{n}{a}{c, d}\dld{n}{n},$$
$$\qcfs{n}{a, b}{n}{n, d}=\tcfd{a, b}{n}{d}\dur{n}{n},$$
$$\qcfs{n}{a, b}{n}{c, n}=\tcfd{a, b}{n}{c}\dur{n}{n}.$$

\item The final set of identities was mentioned previously. Let $0 \leq a, b, c, d \leq N$ such that $a+b+c+d=4n=2N$, $a \leq b$ and $c \leq d$. When $a=0$, $b=N$, $c=0$ or $d=N$, we have
$$\qcfs{n}{0, b}{n}{c, d}=:\qcfs{n}{b}{n}{c, d},$$
$$\qcfs{n}{a, N}{n}{c, d}=:\qcfs{n}{a}{n}{c, d},$$
$$\qcfs{n}{a, b}{n}{0, d}=:\qcfs{n}{a, b}{n}{d},$$
$$\qcfs{n}{a, b}{n}{c, N}=:\qcfs{n}{a, b}{n}{c}.$$
If $a=0$ and $d=N$, we will have
$$\qcfs{n}{0, b}{n}{c, N}=\qcfs{n}{b}{n}{c}.$$
A similar equality holds when $b=N, c=0$. If $a=0$, $b=N$, $c=0$ and $d=N$, we will have that 
$$\qcfs{n}{0, N}{n}{0, N}=\qcfs{n}{0}{n}{0}.$$

\end{itemize}

The first three sets of identities are the most important. They are variations on the octahedron recurrence. When $i+j < n$ and $i < j$, the three stages use the first, second and third set of identities, respectively. When $i+j \geq n$ and $i < j$, the two stages use the second and third set of identities, respectively. When $i+j < n$ and $i \geq j$, the two stages use the first and second set of identities, respectively. When $i+j \geq n$ and $i \geq j$, the one stage uses only the second set of identities.

The last three sets of identities are used to give degenerate versions of the previous three sets of identities.

\end{proof}

\section{The cluster algebra structure on $\Conf_m G/U$ for $G=Spin_{2n+1}$}

We now define the cluster algebra structure on $\Conf_m G/U$ when $G=Spin_{2n+1}$. The story will be parallel to the previous case when $G=Sp_{2n}$. The similarities are quite striking, and reflect the Langlands duality between the seeds as predicted in \cite{FG2}. As in that case we will utilize what we understand about functions on $\Conf_m \A_{SL_{2n+1}}$ in order to study $\Conf_m \A_{Spin_{2n+1}}$.

Recall that $Spin_{2n+1}$ is the double cover of the group $SO_{2n+1}$, which is the subgroup of $SL_{2n+1}$ preserving a symmetric quadratic form. We take the quadratic form given in the basis standard basis $e_1, \dots, e_{2n+1}$ by
$$<e_i, e_{2n+2-i}>=(-1)^{i-1}$$
and $<e_i,e_j>=0$ otherwise.

\begin{rmk} Here the signs are chosen so that the embedding is compatible with the positive structures on $Spin_{2n+1}$ and $SL_{2n+1}$. Note that the signature of the quadratic form is $(n+1,n)$, so that taking real points gives the split real form of $SO_{2n+1}$. The cluster algebra structure on $\Conf_m \A_{Spin_{2n+1}}$ gives another way of defining the positive structure on $\A_{Spin_{2n+1}, S}$, which gives a parameterization of the Hitchin component for the group $Spin_{2n+1}$ and the surface $S$. This Hitchin component is a component of the character variety for a split real group $G$ and a surface $S$.
\end{rmk}

The maps
$$Spin_{2n+1} \twoheadrightarrow SO_{2n+1} \hookrightarrow SL_{2n+1}$$
induce maps
$$\Conf_m \A_{Spin_{2n+1}} \rightarrow \Conf_m \A_{SO_{2n+1}} \rightarrow \Conf_m \A_{SL_{2n+1}}.$$

Let us describe these maps concretely. The variety $\A_{SO_{2n+1}}$ parameterizes chains of isotropic vector spaces 
$$V_1 \subset V_2 \subset \cdots \subset V_n \subset V$$ inside the $2n+1$-dimensional standard representation $V$, where $\operatorname{dim} V_i= i$, and where each $V_i$ is equipped with a volume form.

Equivalently, a point of $\A_{SO_{2n+1}}$ is given by a sequence of vectors 
$$v_1, v_2, \dots, v_n,$$
where $$V_i:=<v_1, \dots, v_i>$$ is isotropic, and where $v_i$ is only determined up to adding linear combinations of $v_j$ for $j < i$.

The volume form on $V_i$ is then $v_1 \wedge \cdots \wedge v_i$.

From the sequence of vectors $v_1, \dots, v_n$, we can complete to a basis $v_1, v_2, \dots v_{2n+1}$, where $<v_i, v_{2n+2-i}>=(-1)^{i-1}$, and $<v_i,v_j>=0$ otherwise. Equivalently, the quadratic form induces an isomorphism $<-,-> : V \rightarrow V^*$. At the same time, there are perfect pairings
$$\bigwedge\nolimits^k V \times \bigwedge\nolimits^k V^* \rightarrow F$$
$$\bigwedge\nolimits^{2n+1-k} V \times \bigwedge\nolimits^k V \rightarrow F$$
that induce an isomorphism
$$\bigwedge\nolimits^{2n+1-k} V \simeq \bigwedge\nolimits^k V^*.$$ Composing this with the inverse of the isomorphism 
$$<-,-> : \bigwedge\nolimits^k V \rightarrow \bigwedge\nolimits^k V^*$$
gives an isomorphism 
$$\bigwedge\nolimits^{2n+1-k} V \simeq \bigwedge\nolimits^k V^* \simeq \bigwedge\nolimits^k V.$$
Then $v_{n+1}, \dots, v_{2n+1}$ are chosen so that this isomorphism takes $v_1 \wedge \cdots v_{i}$ to $v_1 \wedge \cdots v_{2n+1-i}$.

Then $v_1, v_2, \dots v_{2n+1}$ determines a point of $\A_{SL_{2n+1}}$, as $\A_{SL_{2n+1}}$ parameterizes chains of vector subspaces  
$$V_1 \subset V_2 \subset \cdots \subset V_{2n+1} = V$$ along with volume forms $v_1 \wedge \cdots v_{i}$, $1 \leq i \leq 2n+1$.

From the embedding 
$$\A_{SO_{2n+1}} \hookrightarrow \A_{SL_{2n+1}},$$
 one naturally gets an embedding $\Conf_m \A_{SO_{2n+1}} \hookrightarrow \Conf_m \A_{SL_{2n+1}}$. We can then pull back functions from $\Conf_m \A_{SL_{2n+1}}$ to get functions on $\Conf_m \A_{SO_{2n+1}}$. However, we are ultimately interested in functions on $\Conf_m \A_{Spin_{2n+1}}$.

The functions on $\Conf_m \A_{Spin_{2n+1}}$ that we will use to define the cluster structure on $\Conf_m \A_{Spin_{2n+1}}$ will be invariants of tensor products of representations of $Spin_{2n+1}$. For $m=3$, they will lie inside 
$$[V_{\lambda} \otimes V_{\mu} \otimes V_{\nu}]^G$$
where $\lambda, \mu, \nu$ are elements of the dominant cone inside the weight lattice. In general, not all such functions will come from pulling back functions on $\Conf_m \A_{SL_{2n+1}}$.

However, suppose that 
$$f \in [V_{\lambda} \otimes V_{\mu} \otimes V_{\nu}]^G \subset \mathcal{O}(\Conf_m \A_{Spin_{2n+1}}).$$
Then
$$f^2 \in [V_{2\lambda} \otimes V_{2\mu} \otimes V_{2\nu}]^G \subset \mathcal{O}(\Conf_m \A_{Spin_{2n+1}}).$$
However, because $2\lambda, 2\mu, 2\nu$ are dominant weights for $SO_{2n+1}$, $f^2$ may be viewed as a function on $\Conf_m \A_{SO_{2n+1}}.$ This function is then a pull-back of a function on $\Conf_m \A_{SL_{2n+1}}.$ Therefore functions on $\Conf_m \A_{Spin_{2n+1}}$ are either the pull-backs of functions on $\Conf_m \A_{SL_{2n+1}}$ or square-roots of such functions. The square-root here corresponds to the fact that $Spin_{2n+1}$ is a double cover of $SO_{2n+1}$. The choice of the branch of the square-root that we take is determined by the positive structure on $\Conf_m \A_{Spin_{2n+1}}$: if $f$ is a positive function on $\Conf_m \A_{SL_{2n+1}}$ such that its square-root is a function on $\Conf_m \A_{Spin_{2n+1}}$, there is a unique choice of $\sqrt{f}$ that is positive on $\Conf_m \A_{Spin_{2n+1}}$. That is the square-root that we will always take.

\subsection{Construction of the seed}

We are now ready to construct the seed for the cluster structure on $\Conf_m \A$ when $G=Spin_{2n+1}$. Throughout this section, $G=Spin_{2n+1}$ unless otherwise noted.

Recall that $Spin_{2n+1}$ is associated to the type $B$ Dynkin diagram:

\begin{multicols}{2}
\begin{center}\begin{tikzpicture}
    \draw (-1,0) node[anchor=east]  {$B_n$};

    \node[dnode,label=below:$1$] (1) at (0,0) {};
    \node[dnode,label=below:$2$] (2) at (1,0) {};
    \node[dnode,label=below:$n-2$] (3) at (3,0) {};
    \node[dnode,label=below:$n-1$] (4) at (4,0) {};
    \node[dnode,label=below:$n$] (5) at (5,0) {};

    \path (1) edge[sedge] (2)
          (2) edge[sedge,dashed] (3)
          (3) edge[sedge] (4)
          (4) edge[dedge] (5)
          ;

\draw[yshift=-2cm]
  node[below,text width=6cm] 
  {
  Figure 23 $B_n$ Dynkin diagram
  };

\end{tikzpicture}\end{center}
\end{multicols}

The nodes of the diagram correspond to $n-1$ long roots, numbered $1, 2, \dots, n-1$, and one short root, which is numbered $n$. To describe the cluster structure on $\Conf_3 \A$, we need to give the following data: the set $I$ parameterizing vertices, the functions on $\Conf_3 \A$ corresponding to each vertex, and the $B$-matrix for this seed.

The $B$-matrix is encoded via a quiver which  consists of $n^2+2n$ vertices, of which $n+2$ have $d_i=1$, while the remaining vertices have $d_i=\frac{1}{2}$. We color the vertices with $d_i=1$ black, while the remaining vertices are white. There are $n$ edge functions for each edge of the triangle, and $n^2-n$ face functions. There is one black vertex for each edge. The $B$-matrix is read off from the quiver by the following rules:

\begin{itemize}
\item An arrow from $i$ to $j$ means that $b_{ij}>0$.
\item $|b_{ij}|=2$ if $d_i=1$ and $d_j=\frac{1}{2}$.
\item $|b_{ij}|=1$ otherwise.
\end{itemize}

In Figure 24, we see the quiver for $Spin_7$. The generalization for other values of $n$ should be clear.

\begin{center}
\begin{tikzpicture}[scale=2]
\begin{scope}[xshift=-1cm]
  \foreach \x in {1, 2}
    \foreach \y in {0, 1, 2, 3}
      \node[] (x\y\x) at (\x,-\y) {\Large $\ontop{x_{\y\x}}{\circ}$};
  \foreach \x in {3}
    \foreach \y in {0, 1, 2, 3}
      \node[] (x\y\x) at (\x,-\y) {\Large $\ontop{x_{\y\x}}{\bullet}$};
  \node[] (y1) at (0,-2)  {\Large $\ontop{y_1}{\circ}$};
  \node[] (y2) at (0,-1)  {\Large $\ontop{y_2}{\circ}$};
  \node[] (y3) at (4,-1)  {\Large $\ontop{y_3}{\bullet}$};
  \draw [->] (x31) to (x21);
  \draw [->] (x21) to (x11);
  \draw [->] (x11) to (x01);
  \draw [->] (x32) to (x22);
  \draw [->] (x22) to (x12);
  \draw [->] (x12) to (x02);
  \draw [->] (x33) to (x23);
  \draw [->] (x23) to (x13);
  \draw [->] (x13) to (x03);
  \draw [->, dashed] (y2) .. controls +(up:2) and +(up:2) .. (y3);
  \draw [->, dashed] (y1) to (y2);

  \draw [->, dashed] (x03) to (x02);
  \draw [->, dashed] (x02) to (x01);
  \draw [->] (y3) to (x13);
  \draw [->] (x13) to (x12);
  \draw [->] (x12) to (x11);
  \draw [->] (x11) to (y2);
  \draw [->] (x23) to (x22);
  \draw [->] (x22) to (x21);
  \draw [->] (x21) to (y1);
  \draw [->, dashed] (x33) to (x32);
  \draw [->, dashed] (x32) to (x31);

  \draw [->] (x01) to (x12);
  \draw [->] (x02) to (x13);
  \draw [->] (x03) to (y3);
  \draw [->] (y2) to (x21);
  \draw [->] (x11) to (x22);
  \draw [->] (x12) to (x23);
  \draw [->] (y1) to (x31);
  \draw [->] (x21) to (x32);
  \draw [->] (x22) to (x33);

\end{scope}

\draw[yshift=-3.5cm,xshift=1cm]
  node[below,text width=6cm] 
  {
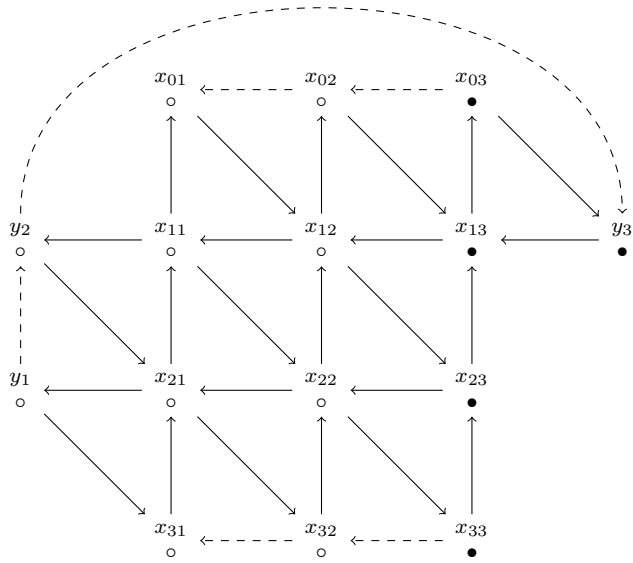
  Figure 24. Quiver encoding the cluster structure for $\Conf_3 \A_{Spin_{7}}$
  };

\end{tikzpicture}
\end{center}

A dotted arrow means that $b_{ij}$ is half the value it would be if the arrow were solid. In other words, $|b_{ij}|=1$ if $d_i=1$ and $d_j=\frac{1}{2}$.
and $|b_{ij}|=\frac{1}{2}$ otherwise.

We will no longer use single letters like $i, j$ to denote vertices of the quiver, because it will be convenient for us to use the pairs $(i, j)$ to parameterize the vertices of the quiver. In the formulas in the remainder of this section, we will not refer to the particular entries of the $B$-matrix, $b_{ij}$. Instead, the values of the entries of the $B$-matrix will be encoded in quivers. This will hopefully avoid any notational confusion.

Label the vertices of the quiver $x_{ij}$ and $y_k$, where $0 \leq i \leq n+1$, $1 \leq j \leq n$, $1 \leq k \leq n$. The white vertices correspond to $x_{in}$ and $y_n$. The vertices $y_k$ and $x_{ij}$ for $i= 0$ or $n$ are frozen. We will sometimes write $x_{i,j}$ for $x_{ij}$ for orthographic reasons. Note that the dotted arrows only go between frozen vertices, thus the entries $b_{ij}$ of the $B$-matrix are integral unless $i$ and $j$ are both frozen, and thus the $B$-matrix defines a cluster algebra.

Let us now recall some facts about the representation theory of $Spin_{2n+1}$. The fundamental representations of $Spin_{2n+1}$ are labelled by the fundamental weights $\omega_1, \dots, \omega_n$. $Spin_{2n+1}$ has a standard $2n+1$-dimensional representation $V$. Let $<-,->$ be the orthogonal pairing. Then for $i < n$ the representation $V_{\omega_i}$ corresponding to $\omega_i$ is precisely $\bigwedge\nolimits^i V$. The representation $V_{\omega_n}$ is the \emph{spin representation} of $Spin_{2n+1}$. The representation $V_{2\omega_n}$ is isomorphic to $\bigwedge\nolimits^n V$.

We now say which functions are attached to the vertices of the quiver. Recall the functions defined via the webs from Figures 3, 4, and 6. It turns out to be easier to describe the functions attached to the white vertices and the square of the functions attached to the black vertices. In other words, if the function $f_{ij}$ is attached to vertex $x_{ij}$, it is sometimes more convenient to consider the function $f_{ij}^{2d_{ij}}$. Thus we break down the description of the functions attached to $x_{ij}$ into steps:

\begin{enumerate}
\item For $k<n$, assign the function $\dud{k}{2n+1-k}$ to $y_k$. 
\item When $i \geq j$, assign the function $\tcfr{n-i}{n+1+i-j}{j}$ to $x_{ij}$.
\item When $i < j$ and $i \neq 0$, we assign the function $\tcfr{n-i, 2n+1+i-j}{n+1}{j}$ to $x_{ij}$. 
\item When $i=0$, we assign the function $\tcfr{2n+1-j}{}{j}$ to $x_{ij}$.
\item Now take the square root of the functions assigned to $x_{in}$ and $y_n$.
\end{enumerate}

This completely describes the cluster structure on $\Conf_3 \A_{Spin_{2n+1}}$. The fact that we can take the square-roots of the functions assigned to $x_{in}$ and $y_n$ and get functions that are well-defined on $\Conf_3 \A_{Spin_{2n+1}}$ follows from the computations of the next section. Note that the cluster structure is not symmetric with respect to the three flags. Performing various $S_3$ symmetries, we obtain six different possible cluster structures on  $\Conf_3 \A_{Spin_{2n+1}}$. These six structures are related by sequences of mutations that we describe in the next section. Below, in Figure 25, we depict two of the six cluster structures for $\Conf_3 \A_{Spin_{9}}$ that are obtained from the original cluster structure by an $S_3$ symmetry.

\begin{center}
\begin{tikzpicture}[scale=2.4]
  \node (x01) at (-1,0) {$\tcfr{8}{}{1}$};
  \node (x02) at (0,0) {$\tcfr{7}{}{2}$};
  \node (x03) at (1,0) {$\tcfr{6}{}{3}$};
  \node (x04) at (2,0) {$\sqrt{\tcfr{5}{}{4}}$};
  \node (x11) at (-1,-1) {$\tcfr{3}{5}{1}$};
  \node (x12) at (0,-1) {$\tcfr{3, 8}{5}{2}$};
  \node (x13) at (1,-1) {$\tcfr{3, 7}{5}{3}$};
  \node (x14) at (2,-1) {$\sqrt{\tcfr{3, 6}{5}{4}}$};
  \node (x21) at (-1,-2) {$\tcfr{2}{6}{1}$};
  \node (x22) at (0,-2) {$\tcfr{2}{5}{2}$};
  \node (x23) at (1,-2) {$\tcfr{2, 8}{5}{3}$};
  \node (x24) at (2,-2) {$\sqrt{\tcfr{2, 7}{5}{4}}$};
  \node (x31) at (-1,-3) {$\tcfr{1}{7}{1}$};
  \node (x32) at (0,-3) {$\tcfr{1}{6}{2}$};
  \node (x33) at (1,-3) {$\tcfr{1}{5}{3}$};
  \node (x34) at (2,-3) {$\sqrt{\tcfr{1, 8}{5}{4}}$};
  \node (x41) at (-1,-4) {$\tcfr{}{8}{1}$};
  \node (x42) at (0,-4) {$\tcfr{}{7}{2}$};
  \node (x43) at (1,-4) {$\tcfr{}{6}{3}$};
  \node (x44) at (2,-4) {$\sqrt{\tcfr{}{5}{4}}$};
  \node (y1) at (-2,-3) {$\dud{1}{8}$};
  \node (y2) at (-2,-2) {$\dud{2}{7}$};
  \node (y3) at (-2,-1) {$\dud{3}{6}$};
  \node (y4) at (3,-1) {$\sqrt{\dud{4}{5}}$};

  \draw [->] (x41) to (x31);
  \draw [->] (x31) to (x21);
  \draw [->] (x21) to (x11);
  \draw [->] (x11) to (x01);
  \draw [->] (x42) to (x32);
  \draw [->] (x32) to (x22);
  \draw [->] (x22) to (x12);
  \draw [->] (x12) to (x02);
  \draw [->] (x43) to (x33);
  \draw [->] (x33) to (x23);
  \draw [->] (x23) to (x13);
  \draw [->] (x13) to (x03);
  \draw [->] (x44) to (x34);
  \draw [->] (x34) to (x24);
  \draw [->] (x24) to (x14);
  \draw [->] (x14) to (x04);
  \draw [->, dashed] (y3) .. controls +(up:2) and +(up:2) .. (y4);
  \draw [->, dashed] (y2) to (y3);
  \draw [->, dashed] (y1) to (y2);

  \draw [->, dashed] (x04) to (x03);
  \draw [->, dashed] (x03) to (x02);
  \draw [->, dashed] (x02) to (x01);
  \draw [->] (y4) to (x14);
  \draw [->] (x14) to (x13);
  \draw [->] (x13) to (x12);
  \draw [->] (x12) to (x11);
  \draw [->] (x11) to (y3);
  \draw [->] (x24) to (x23);
  \draw [->] (x23) to (x22);
  \draw [->] (x22) to (x21);
  \draw [->] (x21) to (y2);
  \draw [->] (x34) to (x33);
  \draw [->] (x33) to (x32);
  \draw [->] (x32) to (x31);
  \draw [->] (x31) to (y1);
  \draw [->, dashed] (x44) to (x43);
  \draw [->, dashed] (x43) to (x42);
  \draw [->, dashed] (x42) to (x41);

  \draw [->] (x01) to (x12);
  \draw [->] (x02) to (x13);
  \draw [->] (x03) to (x14);
  \draw [->] (x04) to (y4);
  \draw [->] (y3) to (x21);
  \draw [->] (x11) to (x22);
  \draw [->] (x12) to (x23);
  \draw [->] (x13) to (x24);
  \draw [->] (y2) to (x31);
  \draw [->] (x21) to (x32);
  \draw [->] (x22) to (x33);
  \draw [->] (x23) to (x34);
  \draw [->] (y1) to (x41);
  \draw [->] (x31) to (x42);
  \draw [->] (x32) to (x43);
  \draw [->] (x33) to (x44);
xshift=1cm
\draw[yshift=-5.5cm,xshift=1cm]
  node[below,text width=6cm] 
  {
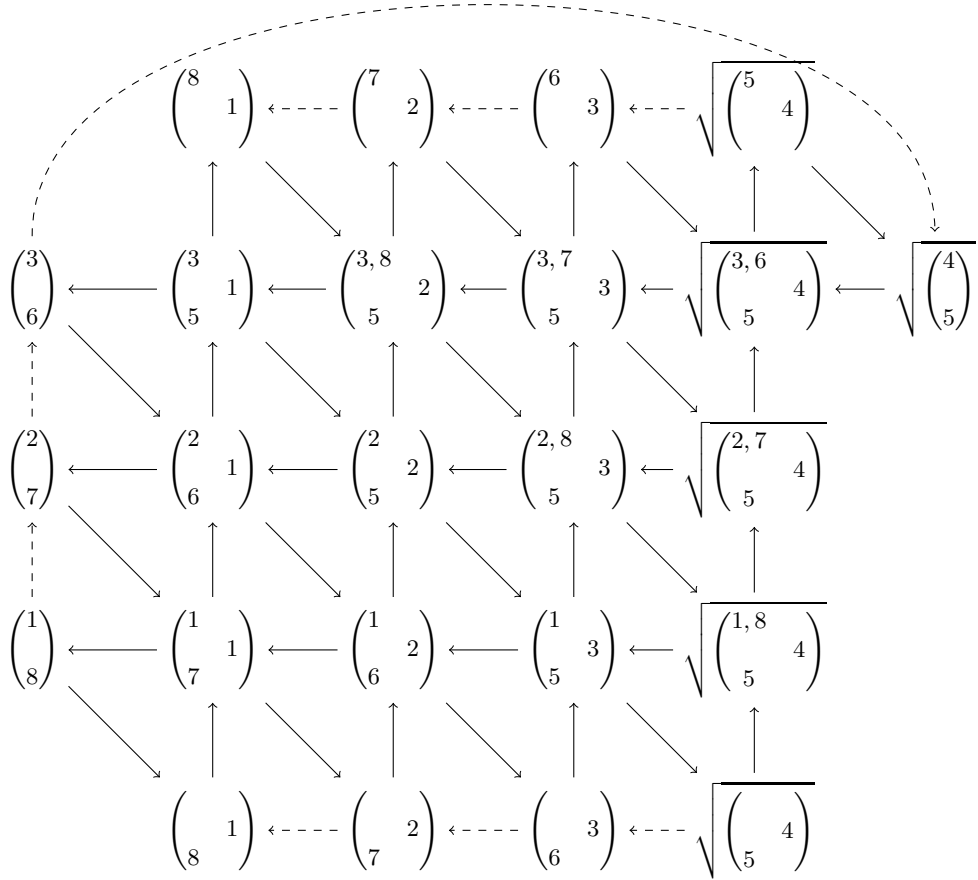
  Figure 25a. One cluster structure for $\Conf_3 \A_{Spin_{9}}$.
  };

\end{tikzpicture}
\end{center}

\begin{center}
\begin{tikzpicture}[scale=2.4]

  \node (x01) at (-2,-1) {$\dud{8}{1}$};
  \node (x02) at (-2,-2) {$\dud{7}{2}$};
  \node (x03) at (-2,-3) {$\dud{6}{3}$};
  \node (x04) at (-2,-4) {$\sqrt{\dud{5}{4}}$};
  \node (x11) at (-1,-1) {$\tcfr{3}{1}{5}$};
  \node (x12) at (-1,-2) {$\tcfr{3, 8}{2}{5}$};
  \node (x13) at (-1,-3) {$\tcfr{3, 7}{3}{5}$};
  \node (x14) at (-1,-4) {$\sqrt{\tcfr{3, 6}{4}{5}}$};
  \node (x21) at (0,-1) {$\tcfr{2}{1}{6}$};
  \node (x22) at (0,-2) {$\tcfr{2}{2}{5}$};
  \node (x23) at (0,-3) {$\tcfr{2, 8}{3}{5}$};
  \node (x24) at (0,-4) {$\sqrt{\tcfr{2, 7}{4}{5}}$};
  \node (x31) at (1,-1) {$\tcfr{1}{1}{7}$};
  \node (x32) at (1,-2) {$\tcfr{1}{2}{6}$};
  \node (x33) at (1,-3) {$\tcfr{1}{3}{5}$};
  \node (x34) at (1,-4) {$\sqrt{\tcfr{1, 8}{4}{5}}$};
  \node (x41) at (2,-1) {$\tcfr{}{1}{8}$};
  \node (x42) at (2,-2) {$\tcfr{}{2}{7}$};
  \node (x43) at (2,-3) {$\tcfr{}{3}{6}$};
  \node (x44) at (2,-4) {$\sqrt{\tcfr{}{4}{5}}$};
  \node (y1) at (1,0) {$\dud{8}{1}$};
  \node (y2) at (0,0) {$\dud{7}{2}$};
  \node (y3) at (-1,0) {$\dud{6}{3}$};
  \node (y4) at (-1,-5) {$\sqrt{\dud{5}{4}}$};

  \draw [->] (x01) to (x11);
  \draw [->] (x11) to (x21);
  \draw [->] (x21) to (x31);
  \draw [->] (x31) to (x41);
  \draw [->] (x02) to (x12);
  \draw [->] (x12) to (x22);
  \draw [->] (x22) to (x32);
  \draw [->] (x32) to (x42);
  \draw [->] (x03) to (x13);
  \draw [->] (x13) to (x23);
  \draw [->] (x23) to (x33);
  \draw [->] (x33) to (x43);
  \draw [->] (x04) to (x14);
  \draw [->] (x14) to (x24);
  \draw [->] (x24) to (x34);
  \draw [->] (x34) to (x44);
  \draw [->, dashed] (y4) .. controls +(left:2) and +(left:2) .. (y3);
  \draw [->, dashed] (y3) to (y2);
  \draw [->, dashed] (y2) to (y1);

  \draw [->, dashed] (x01) to (x02);
  \draw [->, dashed] (x02) to (x03);
  \draw [->, dashed] (x03) to (x04);
  \draw [->] (y3) to (x11);
  \draw [->] (x11) to (x12);
  \draw [->] (x12) to (x13);
  \draw [->] (x13) to (x14);
  \draw [->] (x14) to (y4);
  \draw [->] (y2) to (x21);
  \draw [->] (x21) to (x22);
  \draw [->] (x22) to (x23);
  \draw [->] (x23) to (x24);
  \draw [->] (y1) to (x31);
  \draw [->] (x31) to (x32);
  \draw [->] (x32) to (x33);
  \draw [->] (x33) to (x34);
  \draw [->, dashed] (x41) to (x42);
  \draw [->, dashed] (x42) to (x43);
  \draw [->, dashed] (x43) to (x44);

  \draw [->] (x12) to (x01);
  \draw [->] (x13) to (x02);
  \draw [->] (x14) to (x03);
  \draw [->] (y4) to (x04);
  \draw [->] (x21) to (y3);
  \draw [->] (x22) to (x11);
  \draw [->] (x23) to (x12);
  \draw [->] (x24) to (x13);
  \draw [->] (x31) to (y2);
  \draw [->] (x32) to (x21);
  \draw [->] (x33) to (x22);
  \draw [->] (x34) to (x23);
  \draw [->] (x41) to (y1);
  \draw [->] (x42) to (x31);
  \draw [->] (x43) to (x32);
  \draw [->] (x44) to (x33);

\draw[yshift=-5.5cm]
  node[below,text width=6cm] 
  {
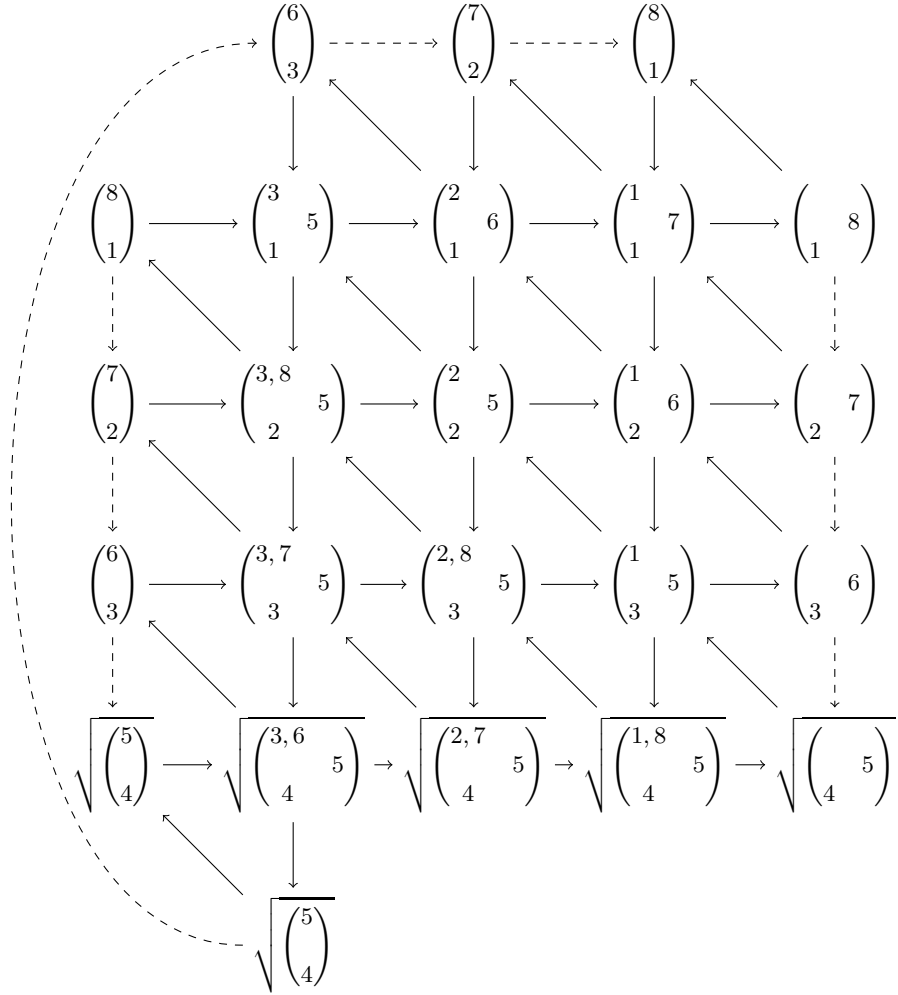
  Figure 25b. Another cluster structures for $\Conf_3 \A_{Spin_{9}}$ related by to the first by an $S_3$ symmetry.
  };

\end{tikzpicture}
\end{center}

Let us describe in more detail how to obtain these other cluster structures. As before, if the $S_3$ symmetry is a rotation, we just rotate the quiver. If the $S_3$ symmetry is a transposition, we transpose the quiver and also reverse the arrows.

The functions attached to the vertices come from permuting the arguments in our notation for the function. For example, rotating the function $\tcfr{n-i, 2n+1+i-j}{n+1}{j}$ gives the function $\tcfr{j}{n-i, 2n+1+i-j}{n+1}$, while transposing the first two arguments gives $\tcfr{n+1}{n-i, 2n+1+i-j}{j}$. A function of the form $\tcfr{d}{a, b}{c}$ or ${c; d; a, b}$ is defined by applying the cyclic shift map to the function $\tcfr{a, b}{c}{d}$.

The cluster structure for $\Conf_m \A$ comes from triangulating an $m$-gon and then attaching the cluster structure on $\Conf_3 \A$ to each triangle and then using the procedure of amalgamation. As before, the vertices $y_k$, $x_{0j}$ and $x_{nj}$ will lie on the edges of the triangle, and the functions attached to them are {\em edge functions}. On $\Conf_3 \A$ the edge vertices are frozen, and the functions attached to these edges depend on only two of the three flags. All other functions are {\em face functions}.

To form the quiver for $\Conf_m \A_{Spin_{2n+1}}$, we first take a triangulation of an $m$-gon. On each of the $m-2$ triangles, attach any one of the six quivers formed from performing $S_3$ symmetries on the quiver for $\Conf_3 \A_{Spin_{2n+1}}$. Each edge of each of these triangles has $n$ frozen vertices. For any two triangles sharing an edge, the corresponding vertices on those edges are identified. Those vertices become unfrozen.

If vertices $i$ and $j$ are glued with $i'$ and $j'$ to get new vertices $i''$ and $j''$, then we declare that $$b_{i''j''}=b_{ij}+b_{i'j'}.$$
In other words, two dotted arrows in the same direction glue to give us a solid arrow, whereas two dotted arrows in the opposite direction cancel to give us no arrow. One can easily check that, again, any gluing will result in no dotted arrows using the unfrozen vertices. The arrows involving vertices that were not previously frozen remain the same. Figure 26 below shows one gluing between two triangles for $Spin_7$.

\begin{center}
\begin{tikzpicture}[scale=2.4]
  \node (x-31) at (-3,0) {$\dul{1}{6}$};
  \node (x-21) at (-2,0) {$\tcfl{1}{2}{4}$};
  \node (x-11) at (-1,0) {$\tcfl{1}{1}{5}$};
  \node (x01) at (0,0) {$\dud{1}{6}$};
  \node (x11) at (1,0) {$\tcfr{2}{6}{6}$};
  \node (x21) at (2,0) {$\tcfr{3}{6}{5}$};
  \node (x31) at (3,0) {$\ddr{6}{1}$};
  \node (x-32) at (-3,-1) {$\dul{2}{5}$};
  \node (x-22) at (-2,-1) {$\tcfl{2}{2, 6}{4}$};
  \node (x-12) at (-1,-1) {$\tcfl{2}{1}{4}$};
  \node (x02) at (0,-1) {$\dud{2}{5}$};
  \node (x12) at (1,-1) {$\tcfr{3}{5}{6}$};
  \node (x22) at (2,-1) {$\tcfr{3}{5}{1, 5}$};
  \node (x32) at (3,-1) {$\ddr{5}{2}$};
  \node (x-33) at (-3,-2) {$\sqrt{\dul{3}{4}}$};
  \node (x-23) at (-2,-2) {$\sqrt{\tcfl{3}{2, 5}{4}}$};
  \node (x-13) at (-1,-2) {$\sqrt{\tcfl{3}{1, 6}{4}}$};
  \node (x03) at (0,-2) {$\sqrt{\dud{3}{4}}$};
  \node (x13) at (1,-2) {$\sqrt{\tcfr{3}{4}{1, 6}}$};
  \node (x23) at (2,-2) {$\sqrt{\tcfr{3}{4}{2, 5}}$};
  \node (x33) at (3,-2) {$\sqrt{\ddr{4}{3}}$};
  \node (y-1) at (-0.5,1) {$\dld{1}{6}$};
  \node (y-2) at (-1.5,1) {$\dld{2}{5}$};
  \node (y-3) at (-2.5,-3) {$\sqrt{\dld{3}{4}}$};
  \node (y1) at (0.5,1) {$\tcfr{1}{}{6}$};
  \node (y2) at (1.5,1) {$\tcfr{2}{}{5}$};
  \node (y3) at (2.5,-3) {$\sqrt{\tcfr{3}{}{4}}$};

  \draw [->] (x01) to (x11);
  \draw [->] (x11) to (x21);
  \draw [->] (x21) to (x31);
  \draw [->] (x02) to (x12);
  \draw [->] (x12) to (x22);
  \draw [->] (x22) to (x32);
  \draw [->] (x03) to (x13);
  \draw [->] (x13) to (x23);
  \draw [->] (x23) to (x33);
  \draw [->, dashed] (y1) to (y2);
  \draw [->, dashed] (y2) .. controls +(right:1) and +(up:1) .. (y3);

  \draw [->] (x03) to (x02);
  \draw [->] (x02) to (x01);

  \draw [->] (x13) to (x12);
  \draw [->] (x12) to (x11);
  \draw [->] (x11) to (y1);
  \draw [->] (y3) to (x23);
  \draw [->] (x23) to (x22);
  \draw [->] (x22) to (x21);
  \draw [->] (x21) to (y2);
  \draw [->, dashed] (x33) to (x32);
  \draw [->, dashed] (x32) to (x31);

  \draw [->] (y1) to (x01);
  \draw [->] (x11) to (x02);
  \draw [->] (x12) to (x03);
 \draw [->] (y2) to (x11);
  \draw [->] (x21) to (x12);
  \draw [->] (x22) to (x13);
  \draw [->] (x31) to (x22);
  \draw [->] (x32) to (x23);
  \draw [->] (x33) to (y3);

  \draw [->] (x01) to (x-11);
  \draw [->] (x-11) to (x-21);
  \draw [->] (x-21) to (x-31);
  \draw [->] (x02) to (x-12);
  \draw [->] (x-12) to (x-22);
  \draw [->] (x-22) to (x-32);
  \draw [->] (x03) to (x-13);
  \draw [->] (x-13) to (x-23);
  \draw [->] (x-23) to (x-33);
  \draw [->, dashed] (y-1) to (y-2);
  \draw [->, dashed] (y-2) .. controls +(left:1) and +(up:1) .. (y-3);

  \draw [->] (x-13) to (x-12);
  \draw [->] (x-12) to (x-11);
  \draw [->] (x-11) to (y-1);
  \draw [->] (y-3) to (x-23);
  \draw [->] (x-23) to (x-22);
  \draw [->] (x-22) to (x-21);
  \draw [->] (x-21) to (y-2);
  \draw [->, dashed] (x-33) to (x-32);
  \draw [->, dashed] (x-32) to (x-31);

  \draw [->] (y-1) to (x01);
  \draw [->] (x-11) to (x02);
  \draw [->] (x-12) to (x03);
  \draw [->] (y-2) to (x-11);
  \draw [->] (x-21) to (x-12);
  \draw [->] (x-22) to (x-13);
  \draw [->] (x-31) to (x-22);
  \draw [->] (x-32) to (x-23);
  \draw [->] (x-33) to (y-3);

\draw[yshift=-3.85cm]
  node[below,text width=6cm] 
  {
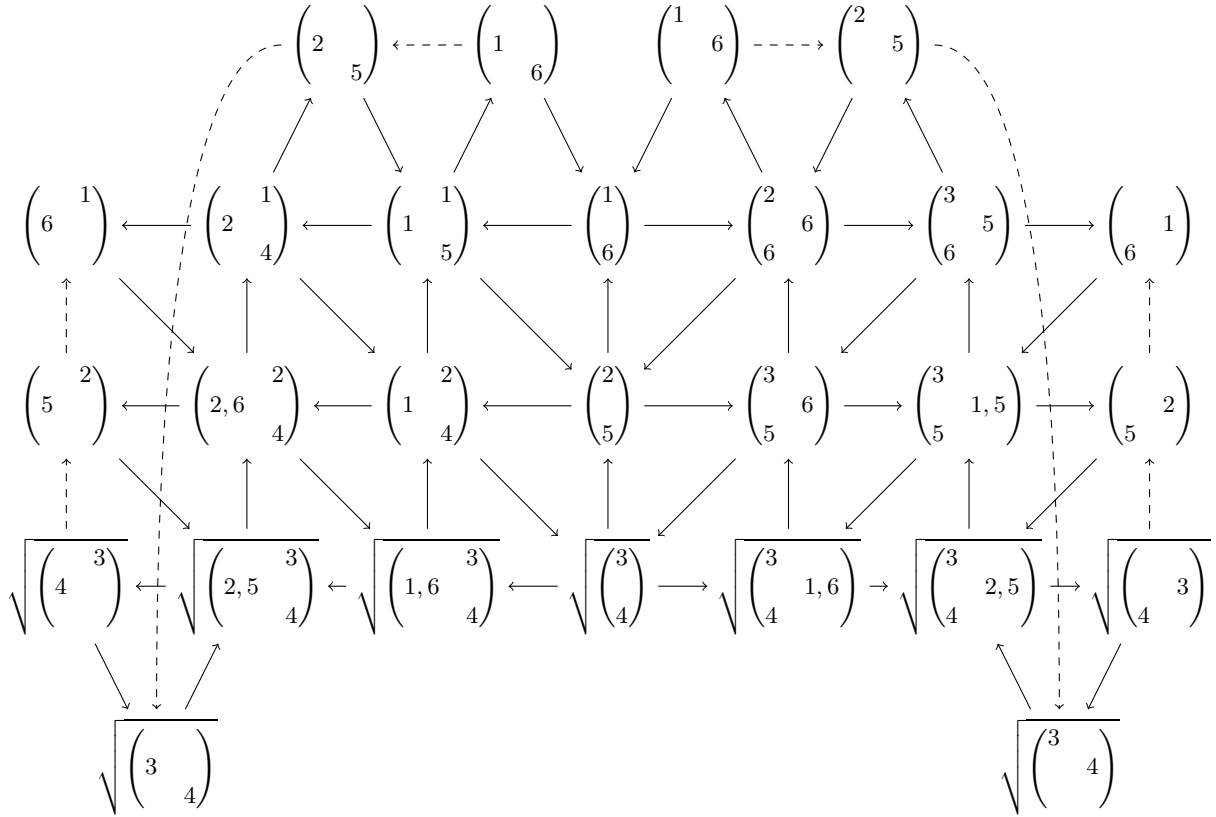
  Figure 26. The functions and quiver for the cluster algebra on $\Conf_4 \A_{Spin_{7}}$.
  };

\end{tikzpicture}
\end{center}

In the next section, we discuss Langlands duality, which will give us a framework for relating, and explaining the similarities between, the cluster algebra structures on $\Conf_m \A_{G}$ for $G=Sp_{2n}$ and $Spin_{2n+1}$.

\subsection{Langlands duality}

In this section, we make some remarks related to Langlands duality. Note that the quivers for $\Conf_3 \A_{Spin_{2n+1}}$ and $\Conf_3 \A_{Sp_{2n}}$ are the same, except that white and black vertices switch colors. Moreover, if we mutate corresponding vertices in the quivers for $\Conf_3 \A_{Spin_{2n+1}}$ and $\Conf_3 \A_{Sp_{2n}}$, the relationship persists: we still get the same quivers, but with colors reversed. This is because the seeds we have constructed for the cluster algebras on $\Conf_3 \A_{Spin_{2n+1}}$ and $\Conf_3 \A_{Sp_{2n}}$ are Langlands dual seeds:

\begin{definition} \cite{FG2} Two seeds that have the same set of vertices $I$ and the same set of frozen vertices $I_0$ are said to be {\em Langlands dual} if they have $B$-matrices $b_{ij}$ and $b^{\vee}_{ij}$ and multipliers $d_i$ and $d^{\vee}_i$ where
$$d_i=(d^{\vee}_i)^{-1}D,$$
$$b_{ij}d_j=-b^{\vee}_{ij}d^{\vee}_j,$$
for some rational number $D$.
\end{definition}

\begin{rmk} Note that the multipliers $d_i$ for a cluster algebra are determined only up to simultaneous scaling by a rational number. Conventions sometimes differ on how to specify the values for the $d_i$. This is one reason the rational number $D$ appears in the above definition.
\end{rmk}

\begin{rmk} The cluster algebras for $\Conf_3 \A_{Spin_{2n+1}}$ and $\Conf_3 \A_{Sp_{2n}}$ as we have defined them satisfy
$$b_{ij}d_j=b^{\vee}_{ij}d^{\vee}_j,$$
without the negative sign. Negating the $B$-matrix does note change the cluster algebra, and for that reason a $B$-matrix and its negative are often considered to be equivalent. The reason that \cite{FG2} include a negative sign in their definition is to make the quantization of the ``modular double'' cleaner. We will ignore this issue for most of this paper.
\end{rmk}

In fact, the seeds for $\Conf_m \A_{Spin_{2n+1}}$ and $\Conf_m \A_{Sp_{2n}}$ are Langlands dual; the property is preserved under amalgamation.

When two seeds are Langlands dual, there is a close relationship between the resulting cluster algebras. Suppose that $(I, I_0, b_{ij}, d_i)$ and $(I, I_0, b^{\vee}_{ij}, d^{\vee}_i)$ are Langlands dual seeds. Let the cluster variables for the initial seeds be $x_1, \dots, x_n$ and $x^{\vee}_1, \dots, x^{\vee}_n$, respectively. These cluster variables are naturally in bijection. Then if we mutate $x_k$ to obtain the new cluster variable $x_k'$, we can do the same to $x^{\vee}_k$ to get $(x'_k)^{\vee}$ and then match $x_k'$ and $(x'_k)^{\vee}$. Continuing in this manner, one conjecturally gets a bijection between all the cluster variables for the Langlands dual seeds. Let us make an observation:

\begin{observation} Suppose that we have a cluster variable
$$f \in [V_{\lambda} \otimes V_{\mu} \otimes V_{\nu}]^{Sp_{2n}} \subset \mathcal{O}(\Conf_3 \A_{Sp_{2n}}).$$
Then if $f^{\vee}$ is the dual cluster variable to $f$, then
$$f^{\vee} \in [V_{\lambda} \otimes V_{\mu} \otimes V_{\nu}]^{Spin_{2n+1}} \subset \mathcal{O}(\Conf_3 \A_{Spin_{2n+1}})$$
if $f$ is associated to a black vertex and $f^{\vee}$ is associated to a white vertex, while
$$f^{\vee} \in [V_{\lambda/2} \otimes V_{\mu/2} \otimes V_{\nu/2}]^{Spin_{2n+1}} \subset \mathcal{O}(\Conf_3 \A_{Spin_{2n+1}})$$
if $f$ is associated to a white vertex and $f^{\vee}$ is associated to a black vertex.
\end{observation}

This is clearly true in the initial cluster, and as long as all the cluster variables are functions on $\Conf_3 \A_{Sp_{2n}}$ and $\Conf_3 \A_{Spin_{2n+1}}$, and not just rational functions on those spaces (as we expect, but do not know how to prove), it is easy to check that the above observation remains true under mutation. Certainly, in all the clusters we consider in this paper this will be the case, as is borne out in the computations that follow.

This gives us the following principle which will underlie the computations of the $S_3$ symmetries on $\Conf_3 \A_{Spin_{2n+1}}$ and the flip on $\Conf_4 \A_{Spin_{2n+1}}$:

\begin{observation} We can compute the formulas for the cluster variables on $\Conf_3 \A_{Spin_{2n+1}}$ and $\Conf_4 \A_{Spin_{2n+1}}$ that appear at various stages of mutation in the following way: Start with the formula for the corresponding cluster variable on $\Conf_m \A_{Sp_{2n}}$. Replace every instance of ``$a$'' where $1 \leq a \leq n-1$ by ``$a$,'' and replace every instance of ``$2n-a$'' where $1 \leq a \leq n-1$ by ``$2n+1-a$.'' Every instance of ``$n$'' should be replaced by either ``$n$'' or ``$n+1$,'' depending on the context. Take a square-root if the cluster variable corresponds to a black vertex for $\Conf_m \A_{Spin_{2n+1}}$. One then obtains the formula for the cluster variable on $\Conf_m \A_{Spin_{2n+1}}$.
\end{observation}

All the formulas we derive for $Spin_{2n+1}$ will follow this principle.

\subsection{Reduced words}

We now relate the cluster structure on $\Conf_3 \A_{Spin_{2n+1}}$ given in the previous section to Berenstein, Fomin and Zelevinsky's cluster structure on $B$, the Borel of the group $G$ (\cite{BFZ}). This will imply that the cluster we constructed induces the same positive structure on $\A_{G,S}$ as given in \cite{FG1}. For a reader not interested in the positive structure on $\A_{G,S}$, and more interested in just understanding the cluster structure on $\A_{G,S}$, this section is not logically necessary.

Let us now restrict our attention to triples of principal flags of the form $(U^-, \overline{w_0}U^-, b\cdot \overline{w_0}U^-).$ We can consider the map

$$i: b \in B^- \rightarrow (U^-, \overline{w_0}U^-, b\cdot \overline{w_0}U^-) \in \Conf_3 \A_{Spin_{2n+1}}.$$

For $u, v$ elements of the Weyl group $W$ of $G$, we have the double Bruhat cell
$$G^{u,v}=B^+ \cdot u \cdot B^+ \cap B^- \cdot v \cdot B^v.$$
The cell $G^{w_0,e}$ is the on open part of $B^-$.

\begin{prop} The cluster algebra constructed above on $\Conf_3 \A_{Spin_{2n+1}}$, when restricted to the image of $i$, coincides with the cluster algebra structure given in \cite{BFZ} on $B^-=G^{w_0,e}.$
\end{prop}

\begin{proof} $G^{w_0,e}$ is the on open part of $B^-$. Following \cite{BFZ}, to get a cluster structure on this subset, we must choose a reduced-word for $w_0$. In the numbering of the nodes of the Dynkin diagram given above for $Spin_{2n+2}$, we choose the reduced word expression $$w_0=(s_n s_{n^*} s_{n-1} \cdots s_2 s_1)^n.$$
Here our convention is that the above word corresponds to the string $i_1, i_2, \dots, i_{n-1}, i_{n^*} i_n$ repeated $n$ times. 

Now let $G_0 = U^- H U^+ \subset G$ be the open subset of elements of $G$ having Gaussian decomposition $x=[x]_- [x]_0 [x]_+$. Then for any two elements $u, v \in W$, and any fundamental weight $\omega_i$, we have the {\em generalized minor} $\Delta_{u\omega_i, v\omega_i}(x)$ defined by

$$\Delta_{u\omega_i, v\omega_i}(x) := ([\overline{u}^{-1}x \overline{v}]_0)^{\omega_i}.$$

In our situation, we are interested in such minors when $u, v=e$, or when $v=e$ and $u=u_{ij}=(s_n s_{n^*} s_{n-1} \cdots s_2s_1)^{i-1} s_{n} s_{n^*} s_{n-1}\cdots s_j$ for $1 \leq i \leq n$ and $n \geq j \geq 1$.

Then the cluster functions on $B^-$ given in \cite{BFZ} are $\Delta_{\omega_i, \omega_i}$ for $1 \leq i \leq n$ (these are the functions associated to $u, v=e$), and, for $1 \leq i \leq n$ and $j=1, 2, \dots, n-1, n^*, n$,
$$\Delta_{u_{ij}\omega_j, \omega_j},$$
which are the functions associated to $v=e$ and $u=u_{ij}=(s_n s_{n^*} s_{n-1} \cdots s_2s_1)^i s_{n}s_{n^*}s_{n-1}\cdots s_j$. Note that $u_{ij}$ is the subword of $u$ that stops on the $i^{\textrm th}$ iteration of $s_j$.

We have the following claims:

\begin{enumerate}
\item Recall that when $i=0$, we assign the function $\tcfr{2n+1-j}{}{j}$ to $x_{ij}$ for $j<n$ and $\sqrt{\tcfr{n+1}{}{n}}$ to $x_{in}$. Then 
$$\tcfr{2n+1-j}{}{j} = \Delta_{\omega_j, \omega_j},$$
for $j < n$ and
$$\sqrt{\tcfr{n+1}{}{n}} = \Delta_{\omega_n, \omega_n}.$$

\item Recall that for $i \geq j \neq n$, we assign the function $\tcfr{n-i}{n+1+i-j}{j}$ to $x_{ij}$. When $i=j=n$, we assign the function $\sqrt{\tcfr{}{n+1}{n}}$ to $x_{nn}$. Then 
$$\tcfr{n-i}{n+1+i-j}{j} = \Delta_{u_{ij}\omega_j, \omega_j}$$ 
for $j \neq n$ and
$$\sqrt{\tcfr{}{n+1}{n}} = \Delta_{u_{nn}\omega_n, \omega_n}$$

\item Recall that when $i < j <n$, we assign the function $\tcfr{n-i, 2n+1+i-j}{n+1}{j}$ to $x_{ij}$. Then 
$$\tcfr{n-i, 2n+1+i-j}{n+1}{j} = \Delta_{u_{ij}\omega_j, \omega_j}.$$
When $i < j = n$, we assign the function $\sqrt{\tcfr{n-i, n+1+i}{n+1}{n}}$ to $x_{in}$. Then 
$$\sqrt{\tcfr{n-i, n+1+i}{n+1}{n}} = \Delta_{u_{in}\omega_n, \omega_n}.$$

\end{enumerate}

Thus, in all cases the function assigned to $x_{ij}$ is precisely $\Delta_{u_{ij}\omega_j, \omega_j}$.

The proof of these claims is a straightforward calculation.

It is convenient to fix maps 
$Spin_{2n+1} \twoheadrightarrow SO_{2n+1} \hookrightarrow SL_{2n+1}.$
Choose the quadratic form so that $<e_i,e_{2n+2-i}>=(-1)^{i-1}$ and all other pairings of basis elements are zero. The sign is chosen so that the representation $Spin_{2n+1} \hookrightarrow SL_{2n+1}$ preserves positive structure. Now choose a pinning such that under the map $Spin_{2n+1} \rightarrow SL_{2n+1}$,

$$E_{\alpha_n}=\sqrt{2}(E_{n,n+1}+E_{n+1,n+2}), F_{\alpha_n}=\sqrt{2}(E_{n+1,n}+E_{n+2,n+1}),$$
and for $1\leq i<n$,
$$E_{\alpha_i}=E_{i,i+1}+E_{2n+1-i,2n+2-i},  F_{\alpha_i}=E_{i+1,i}+E_{2n+2-i,2n+1-i}.$$
Here $E_{i,j}$ is the $(i,j)$-elementary matrix, i.e., the matrix with a $1$ in the $(i,j)$ position and $0$ in all other positions.

With respect to this embedding, we can directly calculate $\Delta_{\omega_j, \omega_j}(x)$ where $x \in B^-$. When $x$ is embedded in $SL_{2n+1}$, $\Delta_{\omega_j, \omega_j}(x)$ is simply the determinant of the minor consisting of the first $j$ rows and the first $j$ columns for $j<n$, while $\Delta_{\omega_n, \omega_n}(x)$ is the square-root of the determinant of the minor consisting of the first $n$ rows and columns.

Similarly, one can calculate that $\Delta_{u_{ij}\omega_j, \omega_j}(x)$ only depends on the entries in rows $n+i+1, n+i, \dots, n+2, 1, 2, \dots n-i$ and the first $n$ columns. In particular, for $j< n$, $\Delta_{u_{ij}\omega_j, \omega_j}(x)$ is the determinant of the minor consisting of the first $j$ of those rows (in the order listed above) and the first $j$ columns, while for $j=n$, it is the square-root of the determinant of the minor consisting of all $n$ rows listed and the first $n$ columns.

We then must calculate the functions 
$$\tcfr{2n+1-j}{}{j}, \tcfr{n-i}{n+1+i-j}{j}, \tcfr{n-i, 2n+1+i-j}{n+1}{j}$$ on the triple of flags $(U^-, \overline{w_0}U^-, b\cdot \overline{w_0}U^-)$. Under the embedding  $Sp_{2n+1} \hookrightarrow SL_{2n+1}$, we should choose the flag $U^-$ to be be $e_{2n+1}, -e_{2n}, e_{2n-1}, dots, (-1)^n e_{n+1}$, so that $\overline{w_0}U^-$ is given by the flag $e_1, e_2, \dots, e_n$. Direct calculation then shows that $\tcfr{2n-j}{}{j}, \tcfr{n-i}{n+i-j}{j}$ and $\tcfr{n-i, 2n+i-j}{n}{j}$ are given by the appropriate determinants of minors or square-roots of these determinants of minors.

\end{proof}

Finally, note that we have the following equalities of functions:

\begin{equation} \label{dualities2}
\begin{split}
\dud{k}{2n+1-k}&=\dud{2n+1-k}{k} \\
\tcfr{n-i}{n+1+i-j}{j}&=\tcfr{n+1+i}{n-i+j}{2n+1-j} \\
\tcfr{n-i, 2n+1+i-j}{n+1}{j}&=\tcfr{n+1+i, j-i}{n}{2n+1- j}
\end{split}
\end{equation}
These equalities arise because the quadratic form induces an isomorphism between $\bigwedge\nolimits^i V$ and $\bigwedge\nolimits^{2n+1-i} V$.

\subsection{Relationship with cactus transformations}

As in the case of $Sp_{2n}$, the cactus sequence can be used to construct the cluster algebra structure on $\Conf_3 \A_{Spin_{2n+1}}$. The material of this section is not relevant to the remainder of the paper. We include this material for the following reasons: for the sake of completeness; to emphasize the parallels between $Sp_{2n}$ and $Spin_{2n+1}$; to motivate some computations that we perform later; and to point out a new phenomenon similar to folding of cluster seeds that deserves study in its own right.

Recall the cactus sequence of mutations on $\Conf_3 \A_{SL_{N}}$. We now specialize to $N=2n+1$. The cactus sequence will allow us to give an alternative way to construct the cluster algebra structure for $\Conf_3 \A_{Spin_{2n+1}}$. We perform the following sequence of mutations on $\Conf_3 \A_{SL_{2n+1}}$:

$$x_{N-2,1,1}, x_{N-3,1,2}, \dots, x_{1,1,N-2}$$
$$x_{N-3,2,1}, x_{N-4,2,2}, \dots, x_{1,2,N-3}$$
$$x_{N-2,1,1}, x_{N-3,1,2}, \dots, x_{2,1,N-3}$$
$$x_{N-4,3,1}, x_{N-5,3,2}, \dots, x_{1,3,N-4}$$
$$x_{N-3,2,1}, x_{N-4,2,2}, \dots, x_{2,2,N-4}$$
$$x_{N-2,1,1}, x_{N-3,1,2}, \dots, x_{3,1,N-4}$$
$$\dots$$
$$x_{n+1,n-1,1}, x_{n,n-1,2}, \dots, x_{1,n-1,n+1}$$
$$x_{n+2,n-2,1}, x_{n,n-1,2}, \dots, x_{2,n-2,n+1}$$
$$\dots$$
$$x_{N-2,1,1}, x_{N-3,1,2}, \dots, x_{n-1,1,n+1}$$

One can picture the above sequence as consisting of $n-1$ stages, where each stage consists of mutating all vertices lying in a parallelogram.

Additionally, we perform the following sequence of mutations, which are ``half'' of the mutations in the next stage.
$$x_{n,n,1}, x_{n-1,n,2}, \dots, x_{2,n,n-1}$$
$$x_{n+1,n-1,1}, x_{n,n-1,2}, \dots, x_{4,n-1,n-2}$$
$$\dots$$
$$x_{N-4,3,1}, x_{N-5,3,2}$$
$$x_{N-3,2,1}$$

The result will be that we get the quiver pictured in Figure 27 for $Spin_7$.

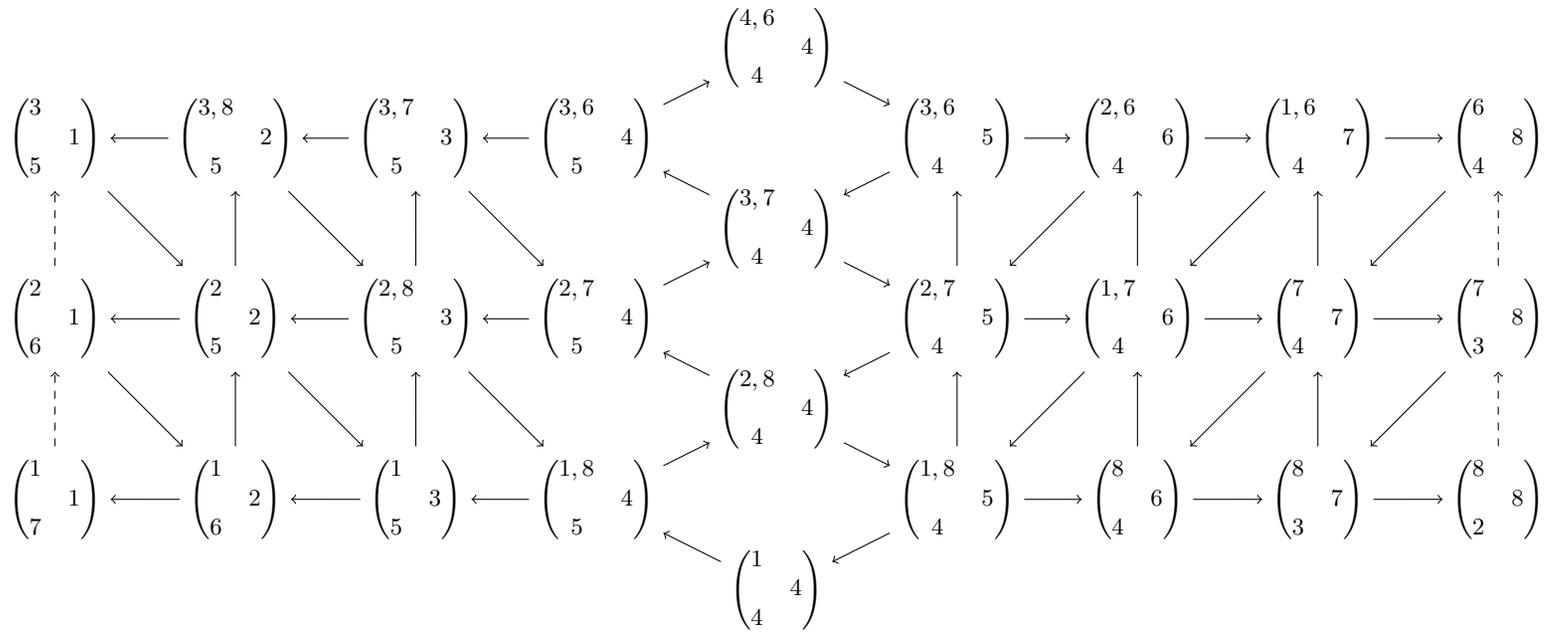
\begin{center}
\begin{tikzpicture}[scale=2.4,rotate=90]

  \node (x-31) at (-4,0) [rotate=90] {$\tcfr{3}{5}{1}$};
  \node (x-21) at (-3,0) [rotate=90] {$\tcfr{3,8}{5}{2}$};
  \node (x-11) at (-2,0) [rotate=90] {$\tcfr{3,7}{5}{3}$};
  \node (x-01) at (-1,0) [rotate=90] {$\tcfr{3,6}{5}{4}$};
  \node (x01) at (1,0) [rotate=90] {$\tcfr{3,6}{4}{5}$};
  \node (x11) at (2,0) [rotate=90] {$\tcfr{2,6}{4}{6}$};
  \node (x21) at (3,0) [rotate=90] {$\tcfr{1,6}{4}{7}$};
  \node (x31) at (4,0) [rotate=90] {$\tcfr{6}{4}{8}$};
  \node (x-32) at (-4,-1) [rotate=90] {$\tcfr{2}{6}{1}$};
  \node (x-22) at (-3,-1) [rotate=90] {$\tcfr{2}{5}{2}$};
  \node (x-12) at (-2,-1) [rotate=90] {$\tcfr{2,8}{5}{3}$};
  \node (x-02) at (-1,-1) [rotate=90] {$\tcfr{2,7}{5}{4}$};
  \node (x02) at (1,-1) [rotate=90] {$\tcfr{2,7}{4}{5}$};
  \node (x12) at (2,-1) [rotate=90] {$\tcfr{1,7}{4}{6}$};
  \node (x22) at (3,-1) [rotate=90] {$\tcfr{7}{4}{7}$};
  \node (x32) at (4,-1) [rotate=90] {$\tcfr{7}{3}{8}$};
  \node (x-33) at (-4,-2) [rotate=90] {$\tcfr{1}{7}{1}$};
  \node (x-23) at (-3,-2) [rotate=90] {$\tcfr{1}{6}{2}$};
  \node (x-13) at (-2,-2) [rotate=90] {$\tcfr{1}{5}{3}$};
  \node (x-03) at (-1,-2) [rotate=90] {$\tcfr{1,8}{5}{4}$};
  \node (x03) at (1,-2) [rotate=90] {$\tcfr{1,8}{4}{5}$};
  \node (x13) at (2,-2) [rotate=90] {$\tcfr{8}{4}{6}$};
  \node (x23) at (3,-2) [rotate=90] {$\tcfr{8}{3}{7}$};
  \node (x33) at (4,-2) [rotate=90] {$\tcfr{8}{2}{8}$};
  \node (x1) at (0,0.5) [rotate=90] {$\tcfr{4,6}{4}{4}$};
  \node (x2) at (0,-0.5) [rotate=90] {$\tcfr{3,7}{4}{4}$};
  \node (x3) at (0,-1.5) [rotate=90] {$\tcfr{2,8}{4}{4}$};
  \node (x4) at (0,-2.5) [rotate=90] {$\tcfr{1}{4}{4}$};

  \draw [->] (x4) to (x-03);
  \draw [->] (x-03) to (x3);
  \draw [->] (x3) to (x-02);
  \draw [->] (x-02) to (x2);  
  \draw [->] (x2) to (x-01);
  \draw [->] (x-01) to (x1);

  \draw [->] (x1) to (x01);
  \draw [->] (x01) to (x2);
  \draw [->] (x2) to (x02);
  \draw [->] (x02) to (x3);  
  \draw [->] (x3) to (x03);
  \draw [->] (x03) to (x4);

  \draw [->] (x01) to (x11);
  \draw [->] (x11) to (x21);
  \draw [->] (x21) to (x31);
  \draw [->] (x02) to (x12);
  \draw [->] (x12) to (x22);
  \draw [->] (x22) to (x32);
  \draw [->] (x03) to (x13);
  \draw [->] (x13) to (x23);
  \draw [->] (x23) to (x33);

  \draw [->, dashed] (x33) to (x32);
  \draw [->, dashed] (x32) to (x31);

  \draw [->] (x03) to (x02);
  \draw [->] (x02) to (x01);
  \draw [->] (x13) to (x12);
  \draw [->] (x12) to (x11);
  \draw [->] (x23) to (x22);
  \draw [->] (x22) to (x21);

  \draw [->] (x11) to (x02);
  \draw [->] (x12) to (x03); 
  \draw [->] (x21) to (x12);
  \draw [->] (x22) to (x13);
  \draw [->] (x31) to (x22);
  \draw [->] (x32) to (x23);

  \draw [->] (x-01) to (x-11);
  \draw [->] (x-11) to (x-21);
  \draw [->] (x-21) to (x-31);
  \draw [->] (x-02) to (x-12);
  \draw [->] (x-12) to (x-22);
  \draw [->] (x-22) to (x-32);
  \draw [->] (x-03) to (x-13);
  \draw [->] (x-13) to (x-23);
  \draw [->] (x-23) to (x-33);

  \draw [->, dashed] (x-33) to (x-32);
  \draw [->, dashed] (x-32) to (x-31);

  \draw [->] (x-13) to (x-12);
  \draw [->] (x-12) to (x-11);
  \draw [->] (x-23) to (x-22);
  \draw [->] (x-22) to (x-21);

  \draw [->] (x-11) to (x-02);
  \draw [->] (x-12) to (x-03); 
  \draw [->] (x-21) to (x-12);
  \draw [->] (x-22) to (x-13);
  \draw [->] (x-31) to (x-22);
  \draw [->] (x-32) to (x-23);

\draw[yshift=-3.5cm]
  node[below,text width=10cm,rotate=90] 
  {
  Figure 27. The functions and quiver for the cluster $\Conf_3 \A_{SL_{9}}$ which folds to give a cluster for $\Conf_3 \A_{Spin_{9}}$.
  };

\end{tikzpicture}
\end{center}

The function attached to $x_{ijk}$ (here we take $i,j,k > 0$ to exclude frozen vertices) will be 
\begin{itemize}
\item $\tcfr{i}{j}{k}$ if $j \geq n+1$,
\item $\tcfr{n+j, i+j-n-1}{n+1}{k+n+1-j}$ if $k <j < n+1$ and $k < n+1$,
\item $\tcfr{n+1+j, i+j-n}{n}{k+n-j}$ if $k \leq j< n+1$,
\item $(j+k, i+j; k+i)$ if $k \geq n$.
\end{itemize}

Thus we see that the functions attached to $x_{ijk}$ and $x_{ikj}$ pull back to the same function on $\Conf_3 \A_{Spin_{2n+1}}$ via the map $\Conf_3 \A_{Spin_{2n+1}} \rightarrow \Conf_3 \A_{SL_{2n+1}}$. Here we use the identities \eqref{dualities}. In fact, for $j > k$ and $j \geq n+1$, we have that the vertex $x_{ijk}$ corresponds to the vertex $x_{n-i,k}$ in the quiver for $\Conf_3 \A_{Spin_{2n+1}}$, while for $j > k$ and $j \leq n$, we have that the vertex $x_{ijk}$ corresponds to the vertex $x_{k,n+1+k-j}$.

The vertices where $j=k$ do not correspond to vertices for the quiver for the cluster algebra on $\Conf_3 \A_{Spin_{2n+1}}$. However, something interesting happens. Note that the function attached to $x_{2n+1-2j, j, j}$ is $\tcfr{n+1+j, n+1-j}{n}{n}$. Mutating this vertex gives the function $\tcfr{n+j, n-j}{n+1}{n+1}$. However, note that

$$\tcfr{n+1+j, n+1-j}{n}{n}=\tcfr{n+j, n-j}{n+1}{n+1}$$
using the identities \eqref{dualities2}.
Using the identity for the cluster transformation we get that 
$$\tcfr{n+1+j, n+1-j}{n}{n}^2=\tcfr{n+1+j, n+1-j}{n}{n}\tcfr{n+j, n-j}{n+1}{n+1}$$
$$=\tcfr{n+1+j, n-j}{n+1}{n}\tcfr{n+j, n+1-j}{n}{n+1}+\tcfr{n+1+j, n-j}{n}{n+1}\tcfr{n+j, n+1-j}{n+1}{n}$$
$$=2\tcfr{n+1+j, n-j}{n+1}{n}\tcfr{n+j, n+1-j}{n}{n+1}$$
(In some cases we will get a degenerate version of this identity.)

Thus, although the functions $\tcfr{n+1+j, n+1-j}{n}{n}$ attached to $x_{2n+1-2j, j, j}$ do not get used as cluster functions on  $\Conf_3 \A_{Spin_{2n+1}}$, we can use the above identity to express them in terms cluster functions. This identity will become useful to us later.

The result of this is that we can find the cluster functions on $\Conf_3 \A_{Spin_{2n+1}}$ as follows: Start with the standard cluster algebra on $\Conf_3 \A_{SL_{2n+1}}$, which has functions attached to vertices $x_{ijk}$. Perform series of mutations given above, which is a subsequence of the cactus sequence. Discard the functions $\tcfr{n+1+j, n+1-j}{n}{n}$ attached to $x_{2n+1-2j, j, j}$. Identify the following pairs of vertices, which have functions which are equal when restricted to the image of $\Conf_3 \A_{Spin_{2n+1}}$: $x_{i,2n-i,0}$ and $x_{2n-i,i,0}$; $x_{i,0,2n-i}$ and $x_{2n-i,0,i}$; $x_{0,j,2n-j}$ and $x_{0,2n-j,j}$; $x_{ijk}$ and $x_{ikj}$ for $i,j,k >0$ and $j \neq k$. Now take the square-roots of the appropriate functions (those that correspond to $x_{i,n}$ or $y_n$). The only problem is that it is unclear to us how to read off the correct $B$-matrix/quiver. This phenomenon, which is reminiscent of folding of cluster algebras, deserves to be investigated more generally.

More invariantly, there is an outer automorphism of $SL_{2n+1}$ that has $SO_{2n+1}$ as its fixed locus. This gives an involution of $\Conf_3 \A_{SL_{2n+1}}$ (and more generally $\Conf_m \A_{SL_{2n+1}}$) that has $\Conf_3 \A_{SO_{2n+1}}$ (respectively, $\Conf_m \A_{SO_{2n+1}}$) as its fixed locus. It turns out that the cluster algebra structure on $\Conf_3 \A_{SL_{2n}}$ is preserved by this involution. Moreover, there is a particular seed, constructed above, that is {\em almost} preserved by this involution: only the functions attached to $x_{2n+1-2j, j, j}$ change.

It is also clarifying to step back and motivate the sequence of mutations realizing the cactus sequence above. As explained in section 3.2, the cluster structure on $\Conf_3 \A$ comes from a reduced word for the longest element $w_0$ in the Weyl group of $G$. The initial seed for $\Conf_3 \A_{SL_{2n+1}}$ is built using the reduced word 
$$s_1 s_2 \dots s_{2n} s_1 s_2 \dots s_{2n-1} \dots s_1 s_2 s_3 s_1 s_2 s_1.$$
Here $s_1, \dots, s_{2n}$ are the generators of the Weyl group for $\Conf_3 \A_{SL_{2n+1}}$. It is known how to use cluster transformations to pass between the clusters that are associated to different reduced words (\cite{BFZ}). The cactus sequence transforms between the cluster above and the cluster associated to the reduced word
$$s_{2n} s_{2n-1} \dots s_{1} s_{2n} s_{2n-1} \dots s_{2} \dots s_{2n} s_{2n-1} s_{2n-2} s_{2n} s_{2n-1} s_{2n}.$$

The subsequence given above transforms the initial cluster into the cluster associated with the reduced word
$$ (s_n s_{n+1} s_n s_{n-1} s_{n+2} s_{n-2} s_{n+3} \dots s_{1} s_{2n})^n.$$

Now let $s'_1, s'_2, \dots s'_n$ be the generators of the Weyl group of $\Conf_3 \A_{Sp_{2n}}$. There is an injection from the Weyl group of $\Conf_3 \A_{Spin_{2n+1}}$ to the Weyl group of $\Conf_3 \A_{SL_{2n+1}}$ that takes
$$s'_n \rightarrow s_n s_{n+1} s_n,$$
$$s'_i \rightarrow s_i s_{2n+1-i}.$$
Thus we see that the folding-like phenomenon that relates the cluster structures on $\Conf_3 \A_{Spin_{2n+1}}$ and $\Conf_3 \A_{SL_{2n+1}}$ has an incarnation on the level of Weyl groups.

\subsection{The sequence of mutations realizing $S_3$ symmetries and flips}

We have described six different cluster structures on $\Conf_3 \A_{Spin_{2n+1}}$. We would now like to give sequences of mutations relating these six clusters to show that they are actually all clusters in the same cluster algebra.

As in the case of $Sp_{2n}$, we will realize the $S_3$ symmetries on $\Conf_3 \A_{Spin_{2n+1}}$ by exhibiting sequences of mutations that realize two different transpositions in the group $S_3$. In fact, the sequences of mutations used to realize these transpositions is the same as for $Sp_{2n}$. This is a reflection of the Langlands duality between the seeds for $\Conf_3 \A_{Spin_{2n+1}}$ and $\Conf_3 \A_{Sp_{2n}}$. As before, the identities used in the mutation sequence of one of these transpositions are exactly those of the cactus sequence. The other transposition requires a different analysis.

\subsubsection{The first transposition}

Let $(A,B,C) \in \Conf_3 \A_{Spin_{2n+1}}$ be a triple of flags. The sequence of mutations that realizes that $S_3$ symmetry $(A,B,C) \rightarrow (A,C,B)$ is the same as \eqref{23Sp}:

\begin{equation}
\begin{gathered}
x_{11}, x_{21}, x_{22}, x_{12}, x_{13}, x_{23}, x_{33}, x_{32}, x_{33,} \dots, x_{1,n-1}, \dots, x_{n-1, n-1}, \dots, x_{n-1, 1}, \\
x_{11}, x_{21}, x_{22}, x_{12},  \dots x_{1,n-2}, \dots, x_{n-2, n-2}, \dots, x_{n-2, 1}, \\
\dots, \\
x_{11}, x_{21}, x_{22}, x_{12}, \\
x_{11} \\
\end{gathered}
\end{equation}

The sequence can be thought of as follows: At any step of the process, we mutate all $x_{ij}$ such that $\max(i,j)$ is constant. It will not matter in which order we mutate these $x_{ij}$ because the vertices we mutate have no arrows between them. So we first mutate the $x_{ij}$ such that $\max(i,j)=1$, then the $x_{ij}$ such that $\max(i,j)=2$, then the $x_{ij}$ such that $\max(i,j)=3$, etc. The sequence of maximums that we use is 
$$1, 2, 3, \dots n-1, 1, 2, \dots n-2 \dots, 1, 2, 3, 1, 2, 1.$$

The evolution of the quiver for $\Conf_3 \A_{Spin_{2n+1}}$ is just as in the case for $\Conf_3 \A_{Sp_{2n}}$, as pictured in Figures 13 and 14, with the only difference being that black and white vertices switch colors. 

In Figure 28, we depict how the quiver for $\Conf_3 \A_{Spin_9}$ changes after performing the sequence of mutations of $x_{ij}$ having maximums $1$; $1, 2$; and $1, 2, 3$.

\begin{center}
\begin{tikzpicture}[scale=2.2]
\begin{scope}[xshift=-1.5cm]

  \node (x01) at (1,0) {$\dud{3}{6}$};
  \node (x02) at (2,0) {$\tcfr{7}{}{2}$};
  \node (x03) at (3,0) {$\tcfr{6}{}{3}$};
  \node (x04) at (4,0) {$\sqrt{\tcfr{5}{}{4}}$};
  \node (x11) at (1,-1) {$\boldsymbol{\tcfr{2,8}{6}{2}}$};
  \node (x12) at (2,-1) {$\tcfr{3, 8}{5}{2}$};
  \node (x13) at (3,-1) {$\tcfr{3, 7}{5}{3}$};
  \node (x14) at (4,-1) {$\sqrt{\tcfr{3, 6}{5}{4}}$};
  \node (x21) at (1,-2) {$\tcfr{2}{6}{1}$};
  \node (x22) at (2,-2) {$\tcfr{2}{5}{2}$};
  \node (x23) at (3,-2) {$\tcfr{2, 8}{5}{3}$};
  \node (x24) at (4,-2) {$\sqrt{\tcfr{2, 7}{5}{4}}$};
  \node (x31) at (1,-3) {$\tcfr{1}{7}{1}$};
  \node (x32) at (2,-3) {$\tcfr{1}{6}{2}$};
  \node (x33) at (3,-3) {$\tcfr{1}{5}{3}$};
  \node (x34) at (4,-3) {$\sqrt{\tcfr{1, 8}{5}{4}}$};
  \node (x41) at (1,-4) {$\tcfr{}{8}{1}$};
  \node (x42) at (2,-4) {$\tcfr{}{7}{2}$};
  \node (x43) at (3,-4) {$\tcfr{}{6}{3}$};
  \node (x44) at (4,-4) {$\sqrt{\tcfr{}{5}{4}}$};
  \node (y1) at (0,-3) {$\dud{1}{8}$};
  \node (y2) at (0,-2) {$\dud{2}{7}$};
  \node (y3) at (0,-1) {$\tcfr{8}{}{1}$};
  \node (y4) at (5,-1) {$\sqrt{\dud{4}{5}}$};

  \draw [->] (x41) to (x31);
  \draw [->] (x31) to (x21);
  \draw [->] (x11) to (x21);
  \draw [->] (x01) to (x11);
  \draw [->] (x42) to (x32);
  \draw [->] (x32) to (x22);
  \draw [->] (x12) to (x02);
  \draw [->] (x43) to (x33);
  \draw [->] (x33) to (x23);
  \draw [->] (x23) to (x13);
  \draw [->] (x13) to (x03);
  \draw [->] (x44) to (x34);
  \draw [->] (x34) to (x24);
  \draw [->] (x24) to (x14);
  \draw [->] (x14) to (x04);
  \draw [->, dashed] (x01) .. controls +(45:1) and +(up:2) .. (y4);
  \draw [->, dashed] (y2) .. controls +(135:1) and +(left:2) .. (x01);
  \draw [->, dashed] (y1) to (y2);

  \draw [->, dashed] (x04) to (x03);
  \draw [->, dashed] (x03) to (x02);
  \draw [->, dashed] (x02) .. controls +(135:1) and +(up:1.5) .. (y3);
  \draw [->] (y4) to (x14);
  \draw [->] (x14) to (x13);
  \draw [->] (x13) to (x12);
  \draw [->] (x11) to (x12);
  \draw [->] (y3) to (x11);
  \draw [->] (x24) to (x23);
  \draw [->] (x23) to (x22);
  \draw [->] (x21) to (y2);
  \draw [->] (x34) to (x33);
  \draw [->] (x33) to (x32);
  \draw [->] (x32) to (x31);
  \draw [->] (x31) to (y1);
  \draw [->, dashed] (x44) to (x43);
  \draw [->, dashed] (x43) to (x42);
  \draw [->, dashed] (x42) to (x41);

  \draw [->] (x12) to (x01);
  \draw [->] (x02) to (x13);
  \draw [->] (x03) to (x14);
  \draw [->] (x04) to (y4);
  \draw [->] (x21) to (y3);
  \draw [->] (x22) to (x11);
  \draw [->] (x12) to (x23);
  \draw [->] (x13) to (x24);
  \draw [->] (y2) to (x31);
  \draw [->] (x21) to (x32);
  \draw [->] (x22) to (x33);
  \draw [->] (x23) to (x34);
  \draw [->] (y1) to (x41);
  \draw [->] (x31) to (x42);
  \draw [->] (x32) to (x43);
  \draw [->] (x33) to (x44);

\end{scope}

\draw[yshift=-5cm,xshift=1cm]
  node[below,text width=6cm] 
  {
  Figure 28a The quiver and the functions for $\Conf_3 \A_{Spin_{9}}$ after performing the sequence of mutations of $x_{ij}$ having maximums $1$.
  };

\end{tikzpicture}
\end{center}

\begin{center}
\begin{tikzpicture}[scale=2.2]

\begin{scope}[xshift=-1.5cm]

  \node (x01) at (1,0) {$\tcfr{7}{}{2}$};
  \node (x02) at (2,0) {$\dud{3}{6}$};
  \node (x03) at (3,0) {$\tcfr{6}{}{3}$};
  \node (x04) at (4,0) {$\sqrt{\tcfr{5}{}{4}}$};
  \node (x11) at (1,-1) {${\tcfr{2,8}{6}{2}}$};
  \node (x12) at (2,-1) {$\boldsymbol{\tcfr{2, 7}{6}{3}}$};
  \node (x13) at (3,-1) {$\tcfr{3, 7}{5}{3}$};
  \node (x14) at (4,-1) {$\sqrt{\tcfr{3, 6}{5}{4}}$};
  \node (x21) at (1,-2) {$\boldsymbol{\tcfr{1,8}{7}{2}}$};
  \node (x22) at (2,-2) {$\boldsymbol{\tcfr{1,8}{6}{3}}$};
  \node (x23) at (3,-2) {$\tcfr{2, 8}{5}{3}$};
  \node (x24) at (4,-2) {$\sqrt{\tcfr{2, 7}{5}{4}}$};
  \node (x31) at (1,-3) {$\tcfr{1}{7}{1}$};
  \node (x32) at (2,-3) {$\tcfr{1}{6}{2}$};
  \node (x33) at (3,-3) {$\tcfr{1}{5}{3}$};
  \node (x34) at (4,-3) {$\sqrt{\tcfr{1, 8}{5}{4}}$};
  \node (x41) at (1,-4) {$\tcfr{}{8}{1}$};
  \node (x42) at (2,-4) {$\tcfr{}{7}{2}$};
  \node (x43) at (3,-4) {$\tcfr{}{6}{3}$};
  \node (x44) at (4,-4) {$\sqrt{\tcfr{}{5}{4}}$};
  \node (y1) at (0,-3) {$\dud{1}{8}$};
  \node (y2) at (0,-2) {$\tcfr{8}{}{1}$};
  \node (y3) at (0,-1) {$\dud{2}{7}$};
  \node (y4) at (5,-1) {$\sqrt{\dud{4}{5}}$};

  \draw [->] (x41) to (x31);
  \draw [->] (x21) to (x31);
  \draw [->] (x21) to (x11);
  \draw [->] (x11) to (x01);
  \draw [->] (x42) to (x32);
  \draw [->] (x22) to (x32);
  \draw [->] (x02) to (x12);
  \draw [->] (x43) to (x33);
  \draw [->] (x13) to (x03);
  \draw [->] (x44) to (x34);
  \draw [->] (x34) to (x24);
  \draw [->] (x24) to (x14);
  \draw [->] (x14) to (x04);
  \draw [->, dashed] (x02) .. controls +(up:1) and +(up:2) .. (y4);
  \draw [->, dashed] (y3) .. controls +(up:2) and +(up:1) .. (x02);
  \draw [->, dashed] (y1) .. controls +(left:1) and +(left:1) .. (y3);

  \draw [->, dashed] (x04) to (x03);
  \draw [->, dashed] (x03) .. controls +(up:1) and +(up:1) .. (x01);
  \draw [->, dashed] (x01)  .. controls +(left:2) and +(left:1) .. (y2);
  \draw [->] (y4) to (x14);
  \draw [->] (x14) to (x13);
  \draw [->] (x12) to (x13);
  \draw [->] (x12) to (x11);
  \draw [->] (x11) to (y3);
  \draw [->] (x24) to (x23);
  \draw [->] (x22) to (x23);
  \draw [->] (y2) to (x21);
  \draw [->] (x34) to (x33);
  \draw [->] (x31) to (y1);
  \draw [->, dashed] (x44) to (x43);
  \draw [->, dashed] (x43) to (x42);
  \draw [->, dashed] (x42) to (x41);

  \draw [->] (x01) to (x12);
  \draw [->] (x13) to (x02);
  \draw [->] (x03) to (x14);
  \draw [->] (x04) to (y4);
  \draw [->] (y3) to (x21);
  \draw [->] (x11) to (x22);
  \draw [->] (x23) to (x12);
  \draw [->] (x13) to (x24);
  \draw [->] (x31) to (y2);
  \draw [->] (x32) to (x21);
  \draw [->] (x33) to (x22);
  \draw [->] (x23) to (x34);
  \draw [->] (y1) to (x41);
  \draw [->] (x31) to (x42);
  \draw [->] (x32) to (x43);
  \draw [->] (x33) to (x44);

\end{scope}

\draw[yshift=-5cm,xshift=1cm]
  node[below,text width=6cm] 
  {
  Figure 28b The quiver and the functions for $\Conf_3 \A_{Spin_{9}}$ after performing the sequence of mutations of $x_{ij}$ having maximums $1, 2$.
  };

\end{tikzpicture}
\end{center}

\begin{center}
\begin{tikzpicture}[scale=2.2]

\begin{scope}[xshift=-1.5cm]

  \node (x01) at (1,0) {$\tcfr{7}{}{2}$};
  \node (x02) at (2,0) {$\tcfr{6}{}{3}$};
  \node (x03) at (3,0) {$\dud{3}{6}$};
  \node (x04) at (4,0) {$\sqrt{\dud{4}{5}}$};
  \node (x11) at (1,-1) {${\tcfr{2,8}{6}{2}}$};
  \node (x12) at (2,-1) {${\tcfr{2, 7}{6}{3}}$};
  \node (x13) at (3,-1) {$\boldsymbol{\tcfr{2, 6}{6}{4}}$};
  \node (x14) at (4,-1) {$\sqrt{\tcfr{3, 6}{5}{4}}$};
  \node (x21) at (1,-2) {${\tcfr{1,8}{7}{2}}$};
  \node (x22) at (2,-2) {${\tcfr{1,8}{6}{3}}$};
  \node (x23) at (3,-2) {$\boldsymbol{\tcfr{1,7}{6}{4}}$};
  \node (x24) at (4,-2) {$\sqrt{\tcfr{2, 7}{5}{4}}$};
  \node (x31) at (1,-3) {$\boldsymbol{\tcfr{8}{8}{2}}$};
  \node (x32) at (2,-3) {$\boldsymbol{\tcfr{8}{7}{3}}$};
  \node (x33) at (3,-3) {$\boldsymbol{\tcfr{8}{6}{4}}$};
  \node (x34) at (4,-3) {$\sqrt{\tcfr{1, 8}{5}{4}}$};
  \node (x41) at (1,-4) {$\tcfr{}{8}{1}$};
  \node (x42) at (2,-4) {$\tcfr{}{7}{2}$};
  \node (x43) at (3,-4) {$\tcfr{}{6}{3}$};
  \node (x44) at (4,-4) {$\sqrt{\tcfr{}{5}{4}}$};
  \node (y1) at (0,-3) {$\tcfr{8}{}{1}$};
  \node (y2) at (0,-2) {$\dud{1}{8}$};
  \node (y3) at (0,-1) {$\dud{2}{7}$};
  \node (y4) at (5,-1) {$\sqrt{\tcfr{5}{}{4}}$};

  \draw [->] (x31) to (x41);
  \draw [->] (x31) to (x21);
  \draw [->] (x21) to (x11);
  \draw [->] (x11) to (x01);
  \draw [->] (x32) to (x42);
  \draw [->] (x32) to (x22);
  \draw [->] (x22) to (x12);
  \draw [->] (x12) to (x02);
  \draw [->] (x33) to (x43);
  \draw [->] (x03) to (x13);
  \draw [->] (x34) to (x44);
  \draw [->] (x24) to (x34);
  \draw [->] (x14) to (x24);
  \draw [->] (x04) to (x14);
  \draw [->, dashed] (x03) to (x04);
  \draw [->, dashed] (y3) .. controls +(up:2) and +(up:1) .. (x03);
  \draw [->, dashed] (y2) to (y3);

  \draw [->, dashed] (y4) .. controls +(up:2) and +(up:1) .. (x02);
  \draw [->, dashed] (x02) to (x01);
  \draw [->, dashed] (x01) .. controls +(left:2) and +(left:1) .. (y1);
  \draw [->] (x14) to (y4);
  \draw [->] (x13) to (x14);
  \draw [->] (x13) to (x12);
  \draw [->] (x12) to (x11);
  \draw [->] (x11) to (y3);
  \draw [->] (x23) to (x24);
  \draw [->] (x23) to (x22);
  \draw [->] (x22) to (x21);
  \draw [->] (x21) to (y2);
  \draw [->] (x33) to (x34);
  \draw [->] (y1) to (x31);
  \draw [->, dashed] (x43) to (x44);
  \draw [->, dashed] (x42) to (x43);
  \draw [->, dashed] (x41) to (x42);

  \draw [->] (x01) to (x12);
  \draw [->] (x02) to (x13);
  \draw [->] (x14) to (x03);
  \draw [->] (x04) to (y4);
  \draw [->] (y3) to (x21);
  \draw [->] (x11) to (x22);
  \draw [->] (x12) to (x23);
  \draw [->] (x24) to (x13);
  \draw [->] (y2) to (x31);
  \draw [->] (x21) to (x32);
  \draw [->] (x22) to (x33);
  \draw [->] (x34) to (x23);
  \draw [->] (x41) to (y1);
  \draw [->] (x42) to (x31);
  \draw [->] (x43) to (x32);
  \draw [->] (x44) to (x33);

\end{scope}

\draw[yshift=-5cm,xshift=1cm]
  node[below,text width=6cm] 
  {
  Figure 28c The quiver and the functions for $\Conf_3 \A_{Spin_{9}}$ after performing the sequence of mutations of $x_{ij}$ having maximums $1,  2, 3$.
  };

\end{tikzpicture}
\end{center}

In Figure 29, we depict the state of the quiver after performing the sequence of mutations of $x_{ij}$ having maximums $1, 2, 3$; $1, 2, 3, 1, 2$; and $1, 2, 3, 1, 2, 1$.

\begin{center}
\begin{tikzpicture}[scale=2.2]

\begin{scope}[xshift=-1.5cm]

  \node (x01) at (1,0) {$\tcfr{6}{}{3}$};
  \node (x02) at (2,0) {$\dud{2}{7}$};
  \node (x03) at (3,0) {$\dud{3}{6}$};
  \node (x04) at (4,0) {$\sqrt{\dud{4}{5}}$};
  \node (x11) at (1,-1) {$\boldsymbol{\tcfr{1,7}{7}{3}}$};
  \node (x12) at (2,-1) {$\boldsymbol{\tcfr{1,6}{7}{4}}$};
  \node (x13) at (3,-1) {${\tcfr{2, 6}{6}{4}}$};
  \node (x14) at (4,-1) {$\sqrt{\tcfr{3, 6}{5}{4}}$};
  \node (x21) at (1,-2) {$\boldsymbol{\tcfr{7}{8}{3}}$};
  \node (x22) at (2,-2) {$\boldsymbol{\tcfr{7}{7}{4}}$};
  \node (x23) at (3,-2) {${\tcfr{1,7}{6}{4}}$};
  \node (x24) at (4,-2) {$\sqrt{\tcfr{2, 7}{5}{4}}$};
  \node (x31) at (1,-3) {${\tcfr{8}{8}{2}}$};
  \node (x32) at (2,-3) {${\tcfr{8}{7}{3}}$};
  \node (x33) at (3,-3) {${\tcfr{8}{6}{4}}$};
  \node (x34) at (4,-3) {$\sqrt{\tcfr{1, 8}{5}{4}}$};
  \node (x41) at (1,-4) {$\tcfr{}{8}{1}$};
  \node (x42) at (2,-4) {$\tcfr{}{7}{2}$};
  \node (x43) at (3,-4) {$\tcfr{}{6}{3}$};
  \node (x44) at (4,-4) {$\sqrt{\tcfr{}{5}{4}}$};
  \node (y1) at (0,-3) {$\tcfr{8}{}{1}$};
  \node (y2) at (0,-2) {$\tcfr{7}{}{2}$};
  \node (y3) at (0,-1) {$\dud{1}{8}$};
  \node (y4) at (5,-1) {$\sqrt{\tcfr{5}{}{4}}$};

  \draw [->] (x11) to (x01);
  \draw [->] (x21) to (x11);
  \draw [->] (x21) to (x31);
  \draw [->] (x31) to (x41);
  \draw [->] (x02) to (x12);
  \draw [->] (x22) to (x32);
  \draw [->] (x32) to (x42);
  \draw [->] (x03) to (x13);
  \draw [->] (x13) to (x23);
  \draw [->] (x23) to (x33);
  \draw [->] (x33) to (x43);
  \draw [->] (x04) to (x14);
  \draw [->] (x14) to (x24);
  \draw [->] (x24) to (x34);
  \draw [->] (x34) to (x44);
  \draw [->, dashed] (y4) .. controls +(up:2) and +(up:1) .. (x01);
  \draw [->, dashed] (x01) .. controls +(left:2) and +(left:1) .. (y2);
  \draw [->, dashed] (y2) to (y1);

  \draw [->, dashed] (y3) .. controls +(up:2) and +(up:1) .. (x02);
  \draw [->, dashed] (x02) to (x03);
  \draw [->, dashed] (x03) to (x04);
  \draw [->] (x11) to (y3);
  \draw [->] (x12) to (x11);
  \draw [->] (x12) to (x13);
  \draw [->] (x13) to (x14);
  \draw [->] (x14) to (y4);
  \draw [->] (y2) to (x21);
  \draw [->] (x22) to (x23);
  \draw [->] (x23) to (x24);
  \draw [->] (y1) to (x31);
  \draw [->] (x31) to (x32);
  \draw [->] (x32) to (x33);
  \draw [->] (x33) to (x34);
  \draw [->, dashed] (x41) to (x42);
  \draw [->, dashed] (x42) to (x43);
  \draw [->, dashed] (x43) to (x44);

  \draw [->] (x01) to (x12);
  \draw [->] (x13) to (x02);
  \draw [->] (x14) to (x03);
  \draw [->] (y4) to (x04);
  \draw [->] (y3) to (x21);
  \draw [->] (x11) to (x22);
  \draw [->] (x23) to (x12);
  \draw [->] (x24) to (x13);
  \draw [->] (x31) to (y2);
  \draw [->] (x32) to (x21);
  \draw [->] (x33) to (x22);
  \draw [->] (x34) to (x23);
  \draw [->] (x41) to (y1);
  \draw [->] (x42) to (x31);
  \draw [->] (x43) to (x32);
  \draw [->] (x44) to (x33);

\end{scope}

\draw[yshift=-5cm,xshift=1cm]
  node[below,text width=6cm] 
  {
  Figure 29a The quiver and the functions for $\Conf_3 \A_{Spin_{9}}$ after performing the sequence of mutations of $x_{ij}$ having maximums $1,  2, 3, 1, 2$.
  };

\end{tikzpicture}
\end{center}

\begin{center}
\begin{tikzpicture}[scale=2.2]

\begin{scope}[xshift=-1.5cm]

  \node (x01) at (1,0) {$\dud{1}{8}$};
  \node (x02) at (2,0) {$\dud{2}{7}$};
  \node (x03) at (3,0) {$\dud{3}{6}$};
  \node (x04) at (4,0) {$\sqrt{\dud{4}{5}}$};
  \node (x11) at (1,-1) {$\boldsymbol{\tcfr{6}{8}{4}}$};
  \node (x12) at (2,-1) {${\tcfr{1,6}{7}{4}}$};
  \node (x13) at (3,-1) {${\tcfr{2, 6}{6}{4}}$};
  \node (x14) at (4,-1) {$\sqrt{\tcfr{3, 6}{5}{4}}$};
  \node (x21) at (1,-2) {${\tcfr{7}{8}{3}}$};
  \node (x22) at (2,-2) {${\tcfr{7}{7}{4}}$};
  \node (x23) at (3,-2) {${\tcfr{1,7}{6}{4}}$};
  \node (x24) at (4,-2) {$\sqrt{\tcfr{2, 7}{5}{4}}$};
  \node (x31) at (1,-3) {${\tcfr{8}{8}{2}}$};
  \node (x32) at (2,-3) {${\tcfr{8}{7}{3}}$};
  \node (x33) at (3,-3) {${\tcfr{8}{6}{4}}$};
  \node (x34) at (4,-3) {$\sqrt{\tcfr{1, 8}{5}{4}}$};
  \node (x41) at (1,-4) {$\tcfr{}{8}{1}$};
  \node (x42) at (2,-4) {$\tcfr{}{7}{2}$};
  \node (x43) at (3,-4) {$\tcfr{}{6}{3}$};
  \node (x44) at (4,-4) {$\sqrt{\tcfr{}{5}{4}}$};
  \node (y1) at (0,-3) {$\tcfr{8}{}{1}$};
  \node (y2) at (0,-2) {$\tcfr{7}{}{2}$};
  \node (y3) at (0,-1) {$\tcfr{6}{}{3}$};
  \node (y4) at (5,-1) {$\sqrt{\tcfr{5}{}{4}}$};

  \draw [->] (x01) to (x11);
  \draw [->] (x11) to (x21);
  \draw [->] (x21) to (x31);
  \draw [->] (x31) to (x41);
  \draw [->] (x02) to (x12);
  \draw [->] (x12) to (x22);
  \draw [->] (x22) to (x32);
  \draw [->] (x32) to (x42);
  \draw [->] (x03) to (x13);
  \draw [->] (x13) to (x23);
  \draw [->] (x23) to (x33);
  \draw [->] (x33) to (x43);
  \draw [->] (x04) to (x14);
  \draw [->] (x14) to (x24);
  \draw [->] (x24) to (x34);
  \draw [->] (x34) to (x44);
  \draw [->, dashed] (y4) .. controls +(up:2) and +(up:2) .. (y3);
  \draw [->, dashed] (y3) to (y2);
  \draw [->, dashed] (y2) to (y1);

  \draw [->, dashed] (x01) to (x02);
  \draw [->, dashed] (x02) to (x03);
  \draw [->, dashed] (x03) to (x04);
  \draw [->] (y3) to (x11);
  \draw [->] (x11) to (x12);
  \draw [->] (x12) to (x13);
  \draw [->] (x13) to (x14);
  \draw [->] (x14) to (y4);
  \draw [->] (y2) to (x21);
  \draw [->] (x21) to (x22);
  \draw [->] (x22) to (x23);
  \draw [->] (x23) to (x24);
  \draw [->] (y1) to (x31);
  \draw [->] (x31) to (x32);
  \draw [->] (x32) to (x33);
  \draw [->] (x33) to (x34);
  \draw [->, dashed] (x41) to (x42);
  \draw [->, dashed] (x42) to (x43);
  \draw [->, dashed] (x43) to (x44);

  \draw [->] (x12) to (x01);
  \draw [->] (x13) to (x02);
  \draw [->] (x14) to (x03);
  \draw [->] (y4) to (x04);
  \draw [->] (x21) to (y3);
  \draw [->] (x22) to (x11);
  \draw [->] (x23) to (x12);
  \draw [->] (x24) to (x13);
  \draw [->] (x31) to (y2);
  \draw [->] (x32) to (x21);
  \draw [->] (x33) to (x22);
  \draw [->] (x34) to (x23);
  \draw [->] (x41) to (y1);
  \draw [->] (x42) to (x31);
  \draw [->] (x43) to (x32);
  \draw [->] (x44) to (x33);

\end{scope}

\draw[yshift=-5cm,xshift=1cm]
  node[below,text width=6cm] 
  {
  Figure 29b The quiver and the functions for $\Conf_3 \A_{Spin_{9}}$ after performing the sequence of mutations of $x_{ij}$ having maximums $1, 2, 3, 1, 2, 1$.
  };

\end{tikzpicture}
\end{center}

From these diagrams the various quivers in the general case of $\Conf_3 \A_{Spin_{2n+1}}$ should be clear.

\begin{theorem}
If $\max(i,j)=k$, then $x_{ij}$ is mutated a total of $n-k$ times. Recall that when $i \geq j$, we assign the function $\tcfr{n-i}{n+1+i-j}{j}$ to $x_{ij}$. Thus the function attached to $x_{ij}$ transforms as follows:
$$\tcfr{n-i}{n+1+i-j}{j} \rightarrow \tcfr{2n, n-i-1}{n+i-j+2}{j+1} \rightarrow \tcfr{2n-1, n-i-2}{n+i-j+3}{j+2} \rightarrow \dots$$ 
$$\rightarrow \tcfr{n+i+2, 1}{2n-j}{n-i+j-1} \rightarrow \tcfr{n+i+1}{2n+1-j}{n-i+j}=\tcfr{n-i}{j}{n+1+i-j}$$

When $i < j$ and $i \neq 0$, we assign the function $\tcfr{n-i, 2n+1+i-j}{n+1}{j}$ to $x_{ij}$. Thus the function attached to $x_{ij}$ transforms as follows:
$$\tcfr{n-i, 2n+1+i-j}{n+1}{j} \rightarrow \tcfr{n-i-1, 2n+i-j}{n+2}{j+1} \rightarrow \tcfr{n-i-2, 2n+i-j-1}{n+3}{j+2} \rightarrow \dots$$
$$\rightarrow \tcfr{j-i+1, n+i+2}{2n-j}{n-1} \rightarrow \tcfr{j-i, n+i+1}{2n+1-j}{n}=\tcfr{n-i, 2n+1+i-j}{j}{n+1}$$

\end{theorem}

\begin{proof}

The proof is identical to the case where $G=Sp_{2n}$. Note that we do note mutate any of the vertices that have square-roots. The functions with square-roots are involved in the mutations, as we mutate vertices adjacent to them; however, because of the values of the multipliers $d_i$, the mutation of a white vertex involves only squares of the functions attached to black vertices. Therefore, we see that the identities we need to use are exactly those appearing in the cactus sequence.

\end{proof}

The above sequence of mutations takes us from one seed for the cluster algebra structure on $\Conf_3 \A_{Spin_{2n+1}}$ to another seed where the roles of the second and third principal flags have been reversed. Thus, we have realized the first of the transpositions necessary to construct all the $S_3$ symmetries of $\Conf_3 \A_{Spin_{2n+1}}$.

\subsubsection{The second transposition}

Let us now give the sequence of mutations that realizes that $S_3$ symmetry $(A,B,C) \rightarrow (C,B,A).$

The sequence of mutations is as in the case of $Sp_{2n}$, \eqref{13Sp}:
\begin{equation}
\begin{gathered}
x_{n-1,n}, x_{n-2,n-1}, x_{n-2,n}, x_{n-3,n-2}, x_{n-3,n-1}, x_{n-3,n}, \dots, x_{1,2}, \dots, x_{1,n}, \\
x_{n-1,n}, x_{n-2,n-1}, x_{n-2,n}, x_{n-3,n-2}, x_{n-3,n-1}, x_{n-3,n}, \dots, x_{2,3}, \dots, x_{2,n}, \\
x_{n-1,n}, x_{n-2,n-1}, x_{n-2,n}, x_{n-3,n-2}, x_{n-3,n-1}, x_{n-3,n}, \dots, x_{3,4}, \dots, x_{3,n}, \\
\dots
x_{n-1,n}, x_{n-2,n-1}, x_{n-2,n}, x_{n-3,n-2}, x_{n-3,n-1}, x_{n-3,n}, \\
x_{n-1,n}, x_{n-2,n-1}, x_{n-2,n} \\
x_{n-1,n} \\
\end{gathered}
\end{equation}

The sequence can be thought of as follows: We only mutate those $x_{ij}$ with $i < j$. At any step of the process, we mutate all $x_{ij}$ in the $k^{\textrm th}$ row (the $k^{\textrm th}$ row consists of $x_{ij}$ such that $i=k$) such that $i < j$. It will not matter in which order we mutate these $x_{ij}$. The sequence of rows that we mutate is
$$n-1, n-2 \dots, 2,1 n-1, n-2, \dots, 2, n, \dots, 3, \dots, n-1, n-2, n-1.$$ 

As in the previous transposition, the evolution of the quiver for $\Conf_3 \A_{Spin_{2n+1}}$ is just as in the case for $\Conf_3 \A_{Sp_{2n}}$, as pictured in Figures 15 and 16, with the only difference being that black and white vertices switch colors. 

In Figure 30, we depict how the quiver for $\Conf_3 \A_{Spin_{11}}$ changes after performing the sequence of mutations of $x_{ij}$ in rows $4$; $4,3$; $4,3,2$; and $4,3,2,1$.

\begin{center}

\end{center}

In Figure 31, we depict the state of the quiver after performing the sequence of mutations of $x_{ij}$ in rows $4,3,2,1$; $4,3,2,1,4,3,2$; $4,3,2,1,4,3,2,4,3$; and $4,3,2,1,4,3,2,4,3,4$.

\begin{center}
%
\end{center}

The circles on several of the arrows are a bookkeeping device that tell us how to lift a function from $\Conf_3 \A_{Spin_{11}}$ to $\Conf_3 \A_{SL_{11}}$ in the way that is most convenient for our computations. From these diagrams the various quivers in the general case of $\Conf_3 \A_{Spin_{2n+1}}$ should be clear. 

Recall the functions defined via Figures 17. They will appear when we perform the sequence of mutations above. In particular, we make use of functions of the form
$$\tcfr{n-i, n+1+i}{n, n+1}{n-j, n+1+j}.$$
These functions are invariants (unique up to scale) of the tensor product
$$[V_{2\omega_{n-i}} \otimes V_{4\omega_{n}} \otimes V_{2\omega_{n-j}}]^{Spin_{2n+1}}.$$
These functions will have square-roots which are invariants (again, unique up to scale) of the tensor product
$$[V_{\omega_{n-i}} \otimes V_{2\omega_{n}} \otimes V_{\omega_{n-j}}]^{Spin_{2n+1}}.$$
Thus
$$\sqrt{\tcfr{n-i, n+1+i}{n, n+1}{n-j, n+1+j}}$$
is a well-defined function on $\Conf_3 \A_{Spin_{2n+1}}$.

Note that in the above sequence of mutations, $x_{ij}$ is mutated $i$ times if $i<j$. We can now state the main theorem of this section.

\begin{theorem}
If $i < j$, then $x_{ij}$ is mutated a total of $i$ times. Recall that when $i < j$, we assign the either the function $\tcfr{n-i, 2n+1+i-j}{n}{j}$ or its square-root to $x_{ij}$ depending on whether $j < n$ or $j=n$. For $j < n$, the function attached to $x_{ij}$ transforms as follows:
$$\tcfr{n-i, 2n+1+i-j}{n+1}{j} \rightarrow \tcfr{n-i+1, 2n+i-j}{n, n+1}{j-1, n+2} \rightarrow $$
$$\tcfr{n-i+2, 2n+i-j-1}{n, n+1}{j-2, n+3} \rightarrow \dots \rightarrow \tcfr{n-1, 2n+2-j}{n, n+1}{j-i+1, n+i} $$
$$\rightarrow \tcfr{2n+1-j}{n}{j-i, n+i+1}=\tcfr{j}{n+1}{n-i, 2n+1+i-j}$$

The first transformation can be seen as the composite of two steps,
$$\tcfr{n-i, 2n+1+i-j}{n+1}{j} \rightarrow \tcfr{n-i, 2n+1+i-j}{n, n+1}{j, n+1} $$
$$\rightarrow \tcfr{n-i+1, 2n+i-j}{n, n+1}{j-1, n+2},$$
while the last transformation can also be seen as the composite of two steps, 
$$\tcfr{n-1, 2n+2-j}{n, n+1}{j-i+1, n+i} \rightarrow \tcfr{n, 2n+1-j}{n, n+1}{j-i, n+i+1} $$
$$\rightarrow \tcfr{2n+1-j}{n}{j-i, n+i+1}.$$
Then with each transformation, two of the parameters increase by one, and two decrease by one.

For $j=n$, we use the same formulas, but take the square roots of all the resulting functions.

\end{theorem}

\begin{proof}

We have already described the quivers at the various stages of mutation. We must then check that the functions above satisfy the identities of the associated cluster transformations.

This is the first sequence of mutations where we have to mutate the black vertices $x_{in}$ ($y_n$ is also black, but it does not get mutated). Proving that the black vertices mutate as expected will be our most difficult task.

In verifying the cluster identities that we need, we will actually be computing functions on $\Conf_3 \A_{SL_N}$. Thus we will use the arrows with a circle on them as a bookkeeping device. Circled arrows only occur between two white vertices. If a white vertex $x$ has an incoming or outgoing circled arrow from another white vertex $x'$, this means that in computing the mutation of $x$, we should use the dual of the function attached to $x'$.

With these rules in mind, we are ready to compute mutations.

We will need the following facts, which are analogous to the previous facts used for the case of $Sp_{2n}$. Let $N=2n+1$:
\begin{itemize}
\item Let $1 \leq a, b, c, d \leq N$, and $a+b+c+d=2N$. 
$$\tcfr{a, b}{n+1, n}{c, d}\tcfr{a+1, b-1}{n+1, n}{c-1, d+1}=$$
$$\tcfr{a, b}{n+1, n}{c-1, d+1}\tcfr{a+1, b-1}{n+1, n}{c, d}+\tcfr{a+1, b}{n+1, n}{c-1,d}\tcfr{a, b-1}{n+1, n}{c, d+1}.$$
\item If $a+c=n+1$ and $b+d=3n+1$,
$$\tcfr{a, b}{n+1, n}{c, d}=\tcfr{a}{n}{c}\tcfr{b}{n+1}{d}.$$
\item If $a+c=n$ and $b+d=3n+2$,
$$\tcfr{a, b}{n+1, n}{c, d}=\tcfr{a}{n+1}{c}\tcfr{b}{n}{d}.$$
\item If $a=n$,
$$\tcfr{n, b}{n+1, n}{c, d}=\dud{n}{n+1}\tcfr{b}{n}{c, d}.$$
Similarly, we have
$$\tcfr{a, b}{n+1, n}{c, n+1}=\tcfr{}{n}{n+1}\tcfr{a, b}{n+1}{c}.$$
\item We will need the duality identities of \eqref{dualities2}, and also the following duality identity:
$$\tcfr{a, b}{n+1, n}{c, d}=\tcfr{N-b, N-a}{n+1, n}{N-d, N-c}$$

\end{itemize}

All these identities can be proved by the same method as we used for $G=Sp_{2n}$ and $N=2n$. 

All cluster mutations of white vertices are obtained as before: they are either the first identity in the above list, or they are degenerations of this identity, and are obtained from the first one by applying the other three identities.

The mutation of black vertices requires some more work. We will also make use of functions of the form
$$\tcfr{a, b}{x, y}{c, d}$$
where $x=n$ or $n+1$, $y=n$ or $n+1$, and $a+b+x+y+c+d=3N$. These functions are defined just as in Figure 17. We depict $\tcfr{a, b}{n+1, n}{c, d}$ in Figure 17b.

\begin{center}
\begin{tikzpicture}[scale=0.6]
\begin{scope}[decoration={
    markings,
    mark=at position 0.5 with {\arrow{>}}}
    ] 
\draw [postaction={decorate}] (-8,6) -- (-6,3) node [midway,below left] {$a$}; 
\draw [postaction={decorate}] (-4,0) -- (-6,3) node [midway,below left] {$N-a$}; 
\draw [postaction={decorate}] (-2,9) -- (0,6) node [midway,below left] {$b$};
\draw [postaction={decorate}] (2,3) -- (0,6) node [midway,below left] {$N-b$};
\draw [postaction={decorate}] (-8,-6) -- (-4,0) node [midway,below right] {$n+1$};
\draw [postaction={decorate}] (-2,-9) -- (2,-3) node [midway,below right] {$n$};
\draw [postaction={decorate}] (0,0) -- (-4,0) node [midway,below] {$n-a$};
\draw [postaction={decorate}] (14,-3) -- (2,-3) node [midway,above] {$c$};
\draw [postaction={decorate}] (2,-3) -- (0,0) node [midway,below left] {$n+c$};
\draw [postaction={decorate}] (0,0) -- (2,3) node [midway,below right] {$a+c$};
\draw [postaction={decorate}] (14,3) -- (8,3) node [midway,above] {$d$};
\draw [postaction={decorate}] (2,3) -- (8,3) node [midway,above] {$N-d$};

\draw (8,2.5) -- (8,3) ;
\draw (-6.3, 2.8) -- (-6,3) ;
\draw (-0.3, 5.8) -- (0,6) ;

\end{scope}

\draw[yshift=-10cm]
  node[below,text width=6cm] 
  {
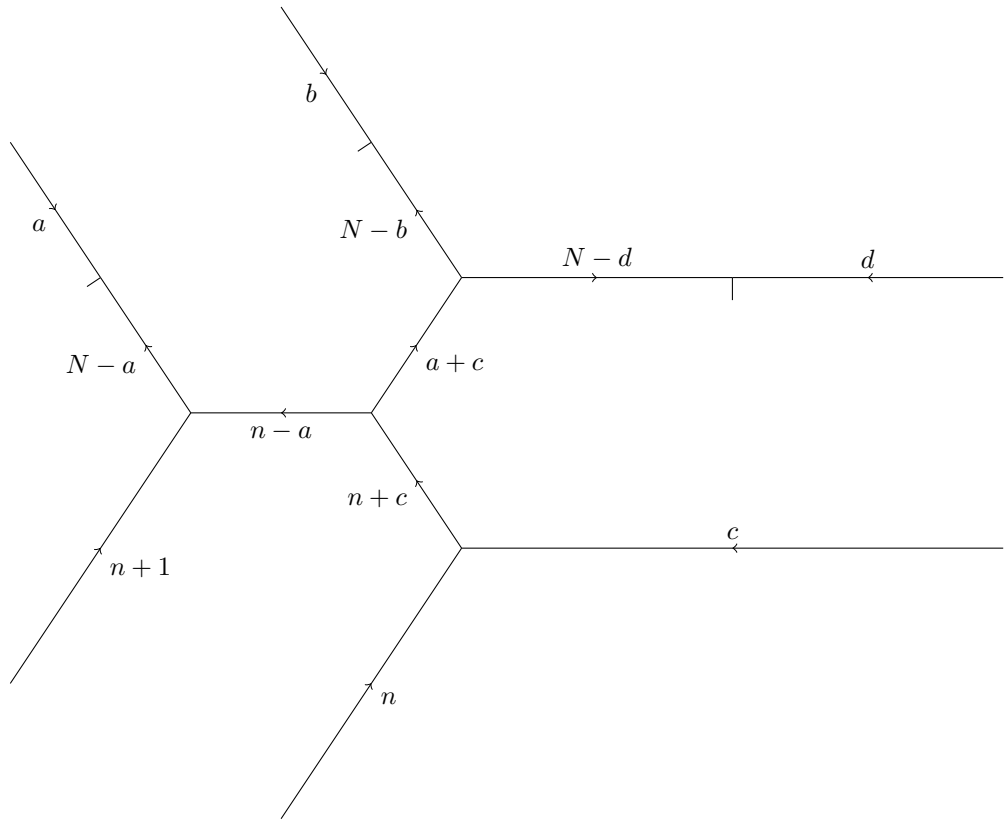
  Figure 17b. Web for the function $\tcfr{a,b}{n+1,n}{c,d}$ where $a+b+c+d=2N$.
  };

\end{tikzpicture}
\end{center}

The most general identity we will need to show has the form

$$\sqrt{\tcfr{n-i, n+1+i}{n+1, n}{n-j, n+1+j}} \sqrt{\tcfr{n+1-i, n+i}{n+1, n}{n-1-j, n+2+j}} = $$
$$\sqrt{\tcfr{n-i, n+1+i}{n+1, n}{n-1-j, n+2+j}}\sqrt{\tcfr{n+1-i, n+i}{n+1, n}{n-j, n+1+j}}+$$
$$\tcfr{n+1-i, n+1+i}{n+1, n}{n-1-j, n+1-j}.$$
Note that by duality \eqref{dualities2},
$$\tcfr{n-i, n+i}{n+1, n}{n-j, n+2-j}=\tcfr{n+1-i, n+1+i}{n, n+1}{n-1-j, n+1-j},$$
which explains the seeming asymmetry of the last term.

The general identity above follows directly from the following identities. Let $i, j < n$. Then 

\begin{itemize}

\item 
$$\tcfr{n-i, n+1+i}{n+1, n}{n-j, n+1+j}\tcfr{n+1-i, n+i}{n+1, n+1}{n-1-j, n+1+j}=$$
$$\tcfr{n+1-i, n+i}{n+1, n}{n-j, n+1+j}\tcfr{n-i, n+1+i}{n+1, n+1}{n-1-j, n+1+j}+$$
$$\tcfr{n+1-i, n+1+i}{n+1, n}{n-1-j, n+1-j}\tcfr{n-i, n+i}{n+1, n+1}{n-j, n+1+j},$$

\item 
$$\tcfr{n+1-i, n+i}{n+1, n+1}{n-1-j, n+1+j}=$$
$$\sqrt{2\tcfr{n+1-i, n+i}{n+1, n}{n-1-j, n+2+j}\tcfr{n+1-i, n+i}{n+1, n}{n-j, n+1+j}},$$

\item 
$$\tcfr{n-i, n+1+i}{n+1, n+1}{n-1-j, n+1+j}=$$
$$\sqrt{2\tcfr{n-i, n+1+i}{n+1, n}{n-j, n+1+j}\tcfr{n-i, n+1+i}{n+1, n}{n-1-j, n+2+j}},$$

\item 
$$\tcfr{n-i, n+i}{n+1, n+1}{n-j, n+1+j}=$$
$$\sqrt{2\tcfr{n-i, n+1+i}{n+1, n}{n-j, n+1+j}\tcfr{n+1-i, n+i}{n+1, n}{n-j, n+1+j}}.$$

\end{itemize}

The first identity is a relative of the octahedron recurrence, which we treated previously. The last three identities are all of the form
$$\tcfr{a, N-a}{n+1, n+1}{b-1, N-b}=$$
$$\sqrt{2\tcfr{a, N-a}{n+1, n}{b-1, N-b+1}\tcfr{a, N-a}{n+1, n}{b, N-b}},$$

Let us prove this identity. We know by a variant of the octahedron recurrence that

$$\tcfr{a, N-a}{n+1, n+1}{b-1, N-b}\tcfr{a, N-a}{n, n}{b, N-b+1}=$$
$$\tcfr{a, N-a}{n, n+1}{b, N-b}\tcfr{a, N-a}{n+1, n}{b-1, N-b+1}+$$
$$\tcfr{a, N-a}{n+1, n}{b, N-b}\tcfr{a, N-a}{n, n+1}{b-1, N-b+1}$$

$$\tcfr{n+1-i, n+i}{n+1, n+1}{n-1-j, n+1+j}\tcfr{n+1-i, n+i}{n, n}{n-j, n+2+j}=$$
$$\tcfr{n+1-i, n+i}{n, n+1}{n-j, n+1+j}\tcfr{n+1-i, n+i}{n+1, n}{n-1-j, n+2+j}+$$
$$\tcfr{n+1-i, n+i}{n+1, n}{n-j, n+1+j}\tcfr{n+1-i, n+i}{n, n+1}{n-1-j, n+2+j}.$$

However, by \eqref{dualities2} we have
$$\tcfr{a, N-a}{n, n}{b, N-b+1}=\tcfr{a, N-a}{n+1, n+1}{b-1, N-b},$$
while we also have
$$\tcfr{a, N-a}{n, n+1}{b, N-b}=\tcfr{a, N-a}{n+1, n}{b, N-b}$$
and
$$\tcfr{a, N-a}{n+1, n}{b-1, N-b+1}=\tcfr{a, N-a}{n, n+1}{b-1, N-b+1}$$
so that putting these together, we get the second identity above.

The other mutations of black vertices will be degenerate specializations of the above general identity. For example, for $n \geq 3$, the first mutation is
$$\sqrt{\tcfr{1, N-1}{n}{n+1}\tcfr{2, N-2}{n, n}{n-1,n+2}}=$$
$$\sqrt{\tcfr{2, N-2}{n}{n+1}}\tcfr{1}{n+1}{n-1}+\sqrt{\tcfr{}{n}{n+1}}\tcfr{2, N-1}{n+1}{n-1}.$$ 
To derive this, we use the general identity
$$\sqrt{\tcfr{1, N-1}{n+1, n}{n, n+1}} \sqrt{\tcfr{2, N-2}{n+1, n}{n-1, n+2}} = $$
$$\sqrt{\tcfr{1, N-1}{n+1, n}{n-1, n+2}}\sqrt{\tcfr{2, N-2}{n+1, n}{n, n+1}}+$$
$$\tcfr{2, N-1}{n+1, n}{n-1, n+1}.$$
plus the facts
$$\tcfr{1, N-1}{n+1, n}{n, n+1}=\tcfr{1, N-1}{n+1}{n}\tcfr{}{n}{n+1},$$
$$\tcfr{1, N-1}{n+1, n}{n-1, n+2}=\tcfr{1}{n+1}{n-1}\tcfr{N-1}{n}{n+2}=\tcfr{1}{n+1}{n-1}^2,$$
$$\tcfr{2, N-2}{n+1, n}{n, n+1}=\tcfr{2, N-2}{n+1}{n}\tcfr{}{n}{n+1},$$
$$\tcfr{2, N-1}{n, n+1}{n-1, n+1}=\tcfr{2, N-1}{n+1}{n-1}\tcfr{}{n}{n+1},$$
$$\tcfr{}{n}{n+1}=\tcfr{}{n+1}{n}.$$

\end{proof}

\subsection{The sequence of mutations for a flip}

In this section, we will give a sequence of mutations that relates two of the clusters coming from different triangulations of the $4$-gon. Combined with the previous section, this allows us to connect by mutations all $72$ different clusters we have constructed for $\Conf_4 \A_{Spin_{2n+1}}$.

Given a configuration $(A,B,C,D) \in \Conf_4 \A_{Spin_{2n+1}}$, we will give a sequence of mutations that relates a cluster coming from the triangulation $ABC, ACD$ to a cluster coming from the triangulation $ABD, BCD$.

We will need to relabel the quiver with vertices $x_{ij}$, $y_k$, with $-n \leq i \leq n$, $1 \leq j \leq n$ and $1 \leq |k| \leq n$. The quiver we will start with is as in Figure 32, pictured for $Spin_7$.

\begin{center}
\begin{tikzpicture}[scale=2.4]
  \foreach \x in {-3,-2,-1,0,1,2,3}
    \foreach \y in {1, 2}
      \node[] (x\x\y) at (\x,-\y) {\Large $\ontop{x_{\x\y}}{\circ}$};
  \foreach \x in {-3,-2,-1,0,1,2,3}
    \foreach \y in {3}
      \node[] (x\x\y) at (\x,-\y) {\Large $\ontop{x_{\x\y}}{\bullet}$};
  \node (y-1) at (-0.5,0) {\Large $\ontop{y_{-1}}{\circ}$};
  \node (y-2) at (-1.5,0) {\Large $\ontop{y_{-2}}{\circ}$};
  \node (y-3) at (-2.5,-4) {\Large $\ontop{y_{-3}}{\bullet}$};
  \node (y1) at (0.5,0) {\Large $\ontop{y_1}{\circ}$};
  \node (y2) at (1.5,0) {\Large $\ontop{y_2}{\circ}$};
  \node (y3) at (2.5,-4) {\Large $\ontop{y_{3}}{\bullet}$};

  \draw [->] (x01) to (x11);
  \draw [->] (x11) to (x21);
  \draw [->] (x21) to (x31);
  \draw [->] (x02) to (x12);
  \draw [->] (x12) to (x22);
  \draw [->] (x22) to (x32);
  \draw [->] (x03) to (x13);
  \draw [->] (x13) to (x23);
  \draw [->] (x23) to (x33);
  \draw [->, dashed] (y1) to (y2);
  \draw [->, dashed] (y2) .. controls +(right:1) and +(up:1) .. (y3);

  \draw [->] (x03) to (x02);
  \draw [->] (x02) to (x01);

  \draw [->] (x13) to (x12);
  \draw [->] (x12) to (x11);
  \draw [->] (x11) to (y1);
  \draw [->] (y3) to (x23);
  \draw [->] (x23) to (x22);
  \draw [->] (x22) to (x21);
  \draw [->] (x21) to (y2);
  \draw [->, dashed] (x33) to (x32);
  \draw [->, dashed] (x32) to (x31);

  \draw [->] (y1) to (x01);
  \draw [->] (x11) to (x02);
  \draw [->] (x12) to (x03);
 \draw [->] (y2) to (x11);
  \draw [->] (x21) to (x12);
  \draw [->] (x22) to (x13);
  \draw [->] (x31) to (x22);
  \draw [->] (x32) to (x23);
  \draw [->] (x33) to (y3);

  \draw [->] (x01) to (x-11);
  \draw [->] (x-11) to (x-21);
  \draw [->] (x-21) to (x-31);
  \draw [->] (x02) to (x-12);
  \draw [->] (x-12) to (x-22);
  \draw [->] (x-22) to (x-32);
  \draw [->] (x03) to (x-13);
  \draw [->] (x-13) to (x-23);
  \draw [->] (x-23) to (x-33);
  \draw [->, dashed] (y-1) to (y-2);
  \draw [->, dashed] (y-2) .. controls +(left:1) and +(up:1) .. (y-3);

  \draw [->] (x-13) to (x-12);
  \draw [->] (x-12) to (x-11);
  \draw [->] (x-11) to (y-1);
  \draw [->] (y-3) to (x-23);
  \draw [->] (x-23) to (x-22);
  \draw [->] (x-22) to (x-21);
  \draw [->] (x-21) to (y-2);
  \draw [->, dashed] (x-33) to (x-32);
  \draw [->, dashed] (x-32) to (x-31);

  \draw [->] (y-1) to (x01);
  \draw [->] (x-11) to (x02);
  \draw [->] (x-12) to (x03);
  \draw [->] (y-2) to (x-11);
  \draw [->] (x-21) to (x-12);
  \draw [->] (x-22) to (x-13);
  \draw [->] (x-31) to (x-22);
  \draw [->] (x-32) to (x-23);
  \draw [->] (x-33) to (y-3);

\draw[yshift=-3.85cm]
  node[below,text width=6cm] 
  {
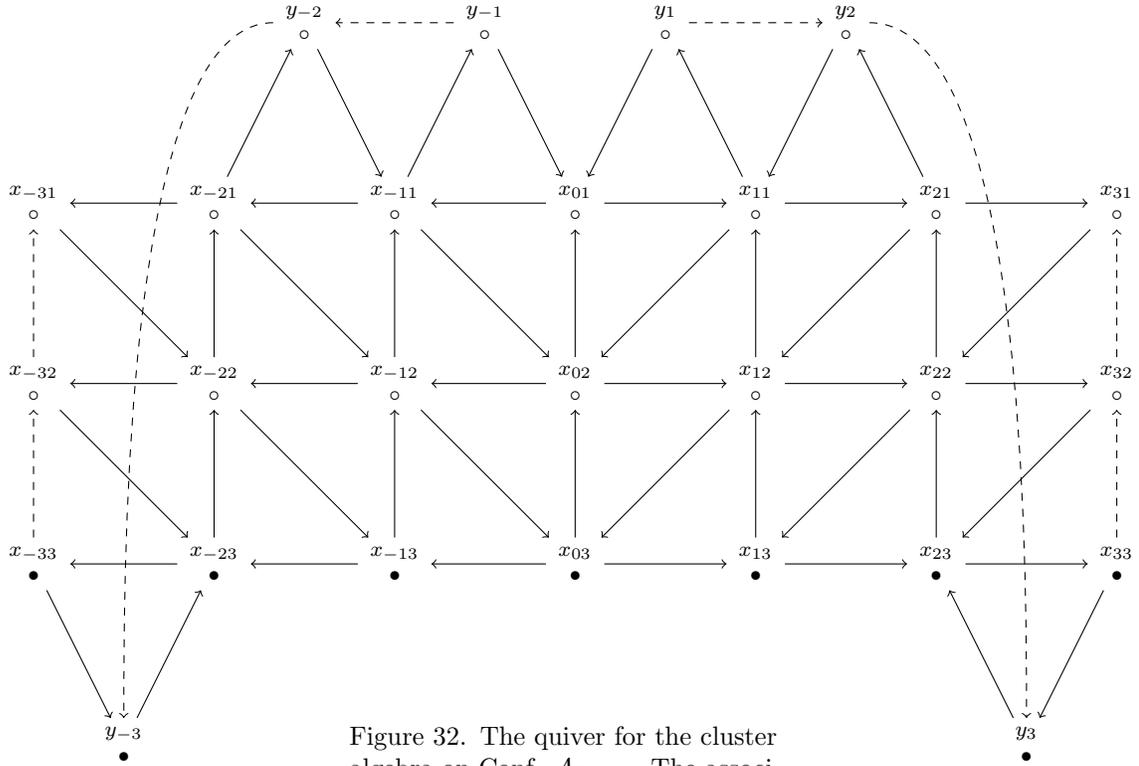
  Figure 32. The quiver for the cluster algebra on $\Conf_4 \A_{Spin_{7}}$. The associated functions are pictured in Figure 26.
  };

\end{tikzpicture}
\end{center}

We will describe the functions attached to these vertices in two steps. First we will assign some functions to each vertex. Then we will take the square-root of the functions assigned to the black vertices. 

Let $N=2n+1$. First make an assigment of functions to vertices as follows:

\begin{alignat*}{1}
\dur{k}{N-k} &\longleftrightarrow y_k, \textrm{ for } k >0; \\
\dld{|k|}{N-|k|} &\longleftrightarrow y_k, \textrm{ for } k <0; \\
\dul{j}{N-j} &\longleftrightarrow x_{-n,j}; \\
\ddr{N-j}{j} &\longleftrightarrow x_{nj}; \\
\dud{j}{N-j} &\longleftrightarrow x_{0j}. \\
\end{alignat*}
The face functions in the triangle where $i>0$ are
\begin{alignat*}{1}
\tcfr{i+j}{N-j}{N-i} &\longleftrightarrow x_{ij}, \textrm{ for } 0<i<n, i+j \leq n; \\
\tcfr{n}{N-j}{j+i-n, N-i} &\longleftrightarrow x_{ij}, \textrm{ for } 0<i<n, i+j > n; \\
\end{alignat*}
while the face functions in the triangle where $i<0$ are
\begin{alignat*}{1}
\tcfl{j}{|i|}{N-|i|-j} &\longleftrightarrow x_{ij}, \textrm{ for } -n<i<0, |i|+j \leq n; \\
\tcfl{j}{|i|, N+n-|i|-j}{n+1} &\longleftrightarrow x_{ij}, \textrm{ for } -n<i<0, |i|+j > n.
\end{alignat*}

\begin{rmk} Note that our labelling of the vertices is somewhat different from before. The vertices labelled $x_{ij}$ correspond to the vertices labelled $x_{n-|i|, j}$ in  $\Conf_3 \A_{Spin_{2n+1}}$.
\end{rmk}

Now, take the square-root of the functions assigned to $x_{in}, y_n$. For example, when $i > 0$ we assign
$$\sqrt{\tcfr{n}{n+1}{i, N-i}} \longleftrightarrow x_{in}$$
and for $i<0$, we assign
$$\sqrt{\tcfl{n}{|i|, N-|i|}{n+1}} \longleftrightarrow x_{in}.$$

The functions above are defined by pulling back via the various natural maps 
$$p_1, p_2, p_3, p_4: \Conf_4 \A_{Spin_{2n+1}} \rightarrow \Conf_3 \A_{Spin_{2n+1}}$$
that map a configuration $(A,B,C,D)$ to $(B, C, D)$, $(A, C, D)$, $(A, B, D)$, $(A,B,C)$, respectively. Pulling back functions from $\Conf_3 \A_{Spin_{2n+1}}$ allows us to define functions on $\Conf_4 \A_{Spin_{2n+1}}$. For example,
$$p_2^*\tcfr{n}{N-j}{j+i-n, N-i} =: \tcfr{n}{N-j}{j+i-n, N-i}.$$
Similarly, we can pull back functions from various maps 
$$\Conf_4 \A_{Spin_{2n+1}} \rightarrow \Conf_2 \A_{Spin_{2n+1}}$$
to define functions such as 
$$\dld{j}{N-j}.$$

There is also a map
$$T: \Conf_4 \A  \rightarrow \Conf_4 \A$$
which sends
$$(A,B,C,D) \rightarrow (s_G \cdot D, A, B, C)$$
which allows us to define, for example
$$T^*\tcfu{n}{j}{n+1-i, N+i-j} =: \tcfr{j}{n+1-i, N+i-j}{n}.$$ The forgetful maps and twist maps, combined with the constructions below, will furnish all the functions necessary for the computation of the flip mutation sequence.

As in the case of $G=Sp_{2n}$, we will have to use some functions which depend on all four flags. Let $N=2n+1$. Let $0 \leq a, b, c, d \leq N$ such that $a+b+c+d=4n+2=2N$ and $b+c \leq N$. Then we would like to define a function that we will call 
$$\qcfs{a}{b}{c}{d}.$$

We again use the symbol ``:'' because the function does not have cyclic symmetry. In other words,
$$T^*\qcfs{a}{b}{c}{d} \neq \qcfs{b}{c}{d}{a}.$$
Instead, we use the notation
$$T^*\qcfs{a}{b}{c}{d} =: \qcfb{b}{c}{d}{a}.$$
We can also define
$$(T^2)^*\qcfs{a}{b}{c}{d}=:\qcfs{c}{d}{a}{b}.$$

The function $\qcfs{a}{b}{c}{d}$ on $\Conf_4 \A_{Spin_{2n+1}}$ is pulled back from a function on $\Conf_4 \A_{SL_{N}}$. The function on $\Conf_4 \A_{SL_{N}}$ is given by an invariant in the space
$$[V_{\omega_a}  \otimes V_{\omega_b}  \otimes V_{\omega_c} \otimes V_{\omega_d}]^{SL_N}.$$
The function is given, as before, by the web in Figure 19.

$$\qcfs{n}{a}{b}{c, d}.$$



We now define a second type of function on  $\Conf_4 \A_{Spin_{2n+1}}$. If $a+b+c+d=2N+n$, we define the function
$$\qcfs{n+1}{a}{b}{c, d}.$$
It is given by the invariant in the space
$$[V_{\omega_n+1}  \otimes V_{\omega_a}  \otimes V_{\omega_b} \otimes V_{\omega_c+\omega_d}]^{SL_N}$$
picked out by the web in Figure 33.

\begin{center}
\begin{tikzpicture}[scale=1.4]
\begin{scope}[decoration={
    markings,
    mark=at position 0.5 with {\arrow{>}}},
    ] 
\draw [postaction={decorate}] (-4,-1) -- (-2.5,-1) node [midway,below] {$a$}; 
\draw [postaction={decorate}] (-1,-1) -- (-2.5,-1) node [midway,below] {$N-a$}; 
\draw [postaction={decorate}] (-1,-4) -- (-1,-1) node [midway,right] {$b$}; 
\draw [postaction={decorate}] (-1,-1) -- (1,1) node [midway,below right] {$a+b-N$}; 

\draw [postaction={decorate}] (4,1) -- (2.5,1) node [midway,above] {$c$}; 
\draw [postaction={decorate}] (1,1) -- (2.5,1) node [midway,below] {$N-c$}; 
\draw [postaction={decorate}] (1,1) -- (1,2) node [midway,left] {$a+b+c-2N$}; 
\draw [postaction={decorate}] (4,2) -- (1,2) node [midway,below] {$d$}; 
\draw [postaction={decorate}] (1,2) -- (1,3) node [midway,left] {$a+b+c+d-2N$}; 
\draw [postaction={decorate}] (1,4) -- (1,3) node [midway,left] {$n+1$};

\draw (2.5,0.9) -- (2.5,1);
\draw (0.9,3) -- (1,3);
\draw (-2.5,-1.1) -- (-2.5,-1);

\end{scope}

\draw[yshift=-4cm]
  node[below,text width=8cm] 
  {
  Figure 33 Web for the function $\qcfs{n+1}{a}{b}{c, d}$, where $a+b+c+d=2N+n$.
  };

\end{tikzpicture}
\end{center}

Let us give a concrete description of this functions.

Given four flags 
$$u_1, \dots, u_N;$$
$$v_1, \dots, v_N;$$
$$w_1, \dots, w_N;$$
$$x_1, \dots, x_N;$$

first consider the forms
$$U_{n+1} := u_1 \wedge \cdots \wedge u_{n+1},$$
$$V_a := v_1 \wedge \cdots \wedge v_a,$$
$$W_b := w_1 \wedge \cdots \wedge w_b,$$
$$X_c := x_1 \wedge \cdots \wedge x_c.$$
$$X_d := x_1 \wedge \cdots \wedge x_d.$$

Because $b=2N+n-a-c-d$, there is a natural map 
$$\phi_{N-c, n-d, N-a}: \bigwedge\nolimits^{b} V \rightarrow \bigwedge\nolimits^{N-c} V \otimes \bigwedge\nolimits^{n-d} V \otimes \bigwedge\nolimits^{N-a} V.$$
There are also natural maps 
$$- \wedge X_{c} : \bigwedge\nolimits^{N-c} V \rightarrow \bigwedge\nolimits^{N} V \simeq F,$$
$$U_{n+1} \wedge - \wedge X_d: \bigwedge\nolimits^{n-d} V \rightarrow \bigwedge\nolimits^{N} V \simeq F,$$ and
$$V_a \wedge - :  \bigwedge\nolimits^{N-a} V \rightarrow \bigwedge\nolimits^{N} V \simeq F.$$ 
Applying these maps to the first, second, and third factors of $\phi_{N-c, n-d, N-a}(W_b)$, respectively, and then multiplying, we get get the value of our function $\qcfs{n+1}{a}{b}{c, d}$. Using the twist map $T$, we can also define the functions $\qcfb{a}{b}{c, d}{n+1}$, $\qcfs{b}{c, d}{n+1}{a}$, and $\qcfb{c, d}{n+1}{a}{b}$.

Using duality, there is also a function $\qcfs{n}{a}{b}{c, d}$ on $\Conf_4 \A_{Spin_{2n+1}}$ for $0 \leq a, b, c, d \leq N$, $a+b+c+d=3n+2=N+n+1$, and $c \leq d$. 

This function is pulled back from the function on $\Conf_4 \A_{SL_{2n+1}}$ given by an invariant in the space
$$[V_{\omega_n}  \otimes V_{\omega_a}  \otimes V_{\omega_b} \otimes V_{\omega_c+\omega_d}]^{SL_N}.$$
This vector space is generally multi-dimensional. To pick out the correct invariant, we use the same web as in Figure 20b.

Let us give a concrete description of this function. We use the same notation as before.

Because $a+b=N+n+1-c-d$, there is a natural map 
$$\phi_{N-c, n+1-d}: \bigwedge\nolimits^{a+b} V \rightarrow \bigwedge\nolimits^{N-c} V \otimes \bigwedge\nolimits^{n-d} V.$$
There are also natural maps 
$$- \wedge X_{c} : \bigwedge\nolimits^{N-c} V \rightarrow \bigwedge\nolimits^{N} V \simeq F,$$ and
$$U_n \wedge - \wedge X_d: \bigwedge\nolimits^{n+1-d} V \rightarrow \bigwedge\nolimits^{N} V \simeq F.$$
Applying these maps to the first and second factors of $\phi_{N-c, n+1-d}(V_a \wedge W_b)$, respectively, and then multiplying, we get get the value of our function $\qcfs{n}{a}{b}{c, d}$. 

Using the twist map $T$, we can also define the functions $\qcfb{a}{b}{c, d}{n}$, $\qcfs{b}{c, d}{n}{a}$, and $\qcfb{c, d}{n}{a}{b}$.

We will need to define one more type of function to do our calculations. Let $0 \leq a, b, c, d \leq N$ such that $a+b+c+d=4n+2=2N$, $a \leq n \leq b$ and $c \leq n \leq d$. Then we would like to define a function that we will call 
$$\qcfs{n}{a, b}{n+1}{c, d}.$$
We will make less frequent use of the functions
$$\qcfs{n+1}{a, b}{n}{c, d}, \qcfs{n}{a, b}{n}{c, d}, \qcfs{n+1}{a, b}{n+1}{c, d}.$$
These functions are defined by similar formulas. Here $a+b+c+d=2N$, $2N+1$ or $2N-1$, respectively.

The function $\qcfs{n}{a, b}{n+1}{c, d}$ on $\Conf_4 \A_{Spin_{2n+1}}$ is pulled back from a function on $\Conf_4 \A_{SL_{N}}$. The function on $\Conf_4 \A_{SL_{N}}$ is given by an invariant in the space
$$[V_{\omega_n}  \otimes V_{\omega_a+\omega_b}  \otimes V_{\omega_{n+1}} \otimes V_{\omega_c+\omega_d}]^{SL_N}.$$
This vector space is generally multi-dimensional. To pick out the correct invariant, we use the web in Figure 34:

\begin{center}
\begin{tikzpicture}[scale=1.4]
\begin{scope}[decoration={
    markings,
    mark=at position 0.5 with {\arrow{>}}},
    ] 
\draw [postaction={decorate}] (-4,-1) -- (-2.5,-1) node [midway,below] {$b$}; 
\draw [postaction={decorate}] (-1,-1) -- (-2.5,-1) node [midway,below] {$N-b$}; 
\draw [postaction={decorate}] (-1,-4) -- (-1,-2) node [midway,right] {$n+1$}; 
\draw [postaction={decorate}] (-4,-2) -- (-1,-2) node [midway,below] {$a$}; 
\draw [postaction={decorate}] (-1,-2) -- (-1,-1) node [midway,right] {$a+n+1$}; 
\draw [postaction={decorate}] (-1,-1) -- (1,1) node [midway,below right] {$a+b+n+1-N$}; 

\draw [postaction={decorate}] (4,1) -- (2.5,1) node [midway,above] {$c$}; 
\draw [postaction={decorate}] (1,1) -- (2.5,1) node [midway,below] {$N-c$}; 
\draw [postaction={decorate}] (1,1) -- (1,2) node [midway,left] {$a+b+c+n+1-2N=n+1-d$}; 
\draw [postaction={decorate}] (4,2) -- (1,2) node [midway,below] {$d$}; 
\draw [postaction={decorate}] (1,2) -- (1,3) node [midway,left] {$n+1$}; 
\draw [postaction={decorate}] (1,4) -- (1,3) node [midway,left] {$n$};

\draw (2.5,0.9) -- (2.5,1);
\draw (0.9,3) -- (1,3);
\draw (-2.5,-1.1) -- (-2.5,-1);

\end{scope}

\draw[yshift=-4cm]
  node[below,text width=8cm] 
  {
  Figure 34 Web for the function $\qcfs{n}{a,b}{n+1}{c, d}$, where $a+b+c+d=2N$.
  };

\end{tikzpicture}
\end{center}

Let us give a concrete description of this function.

We use the forms $U_n, V_a, V_b, W_n, X_c, X_d.$

Because $n+a=2N+n+1-b-c-d$, there is a natural map 
$$\phi_{N-c, n+1-d, N-b}: \bigwedge\nolimits^{n+a} V \rightarrow \bigwedge\nolimits^{N-c} V \otimes \bigwedge\nolimits^{n+1-d} V \otimes \bigwedge\nolimits^{N-b} V.$$
There are also natural maps 
$$- \wedge X_{c} : \bigwedge\nolimits^{N-c} V \rightarrow \bigwedge\nolimits^{N} V \simeq F,$$
$$U_n \wedge - \wedge X_d: \bigwedge\nolimits^{n+1-d} V \rightarrow \bigwedge\nolimits^{N} V \simeq F,$$ and
$$V_b \wedge - :  \bigwedge\nolimits^{N-b} V \rightarrow \bigwedge\nolimits^{N} V \simeq F.$$ 
Applying these maps to the first, second, and third factors of $\phi_{N-c, n-d+1, N-a}(V_a \wedge W_n)$, respectively, and then multiplying, we get get the value of our function $\qcfs{n}{a, b}{n+1}{c, d}$. 
Note that
$$\qcfs{n+1}{a, b}{n}{c, d}=(T^2)^*\qcfs{n}{c, d}{n+1}{a, b}.$$

Note that when $a=0$, $b=N$, $c=0$ or $d=N$, we have
$$\qcfs{n}{0, b}{n+1}{c, d}=:\qcfs{n}{b}{n+1}{c, d},$$
$$\qcfs{n}{a, N}{n+1}{c, d}=:\qcfs{n}{a}{n+1}{c, d},$$
$$\qcfs{n}{a, b}{n+1}{0, d}=:\qcfs{n}{a, b}{n+1}{d},$$
$$\qcfs{n}{a, b}{n+1}{c, N}=:\qcfs{n}{a, b}{n+1}{c}.$$
If $a=0$ and $d=N$, we will have
$$\qcfs{n}{0, b}{n+1}{c, N}=\qcfs{n}{b}{n+1}{c},$$
where $\qcfs{n}{b}{n+1}{c} $ is as defined above. A similar equality holds when $b=N, c=0$. If $a=0$, $b=N$, $c=0$ and $d=N$, we will have that 
$$\qcfs{n}{0, N}{n+1}{0, N}=\qcfs{n}{0}{n+1}{0}.$$

Finally, note that if $a+b=c+d=N$, then 
$$\sqrt{\qcfs{n}{a, b}{n+1}{c, d}}$$
is a well-defined function on $\Conf_4 \A_{Spin_{2n+1}}$. This is because the representations $V_{\omega_a}$ and $V_{\omega_b}$ of $SL_N$ give the same representations of $Spin_N$, and $V_{\omega_n}$ and $V_{\omega_n+1}$, as representations of $Spin_N$, have twice the weight of the spin representation.

Now we give the sequence of mutations realizing the flip of a triangulation. The sequence of mutations is exactly as in the case for $Sp_{2n}$. The sequence of mutations leaves $x_{-n,j}, x_{nj}, y_k$ untouched as they are frozen variables. Hence we only mutate $x_{ij}$ for $-n \leq i \leq n$. We now describe the sequence of mutations. The sequence of mutations will have $3n-2$ stages. At the $r^{\textrm{th}}$ step, we mutate all vertices such that 
$$|i|+j \leq r,$$
$$j-|i| + 2n -2 \geq r,$$
$$|i|+j \equiv r \mod 2.$$

Note that the first inequality is empty for $r \geq 2n-1$, while the second inequality is empty for $r \leq n$. For example, for $Spin_7$, the sequence of mutations is

\begin{equation}
\begin{gathered}
x_{01}, \\
x_{-1,1}, x_{02}, x_{11},  \\
x_{-2,1}, x_{-1,2}, x_{01}, x_{03}, x_{12}, x_{21},  \\
x_{-2,2}, x_{-1,1}, x_{-1,3}, x_{02}, x_{11}, x_{13}, x_{22},  \\
x_{-2,3}, x_{-1,2}, x_{01}, x_{03}, x_{12}, x_{23},  \\
x_{-1,3}, x_{02}, x_{13},  \\
x_{03}.
\end{gathered}
\end{equation}

In Figure 35, we depict how the quiver for $\Conf_4 \A_{Spin_{7}}$ changes after each of the seven stages of mutation.

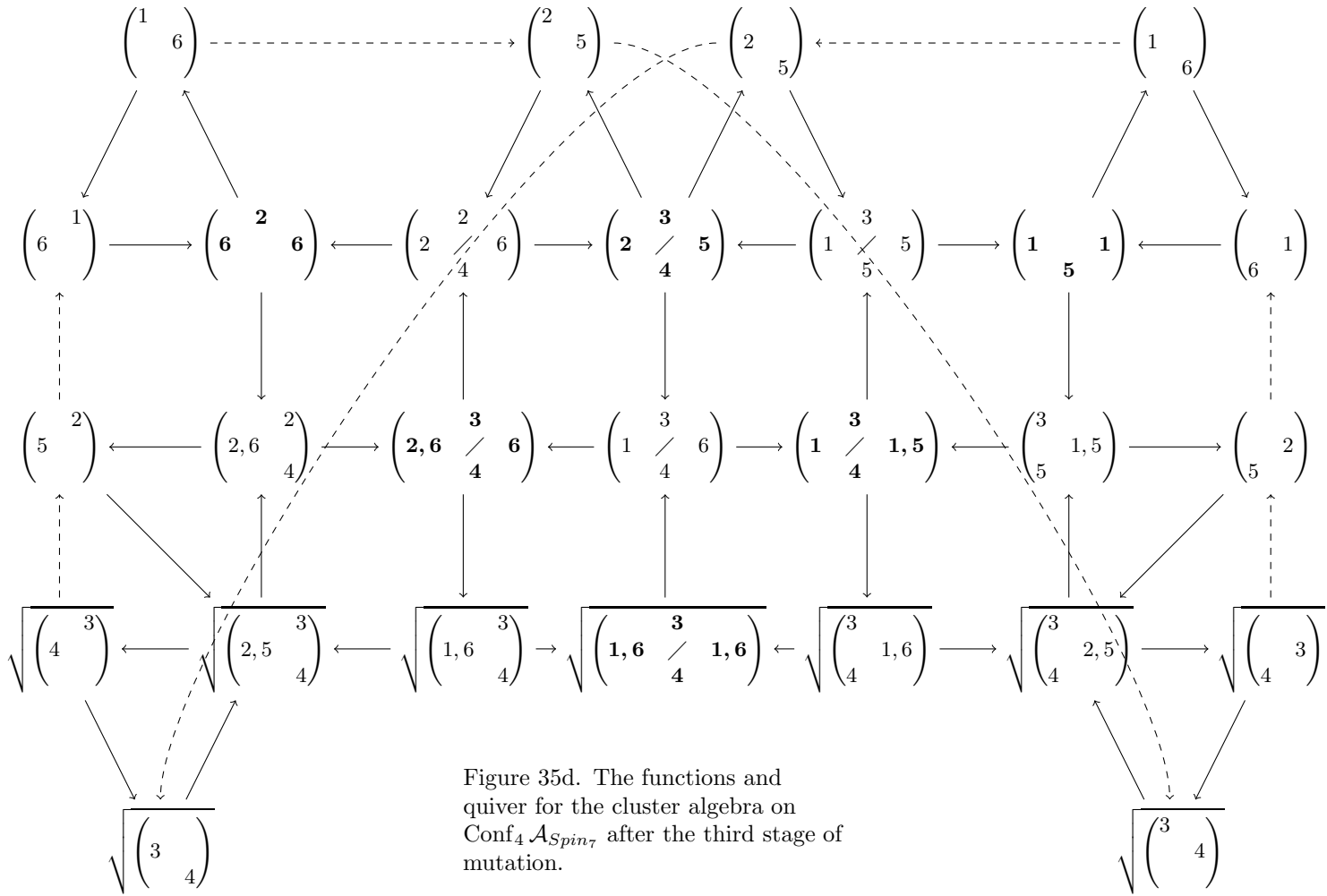
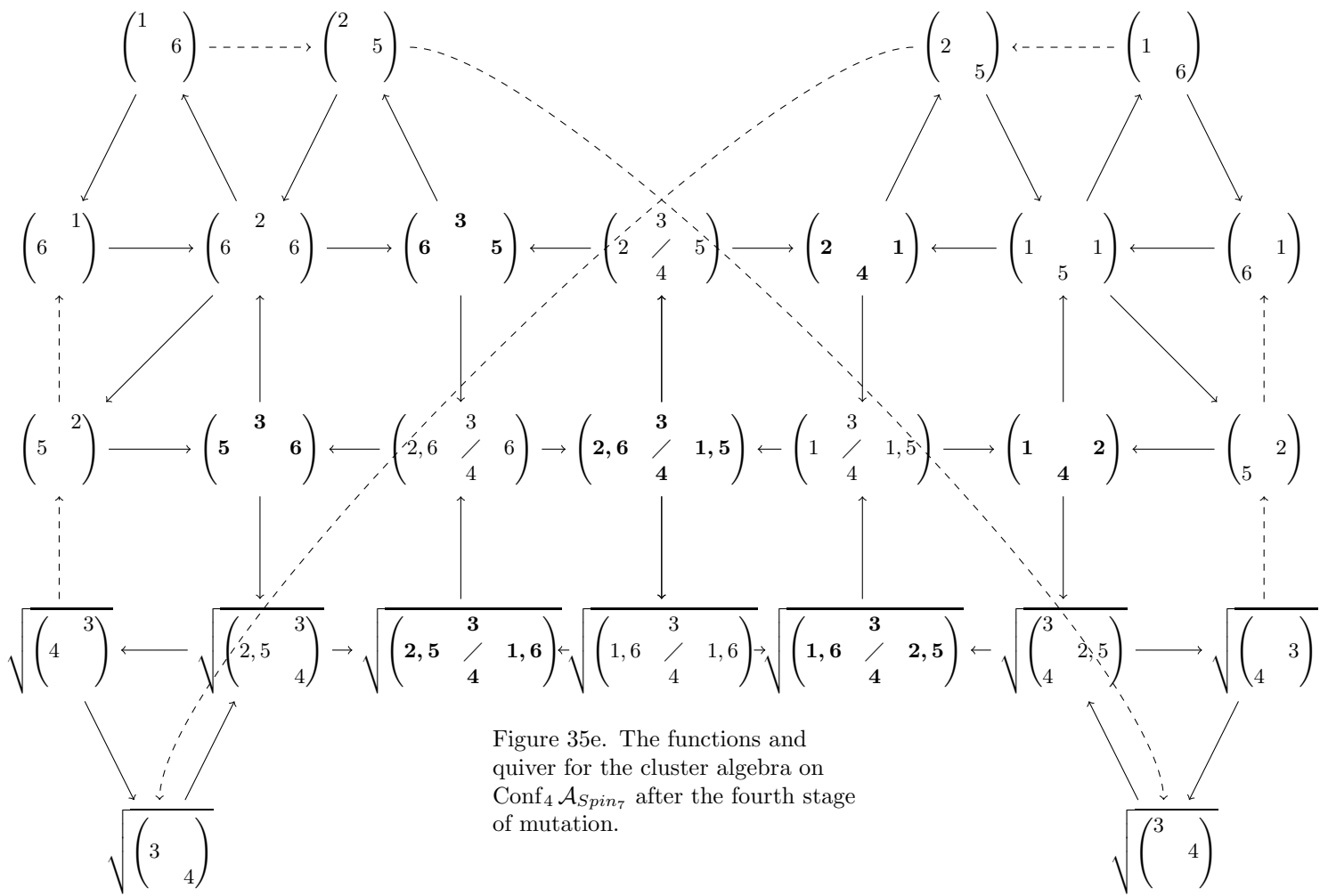
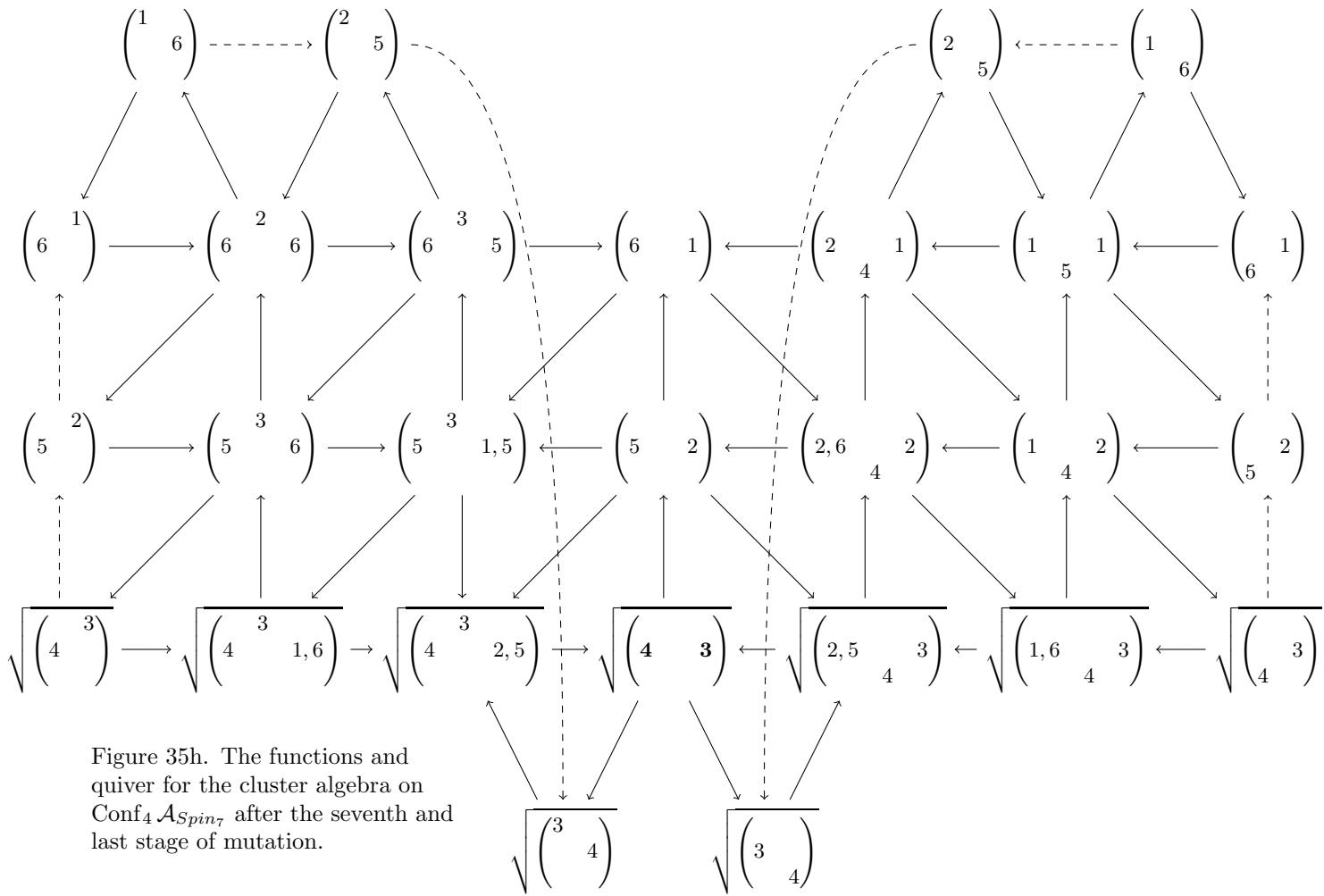
\begin{center}

\end{center}

The analogue $\Conf_4 \A_{Spin_{2n+1}}$ should be clear.

We now have the main theorem of this section:

\begin{theorem}
We first analyze the situation when $i > 0$. The vertex $x_{ij}$ is mutated a total of $n-i$ times. There are four cases.
\begin{itemize}
\item When $i+j < n$ and $i < j$, the function attached to $x_{ij}$ mutates in three stages, consisting of $n-i-j, i,$ and $j-i$ mutations, respectively:
\begin{enumerate}
\item $$\tcfr{i+j}{N-j}{N-i} \rightarrow \qcfs{i+j+1}{1}{N-j-1}{N-i-1} \rightarrow $$
$$\qcfs{i+j+2}{2}{N-j-2}{N-i-2} \rightarrow $$
$$\dots \rightarrow \qcfs{n}{n-i-j}{n+1+i}{n+1+j}$$
\item $$\qcfs{n}{n-i-j}{n+1+i}{n+1+j}=\qcfs{n}{n-i-j}{n+1+i}{0, n+1+j} \rightarrow $$
$$\qcfs{n}{n-i-j+1}{n+i}{1, n+j} \rightarrow$$
$$ \qcfs{n}{n-i-j+2}{n+i-1}{2, n+j-1} \rightarrow $$
$$\dots \rightarrow \qcfs{n}{n-j}{n+1}{i, n+1+j-i}$$
\item $$\qcfs{n}{n-j}{n+1}{i, n+1+j-i}=\qcfs{n}{n-j, N}{n+1}{i, n+1+j-i}$$
$$\rightarrow \qcfs{n}{n-j+1, N-1}{n+1}{i+1, n+j-i} \rightarrow $$
$$\qcfs{n}{n-j+2, N-2}{n+1}{i+2, n+j-i-1} \rightarrow $$
$$\dots \rightarrow [\qcfs{n}{n-i, N-j+i}{n+1}{j, n+1}] \textrm{ } \tcfd{n-i, N-j+i}{n+1}{j}$$
\end{enumerate}

\item When $i+j \geq n$, $i < j$, and $j \neq n$ the function attached to $x_{ij}$ mutates  in two stages, consisting of $n-j$ and $j-i$ mutations, respectively:
\begin{enumerate}
\item $$\tcfr{n}{N-j}{j+i-n, N-i} \rightarrow \qcfs{n}{1}{N-j-1}{j+i-n+1, N-i-1} \rightarrow $$
$$\qcfs{n}{2}{N-j-2}{j+i-n+2, N-i-2} \rightarrow$$
$$\dots  \rightarrow \qcfs{n}{n-j}{n+1}{i, n+1-i+j}$$
\item $$\qcfs{n}{n-j}{n+1}{i, n+1-i+j} [\qcfs{n}{n-j, N}{n+1}{i, n+1-i+j}] \rightarrow $$
$$\qcfs{n}{n-j+1, N-1}{n+1}{i+1, n-i+j} \rightarrow $$
$$\qcfs{n}{n-j+2, N-2}{n+1}{i+2, n-i+j-1} \rightarrow$$
$$\dots  \rightarrow [\qcfs{n}{n-i, N-j+i}{n+1}{j, n+1}]  \textrm{ } \tcfd{n-i, N-j+i}{n+1}{j} $$
\end{enumerate}

\item In the case that is most different from situation when $G=Sp_{2n}$, we have that if $j=n$, then mutations happen in one stage consisting of $n-i$ mutations:
$$\sqrt{\qcfs{n}{0}{n+1}{i, N-i}} [\sqrt{\qcfs{n}{0, N}{n+1}{i, N-i}}] \rightarrow $$
$$\sqrt{\qcfs{n}{1, N-1}{n+1}{i+1, N-i-1}} \rightarrow $$
$$\sqrt{\qcfs{n}{2, N-2}{n+1}{i+2, N-i-2}} \rightarrow$$
$$\dots  \rightarrow [\sqrt{\qcfs{n}{n-i, N-n+i}{n+1}{n, n+1}}]  \textrm{ } \sqrt{\tcfd{n-i, N-n+i}{n+1}{n}}$$

\item When $i+j < n$ and $i \geq j$, the function attached to $x_{ij}$ mutates in two stages, consisting of $n-i-j$ and $j$ mutations, respectively:
\begin{enumerate}
\item $$\tcfr{i+j}{N-j}{N-i} \rightarrow \qcf{i+j+1}{1}{N-j-1}{N-i-1} \rightarrow $$
$$\qcf{i+j+2}{2}{N-j-2}{N-i-2} \rightarrow $$
$$\dots \rightarrow \qcf{n}{n-i-j}{n+1+i}{n+1+j}$$
\item $$\qcf{n}{n-i-j}{n+1+i}{n+1+j} \rightarrow \qcfs{n}{n-i-j+1}{n+i}{1, n+j} \rightarrow $$
$$\qcfs{n}{n-i-j+2}{n+i-1}{2, n+j-1} \rightarrow $$
$$\dots \rightarrow [\qcfs{n}{n-i}{n+1+i-j}{j, n+1}]  \textrm{ } \tcfd{n-i}{n+1+i-j}{j}$$
\end{enumerate}

\item When $i+j \geq n$ and $i \geq j$, the function attached to $x_{ij}$ mutates in one stage consisting of $n-i$ mutations:
$$\tcfr{n}{N-j}{j+i-n, N-i} \rightarrow \qcfs{n}{1}{N-j-1}{j+i-n+1, N-i-1} \rightarrow $$
$$\qcfs{n}{2}{N-j-2}{j+i-n+2, N-i-2} \rightarrow$$
$$\dots  \rightarrow [\qcfs{n}{n-i}{n+1+i-j}{j, n+1}] \textrm{ } \tcfd{n-i}{n+1+i-j}{j}$$

\end{itemize}

The mutation sequence when $i \leq 0$ is completely parallel. We include it in an appendix for reference.

In all these sequences, for each mutation, two parameters increase, and two decrease. Within a stage, the same parameters increase or decrease. The only exception is that sometimes after the last mutation, one removes the factor $\dur{n}{n+1}$ (or $\dld{n}{n+1}$ when $i \leq 0$). The expressions in square brackets indicate the functions before removing factors of $\dur{n}{n+1}$ (or $\dld{n}{n+1}$ when $i \leq 0$).

\end{theorem}

\begin{proof} The proof comes down to a handful of identities used in conjunction, as in previous proofs of this type. Here are the identities we use:

\begin{itemize}
\item Let $0 \leq a, b, c, d \leq N$, and $a+b+c+d=2N$. 
$$\qcfs{a}{b}{c}{d}\qcfs{a+1}{b+1}{c-1}{d-1}=$$
$$\qcfs{a}{b+1}{c-1}{d}\qcfs{a+1}{b}{c}{d-1}+\qcfs{a+1}{b}{c-1}{d}\qcfs{a}{b+1}{c}{d-1}.$$

\item Let $0 \leq a, b, c, d \leq N$, and $a+b+c+d=N+n+1$. 
$$\qcfs{n}{a}{b}{c, d}\qcfs{n}{a+1}{b-1}{c+1, d-1}=$$
$$\qcfs{n}{a+1}{b-1}{c, d}\qcfs{n}{a}{b}{c+1, d-1}+\qcfs{n}{a+1}{b}{c, d-1}\qcfs{n}{a}{b-1}{c+1, d}.$$
There is also a dual identity when $a+b+c+d=2N+n$ that we use when $i < 0$:
$$\qcfs{n+1}{a}{b}{c, d}\qcfs{n+1}{a-1}{b+1}{c+1, d-1}=$$
$$\qcfs{n+1}{a}{b}{c+1, d-1}\qcfs{n+1}{a-1}{b+1}{c, d}+\qcfs{n+1}{a-1}{b}{c+1, d}\qcfs{n+1}{a}{b+1}{c, d-1}.$$

\item Let $0 \leq a, b, c, d \leq N$, and $a+b+c+d=4n+2$. 
$$\qcfs{n}{a, b}{n+1}{c, d}\qcfs{n}{a+1, b-1}{n+1}{c+1, d-1}=$$
$$\qcfs{n}{a+1, b-1}{n+1}{c, d}\qcfs{n}{a, b}{n+1}{c+1, d-1}+\qcfs{n}{a+1, b}{n+1}{c, d-1}\qcfs{n}{a, b-1}{n+1}{c+1, d}.$$

\item Let $0 \leq a, b, c, d \leq N$ such that $a+b+c+d=2N$, $a \leq b$ and $c \leq d$. If $a+d=b+c=N$.,
$$\qcfs{a}{b}{c}{d}=\dld{b}{c}\dur{a}{d}$$

\item Let $0 \leq a, b, c, d \leq N$ such that $a+b+c+d=2N$, $a \leq b$ and $c \leq d$. If $a$ or $b=n$ or $c$ or $d=n+1$,
$$\qcfs{n}{n, b}{n+1}{c, d}=\tcfu{n}{b}{c, d}\dld{n}{n+1},$$
$$\qcfs{n}{a, n}{n+1}{c, d}=\tcfu{n}{a}{c, d}\dld{n}{n+1},$$
$$\qcfs{n}{a, b}{n+1}{n+1, d}=\tcfd{a, b}{n+1}{d}\dur{n}{n+1},$$
$$\qcfs{n}{a, b}{n+1}{c, n+1}=\tcfd{a, b}{n+1}{c}\dur{n}{n+1}.$$

\item The final set of identities was mentioned previously. Let $0 \leq a, b, c, d \leq N$ such that $a+b+c+d=2N$, $a \leq b$ and $c \leq d$. When $a=0$, $b=N$, $c=0$ or $d=N$, we have
$$\qcfs{n}{0, b}{n+1}{c, d}=:\qcfs{n}{b}{n+1}{c, d},$$
$$\qcfs{n}{a, N}{n+1}{c, d}=:\qcfs{n}{a}{n+1}{c, d},$$
$$\qcfs{n}{a, b}{n+1}{0, d}=:\qcfs{n}{a, b}{n+1}{d},$$
$$\qcfs{n}{a, b}{n+1}{c, N}=:\qcfs{n}{a, b}{n+1}{c}.$$
If $a=0$ and $d=N$, we will have
$$\qcfs{n}{0, b}{n+1}{c, N}=\qcfs{n}{b}{n+1}{c}.$$
A similar equality holds when $b=N, c=0$. If $a=0$, $b=N$, $c=0$ and $d=N$, we will have that 
$$\qcfs{n}{0, N}{n+1}{0, N}=\qcfs{n}{0}{n+1}{0}.$$

\end{itemize}

The proof of the mutation identities is much like when $G=Sp_{2n}$. The first three sets of identities are the most important. They are variations on the octahedron recurrence. When $i+j < n$ and $i < j$, the three stages use the first, second and third set of identities, respectively. When $i+j \geq n$ and $i < j \neq n$, the two stages use the second and third set of identities, respectively. When $i+j < n$ and $i \geq j$, the two stages use the first and second set of identities, respectively. When $i+j \geq n$ and $i \geq j$, the one stage uses only the second set of identities.

The last three sets of identities are used to give degenerate versions of the previous three sets of identities.

The main novelty occurs when $j=n$. Here we are mutating black vertices. Here we will need to derive some new identities. The general mutation identity when $j=n$ has the following form:

$$\sqrt{\qcfs{n}{a, N-a}{n+1}{b, N-b}}\sqrt{\qcfs{n}{a+1, N-a-1}{n+1}{b+1, N-b-1}}=$$
$$\sqrt{\qcfs{n}{a+1, N-a-1}{n+1}{b, N-b}}\sqrt{\qcfs{n}{a, N-a}{n+1}{b+1, N-b-1}}+$$
$$\qcfs{n}{a+1, N-a}{n+1}{b, N-b-1}.$$

The above identity in turn follows from the following identities:
\begin{itemize}
\item $$\qcfs{n}{a, N-a}{n+1}{b, N-b}\qcfs{n}{a+1, N-a}{n}{b+1, N-b-1}=$$
$$\qcfs{n}{a+1, N-a}{n}{b, N-b}\qcfs{n}{a, N-a}{n+1}{b+1, N-b-1}+$$
$$\qcfs{n}{a+1, N-a}{n+1}{b, N-b-1}\qcfs{n}{a, N-a}{n}{b+1, N-b}$$
\item $$\qcfs{n}{a+1, N-a}{n}{b+1, N-b-1}=$$
$$\sqrt{2\qcfs{n}{a, N-a}{n+1}{b+1, N-b-1}\qcfs{n}{a+1, N-a-1}{n+1}{b+1, N-b-1}}$$
\item $$\qcfs{n}{a+1, N-a}{n}{b, N-b}=$$
$$\sqrt{2\qcfs{n}{a, N-a}{n+1}{b, N-b}\qcfs{n}{a+1, N-a-1}{n+1}{b, N-b}}$$
\item $$\qcfs{n}{a, N-a}{n}{b+1, N-b}=$$
$$\sqrt{2\qcfs{n}{a, N-a}{n+1}{b+1, N-b-1}\qcfs{n}{a, N-a}{n+1}{b, N-b}}$$
\end{itemize}

Simply substitute each term on the left hand side of the last three identities with the corresponding term on the right-hand side into the first identity. Cancelling will give the general mutation identity. All other mutation identities for $j=n$ come from this one using degeneracies.

The first of the above identities is of a type we have seen before. The last three identities are equivalent. We will prove the first of the last three. Observe that a variation on the octahedron recurrence gives
$$\qcfs{n}{a+1, N-a}{n}{b+1, N-b-1}\qcfs{n+1}{a, N-a-1}{n+1}{b+1, N-b-1}=$$
$$\qcfs{n+1}{a, N-a}{n}{b+1, N-b-1}\qcfs{n}{a+1, N-a-1}{n+1}{b+1, N-b-1}+$$
$$\qcfs{n+1}{a+1, N-a-1}{n}{b+1, N-b-1}\qcfs{n}{a, N-a}{n+1}{b+1, N-b-1}.$$

Note that by duality
$$\qcfs{n}{a+1, N-a}{n}{b+1, N-b-1}=\qcfs{n+1}{a, N-a-1}{n+1}{b+1, N-b-1},$$
$$\qcfs{n+1}{a, N-a}{n}{b+1, N-b-1}=\qcfs{n}{a+1, N-a-1}{n+1}{b+1, N-b-1},$$
$$\qcfs{n+1}{a+1, N-a-1}{n}{b+1, N-b-1}=\qcfs{n}{a, N-a}{n+1}{b+1, N-b-1},$$
so we get 
$$\qcfs{n}{a+1, N-a}{n}{b+1, N-b-1}=$$
$$\sqrt{2\qcfs{n}{a, N-a}{n+1}{b+1, N-b-1}\qcfs{n}{a+1, N-a-1}{n+1}{b+1, N-b-1}}$$
as desired.

\end{proof}

\section{The cluster algebra structure on $\Conf_m G/U$ for $G=Spin_{2n}$}

We now define the cluster algebra structure on $\Conf_m G/U$ when $G=Spin_{2n}$. In fact, to emphasize the parallels with the case of $G=Spin_{2n+1}$, we will let $G=Spin_{2n+2}$. When $G=Spin_{2n+2}$, the cluster algebra structure, along with the mutations realizing $S_3$ symmetries and the flip of triangulation, will be an unfolding of the same structures for $G=Spin_{2n+1}$

We will utilize what we understand about functions on $\Conf_m \A_{SL_{2n+2}}$ in order to study $\Conf_m \A_{Spin_{2n+2}}$. However, because the Dynkin diagram of $Spin_{2n+2}$ is not obtained from $SL_{2n+2}$ by folding, as was the case for $Sp_{2n}$ and $Spin_{2n+1}$, there will be additional complications, particularly regarding signs.

Recall that $Spin_{2n+2}$ is the double cover of the group $SO_{2n+2}$, which is the subgroup of $SL_{2n+2}$ preserving a symmetric quadratic form. We take the quadratic form given in the basis standard basis $e_1, \dots, e_{2n+2}$ by
$$<e_i, e_{2n+3-i}>=(-1)^{i-1}$$
for $1 \leq i \leq n+1$, and $<e_i,e_j>=0$ otherwise.

\begin{rmk} Note that the signature of the quadratic form is $(n+1,n+1)$, so that taking real points gives the split real form of $SO_{2n+2}$. The cluster algebra structure on $\Conf_m \A_{Spin_{2n+2}}$ gives another way of defining the positive structure on $\A_{Spin_{2n+2}, S}$, which gives a parameterization of the Hitchin component for the group $Spin_{2n+2}$ and the surface $S$.
\end{rmk}

The maps
$$Spin_{2n+2} \twoheadrightarrow SO_{2n+2} \hookrightarrow SL_{2n+2}$$
induce maps
$$\Conf_m \A_{Spin_{2n+2}} \rightarrow \Conf_m \A_{SO_{2n+2}} \rightarrow \Conf_m \A_{SL_{2n+2}}.$$

Let us describe these maps concretely. The variety $\A_{SO_{2n+2}}$ parameterizes chains of isotropic vector spaces 
$$V_1 \subset V_2 \subset \cdots \subset V_{n+1} \subset V$$ inside the $2n+2$-dimensional standard representation $V$, where $\operatorname{dim} V_i= i$, and where each $V_i$ is equipped with a volume form.

Equivalently, a point of $\A_{SO_{2n+2}}$ is given by a sequence of vectors 
$$v_1, v_2, \dots, v_{n+1},$$
where $$V_i:=<v_1, \dots, v_i>$$ is isotropic, and where $v_i$ is only determined up to adding linear combinations of $v_j$ for $j < i$.

The volume form on $V_i$ is then $v_1 \wedge \cdots \wedge v_i$.

From the sequence of vectors $v_1, \dots, v_{n+1}$, we can complete to a basis $v_1, v_2, \dots v_{2n+2}$, where $<v_i, v_{2n+3-i}>=(-1)^{i-1}$, and $<v_i,v_j>=0$ otherwise. Equivalently, the quadratic form induces an isomorphism $<-,-> : V \rightarrow V^*$. At the same time, there are perfect pairings
$$\bigwedge\nolimits^k V \times \bigwedge\nolimits^k V^* \rightarrow F$$
$$\bigwedge\nolimits^{2n+2-k} V \times \bigwedge\nolimits^k V \rightarrow F$$
that induce an isomorphism
$$\bigwedge\nolimits^{2n+2-k} V \simeq \bigwedge\nolimits^k V^*.$$ Composing this with the inverse of the isomorphism 
$$<-,-> : \bigwedge\nolimits^k V \rightarrow \bigwedge\nolimits^k V^*$$
gives an isomorphism 
$$\bigwedge\nolimits^{2n+2-k} V \simeq \bigwedge\nolimits^k V^* \simeq \bigwedge\nolimits^k V.$$
Then $v_{n+2}, \dots, v_{2n+2}$ are chosen so that this isomorphism takes $v_1 \wedge \cdots v_{i}$ to $v_1 \wedge \cdots v_{2n+2-i}$ for $i\leq n+1$.

Then $v_1, v_2, \dots v_{2n+2}$ determines a point of $\A_{SL_{2n+2}}$, as $\A_{SL_{2n+2}}$ parameterizes chains of vector subspaces  
$$V_1 \subset V_2 \subset \cdots \subset V_{2n+2} = V$$ along with volume forms $v_1 \wedge \cdots v_{i}$, $1 \leq i \leq 2n+1$.

From the embedding 
$$\A_{SO_{2n+2}} \hookrightarrow \A_{SL_{2n+2}},$$
one naturally gets an embedding $\Conf_m \A_{SO_{2n+2}} \hookrightarrow \Conf_m \A_{SL_{2n+2}}$. We can then pull back functions from $\Conf_m \A_{SL_{2n+2}}$ to get functions on $\Conf_m \A_{SO_{2n+2}}$. However, we are ultimately interested in functions on $\Conf_m \A_{Spin_{2n+2}}$.

The functions on $\Conf_m \A_{Spin_{2n+2}}$ that we will use to define the cluster structure on $\Conf_m \A_{Spin_{2n+2}}$ will be invariants of tensor products of representations of $Spin_{2n+2}$. For $m=3$, they will lie inside 
$$[V_{\lambda} \otimes V_{\mu} \otimes V_{\nu}]^G$$
where $\lambda, \mu, \nu$ are elements of the dominant cone inside the weight lattice. In general, not all such functions will come from pulling back functions on $\Conf_m \A_{SL_{2n+2}}$.

However, suppose that 
$$f \in [V_{\lambda} \otimes V_{\mu} \otimes V_{\nu}]^G \subset \mathcal{O}(\Conf_m \A_{Spin_{2n+2}}).$$
Then
$$f^2 \in [V_{2\lambda} \otimes V_{2\mu} \otimes V_{2\nu}]^G \subset \mathcal{O}(\Conf_m \A_{Spin_{2n+2}}).$$
However, because $2\lambda, 2\mu, 2\nu$ are dominant weights for $SO_{2n+2}$, $f^2$ may be viewed as a function on $\Conf_m \A_{SO_{2n+2}}.$ This function is then a pull-back of a function on $\Conf_m \A_{SL_{2n+2}}.$ Therefore functions on $\Conf_m \A_{Spin_{2n+2}}$ are either the pull-backs of functions on $\Conf_m \A_{SL_{2n+2}}$ or square-roots of such functions. The square-root here corresponds to the fact that $Spin_{2n+2}$ is a double cover of $SO_{2n+2}$. The choice of the branch of the square-root that we take is determined by the positive structure on $\Conf_m \A_{Spin_{2n+2}}$: if $f$ is a positive function on $\Conf_m \A_{SL_{2n+2}}$ such that its square-root is a function on $\Conf_m \A_{Spin_{2n+2}}$, there is a unique choice of $\sqrt{f}$ that is positive on $\Conf_m \A_{Spin_{2n+2}}$. That is the square-root that we will always take. We discuss this issue further in the section on signs.

It will often be convenient to write down functions on $\Conf_m \A_{Spin_{2n+2}}$ in a slightly different way. Recall that $Spin_{2n+2}$ is associated to the Dynkin diagram $D_{n+1}$:

\begin{multicols}{2}
\begin{center}
\begin{tikzpicture}
    \draw (-1,0) node[anchor=east]  {$D_n$};

    \node[dnode,label=below:$1$] (1) at (0,0) {};
    \node[dnode,label=below:$2$] (2) at (1,0) {};
    \node[dnode,label=below:$n-3$] (3) at (3,0) {};
    \node[dnode,label=below:$n-2$] (4) at (4,0) {};
    \node[dnode,label=above:$n-1$] (5) at (5,0.5) {};
    \node[dnode,label=below:$n$] (6) at (5,-0.5) {};

    \path (1) edge[sedge] (2)
          (2) edge[sedge,dashed] (3)
          (3) edge[sedge] (4)
          (4) edge[sedge] (5)
              edge[sedge] (6)
          ;

\draw[yshift=-2cm]
  node[below,text width=6cm] 
  {
  Figure 36 $D_n$ Dynkin diagram
  };

\end{tikzpicture}
\end{center}
\end{multicols}

This diagram has an order two automorphism that gives an outer automorphism of $Spin_{2n+2}$ of order two. Under this map, the flag given by the sequence of vectors $v_1, v_2, \dots v_{2n+2}$ gets sent to the flag given by the sequence of vectors 
$$v_1, v_2, \dots, v_n, v_{n+2}, v_{n+1}, v_{n+3}, \dots, v_{2n+2}.$$
In other words, the vectors $v_{n+1}$ and $v_{n+2}$ switch places.

The existence of this automorphism of $Spin_{2n+2}$ means that there is a second map
$$\Conf_m \A_{Spin_{2n+2}} \rightarrow \Conf_m \A_{SL_{2n+2}}.$$
We can also pull back functions from $\Conf_m \A_{SL_{2n+2}}$ under this second map. We do not get any new functions in this way, but we will often get simpler expressions for our functions in this way. Thus the functions we consider will involve the volume forms $v_1 \wedge \cdots \wedge v_i$ as well as the volume form 
$$v_1 \wedge \cdots \wedge v_n \wedge v_{n+2}.$$

\subsection{Construction of the seed}

We are now ready to construct the seed for the cluster structure on $\Conf_m \A$ when $G=Spin_{2n+2}$. Throughout this section, $G=Spin_{2n+2}$ unless otherwise noted.

The nodes of the diagram correspond to $n+1$ roots that all have the same length. To describe the cluster structure on $\Conf_3 \A$, we need to give the following data: the set $I$ parameterizing vertices, the functions on $\Conf_3 \A$ corresponding to each vertex, and the $B$-matrix for this seed.

The $B$-matrix is encoded via a quiver which  consists of $n^2+3n+2$ vertices, all of which have $d_i=1$. We color all the vertices black. There are $n+1$ edge functions for each edge of the triangle, and $n^2-1$ face functions. The entries of the $B$-matrix are $\pm 1$, and an arrow from $i$ to $j$ means that $b_{ij}>0$.

In Figure 37a, we see the quiver for $Spin_8$. The generalization for other values of $n$ should be clear.

\begin{center}
\begin{tikzpicture}[scale=2]
\begin{scope}[xshift=-1cm]
  \foreach \x in {1, 2}
    \foreach \y in {0, 1, 2, 3}
      \node[] (x\y\x) at (\x,-\y) {\Large $\ontop{x_{\y\x}}{\bullet}$};
  \node (x33) at (4,-3.8) {\Large $\ontop{x_{33}}{\bullet}$};
  \node (x23) at (4,-2.8) {\Large $\ontop{x_{23}}{\bullet}$};
  \node (x13) at (4,-1.8) {\Large $\ontop{x_{13}}{\bullet}$};
  \node (x03) at (4,-0.8) {\Large $\ontop{x_{03}}{\bullet}$};
  \node (x33*) at (3.5,-2.2) {\Large $\ontop{x_{33^*}}{\bullet}$};
  \node (x23*) at (3.5,-1.2) {\Large $\ontop{x_{23^*}}{\bullet}$};
  \node (x13*) at (3.5,-0.2) {\Large $\ontop{x_{13^*}}{\bullet}$};
  \node (x03*) at (3.5,0.8) {\Large $\ontop{x_{03^*}}{\bullet}$};
  \node[] (y1) at (0,-2)  {\Large $\ontop{y_1}{\bullet}$};
  \node[] (y2) at (0,-1)  {\Large $\ontop{y_2}{\bullet}$};
  \node[] (y3) at (5,-1.5)  {\Large $\ontop{y_3}{\bullet}$};
  \node[] (y3*) at (4.5,0.1)  {\Large $\ontop{y_3}{\bullet}$};

  \draw [->] (x31) to (x21);
  \draw [->] (x21) to (x11);
  \draw [->] (x11) to (x01);
  \draw [->] (x32) to (x22);
  \draw [->] (x22) to (x12);
  \draw [->] (x12) to (x02);
  \draw [->] (x33) to (x23);
  \draw [->] (x23) to (x13);
  \draw [->] (x13) to (x03);
  \draw [->, dashed] (y2) .. controls +(up:2) and +(up:2) .. (y3);
  \draw [->, dashed] (y1) to (y2);

  \draw [->, dashed] (x03) to (x02);
  \draw [->, dashed] (x02) to (x01);
  \draw [->] (y3) to (x13);
  \draw [->] (x13) to (x12);
  \draw [->] (x12) to (x11);
  \draw [->] (x11) to (y2);
  \draw [->] (x23) to (x22);
  \draw [->] (x22) to (x21);
  \draw [->] (x21) to (y1);
  \draw [->, dashed] (x33) to (x32);
  \draw [->, dashed] (x32) to (x31);

  \draw [->] (x01) to (x12);
  \draw [->] (x02) to (x13);
  \draw [->] (x03) to (y3);
  \draw [->] (y2) to (x21);
  \draw [->] (x11) to (x22);
  \draw [->] (x12) to (x23);
  \draw [->] (y1) to (x31);
  \draw [->] (x21) to (x32);
  \draw [->] (x22) to (x33);

  \draw [->] (x33*) to (x23*);
  \draw [->] (x23*) to (x13*);
  \draw [->] (x13*) to (x03*);
  \draw [->, dashed] (y2) .. controls +(up:2) and +(up:2) .. (y3*);

  \draw [->, dashed] (x03*) to (x02);
  \draw [->] (y3*) to (x13*);
  \draw [->] (x13*) to (x12);
  \draw [->] (x23*) to (x22);
  \draw [->, dashed] (x33*) to (x32);

  \draw [->] (x02) to (x13*);
  \draw [->] (x03*) to (y3*);
  \draw [->] (x12) to (x23*);
  \draw [->] (x22) to (x33*);

\end{scope}

\draw[yshift=-4.5cm,xshift=1cm]
  node[below,text width=6cm] 
  {
  Figure 37a. Quiver encoding the cluster structure for $\Conf_3 \A_{Spin_{8}}$.
  };

\end{tikzpicture}
\end{center}

The diagram is busy because the vertices $x_{i3}$ and $y_3$ are doubled by the vertices $x_{i3^*}$ and $y_{3^*}$. In Figure 37b we can see the quiver without the vertices $x_{i3}$ and $y_3$. In Figure 37c, we can see the quiver without the vertices  $x_{i3^*}$ and $y_{3^*}$.

\begin{center}
\begin{tikzpicture}[scale=2]
\begin{scope}[xshift=-1cm]
  \foreach \x in {1, 2}
    \foreach \y in {0, 1, 2, 3}
      \node[] (x\y\x) at (\x,-\y) {\Large $\ontop{x_{\y\x}}{\bullet}$};
  \node (x33*) at (3.5,-2.2) {\Large $\ontop{x_{33^*}}{\bullet}$};
  \node (x23*) at (3.5,-1.2) {\Large $\ontop{x_{23^*}}{\bullet}$};
  \node (x13*) at (3.5,-0.2) {\Large $\ontop{x_{13^*}}{\bullet}$};
  \node (x03*) at (3.5,0.8) {\Large $\ontop{x_{03^*}}{\bullet}$};
  \node[] (y1) at (0,-2)  {\Large $\ontop{y_1}{\bullet}$};
  \node[] (y2) at (0,-1)  {\Large $\ontop{y_2}{\bullet}$};
  \node[] (y3*) at (4.5,0.1)  {\Large $\ontop{y_3}{\bullet}$};

  \draw [->] (x31) to (x21);
  \draw [->] (x21) to (x11);
  \draw [->] (x11) to (x01);
  \draw [->] (x32) to (x22);
  \draw [->] (x22) to (x12);
  \draw [->] (x12) to (x02);
  \draw [->, dashed] (y1) to (y2);

  \draw [->, dashed] (x02) to (x01);
  \draw [->] (x12) to (x11);
  \draw [->] (x11) to (y2);
  \draw [->] (x22) to (x21);
  \draw [->] (x21) to (y1);
  \draw [->, dashed] (x32) to (x31);

  \draw [->] (x01) to (x12);
  \draw [->] (y2) to (x21);
  \draw [->] (x11) to (x22);
  \draw [->] (y1) to (x31);
  \draw [->] (x21) to (x32);

  \draw [->] (x33*) to (x23*);
  \draw [->] (x23*) to (x13*);
  \draw [->] (x13*) to (x03*);
  \draw [->, dashed] (y2) .. controls +(up:2) and +(up:2) .. (y3*);

  \draw [->, dashed] (x03*) to (x02);
  \draw [->] (y3*) to (x13*);
  \draw [->] (x13*) to (x12);
  \draw [->] (x23*) to (x22);
  \draw [->, dashed] (x33*) to (x32);

  \draw [->] (x02) to (x13*);
  \draw [->] (x03*) to (y3*);
  \draw [->] (x12) to (x23*);
  \draw [->] (x22) to (x33*);

\end{scope}

\draw[yshift=-4.5cm,xshift=1cm]
  node[below,text width=6cm] 
  {
  Figure 37b. Quiver encoding the cluster structure for $\Conf_3 \A_{Spin_{8}}$ without the vertices $x_{i3}$ and $y_3$.
  };

\end{tikzpicture}
\end{center}

\begin{center}
\begin{tikzpicture}[scale=2]
\begin{scope}[xshift=-1cm]
  \foreach \x in {1, 2}
    \foreach \y in {0, 1, 2, 3}
      \node[] (x\y\x) at (\x,-\y) {\Large $\ontop{x_{\y\x}}{\bullet}$};
  \node (x33) at (4,-3.8) {\Large $\ontop{x_{33}}{\bullet}$};
  \node (x23) at (4,-2.8) {\Large $\ontop{x_{23}}{\bullet}$};
  \node (x13) at (4,-1.8) {\Large $\ontop{x_{13}}{\bullet}$};
  \node (x03) at (4,-0.8) {\Large $\ontop{x_{03}}{\bullet}$};
  \node[] (y1) at (0,-2)  {\Large $\ontop{y_1}{\bullet}$};
  \node[] (y2) at (0,-1)  {\Large $\ontop{y_2}{\bullet}$};
  \node[] (y3) at (5,-1.5)  {\Large $\ontop{y_3}{\bullet}$};

  \draw [->] (x31) to (x21);
  \draw [->] (x21) to (x11);
  \draw [->] (x11) to (x01);
  \draw [->] (x32) to (x22);
  \draw [->] (x22) to (x12);
  \draw [->] (x12) to (x02);
  \draw [->] (x33) to (x23);
  \draw [->] (x23) to (x13);
  \draw [->] (x13) to (x03);
  \draw [->, dashed] (y2) .. controls +(up:2) and +(up:2) .. (y3);
  \draw [->, dashed] (y1) to (y2);

  \draw [->, dashed] (x03) to (x02);
  \draw [->, dashed] (x02) to (x01);
  \draw [->] (y3) to (x13);
  \draw [->] (x13) to (x12);
  \draw [->] (x12) to (x11);
  \draw [->] (x11) to (y2);
  \draw [->] (x23) to (x22);
  \draw [->] (x22) to (x21);
  \draw [->] (x21) to (y1);
  \draw [->, dashed] (x33) to (x32);
  \draw [->, dashed] (x32) to (x31);

  \draw [->] (x01) to (x12);
  \draw [->] (x02) to (x13);
  \draw [->] (x03) to (y3);
  \draw [->] (y2) to (x21);
  \draw [->] (x11) to (x22);
  \draw [->] (x12) to (x23);
  \draw [->] (y1) to (x31);
  \draw [->] (x21) to (x32);
  \draw [->] (x22) to (x33);

\end{scope}

\draw[yshift=-4.5cm,xshift=1cm]
  node[below,text width=6cm] 
  {
  Figure 37c. Quiver encoding the cluster structure for $\Conf_3 \A_{Spin_{8}}$ without the vertices  $x_{i3^*}$ and $y_{3^*}$.
  };

\end{tikzpicture}
\end{center}

A dotted arrow means that $b_{ij}=\pm \frac{1}{2}$.

For simplicity, all future diagrams will only contain the vertices $x_{in}$ and $y_n$, and not the vertices $x_{in^*}$ and $y_{n^*}$, which merely double them.

We will no longer use single letters like $i, j$ to denote vertices of the quiver, because it will be convenient for us to use the pairs $(i, j)$ to parameterize the vertices of the quiver. In the formulas in the remainder of this section, we will not refer to the particular entries of the $B$-matrix, $b_{ij}$. Instead, the values of the entries of the $B$-matrix will be encoded in quivers. This will hopefully avoid any notational confusion.

Label the vertices of the quiver $x_{ij}$ and $y_k$, where $0 \leq i \leq n+1$, $j=1, 2, 3, \dots, n, n^*$, $k=1, 2, 3, \dots, n, n^*$. The vertices $y_k$ and $x_{ij}$ for $i= 0$ or $n$ are frozen. We will sometimes write $x_{i,j}$ for $x_{ij}$ for orthographic reasons. Note that the dotted arrows only go between frozen vertices, thus the entries $b_{ij}$ of the $B$-matrix are integral unless $i$ and $j$ are both frozen, and thus the $B$-matrix defines a cluster algebra.

Let us now recall some facts about the representation theory of $Spin_{2n+2}$. The fundamental representations of $Spin_{2n+2}$ are labelled by the fundamental weights $\omega_1, \dots, \omega_n, \omega_{n^*}$. $Spin_{2n+2}$ has a standard $2n+1$-dimensional representation $V$. Let $<-,->$ be the orthogonal pairing. Then for $i < n$ the representation $V_{\omega_i}$ corresponding to $\omega_i$ is precisely $\bigwedge\nolimits^i V$. The representations $V_{\omega_n}, V_{\omega_{n^*}}$ are the \emph{spin representations} of $Spin_{2n+2}$. When $n$ is even, the spin representations are dual to each other. When $n$ is odd, the spin representations are self-dual. The the direct sum of the representations $V_{2\omega_{n^*}}$ and $V_{2\omega_n}$ is isomorphic to $\bigwedge\nolimits^{n+1} V$, and $V_{\omega_n+\omega_{n^*}}$ is isomorphic to $\bigwedge\nolimits^{n} V$.

We now say which functions are attached to the vertices of the quiver. Recall the functions defined via the webs from Figures 3, 4, and 6. It will be convenient to describe the functions attached to $x_{ij}$ for $j \leq n-1$ and $y_k$ for $k \leq n-1$ first:

\begin{enumerate}
\item For $k \leq n-1$, assign the function $\dud{k}{2n+2-k}(-1)^k=\dud{2n+2-k}{k}$ to $y_k$. 
\item When $i \geq j$, assign the function $\tcfr{n-i}{n+2+i-j}{j}(-1)^{n-i}$ to $x_{ij}$.
\item When $i < j$ and $i \neq 0$, we assign the function $\tcfr{n-i, 2n+2+i-j}{n+2}{j}(-1)^{n-i}$ to $x_{ij}$. 
\item When $i=0$, we assign the function $\tcfr{2n+2-j}{}{j}$ to $x_{ij}$.
\end{enumerate}

Now, when $j$ or $k$ are equal to $n$ or $n^*$, the story is somewhat more complicated. As in the case of $Spin_{2n+1}$, the functions involve square-roots. But then there are two additional complications: first, we will need to slightly modify some of the functions which we previously defined in order to deal with the fact that there are two spin representations of $Spin_{2n+2}$; and second, the somewhat different behavior of these spin representations for $n$ odd and $n$ even means we will need to treat these cases separately.

Recall that in Figure 3, we defined functions of the form $\tcfr{a, b}{c}{d}$. We will now need to define some new functions of the form $\tcfr{a, b}{c}{n+1*}$, where $a+b+c=3n+3$. This is a function on the space of configurations of three principal flags for the group $Spin_{2n+2}$. Suppose these flags are given in terms of the three flags
$$u_1, \dots, u_N;$$
$$v_1, \dots, v_N;$$
$$w_1, \dots, w_N;$$
where $N=2n+2$.

Now consider the forms
$$U_{a} := u_1 \wedge \cdots \wedge u_{a},$$
$$U_b := u_1 \wedge \cdots \wedge u_b,$$
$$V_c := v_1 \wedge \cdots \wedge v_c,$$ 
$$W_{n+1*} := w_1 \wedge \cdots \wedge w_n \wedge w_{n+2}.$$

The function $\tcfr{a, b}{c}{n+1^*}$ is defined in the same way as the function $\tcfr{a, b}{c}{n+1}$, except everywhere where one had $W_{n+1}:= w_1 \wedge \cdots \wedge w_n \wedge w_{n+1}$, one replaces this by $W_{n+1^*} := w_1 \wedge \cdots \wedge w_n \wedge w_{n+2}.$

Thus there is a natural map 
$$\phi_{a+c-N, N-a}: \bigwedge\nolimits^{c} V \rightarrow \bigwedge\nolimits^{a+c-N} V \otimes \bigwedge\nolimits^{N-a} V.$$
There are also natural maps 
$$U_b \wedge - \wedge W_{n+1*} :  \bigwedge\nolimits^{a+c-N} V \rightarrow \bigwedge\nolimits^{N} V \simeq F$$ and 
$$U_a \wedge - : \bigwedge\nolimits^{N-a} V \rightarrow \bigwedge\nolimits^{N} V \simeq F.$$
Applying these maps to the first and second factors of $\phi_{a+c-N,N-a}(V_c)$, respectively, and then multiplying, we get get the value of our function. This is a function on $\Conf_3 \A_{SL_{N}}$ which pulls back to a function on $\Conf_3 \A_{Spin_{N}}$.

Recall that there is an outer automorphism of the group $Spin_{2n+2}$. On the level of flags, this automorphism takes the $n+1$-form $W_{n+1}$ to the $n+1$-form $W_{n+1*}$. All the new functions we will need to define the cluster algebra structure on $\Conf_3 \A_{Spin_{N}}$ will involve taking a previously defined function and substituting $n+1*$ for $n+1$ for some subset of the arguments. For example, it is straightforward to define the functions $\sqrt{-\dud{n+1^*}{n+1}}$ and $\sqrt{\tcfr{n-i, n+2+i}{n+1}{n+1^*}}$, which we will use below.

We will distinguish two cases: $n$ even or $n$ odd. First suppose $n$ is even. Then we assign functions as follows:

\begin{enumerate}
\item Assign the function $\sqrt{-\dud{n+1^*}{n+1}}$ to $y_n$ and $\sqrt{\dud{n+1}{n+1^*}}$ to $y_{n^*}$. 
\item Assign the function $\sqrt{-\tcfr{}{n+1^*}{n+1}}$ to $x_{nn}$ and $\sqrt{\tcfr{}{n+1}{n+1^*}}$ to $x_{nn^*}$. 
\item When $0 < i < n$, $i$ odd, we assign the function $\sqrt{-\tcfr{n-i, n+2+i}{n+1}{n+1}}$ to $x_{in}$ and $\sqrt{\tcfr{n-i, n+2+i}{n+1^*}{n+1^*}}$ to $x_{in^*}$. 
\item When $0 < i < n$, $i$ even, we assign the function $\sqrt{-\tcfr{n-i, n+2+i}{n+1^*}{n+1}}$ to $x_{in}$ and $\sqrt{\tcfr{n-i, n+2+i}{n+1}{n+1^*}}$ to $x_{in^*}$.
\item Assign the function $\sqrt{-\tcfr{n+1^*}{}{n+1}}$ to $x_{0n}$ and $\sqrt{\tcfr{n+1}{}{n+1^*}}$ to $x_{0n^*}$. 
\end{enumerate}

When $n$ is odd, we assign functions as follows:

\begin{enumerate}
\item Assign the function $\sqrt{\dud{n+1}{n+1}}$ to $y_n$ and $\sqrt{-\dud{n+1^*}{n+1^*}}$ to $y_{n^*}$. 
\item Assign the function $\sqrt{\tcfr{}{n+1}{n+1}}$ to $x_{nn}$ and $\sqrt{-\tcfr{}{n+1^*}{n+1^*}}$ to $x_{nn^*}$. 
\item When $0 < i < n$, $i$ odd, we assign the function $\sqrt{\tcfr{n-i, n+2+i}{n+1^*}{n+1}}$ to $x_{in}$ and $\sqrt{-\tcfr{n-i, n+2+i}{n+1}{n+1^*}}$ to $x_{in^*}$. 
\item When $0 < i < n$, $i$ even, we assign the function $\sqrt{\tcfr{n-i, n+2+i}{n+1}{n+1}}$ to $x_{in}$ and $\sqrt{-\tcfr{n-i, n+2+i}{n+1^*}{n+1^*}}$ to $x_{in^*}$.
\item Assign the function $\sqrt{\tcfr{n+1}{}{n+1}}$ to $x_{0n}$ and $\sqrt{-\tcfr{n+1^*}{}{n+1^*}}$ to $x_{0n^*}$. 
\end{enumerate}

This completely describes the cluster structure on $\Conf_3 \A_{Spin_{2n+2}}$. The fact that we can take the square-roots of the functions assigned to $x_{in}$ and $y_n$ and get functions that are well-defined on $\Conf_3 \A_{Spin_{2n+2}}$ follows from the computations with reduced words that we perform later. Note that the cluster structure is not symmetric with respect to the three flags. Performing various $S_3$ symmetries, we obtain six different possible cluster structures on  $\Conf_3 \A_{Spin_{2n+2}}$. These six structures are related by sequences of mutations that we describe in the next section. Below, in Figure 38, we depict two of the cluster structures for $\Conf_3 \A_{Spin_{12}}$ obtained from our original cluster structure by an $S_3$ symmetry.

\begin{center}
\begin{tikzpicture}[scale=2.4]
  \node (x01) at (-1,0) {$\tcfr{9}{}{1}$};
  \node (x02) at (0,0) {$\tcfr{8}{}{2}$};
  \node (x03) at (1,0) {$\tcfr{7}{}{3}$};
  \node (x04) at (2,0) {$\sqrt{-\tcfr{5^*}{}{5}}$};
  \node (x11) at (-1,-1) {$-\tcfr{3}{6}{1}$};
  \node (x12) at (0,-1) {$-\tcfr{3, 9}{6}{2}$};
  \node (x13) at (1,-1) {$-\tcfr{3, 8}{6}{3}$};
  \node (x14) at (2,-1) {$\sqrt{-\tcfr{3, 7}{5}{5}}$};
  \node (x21) at (-1,-2) {$\tcfr{2}{7}{1}$};
  \node (x22) at (0,-2) {$\tcfr{2}{6}{2}$};
  \node (x23) at (1,-2) {$\tcfr{2, 9}{6}{3}$};
  \node (x24) at (2,-2) {$\sqrt{-\tcfr{2, 8}{5^*}{5}}$};
  \node (x31) at (-1,-3) {$-\tcfr{1}{8}{1}$};
  \node (x32) at (0,-3) {$-\tcfr{1}{7}{2}$};
  \node (x33) at (1,-3) {$-\tcfr{1}{6}{3}$};
  \node (x34) at (2,-3) {$\sqrt{-\tcfr{1, 9}{5}{5}}$};
  \node (x41) at (-1,-4) {$\tcfr{}{9}{1}$};
  \node (x42) at (0,-4) {$\tcfr{}{8}{2}$};
  \node (x43) at (1,-4) {$\tcfr{}{7}{3}$};
  \node (x44) at (2,-4) {$\sqrt{-\tcfr{}{5^*}{5}}$};
  \node (y1) at (-2,-3) {$-\dud{1}{9}$};
  \node (y2) at (-2,-2) {$\dud{2}{8}$};
  \node (y3) at (-2,-1) {$-\dud{3}{7}$};
  \node (y4) at (3,-1) {$\sqrt{-\dud{5^*}{5}}$};

  \draw [->] (x41) to (x31);
  \draw [->] (x31) to (x21);
  \draw [->] (x21) to (x11);
  \draw [->] (x11) to (x01);
  \draw [->] (x42) to (x32);
  \draw [->] (x32) to (x22);
  \draw [->] (x22) to (x12);
  \draw [->] (x12) to (x02);
  \draw [->] (x43) to (x33);
  \draw [->] (x33) to (x23);
  \draw [->] (x23) to (x13);
  \draw [->] (x13) to (x03);
  \draw [->] (x44) to (x34);
  \draw [->] (x34) to (x24);
  \draw [->] (x24) to (x14);
  \draw [->] (x14) to (x04);
  \draw [->, dashed] (y3) .. controls +(up:2) and +(up:2) .. (y4);
  \draw [->, dashed] (y2) to (y3);
  \draw [->, dashed] (y1) to (y2);

  \draw [->, dashed] (x04) to (x03);
  \draw [->, dashed] (x03) to (x02);
  \draw [->, dashed] (x02) to (x01);
  \draw [->] (y4) to (x14);
  \draw [->] (x14) to (x13);
  \draw [->] (x13) to (x12);
  \draw [->] (x12) to (x11);
  \draw [->] (x11) to (y3);
  \draw [->] (x24) to (x23);
  \draw [->] (x23) to (x22);
  \draw [->] (x22) to (x21);
  \draw [->] (x21) to (y2);
  \draw [->] (x34) to (x33);
  \draw [->] (x33) to (x32);
  \draw [->] (x32) to (x31);
  \draw [->] (x31) to (y1);
  \draw [->, dashed] (x44) to (x43);
  \draw [->, dashed] (x43) to (x42);
  \draw [->, dashed] (x42) to (x41);

  \draw [->] (x01) to (x12);
  \draw [->] (x02) to (x13);
  \draw [->] (x03) to (x14);
  \draw [->] (x04) to (y4);
  \draw [->] (y3) to (x21);
  \draw [->] (x11) to (x22);
  \draw [->] (x12) to (x23);
  \draw [->] (x13) to (x24);
  \draw [->] (y2) to (x31);
  \draw [->] (x21) to (x32);
  \draw [->] (x22) to (x33);
  \draw [->] (x23) to (x34);
  \draw [->] (y1) to (x41);
  \draw [->] (x31) to (x42);
  \draw [->] (x32) to (x43);
  \draw [->] (x33) to (x44);

\draw[yshift=-5.5cm,xshift=1cm]
  node[below,text width=6cm] 
  {
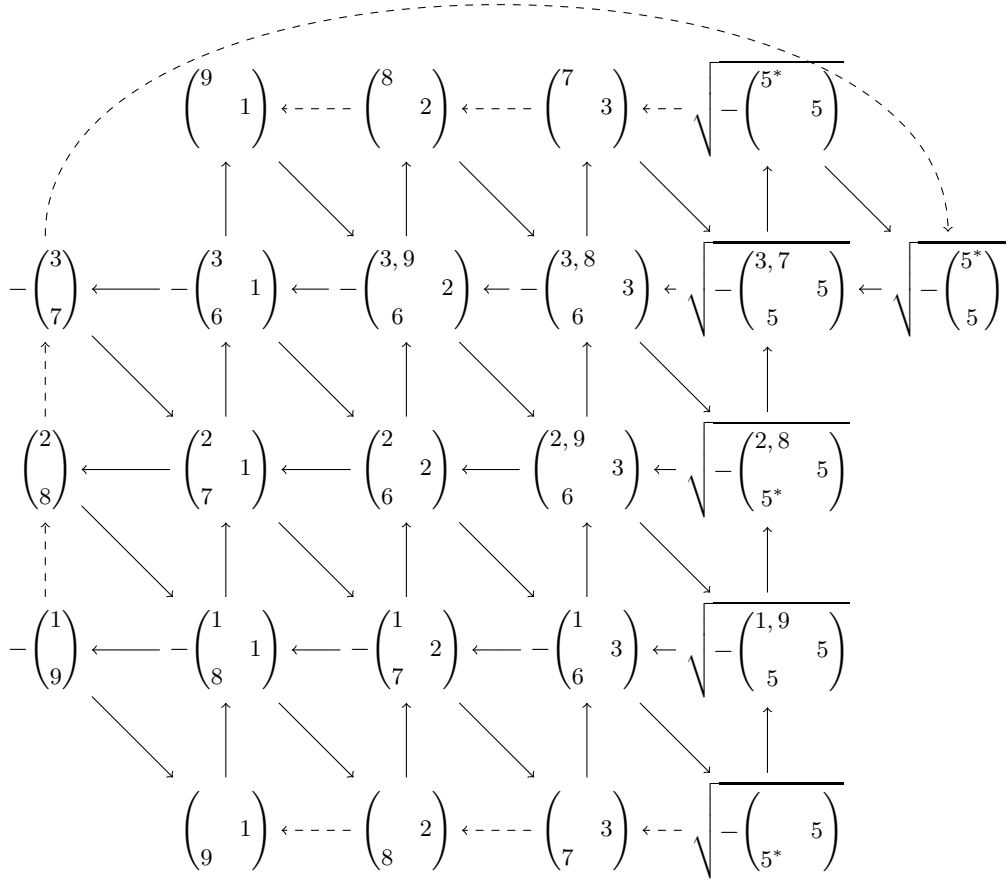
  Figure 38a. One cluster structure for $\Conf_3 \A_{Spin_{10}}$.
  };

\end{tikzpicture}
\end{center}

\begin{center}
\begin{tikzpicture}[scale=2.6]

  \node (x01) at (-2,-1) {$\dud{9}{1}$};
  \node (x02) at (-2,-2) {$\dud{8}{2}$};
  \node (x03) at (-2,-3) {$\dud{7}{3}$};
  \node (x04) at (-2,-4) {$\sqrt{-\dud{5^*}{5}}$};
  \node (x11) at (-1,-1) {$\tcfr{3}{1}{6}$};
  \node (x12) at (-1,-2) {$\tcfr{3, 9}{2}{6}$};
  \node (x13) at (-1,-3) {$\tcfr{3, 8}{3}{6}$};
  \node (x14) at (-1,-4) {$\sqrt{-\tcfr{3, 7}{5}{5}}$};
  \node (x21) at (0,-1) {$-\tcfr{2}{1}{7}$};
  \node (x22) at (0,-2) {$\tcfr{2}{2}{6}$};
  \node (x23) at (0,-3) {$\tcfr{2, 9}{3}{6}$};
  \node (x24) at (0,-4) {$\sqrt{\tcfr{2, 8}{5}{5^*}}$};
  \node (x31) at (1,-1) {$\tcfr{1}{1}{8}$};
  \node (x32) at (1,-2) {$-\tcfr{1}{2}{7}$};
  \node (x33) at (1,-3) {$\tcfr{1}{3}{6}$};
  \node (x34) at (1,-4) {$\sqrt{-\tcfr{1, 9}{5}{5}}$};
  \node (x41) at (2,-1) {$-\tcfr{}{1}{9}$};
  \node (x42) at (2,-2) {$\tcfr{}{2}{8}$};
  \node (x43) at (2,-3) {$-\tcfr{}{3}{7}$};
  \node (x44) at (2,-4) {$\sqrt{\tcfr{}{5}{5^*}}$};
  \node (y1) at (1,0) {$-\dur{1}{9}$};
  \node (y2) at (0,0) {$\dur{2}{8}$};
  \node (y3) at (-1,0) {$-\dur{3}{7}$};
  \node (y4) at (-1,-5) {$\sqrt{-\dur{5^*}{5}}$};

  \draw [->] (x01) to (x11);
  \draw [->] (x11) to (x21);
  \draw [->] (x21) to (x31);
  \draw [->] (x31) to (x41);
  \draw [->] (x02) to (x12);
  \draw [->] (x12) to (x22);
  \draw [->] (x22) to (x32);
  \draw [->] (x32) to (x42);
  \draw [->] (x03) to (x13);
  \draw [->] (x13) to (x23);
  \draw [->] (x23) to (x33);
  \draw [->] (x33) to (x43);
  \draw [->] (x04) to (x14);
  \draw [->] (x14) to (x24);
  \draw [->] (x24) to (x34);
  \draw [->] (x34) to (x44);
  \draw [->, dashed] (y4) .. controls +(left:2) and +(left:2) .. (y3);
  \draw [->, dashed] (y3) to (y2);
  \draw [->, dashed] (y2) to (y1);

  \draw [->, dashed] (x01) to (x02);
  \draw [->, dashed] (x02) to (x03);
  \draw [->, dashed] (x03) to (x04);
  \draw [->] (y3) to (x11);
  \draw [->] (x11) to (x12);
  \draw [->] (x12) to (x13);
  \draw [->] (x13) to (x14);
  \draw [->] (x14) to (y4);
  \draw [->] (y2) to (x21);
  \draw [->] (x21) to (x22);
  \draw [->] (x22) to (x23);
  \draw [->] (x23) to (x24);
  \draw [->] (y1) to (x31);
  \draw [->] (x31) to (x32);
  \draw [->] (x32) to (x33);
  \draw [->] (x33) to (x34);
  \draw [->, dashed] (x41) to (x42);
  \draw [->, dashed] (x42) to (x43);
  \draw [->, dashed] (x43) to (x44);

  \draw [->] (x12) to (x01);
  \draw [->] (x13) to (x02);
  \draw [->] (x14) to (x03);
  \draw [->] (y4) to (x04);
  \draw [->] (x21) to (y3);
  \draw [->] (x22) to (x11);
  \draw [->] (x23) to (x12);
  \draw [->] (x24) to (x13);
  \draw [->] (x31) to (y2);
  \draw [->] (x32) to (x21);
  \draw [->] (x33) to (x22);
  \draw [->] (x34) to (x23);
  \draw [->] (x41) to (y1);
  \draw [->] (x42) to (x31);
  \draw [->] (x43) to (x32);
  \draw [->] (x44) to (x33);

\draw[yshift=-5.5cm]
  node[below,text width=6cm] 
  {
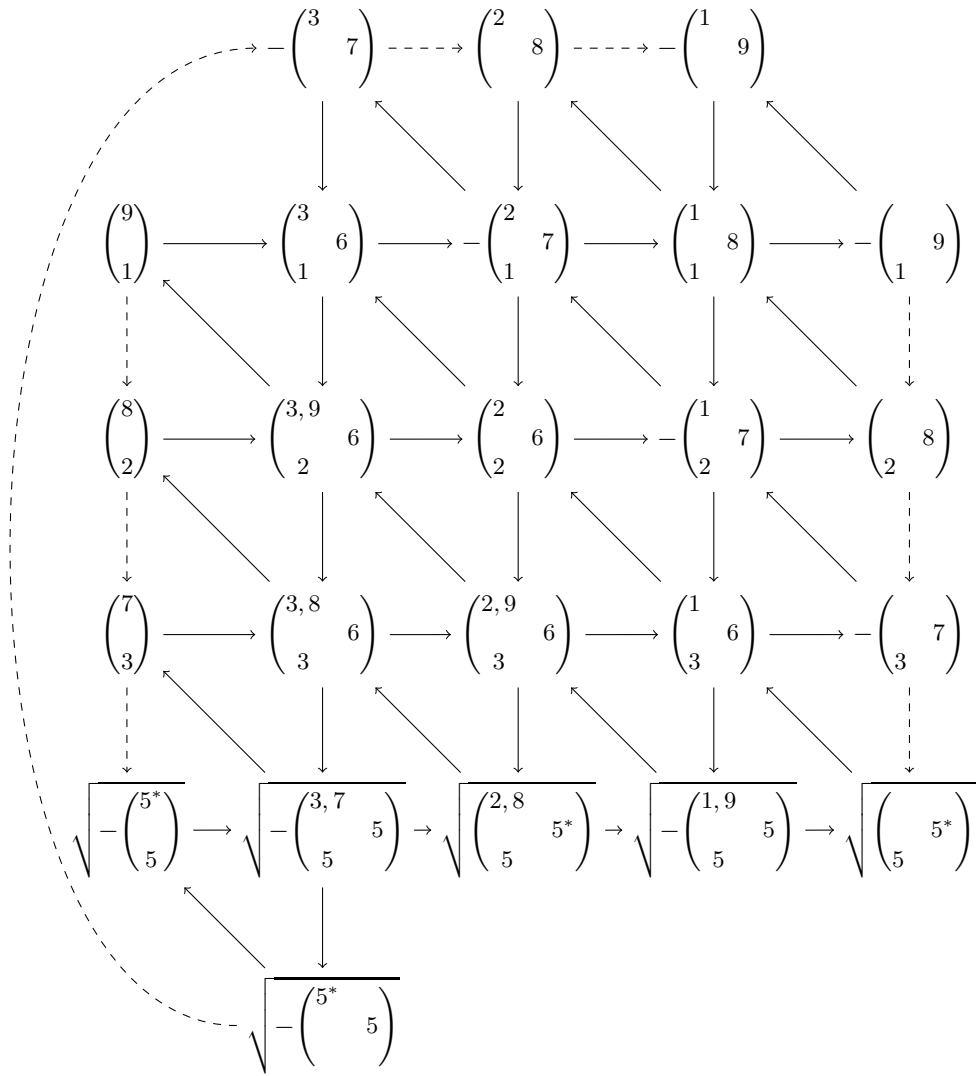
  Figure 38b. Another cluster structures for $\Conf_3 \A_{Spin_{10}}$ related by to the first by an $S_3$ symmetry.
  };

\end{tikzpicture}
\end{center}

Let us describe in more detail how to obtain these other cluster structures. If the $S_3$ symmetry is a rotation, we just rotate the quiver. If the $S_3$ symmetry is a transposition, we transpose the quiver and also reverse the arrows.

Unlike the cases of $G=Sp_{2n}$ and $G=Spin_{2n+1}$, the functions attached to the vertices do not come from permuting the arguments in our notation for the function. Permuting the arguments only gives the correct functions up to a sign. This can be seen in the examples above. This is important enough that we will discuss this separately in the next section.

The cluster structure for $\Conf_m \A$ comes from triangulating an $m$-gon and then attaching the cluster structure on $\Conf_3 \A$ to each triangle and then using the procedure of amalgamation. As before, the vertices $y_k$, $x_{0j}$ and $x_{nj}$ will lie on the edges of the triangle, and the functions attached to them are {\em edge functions}. On $\Conf_3 \A$ the edge vertices are frozen, and the functions attached to these edges depend on only two of the three flags. All other functions are {\em face functions}.

To form the quiver for $\Conf_m \A_{Spin_{2n+2}}$, we first take a triangulation of an $m$-gon. On each of the $m-2$ triangles, attach any one of the six quivers formed from performing $S_3$ symmetries on the quiver for $\Conf_3 \A_{Spin_{2n+2}}$. Each edge of each of these triangles has $n$ frozen vertices. For any two triangles sharing an edge, the corresponding vertices on those edges are identified. Those vertices become unfrozen.

If vertices $i$ and $j$ are glued with $i'$ and $j'$ to get new vertices $i''$ and $j''$, then we declare that $$b_{i''j''}=b_{ij}+b_{i'j'}.$$
In other words, two dotted arrows in the same direction glue to give us a solid arrow, whereas two dotted arrows in the opposite direction cancel to give us no arrow. One can easily check that, again, any gluing will result in no dotted arrows using the unfrozen vertices. The arrows involving vertices that were not previously frozen remain the same. Figure 39 below shows one gluing between two triangles for $Spin_8$. As will be the case in the remainder of the paper, doubled vertices are not shown.

\begin{center}
\begin{tikzpicture}[scale=2.4]
  \node (x-31) at (-3,0) {$-\dul{1}{7}$};
  \node (x-21) at (-2,0) {$-\tcfl{1}{2}{5}$};
  \node (x-11) at (-1,0) {$\tcfl{1}{1}{6}$};
  \node (x01) at (0,0) {$-\dud{1}{7}$};
  \node (x11) at (1,0) {$\tcfr{2}{7}{7}$};
  \node (x21) at (2,0) {$\tcfr{3}{7}{6}$};
  \node (x31) at (3,0) {$\ddr{7}{1}$};
  \node (x-32) at (-3,-1) {$\dul{2}{6}$};
  \node (x-22) at (-2,-1) {$\tcfl{2}{2, 7}{5}$};
  \node (x-12) at (-1,-1) {$-\tcfl{2}{1}{5}$};
  \node (x02) at (0,-1) {$\dud{2}{6}$};
  \node (x12) at (1,-1) {$\tcfr{3}{6}{7}$};
  \node (x22) at (2,-1) {$-\tcfr{3}{6}{1, 6}$};
  \node (x32) at (3,-1) {$\ddr{6}{2}$};
  \node (x-33) at (-3,-2) {$\sqrt{\dul{4}{4}}$};
  \node (x-23) at (-2,-2) {$\sqrt{\tcfl{4}{2, 6}{4}}$};
  \node (x-13) at (-1,-2) {$\sqrt{\tcfl{4}{1, 7}{4^*}}$};
  \node (x03) at (0,-2) {$\sqrt{\dud{4}{4}}$};
  \node (x13) at (1,-2) {$\sqrt{\tcfr{4^*}{4}{1,7}}$};
  \node (x23) at (2,-2) {$\sqrt{\tcfr{4}{4}{2,6}}$};
  \node (x33) at (3,-2) {$\sqrt{\ddr{4}{4}}$};
  \node (y-1) at (-0.5,1) {$\dld{1}{7}$};
  \node (y-2) at (-1.5,1) {$\dld{2}{6}$};
  \node (y-3) at (-2.5,-3) {$\sqrt{\dld{4}{4}}$};
  \node (y1) at (0.5,1) {$\tcfr{1}{}{7}$};
  \node (y2) at (1.5,1) {$\tcfr{2}{}{6}$};
  \node (y3) at (2.5,-3) {$\sqrt{\tcfr{4}{}{4}}$};

  \draw [->] (x01) to (x11);
  \draw [->] (x11) to (x21);
  \draw [->] (x21) to (x31);
  \draw [->] (x02) to (x12);
  \draw [->] (x12) to (x22);
  \draw [->] (x22) to (x32);
  \draw [->] (x03) to (x13);
  \draw [->] (x13) to (x23);
  \draw [->] (x23) to (x33);
  \draw [->, dashed] (y1) to (y2);
  \draw [->, dashed] (y2) .. controls +(right:1) and +(up:1) .. (y3);

  \draw [->] (x03) to (x02);
  \draw [->] (x02) to (x01);

  \draw [->] (x13) to (x12);
  \draw [->] (x12) to (x11);
  \draw [->] (x11) to (y1);
  \draw [->] (y3) to (x23);
  \draw [->] (x23) to (x22);
  \draw [->] (x22) to (x21);
  \draw [->] (x21) to (y2);
  \draw [->, dashed] (x33) to (x32);
  \draw [->, dashed] (x32) to (x31);

  \draw [->] (y1) to (x01);
  \draw [->] (x11) to (x02);
  \draw [->] (x12) to (x03);
 \draw [->] (y2) to (x11);
  \draw [->] (x21) to (x12);
  \draw [->] (x22) to (x13);
  \draw [->] (x31) to (x22);
  \draw [->] (x32) to (x23);
  \draw [->] (x33) to (y3);

  \draw [->] (x01) to (x-11);
  \draw [->] (x-11) to (x-21);
  \draw [->] (x-21) to (x-31);
  \draw [->] (x02) to (x-12);
  \draw [->] (x-12) to (x-22);
  \draw [->] (x-22) to (x-32);
  \draw [->] (x03) to (x-13);
  \draw [->] (x-13) to (x-23);
  \draw [->] (x-23) to (x-33);
  \draw [->, dashed] (y-1) to (y-2);
  \draw [->, dashed] (y-2) .. controls +(left:1) and +(up:1) .. (y-3);

  \draw [->] (x-13) to (x-12);
  \draw [->] (x-12) to (x-11);
  \draw [->] (x-11) to (y-1);
  \draw [->] (y-3) to (x-23);
  \draw [->] (x-23) to (x-22);
  \draw [->] (x-22) to (x-21);
  \draw [->] (x-21) to (y-2);
  \draw [->, dashed] (x-33) to (x-32);
  \draw [->, dashed] (x-32) to (x-31);

  \draw [->] (y-1) to (x01);
  \draw [->] (x-11) to (x02);
  \draw [->] (x-12) to (x03);
  \draw [->] (y-2) to (x-11);
  \draw [->] (x-21) to (x-12);
  \draw [->] (x-22) to (x-13);
  \draw [->] (x-31) to (x-22);
  \draw [->] (x-32) to (x-23);
  \draw [->] (x-33) to (y-3);

\draw[yshift=-3.85cm]
  node[below,text width=6cm] 
  {
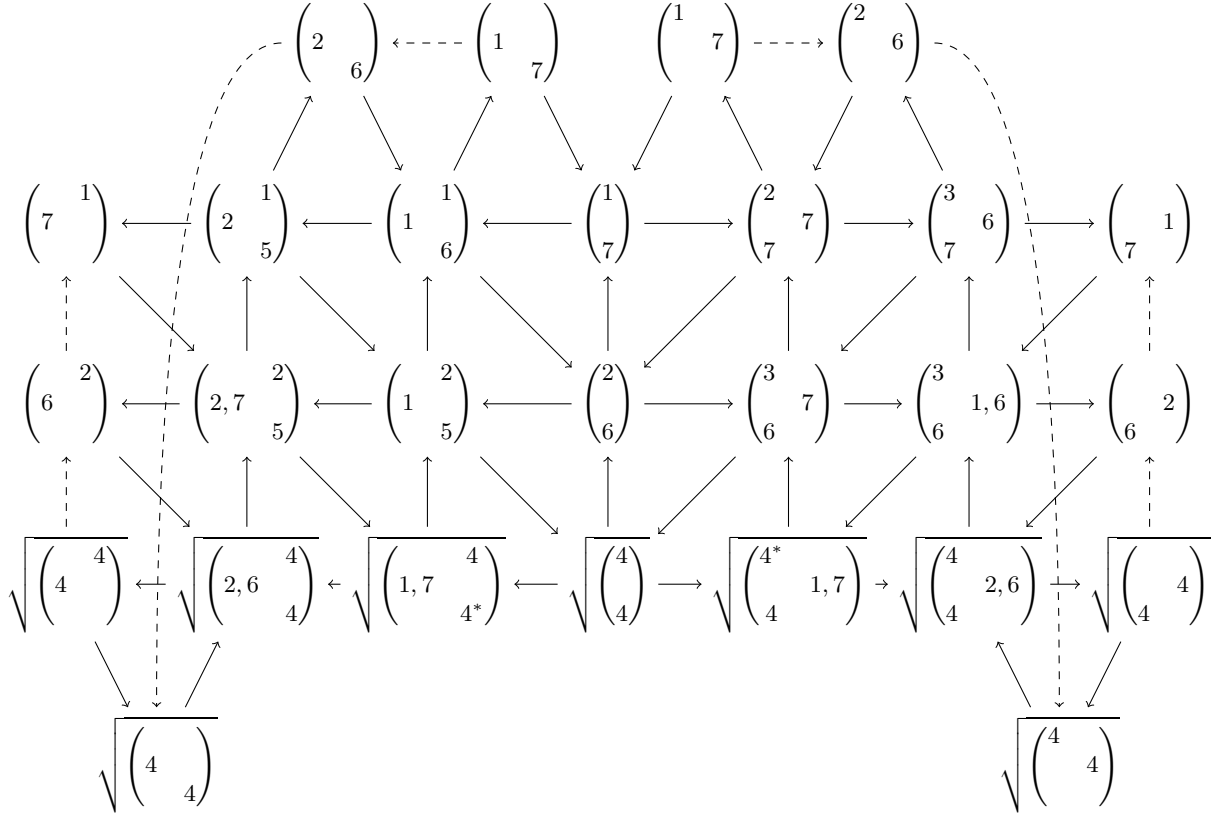
  Figure 39. The functions and quiver for the cluster algebra on $\Conf_4 \A_{Spin_{8}}$.
  };

\end{tikzpicture}
\end{center}

\subsection{Signs and spin representations}

In this section, we discuss in detail the signs involved in defining the cluster variables for $\Conf_m \A_{Spin_{2n+2}}$. This is a somewhat delicate issue, and can certainly be ignored on a first reading. For the remainder of the paper, we would like to ignore sign issues, on the one hand, for the sake of simplicity, and on the other hand, to avoid having to treat the different cases that, as our calculations become more complex, ultimately will depend on the value of $n \mod 4$. Instead of keeping track of all these signs, we will give a framework for computing them.

In defining the various functions we have used, like $\tcfr{a, b}{c}{d}$, we have been careful to define them so that they would be positive functions on  $\Conf_m \A_{SL_{n}}$, $\Conf_m \A_{Sp_{2n}}$,  and $\Conf_m \A_{Spin_{2n+1}}$. Unfortunately, these functions are sometimes positive and sometimes negative on $\Conf_m \A_{Spin_{2n+2}}$. Moreover, as we saw above, the sign of these functions also depends on the parity $n$. As we will see, if we additionally look at the rotations of these functions, the signs will depend on $n \mod 4$. The divergence of all these cases is reflection of Bott periodicity. We will attempt to clarify the situation by isolating the various difficulties and dealing with them separately. 

There are two main difficulties. The first source of complication is that the functions we have defined are well-adapted to computations in $SL_N$. In types $B$ and $C$, the maps $Sp_{2n} \hookrightarrow SL_{2n}$ and $Spin_{2n+1} \rightarrow SL_{2n+1}$ come from folding of the Dynkin diagrams, and therefore preserve positive structures. However, in type $D$, this is not the case. The map $Spin_{2n+2} \rightarrow SL_{2n+2}$ does not preserve positive structures, i.e., the pull-back of a positive function on $\Conf_m \A_{SL_{2n+2}}$ is not necessarily positive on $\Conf_m \A_{Spin_{2n+2}}$.

The second issue is the twisted cyclic shift map, which behaves differently for $SL_{2n+2}$ and $Spin_{2n+2}$. In the following, we will usually let $N=2n+2$, but we will also allow $N$ to be odd for sake of comparison, and to also emphasize the relationship with Bott periodicity.

Let us consider functions of the form $\tcfr{a, b}{c}{d}$ on $\Conf_3 \A_{SL_{N}}$. We defined the functions 
$$\tcfr{c}{d}{a, b} := T^*\tcfr{a, b}{c}{d}$$
$$\tcfr{d}{a, b}{c} := (T^2)^*\tcfr{a, b}{c}{d}$$
where $T$ is the twisted cyclic shift map for $SL_{N}$. This allowed us to define functions in a cluster coming from a rotation of our original cluster.

Suppose we have the functions for a given cluster on $\Conf_3 \A_{Spin_{N}}$. If we want to give the functions in the cluster coming from a rotation of the original cluster, we must pull back by the twisted cyclic shift map for $Spin_N$. However, the elements $s_G$ for $G=Spin_{N}$ and $G=SL_{N}$ are not necessarily the same. In fact, $s_{Spin_{N}} \in Spin_{N}$ is the nontrivial lift of the identity in $SO_{N}$ when 
$$N \equiv 3, 4, 5, 6 \mod 8$$
and $s_{Spin_{N}}$ is the identity when
$$N \equiv 0, 1, 2, 7 \mod 8.$$
Here is one way of determining when $s_{Spin_{N}}$ is the identity. Let $\rho$ be the half-sum of the positive roots, as usual. In the case that $N=2n+2$, $\omega_n$ and $\omega_{n^*}$ are the highest weights of the spin representations, and $s_{Spin_{N}}$ is the identity when $<\omega_n, 2\rho>=<\omega_{n^*}, 2\rho>$ is even and the non-trivial lift of the identity when $<\omega_n, 2\rho>$ is odd. In the case that $N=2n+1$, $\omega_n$ is the highest weight of the spin representation, and $s_{Spin_{N}}$ is the identity when $<\omega_n, 2\rho>$ is even and the non-trivial lift of the identity when $<\omega_n, 2\rho>$ is odd.

On the other hand, $s_{SL_{N}} \in SL_{N}$ is negative of the identity element when $N$ is even and the identity element when $N$ is odd.

(Under the embedding $Sp_{2n} \hookrightarrow SL_{2n}$, $s_{Sp_{2n}}$ is sent to $s_{SL_{2n}}$. Under the map $Spin_{2n+1} \rightarrow SL_{2n+1}$, the element $s_{Spin_{2n+1}}$ is sent to $s_{SL_{2n+1}}$. Note that in the latter case, $s_{Spin_{2n+1}}$ is some lift of the identity in $SO_{2n+1}$, while $s_{SL_{2n+1}}$ is the identity in $SL_{2n+1}$.)

Let us compare what the elements $s_{Spin_{N}}$ and $s_{SL_{N}}$ do on the level of flags. Let $N=2n+2$ be even. Suppose we have a principal affine flag $A \in \A_{Spin_N}$. Then this flag is given an $SL_N$ flag given by a chain of vector spaces 
$$V_1 \subset V_2 \subset \cdots \subset V_N$$
with volume forms on each of these subspaces. Note that the spaces $V_1, \dots, V_{n+1}$ along with their volume forms determine the rest of the flag. This flag can be represented by the sequence of vectors $v_1, v_2, \dots, v_N$, where $V_i$ has volume form $v_1 \wedge v_2 \wedge \cdots \wedge v_i$. Additionally, to make this a $Spin_N$ flag, we choose a ``square-root'' of the form $v_1 \wedge v_2 \wedge \cdots \wedge v_{n+1}$, which is given by choosing a highest weight vector in the spin representation $V_{\omega_n}$. This then forces a choice of a ``square-root'' of the form $v_1 \wedge v_2 \wedge \cdots \wedge v_n \wedge v_{n+2}$, given by a highest weight vector in the  spin representation $V_{\omega_{n^*}}$.

Now $s_{SL_{N}}$ acts by $-1$ on $V$, so 
$$s_{SL_N}(v_1 \wedge v_2 \wedge \cdots \wedge v_i) = (-1)^i v_1 \wedge v_2 \wedge \cdots \wedge v_i.$$
On the other hand, $s_{Spin_N}$ acts by $1$ on $V$, so 
$$s_{Spin_N}(v_1 \wedge v_2 \wedge \cdots \wedge v_i) = v_1 \wedge v_2 \wedge \cdots \wedge v_i.$$
However, in the cases where $s_{Spin_N}$ is a non-trivial element ($N \equiv 4, 6 \mod 8$), it acts by $-1$ on the spin representations.

More generally, the functions on $\A_{Spin_N}$ are naturally isomorphic to the direct sum of its irreducible representations:
$$\bigoplus_{\lambda \in \Lambda_+} V_{\lambda}.$$
The representations come in two types: those that factor through $SO_N$, and those that don't. $s_{Spin_N}$ acts by $1$ on the former and $\pm 1$ on the latter, depending on $N \mod 8$.

For example, if $T$ is the twisted cyclic shift on $\Conf_3 \A_{Spin_{N}}$, then rotation of the function $\tcfr{n-i, 2n+2+i-j}{n+2}{j}$ is given by
$$(T^{-1})^*\tcfr{n-i, 2n+2+i-j}{n+2}{j} = \tcfr{j}{n-i, 2n+2+i-j}{n+2}(-1)^j.$$
As another example, let us apply the twisted cyclic shift to the function 
$$\sqrt{-\tcfr{n-i, n+2+i}{n+1}{n+1}}.$$
Suppose that $N \equiv 4, 6 \mod 8$. Note that
$$(T^{-1})^*\tcfr{n-i, n+2+i}{n+1}{n+1} = (-1)^{n+1}\tcfr{n+1}{n-i, n+2+i}{n+1}.$$
In other words, the function $(T_{-1})^*\tcfr{n-i, n+2+i}{n+1}{n+1}$ is just a strict rotation of the function $\tcfr{n-i, n+2+i}{n+1}{n+1}$. However, 
$$(T^{-1})^*\sqrt{-\tcfr{n-i, n+2+i}{n+1}{n+1}} = -\sqrt{(-1)^{n}\tcfr{n+1}{n-i, n+2+i}{n+1}}.$$
What this means is that to obtain the function $(T^{-1})^*\sqrt{-\tcfr{n-i, n+2+i}{n+1}{n+1}}$, we rotate the function $\sqrt{-\tcfr{n-i, n+2+i}{n+1}{n+1}}$ to get $\sqrt{(-1)^{n}\tcfr{n+1}{n-i, n+2+i}{n+1}}$ and then multiply by $-1$.

The reason for the multiplication by $-1$ is as follows. Note that the function $$\sqrt{-\tcfr{n-i, n+2+i}{n+1}{n+1}}$$ lies in the space
$$[V_{\omega_{n-i}} \otimes V_{\omega_{n}} \otimes V_{\omega_{n}}]^{Spin_N}.$$
Applying the twisted cyclic shift will give a function in the space
$$[V_{\omega_{n}} \otimes V_{\omega_{n-i}} \otimes V_{\omega_{n}}]^{Spin_N}.$$
In the twisted cyclic shift, one factor of $V_{\omega_{n}}$, which is a spin representation, is moved from the third slot to the first slot. The twisted cyclic shift will then act by $-1$ on this factor when $N \equiv 4, 6 \mod 8$.

Similarly, when we want to find the cluster structure corresponding to a transposition, then transposing the arguments only gives the correct function up to a sign. We showed the correct signs for the functions for one of the transpositions of the cluster structure in Figure 38. The negative signs, instead of occuring in every other row, occur here on every other lower diagonal. This pattern persists in general. The signs for other transpositions come from applying the twisted cyclic shift map to this cluster.

Finally, we would like to say something about square-roots. Some of our functions on $\Conf_m \A_{Spin_{N}}$ were defined in terms of square-roots of other functions. This happens when a function lies in an invariant space where two or more of the representations involved is a spin representation. We would like to pick out the correct square-root. In order to do this, we will need a way to describe vectors in these invariant spaces. In the following, we will treat both the cases when $N=2n+1$ is even and when $N=2n$ is odd.

We start by recalling some facts about spin representations. Let $V$ be an $N$-dimensional vector space with a non-degenerate quadratic form $Q(-,-)$. From such data, we can form the Clifford algebra $C(V)$. It is the quotient of the free tensor algebra on $V$ by the relation

$$v \otimes w + w \otimes v = 2Q(v,w).$$

Let $W \subset V$ be as maximal isotropic subspace. It has dimenions $n$. Let $\bigwedge^{\bullet} W$ be the exterior algebra of $W$. $C(V)$ and $\bigwedge\nolimits^{\bullet} W$ have a natural $\Z/2\Z$-grading that comes from considering elements of $V$ to be odd.

Then when $N=2n$ is even, 
$$C(V) \simeq \End \bigwedge\nolimits^{\bullet} W.$$
Moreover, 
$$C^{\textrm{even}}(V) \simeq \End(\bigwedge\nolimits^{\textrm{even}} W) \oplus \End(\bigwedge\nolimits^{\textrm{odd}} W).$$
Let us describe the action of $C(V)$ on $\bigwedge\nolimits^{\bullet} W$. Write $V \simeq W \oplus W'$, where $W'$ is an isotropic subspace complementary to $W$. Then $w \in W$ acts as 
$$w \wedge - : \bigwedge\nolimits^{i} W \rightarrow \bigwedge\nolimits^{i+1} W.$$
$w' \in W'$ maps $\bigwedge\nolimits^{i} W$ to $\bigwedge\nolimits^{i-1} W$. It maps $w \in \bigwedge\nolimits^{1} W$ to $Q(w', w)$, and the action extends to $\bigwedge\nolimits^{i} W$ by using Leibniz and the sign rule (recall that $C(V)$ and $\bigwedge\nolimits^{\bullet} W$ have a natural $\Z/2\Z$-grading.)

When $N=2n+1$ is odd, we can $C(V)$ is isomorphic to two copies of $\End \bigwedge\nolimits^{\bullet} W.$ Let us describe two actions of $C(V)$ on $\bigwedge\nolimits^{\bullet} W$. Write $V \simeq W \oplus U \oplus W'$, where $W'$ is another $n$-dimensional isotropic subspace, and $U$ is a one-dimensional space spanned by $u$ where $Q(u,u)=1$. Then as before, we allow $w \in W$ to act by $w \wedge$, and $w' \in W'$ to act by sending $w \in \bigwedge\nolimits^{1} W$ to $Q(w', w)$, and extending by a signed Leibniz rule. $u$ can act in one of two ways: it can act by $1$ on $\bigwedge\nolimits^{\textrm{even}} W$ and $-1$ on $\bigwedge\nolimits^{\textrm{odd}} W$; or it can act by $-1$ on $\bigwedge\nolimits^{\textrm{even}} W$ and $1$ on $\bigwedge\nolimits^{\textrm{odd}} W$. These two maps of $C(V)$ to $\End \bigwedge\nolimits^{\bullet} W$ realize the isomorphism
$$C(V) \simeq \End(\bigwedge\nolimits^{\bullet} W) \oplus \End(\bigwedge\nolimits^{\bullet} W).$$
Moreover, we have that
$$C^{\textrm{even}}(V) \simeq \End(\bigwedge\nolimits^{\bullet} W).$$

Recall that the lie algebra $\mathfrak{so}_N \simeq \bigwedge\nolimits^2 V$ of $Spin_N$ can be embedded in $C(V)$ as the span of elements of the form $v_1 \cdot v_2 - v_2 \cdot v_1$ in $C(V)$. Then when $N=2n$ is even, we get two representations of $\mathfrak{so}_N \subset C^{\textrm{even}}(V)$: $\bigwedge\nolimits^{\textrm{even}} W$ and $\bigwedge\nolimits^{\textrm{odd}} W$. These are precisely the spin representations of $Spin_N$. When $n$ is even, $\bigwedge\nolimits^{\textrm{even}} W$ has highest weight $\omega_n$ and $\bigwedge\nolimits^{\textrm{odd}} W$ has highest weight $\omega_{n^*}$. If $n$ is odd, $\bigwedge\nolimits^{\textrm{even}} W$ has highest weight $\omega_{n^*}$ and $\bigwedge\nolimits^{\textrm{odd}} W$ has highest weight $\omega_{n}$.

When $N$ is odd, we get one spin representation, given by $\bigwedge\nolimits^{\bullet} W.$

Now we are ready to define invariants in tensor products of representations. First we will define invariants of a tensor product of two spin representations.

Let $N=2n$ be even. Let us write down bases for the spin representations. First choose a basis $e_1, \dots, e_{2n}$ where $Q(e_i, e_{2n+1-i})=(-1)^{i-1}$ for $i \leq n$. Then there is a maximal torus consisting of diagonal elements that preserve $Q$. For each $i$, $1 \leq i \leq n$, there is a cocharacter given by sending $\lambda \in \C^*$ to the map which takes $e_i$ to $\lambda e_i$ and $e_{2n+1-i}$ to $\lambda^{-1} e_{2n+1-i}$ and leaves all other basis elements fixed. These $n$ cocharacters form a basis for the Cartan $\mathfrak{h}$. There is a dual basis  $L_1, \dots, L_n$ of $\mathfrak{h}^*$. We may let $W$ be the span of $e_1, \dots, e_n$. Then if $I = \{i_1, i_2, \dots, i_k\} \subset \{1, 2, \dots, n\}$ where $i_1 < i_2 < \cdots < i_k$, then let 
$$e_I := e_{i_1} \wedge e_{\i_2} \wedge \cdots \wedge e_{i_k}.$$
Then $e_I$ has weight 
$$\omega_I := \frac{1}{2} (\sum_{i \in I} L_i -\sum_{j \notin I} L_j).$$

First suppose $n$ is even. We will then define a pairing $\phi: V_{\omega_n} \times V_{\omega_n} \rightarrow \C$. Then let $e_I, e_J \in \bigwedge\nolimits^{\textrm{even}} W$. If $I$ and $J$ are subsets of $\{1, 2, \dots, n\}$, then we will declare $\phi(e_I,e_J)=0$ unless $I$ and $J$ are complementary subsets of $\{1, 2, \dots, n\}$, in which case
$$\phi(e_I,e_J)=(-1)^{<\omega_n-\omega_I,2\rho>}.$$
Similarly, we have a pairing  $\phi: V_{\omega_{n^*}} \times V_{\omega_{n^*}} \rightarrow \C$. Then let $e_I, e_J \in \bigwedge\nolimits^{\textrm{odd}} W$. If $I$ and $J$ are subsets of $\{1, 2, \dots, n\}$, then we will declare $\phi(e_I,e_J)=0$ unless $I$ and $J$ are complementary subsets of $\{1, 2, \dots, n\}$, in which case
$$\phi(e_I,e_J)=(-1)^{<\omega_{n^*}-\omega_I,2\rho>}.$$

Now suppose $n$ is odd. We can then define a pairing $\phi: V_{\omega_n} \times V_{\omega_{n^*}} \rightarrow \C$. Then let $e_I \bigwedge\nolimits^{\textrm{even}} W$ and $e_J \in \bigwedge\nolimits^{\textrm{odd}} W$. If $I$ and $J$ are subsets of $\{1, 2, \dots, n\}$, then we will declare $\phi(e_I,e_J)=0$ unless $I$ and $J$ are complementary subsets of $\{1, 2, \dots, n\}$, in which case
$$\phi(e_I,e_J)=(-1)^{<\omega_n-\omega_I,2\rho>}.$$
Similarly, we have a pairing  $\phi: V_{\omega_{n^*}} \times V_{\omega_{n}} \rightarrow \C$. Then let $e_I \in \bigwedge\nolimits^{\textrm{odd}} W$ and $e_J \in \bigwedge\nolimits^{\textrm{even}} W$. If $I$ and $J$ are subsets of $\{1, 2, \dots, n\}$, then we will declare $\phi(e_I,e_J)=0$ unless $I$ and $J$ are complementary subsets of $\{1, 2, \dots, n\}$, in which case
$$\phi(e_I,e_J)=(-1)^{<\omega_{n^*}-\omega_I,2\rho>}.$$

Then $\phi$ in all the above cases gives an invariant of the tensor product of two spin representations. This gives correct square-roots for functions of the form $\sqrt{(n; n)}, \sqrt{(n^*; n^*)}, \sqrt{(n; n^*)}$ and $\sqrt{(n^*; n)}$, respectively.

Now let $N=2n+1$ be odd. Here, we may choose a basis $e_1, \dots, e_{2n+1}$ where $Q(e_i, e_{2n+2-i})=(-1)^{i-1}$ for $i \leq n$. Then there is a maximal torus consisting of diagonal elements that preserve $Q$. For each $i$, $1 \leq i \leq n$, there is a cocharacter given by sending $\lambda \in \C^*$ to the map which takes $e_i$ to $\lambda e_i$ and $e_{2n+2-i}$ to $\lambda^{-1} e_{2n+1-i}$ and leaves all other basis elements fixed. These $n$ cocharacters form a basis for the Cartan $\mathfrak{h}$. There is a dual basis  $L_1, \dots, L_n$ of $\mathfrak{h}^*$. We may let $W$ be the span of $e_1, \dots, e_n$. Then if $I = \{i_1, i_2, \dots, i_k\} \subset \{1, 2, \dots, n\}$ where $i_1 < i_2 < \cdots < i_k$, then let 
$$e_I := e_{i_1} \wedge e_{\i_2} \wedge \cdots \wedge e_{i_k}.$$
Then $e_I$ has weight 
$$\omega_I := \frac{1}{2} (\sum_{i \in I} L_i -\sum_{j \notin I} L_j).$$

We can then define a pairing $\phi: V_{\omega_n} \times V_{\omega_n} \rightarrow \C$. If $I$ and $J$ are subsets of $\{1, 2, \dots, n\}$, then we will declare $\phi(e_I,e_J)=0$ unless $I$ and $J$ are complementary subsets of $\{1, 2, \dots, n\}$, in which case
$$\phi(e_I,e_J)=(-1)^{<\omega_n-\omega_I,2\rho>}.$$
This defines the correct square-root for the function $\sqrt{(n; n+1)}$

Finally, let $S_1$ and $S_2$ be two spin representations of $Spin_N$, where $N=2n$ or $2n+1$. Let $k \leq n-1$. We would like to define an invariant, when it exists, in the space
$$[V_{\omega_{k}} \otimes S_1 \otimes S_2]^{Spin_N}.$$
We can do this for  simultaneously for $N$ even or odd. We will do this by constructing a $Spin_N$-invariant map 
$$V_{\omega_{k}} \otimes S_1 \otimes S_2 \rightarrow \C.$$
Note that there is a natural map 
$$\bigwedge\nolimits^k V \rightarrow C(V)$$
given by
$$v_1 \wedge v_2 \wedge \dots \wedge v_k \rightarrow \frac{1}{k!}\sum_{\sigma \in S_n} v_{\sigma(1)} \cdot v_{\sigma(2)} \cdot \dots \cdot v_{\sigma(k)}.$$
When $k$ is even or odd, $\bigwedge\nolimits^k V$ maps to $C^{\textrm{even}}(V)$ or $C^{\textrm{odd}}(V)$, respectively. Thus in the case when $N$ is even, an element of $\bigwedge\nolimits^k V$, viewed inside $C(V)$, gives a map from each spin representation to itself $k$ is even, and gives a map from one spin representation to the other one when $k$ is odd. (When $N$ is odd an element of $\bigwedge\nolimits^k V$ always gives a map from the unique spin representation to itself.) Then if we have $v \otimes s_1 \otimes s_2 \in V_{\omega_{k}} \otimes S_1 \otimes S_2$, we can map this to $\phi{s_1, v \cdot s_2}$ in the instances where $v \cdot s_2$ lies in the spin representation dual to $S_1$ (which depends on $k$, etc.).

The above invariant in $[V_{\omega_{k}} \otimes S_1 \otimes S_2]^{Spin_N}$ gives the correct square-root for the functions $\sqrt{\tcfr{k, N-k}{n}{n+1}}$ when $N$ is odd, and $\sqrt{\pm(k, N-k; n^{(*)}; n^{(*)})}$ when $n$ is even.

(The presence of the superscripts $*$ will depend on the parities of $k$ and $n$, while the sign under the radical will depend on the presence of the superscripts $*$.)

From now on, we will suppress all signs for the sake of simplicity in all our future computations. The analysis above allows the interested reader to supply signs for all the functions that arise in the mutations that follow.

\subsection{Folding}

Let us now discuss how the quiver for $\Conf_3 \A_{Spin_{2n+2}}$ comes from an unfolding for the quiver for $\Conf_3 \A_{Spin_{2n+1}}$. From the description above, it is clear that there is an automorphism $\sigma$ of the quiver for the cluster algebra structure on $\Conf_3 \A_{Spin_{2n+2}}$. Namely, we can define
$$\sigma(x_{in})=x_{in*}$$
$$\sigma(x_{in*})=x_{in}$$
$$\sigma(y_{n})=x_{n*}$$
$$\sigma(y_{n*})=x_{n}$$
and $\sigma$ fixes all other vertices. Then folding the quiver under the automorphism $\sigma$ gives the quiver for $\Conf_3 \A_{Spin_{2n+1}}$.

Let us say this in another way. Let $\sigma$ be the Dynkin diagram automorphism of $D_{n+1}$ having quotient $B_{n}$. Then $\sigma$ induces a map on the root system for $Spin_{2n+2}$, and hence on the fundamental weights and the dominant weights. It also induces an outer automorphism of $Spin_{2n+2}$ having fixed locus $Spin_{2n+1}$, and an involution on the spaces $\Conf_m \A_{Spin_{2n+1}}$. Let $\pi$ be the map from the vertices of $D_{n+1}$ to the vertices of $B_n$. This induces a map $\pi$ sending fundamental weights to corresponding fundamental weights, and therefore projects the weight space for $Spin_{2n+2}$ to the weight space for $Spin_{2n+1}$.

It turns out that the cluster algebra structure on $\Conf_3 \A_{SL_{2n}}$ is preserved by this involution, and that, moreover, the initial seed that we constructed above is preserved by this involution. Folding this seed gives the cluster algebra structure on $\Conf_3 \A_{Spin_{2n+1}}$.

\begin{observation} Let $f$ be a function on $\Conf_3 \A_{Spin_{2n+2}}$ that lies in the invariant space
$$[V_{\lambda} \otimes V_{\mu} \otimes V_{\nu}]^{Spin_{2n+2}}.$$
Then $\sigma^*(f)$ lies in the invariant space
$$[V_{\sigma(\lambda)} \otimes V_{\sigma(\mu)} \otimes V_{\sigma(\nu)}]^{Spin_{2n+2}}.$$
\end{observation}

\begin{observation} Let $f$ be a function on $\Conf_3 \A_{Spin_{2n+2}}$ that lies in the invariant space
$$[V_{\lambda} \otimes V_{\mu} \otimes V_{\nu}]^{SL_{2n}}.$$
Then as a function on $\Conf_3 \A_{Spin_{2n+1}}$, $f$ lies in the invariant space
$$[V_{\pi(\lambda)} \otimes V_{\pi(\mu)} \otimes V_{\pi(\nu)}]^{Spin_{2n+1}}.$$
\end{observation}

\begin{observation} Consider our initial cluster for $\Conf_3 \A_{Spin_{2n+2}}$. Suppose $v$ is a vertex in the quiver for this cluster. Let $f_v$ is the function attached to $v$. Then
$$f_v \in [V_{\lambda} \otimes V_{\mu} \otimes V_{\nu}]^{Spin_{2n+2}},$$
$$f_{\sigma(v)} \in [V_{\sigma(\lambda)} \otimes V_{\sigma(\mu)} \otimes V_{\sigma(\nu)}]^{Spin_{2n+2}}.$$
However, on $\Conf_3 \A_{Spin_{2n+1}}$, $f_v=f_{\sigma(v)}$. This means that we must have $\pi(\lambda)=\pi(\sigma(\lambda))$, $\pi(\mu)=\pi(\sigma(\mu))$, and $\pi(\nu)=\pi(\sigma(\nu))$.
\end{observation}

As we mutate the cluster for $\Conf_3 \A_{Spin_{2n+1}}$, we continue to get clusters that unfold to give clusters for $\Conf_3 \A_{Spin_{2n+2}}$. We will later give sequences of mutations that realize various $S_3$ symmetries for the cluster algebra on $\Conf_3 \A_{Spin_{2n+2}}$, and also the flip on $\Conf_4 \A_{Spin_{2n+2}}$. All these sequences of mutations will just be unfoldings of the analogous sequence of mutations for $\Conf_3 \A_{Spin_{2n+1}}$ or $\Conf_4 \A_{Spin_{2n+1}}$. This gives us the following principle which will underlie the computations of the $S_3$ symmetries on $\Conf_3 \A_{Spin_{2n+2}}$ and the flip on $\Conf_4 \A_{Spin_{2n+2}}$:

\begin{observation} We can compute the formulas for the cluster variables on $\Conf_3 \A_{Spin_{2n+2}}$ and $\Conf_4 \A_{Spin_{2n+2}}$ that appear at various stages of mutation in the following way: Start with the formula for the corresponding cluster variable on $\Conf_m \A_{Spin_{2n+1}}$. Replace every instance of ``$a$'' where $1 \leq a \leq n-1$ by ``$a$,'' and replace every instance of ``$2n+1-a$'' where $1 \leq a \leq n-1$ by ``$2n+2-a$.'' Every instance of ``$n$'' should be replaced by either ``$n$,'' ``$n+1$,'' or ``$n+1^*$,'' depending on the context. One then obtains the formula for the cluster variable on $\Conf_m \A_{Spin_{2n+2}}$.
\end{observation}

All the formulas we derive for $Spin_{2n+2}$ will follow this principle.

Finally, because any two vertices that are identified under the folding of cluster variables (for example $x_{in}$ and $x_{in^*}$) are exchanged under the involution $\sigma$, we have that the formulas for computing the functions attached to these vertices obeys another principle:

\begin{observation} Suppose we have a cluster that is fixed under the involution $\sigma$. (This is the case for our initial cluster and any cluster obtained from the initial one in which whenever we mutate a vertex $v$ we also mutate $\sigma(v)$.) Then if $v$ is a vertex in this cluster, the formula for $f_{\sigma(v)}$ is obtained from the formula for $f_v$ by switching all occurences of $n+1$ and $n+1^*$.
\end{observation}

Finally, let us briefly say something about folding on the level of reduced words in the Weyl group. We will see in the next section that the cluster algebra structure on $\Conf_3 \A_{Spin_{2n+2}}$ comes from the longest word in the Weyl group for $Spin_{2n+2}$. Let the generators for this Weyl group be $s_1, s_2, \dots, s_{n-1}, s_n, s_{n*}$. Then the longest element of the Weyl group is
$$(s'_n s'_{n-1} \dots s'_1)^n \rightarrow (s_n s_{n*} s_{n-1} s_{n-2} \dots s_{1})^n$$

Let $s'_1, s'_2, \dots, s'_n$ be the generators for the Weyl group of $Spin_{2n+1}$. There is an injection from the Weyl group of $Spin_{2n+1}$ to the Weyl group of $Spin_{2n+2}$ that takes
$$s'_n \rightarrow s_n s_{n*},$$
$$s'_i \rightarrow s_i.$$
This map carries the longest element of the Weyl group of $Spin_{2n+1}$ to the longest element of the Weyl group of $Spin_{2n+2}$:
$$(s'_n s'_{n-1} \dots s'_1)^n \rightarrow (s_n s_{n*} s_{n-1} s_{n-2} \dots s_{1})^n$$
Therefore the reduced word for the longest element of the Weyl group of $Spin_{2n+2}$ folds to give the reduced word for the longest element of the Weyl group of $Spin_{2n+1}$, and the folding that gives the cluster structure on $\Conf_3 \A_{Sp_{2n}}$ from the cluster structure on $\Conf_3 \A_{SL_{2n}}$ really takes place on the level of Weyl groups.

\subsection{Reduced words}

We now relate the cluster structure on $\Conf_3 \A_{Spin_{2n+2}}$ given in the previous section to Berenstein, Fomin and Zelevinsky's cluster structure on $B$, the Borel in the group $G$ (\cite{BFZ}). This will allow us to see that the cluster structure described above induces a positive structure on $\A_{G,S}$ identical to the one given in \cite{FG1}. For a reader not interested in the positive structure on $\A_{G,S}$, and more interested in just understanding the cluster structure on $\A_{G,S}$, this section is not logically necessary.

As before, we will restrict our attention to triples of principal flags of the form $(U^-, \overline{w_0}U^-, b\cdot \overline{w_0}U^-).$ Consider the map

$$i: b \in B^- \rightarrow (U^-, \overline{w_0}U^-, b\cdot \overline{w_0}U^-) \in \Conf_3 \A_{Spin_{2n+2}}.$$

Let us recall the constructions of \cite{BFZ}. For $u, v$ elements of the Weyl group $W$ of $G$, we have the double Bruhat cell
$$G^{u,v}=B^+ \cdot u \cdot B^+ \cap B^- \cdot v \cdot B^v.$$
The cell $G^{w_0,e}$ is the on open part of $B^-$.

\begin{prop} The cluster algebra constructed above on $\Conf_3 \A_{Spin_{2n+2}}$, when restricted to the image of $i$, coincides with the cluster algebra structure given in \cite{BFZ} on $B^-=G^{w_0,e}.$
\end{prop}

\begin{proof} $G^{w_0,e}$ is the on open part of $B^-$. Following \cite{BFZ}, to get a cluster structure on this subset, we must choose a reduced-word for $w_0$. In the numbering of the nodes of the Dynkin diagram given above for $Spin_{2n+2}$, we choose the reduced word expression $$w_0=(s_n s_{n^*} s_{n-1} \cdots s_2 s_1)^n.$$
Here our convention is that the above word corresponds to the string $i_1, i_2, \dots, i_{n-1}, i_{n^*}, i_n$ repeated $n$ times. 

Now let $G_0 = U^- H U^+ \subset G$ be the open subset of elements of $G$ having Gaussian decomposition $x=[x]_- [x]_0 [x]_+$. Then for any two elements $u, v \in W$, and any fundamental weight $\omega_i$, we have the {\em generalized minor} $\Delta_{u\omega_i, v\omega_i}(x)$ defined by

$$\Delta_{u\omega_i, v\omega_i}(x) := ([\overline{u}^{-1}x \overline{v}]_0)^{\omega_i}.$$

In our situation, we are interested in such minors when $u, v=e$, or when $v=e$ and $u=u_{ij}=(s_n s_{n^*} s_{n-1} \cdots s_2s_1)^{i-1} s_{n}s_{n^*}s_{n-1}\cdots s_j$ for $1 \leq i \leq n$ and $n \geq j \geq 1$ or $j=n^*$.

Then the cluster functions on $B^-$ given in \cite{BFZ} are $\Delta_{\omega_i, \omega_i}$ for $1 \leq i \leq n$ (these are the functions associated to $u, v=e$), and 
$$\Delta_{u_{ij}\omega_j, \omega_j},$$
which are the functions associated to $v=e$ and $u=u_{ij}=(s_n s_{n-1} \cdots s_2s_1)^i s_{n}s_{n-1}\cdots s_j$. Note that $u_{ij}$ is the subword of $u$ that stops on the $i^{\textrm th}$ iteration of $s_j$.

We have the following claims:

\begin{enumerate}
\item Recall that when $i=0$, we assign the function $\tcfr{2n+2-j}{}{j}$ to $x_{ij}$ for $j<n$,  $\sqrt{\pm(n+1^{(*)}; 0; n+1)}$ to $x_{in}$, and $\sqrt{\pm(n+1^{(*)}; 0; n+1^*)}$ to $x_{in^*}$. Then 
$$\tcfr{2n+2-j}{}{j} = \Delta_{\omega_j, \omega_j},$$
for $j < n$,
$$\sqrt{\pm(n+1^{(*)}; 0; n+1)} = \Delta_{\omega_n, \omega_n}.$$
and
$$\sqrt{\pm(n+1^{(*)}; 0; n+1^*)} = \Delta_{\omega_{n^*}, \omega_{n^*}}.$$

\item Recall that for $i \geq j \neq n$, we assign the function $\tcfr{n-i}{n+2+i-j}{j}(-1)^{n-i}$ to $x_{ij}$. When $i=n$ and $j=n$ or $n^*$, we assign the function $\sqrt{\pm(0; n+1^{(*)}; n+1)}$ to $x_{nn}$ and $\sqrt{\pm(0; n+1^{(*)}; n+1^*)}$ to $x_{nn^*}$. Then 
$$\tcfr{n-i}{n+1+i-j}{j}(-1)^{n-i} = \Delta_{u_{ij}\omega_j, \omega_j}$$ 
for $j \neq n$, while for $j=n$ or $n^*$, 
$$\sqrt{\pm(0; n+1^{(*)}; n+1)} = \Delta_{u_{nn}\omega_n, \omega_n}$$
and
$$\sqrt{\pm(0; n+1^{(*)}; n+1^*)} = \Delta_{u_{nn^*}\omega_{n^*}, \omega_{n^*}}$$

\item Recall that when $i < j <n$, we assign the function $\tcfr{n-i, 2n+2+i-j}{n+2}{j}(-1)^{n-i}$ to $x_{ij}$. Then 
$$\tcfr{n-i, 2n+2+i-j}{n+1}{j} = \Delta_{u_{ij}\omega_j, \omega_j}.$$
When $i <  n$ and $j=n$ or $n^*$, we assign the function $\sqrt{\pm(n-i, n+2+i; n+1^{(*)}; n+1)}$ to $x_{in}$ and the function $\sqrt{\pm(n-i, n+2+i; n+1^{(*)}; n+1^*)}$ to $x_{in}$. Then 
$$\sqrt{\pm(n-i, n+2+i; n+1^{(*)}; n+1)} = \Delta_{u_{in}\omega_n, \omega_n}$$
and
$$\sqrt{\pm(n-i, n+2+i; n+1^{(*)}; n+1^*)} = \Delta_{u_{in^*}\omega_{n^*}, \omega_{n^*}}.$$

\end{enumerate}

Thus, in all cases the function assigned to $x_{ij}$ is precisely $\Delta_{u_{ij}\omega_j, \omega_j}$.

The proof of these claims is a straightforward calculation.

It is convenient to fix maps 
$Spin_{2n+2} \twoheadrightarrow SO_{2n+2} \hookrightarrow SL_{2n+2}.$
Choose the quadratic form so that $<e_i,e_{2n+3-i}>=(-1)^{i-1}$ and all other pairings of basis elements are zero.  Now choose a pinning such that under the map $Spin_{2n+2} \hookrightarrow SL_{2n+2}$,

$$E_{\alpha_n}=E_{n,n+2}+E_{n+1,n+3}, F_{\alpha_n}=E_{n+2,n}+E_{n+3,n+1},$$
$$E_{\alpha_{n^*}}=E_{n,n+1}+E_{n+2,n+3}, F_{\alpha_n}=E_{n+1,n}+E_{n+3,n+2},$$
and for $1\leq i<n$,
$$E_{\alpha_i}=E_{i,i+1}+E_{2n+2-i,2n+3-i},  F_{\alpha_i}=E_{i+1,i}+E_{2n+3-i,2n+2-i},$$
were $E_{i,j}$ is the $(i,j)$-elementary matrix, i.e., the matrix with a $1$ in the $(i,j)$ position and $0$ in all other positions. This map does not preserve positive structures.

With respect to this embedding, we can directly calculate $\Delta_{\omega_j, \omega_j}(x)$ where $x \in B^-$. When $x$ is embedded in $SL_{2n+2}$, $\Delta_{\omega_j, \omega_j}(x)$ is simply the determinant of the minor consisting of the first $j$ rows and the first $j$ columns for $j<n$, $\Delta_{\omega_n+\omega_{n^*}, \omega_n+\omega_{n^*}}(x)$ is the determinant of the minor consisting of the first $n$ rows and columns, $\Delta_{\omega_n, \omega_n}(x)$ is the square-root of the determinant of the minor consisting of the first $n+1$ rows and columns, and 
$$\Delta_{\omega_{n^*}, \omega_{n^*}}(x)=\Delta_{\omega_n+\omega_{n^*}, \omega_n+\omega_{n^*}}(x)/\Delta_{\omega_n, \omega_n}(x).$$
Note that 
$$\Delta_{\omega_{n^*}, \omega_{n^*}}(x)=\sigma^*(\Delta_{\omega_n, \omega_n}(x)).$$

Similarly, one calculates that $\Delta_{u_{ij}\omega_j, \omega_j}(x)$ only depends on the entries in rows $n+i+2, n+i+1, \dots, n+3, 1, 2, \dots n-i$ as well as rows $n+2$ and $n+1$ and the first $n+2$ columns. In particular, for $j< n$, $\Delta_{u_{ij}\omega_j, \omega_j}(x)$ is the determinant of the minor consisting of the first $j$ of those rows (in the order listed above) and the first $j$ columns.  We also have that
$$\Delta_{u_{i,n-1}(\omega_n+\omega_{n^*}), \omega_n+\omega_{n^*}}(x)=\Delta_{u_{in}(\omega_n+\omega_{n^*}), \omega_n+\omega_{n^*}}(x)=\Delta_{u_{in^*}(\omega_n+\omega_{n^*}), \omega_n+\omega_{n^*}}(x)$$
is the determinant of the minor consisting of the $n$ rows $n+i+2, n+i+1, \dots, n+3, 1, 2, \dots n-i$ and the first $n$ columns. 

Meanwhile, $\Delta_{u_{in}\omega_n, \omega_n}(x)$ is calculated by taking the square-root of the minor consisting of the rows $n+i+2, n+i+1, \dots, n+3, 1, 2, \dots n-i$ as well as either row $n+2$ (if $i$ is odd) or $n+1$ (if $i$ is even) and the first $n+1$ columns. $\Delta_{u_{in^*}\omega_{n^*}, \omega_{n^*}}(x)$ can be calculated in one of two ways: using the quotient $\Delta_{u_{in^*}\omega_{n^*}, \omega_{n^*}}(x)=\Delta_{u_{in^*}(\omega_n+\omega_{n^*}), \omega_n+\omega_{n^*}}(x)/\Delta_{u_{in^*}\omega_n, \omega_n}(x)$, or by applying $\sigma$ to $\Delta_{u_{in}\omega_n, \omega_n}(x)$, so that it is the square-root of the minor consisting of the rows $n+i+2, n+i+1, \dots, n+3, 1, 2, \dots n-i$ as well as either row $n+2$ (if $i$ is even) or $n+1$ (if $i$ is odd) and the columns $1, 2, \dots, n, n+2$.

We then must calculate the functions $\tcfr{2n+2-j}{}{j}, \tcfr{n-i}{n+2+i-j}{j}, \tcfr{n-i, 2n+2+i-j}{n+2}{j}$ and $\sqrt{\pm(n-i, n+2+i; n+1^{(*)}; n+1^*)}$ on the triple of flags $(U^-, \overline{w_0}U^-, b\cdot \overline{w_0}U^-)$. Under the embedding  $Spin_{2n+2} \hookrightarrow SL_{2n+2}$, we should choose the flag $U^-$ to be be $e_{2n+2}, -e_{2n+1}, e_{2n}, dots, (-1)^{n} e_{n+2}$, so that $\overline{w_0}U^-$ is given by the flag $e_1, e_2, \dots, e_{n+1}$. Direct calculation then shows that $\tcfr{2n+2-j}{}{j}$, $\tcfr{n-i}{n+2+i-j}{j}$, $\tcfr{n-i, 2n+2+i-j}{n+2}{j}$ and $\sqrt{\pm(n-i, n+2+i; n+1^{(*)}; n+1^*)}$ are given by the appropriate determinants of minors or square-roots of these determinants of minors.

\end{proof}

Finally, note that we have the following equalities of functions:

\begin{equation} \label{dualities3}
\begin{split}
\dud{k}{2n+2-k}&=\dud{2n+2-k}{k} \\
\tcfr{n-i}{n+2+i-j}{j}&=\tcfr{n+2+i}{n-i+j}{2n+2-j} \\
\tcfr{n-i, 2n+2+i-j}{n+2}{j}&=\tcfr{n+2+i, j-i}{n}{2n+2- j}
\end{split}
\end{equation}
These equalities are valid up to sign (see the previous section). They arise because the quadratic form induces an isomorphism between $\bigwedge\nolimits^i V$ and $\bigwedge\nolimits^{2n+2-i} V$.

\subsubsection{The first transposition}

Let $(A,B,C) \in \Conf_3 \A_{Spin_{2n+2}}$ be a triple of flags. The sequence of mutations that realizes that $S_3$ symmetry $(A,B,C) \rightarrow (A,C,B)$ is the same as \eqref{23Sp} (which gave the $S_3$ symmetry for $G=Sp_{2n}$ and $Spin_{2n+1}$):

\begin{equation}
\begin{gathered}
x_{11}, x_{21}, x_{22}, x_{12}, x_{13}, x_{23}, x_{33}, x_{32}, x_{33,} \dots, x_{1,n-1}, \dots, x_{n-1, n-1}, \dots, x_{n-1, 1}, \\
x_{11}, x_{21}, x_{22}, x_{12},  \dots x_{1,n-2}, \dots, x_{n-2, n-2}, \dots, x_{n-2, 1}, \\
\dots, \\
x_{11}, x_{21}, x_{22}, x_{12}, \\
x_{11} \\
\end{gathered}
\end{equation}

The sequence can be thought of as follows: At any step of the process, we mutate all $x_{ij}$ such that $\max(i,j)$ is constant. It will not matter in which order we mutate these $x_{ij}$ because the vertices we mutate have no arrows between them. So we first mutate the $x_{ij}$ such that $\max(i,j)=1$, then the $x_{ij}$ such that $\max(i,j)=2$, then the $x_{ij}$ such that $\max(i,j)=3$, etc. The sequence of maximums that we use is 
$$1, 2, 3, \dots n-1, 1, 2, \dots n-2 \dots, 1, 2, 3, 1, 2, 1.$$

The evolution of the quiver for $\Conf_3 \A_{Spin_{2n+2}}$ is just as in the case for $\Conf_3 \A_{Sp_{2n}}$, as pictured in Figures 13 and 14, with the only difference being that each white vertex is replaced by two black vertices. 

In Figure 40, we depict how the quiver for $\Conf_3 \A_{Spin_10}$ changes after performing the sequence of mutations of $x_{ij}$ having maximums $1$; $1, 2$; and $1, 2, 3$.

\begin{center}
\begin{tikzpicture}[scale=2.2]
\begin{scope}[xshift=-1.5cm]

  \node (x01) at (1,0) {$\dud{3}{7}$};
  \node (x02) at (2,0) {$\tcfr{8}{}{2}$};
  \node (x03) at (3,0) {$\tcfr{7}{}{3}$};
  \node (x04) at (4,0) {$\sqrt{\tcfr{5^*}{}{5}}$};
  \node (x11) at (1,-1) {$\boldsymbol{\tcfr{2,9}{7}{2}}$};
  \node (x12) at (2,-1) {$\tcfr{3, 9}{6}{2}$};
  \node (x13) at (3,-1) {$\tcfr{3, 8}{6}{3}$};
  \node (x14) at (4,-1) {$\sqrt{\tcfr{3, 7}{5}{5}}$};
  \node (x21) at (1,-2) {$\tcfr{2}{7}{1}$};
  \node (x22) at (2,-2) {$\tcfr{2}{6}{2}$};
  \node (x23) at (3,-2) {$\tcfr{2, 9}{6}{3}$};
  \node (x24) at (4,-2) {$\sqrt{\tcfr{2, 8}{5^*}{5}}$};
  \node (x31) at (1,-3) {$\tcfr{1}{8}{1}$};
  \node (x32) at (2,-3) {$\tcfr{1}{7}{2}$};
  \node (x33) at (3,-3) {$\tcfr{1}{6}{3}$};
  \node (x34) at (4,-3) {$\sqrt{\tcfr{1, 9}{5}{5}}$};
  \node (x41) at (1,-4) {$\tcfr{}{9}{1}$};
  \node (x42) at (2,-4) {$\tcfr{}{8}{2}$};
  \node (x43) at (3,-4) {$\tcfr{}{7}{3}$};
  \node (x44) at (4,-4) {$\sqrt{\tcfr{}{5^*}{5}}$};

  \node (y1) at (0,-3) {$\dud{1}{9}$};
  \node (y2) at (0,-2) {$\dud{2}{8}$};
  \node (y3) at (0,-1) {$\tcfr{9}{}{1}$};
  \node (y4) at (5,-1) {$\sqrt{\dud{5^*}{5}}$};

  \draw [->] (x41) to (x31);
  \draw [->] (x31) to (x21);
  \draw [->] (x11) to (x21);
  \draw [->] (x01) to (x11);
  \draw [->] (x42) to (x32);
  \draw [->] (x32) to (x22);
  \draw [->] (x12) to (x02);
  \draw [->] (x43) to (x33);
  \draw [->] (x33) to (x23);
  \draw [->] (x23) to (x13);
  \draw [->] (x13) to (x03);
  \draw [->] (x44) to (x34);
  \draw [->] (x34) to (x24);
  \draw [->] (x24) to (x14);
  \draw [->] (x14) to (x04);
  \draw [->, dashed] (x01) .. controls +(45:1) and +(up:2) .. (y4);
  \draw [->, dashed] (y2) .. controls +(135:1) and +(left:2) .. (x01);
  \draw [->, dashed] (y1) to (y2);

  \draw [->, dashed] (x04) to (x03);
  \draw [->, dashed] (x03) to (x02);
  \draw [->, dashed] (x02) .. controls +(135:1) and +(up:1.5) .. (y3);
  \draw [->] (y4) to (x14);
  \draw [->] (x14) to (x13);
  \draw [->] (x13) to (x12);
  \draw [->] (x11) to (x12);
  \draw [->] (y3) to (x11);
  \draw [->] (x24) to (x23);
  \draw [->] (x23) to (x22);
  \draw [->] (x21) to (y2);
  \draw [->] (x34) to (x33);
  \draw [->] (x33) to (x32);
  \draw [->] (x32) to (x31);
  \draw [->] (x31) to (y1);
  \draw [->, dashed] (x44) to (x43);
  \draw [->, dashed] (x43) to (x42);
  \draw [->, dashed] (x42) to (x41);

  \draw [->] (x12) to (x01);
  \draw [->] (x02) to (x13);
  \draw [->] (x03) to (x14);
  \draw [->] (x04) to (y4);
  \draw [->] (x21) to (y3);
  \draw [->] (x22) to (x11);
  \draw [->] (x12) to (x23);
  \draw [->] (x13) to (x24);
  \draw [->] (y2) to (x31);
  \draw [->] (x21) to (x32);
  \draw [->] (x22) to (x33);
  \draw [->] (x23) to (x34);
  \draw [->] (y1) to (x41);
  \draw [->] (x31) to (x42);
  \draw [->] (x32) to (x43);
  \draw [->] (x33) to (x44);

\end{scope}

\draw[yshift=-5cm,xshift=1cm]
  node[below,text width=6cm] 
  {
  Figure 40a The quiver and the functions for $\Conf_3 \A_{Spin_{10}}$ after performing the sequence of mutations of $x_{ij}$ having maximums $1$.
  };

\end{tikzpicture}
\end{center}

\begin{center}
\begin{tikzpicture}[scale=2.2]

\begin{scope}[xshift=-1.5cm]

  \node (x01) at (1,0) {$\tcfr{8}{}{2}$};
  \node (x02) at (2,0) {$\dud{3}{7}$};
  \node (x03) at (3,0) {$\tcfr{7}{}{3}$};
  \node (x04) at (4,0) {$\sqrt{\tcfr{5^*}{}{5}}$};
  \node (x11) at (1,-1) {${\tcfr{2,9}{7}{2}}$};
  \node (x12) at (2,-1) {$\boldsymbol{\tcfr{2, 8}{7}{3}}$};
  \node (x13) at (3,-1) {$\tcfr{3, 8}{6}{3}$};
  \node (x14) at (4,-1) {$\sqrt{\tcfr{3, 7}{5}{5}}$};
  \node (x21) at (1,-2) {$\boldsymbol{\tcfr{1,9}{8}{2}}$};
  \node (x22) at (2,-2) {$\boldsymbol{\tcfr{1,9}{7}{3}}$};
  \node (x23) at (3,-2) {$\tcfr{2, 9}{6}{3}$};
  \node (x24) at (4,-2) {$\sqrt{\tcfr{2, 8}{5^*}{5}}$};
  \node (x31) at (1,-3) {$\tcfr{1}{8}{1}$};
  \node (x32) at (2,-3) {$\tcfr{1}{7}{2}$};
  \node (x33) at (3,-3) {$\tcfr{1}{6}{3}$};
  \node (x34) at (4,-3) {$\sqrt{\tcfr{1, 9}{5}{5}}$};
  \node (x41) at (1,-4) {$\tcfr{}{9}{1}$};
  \node (x42) at (2,-4) {$\tcfr{}{8}{2}$};
  \node (x43) at (3,-4) {$\tcfr{}{7}{3}$};
  \node (x44) at (4,-4) {$\sqrt{\tcfr{}{5^*}{5}}$};

  \node (y1) at (0,-3) {$\dud{1}{9}$};
  \node (y2) at (0,-2) {$\tcfr{9}{}{1}$};
  \node (y3) at (0,-1) {$\dud{2}{8}$};
  \node (y4) at (5,-1) {$\sqrt{\dud{5^*}{5}}$};

  \draw [->] (x41) to (x31);
  \draw [->] (x21) to (x31);
  \draw [->] (x21) to (x11);
  \draw [->] (x11) to (x01);
  \draw [->] (x42) to (x32);
  \draw [->] (x22) to (x32);
  \draw [->] (x02) to (x12);
  \draw [->] (x43) to (x33);
  \draw [->] (x13) to (x03);
  \draw [->] (x44) to (x34);
  \draw [->] (x34) to (x24);
  \draw [->] (x24) to (x14);
  \draw [->] (x14) to (x04);
  \draw [->, dashed] (x02) .. controls +(up:1) and +(up:2) .. (y4);
  \draw [->, dashed] (y3) .. controls +(up:2) and +(up:1) .. (x02);
  \draw [->, dashed] (y1) .. controls +(left:1) and +(left:1) .. (y3);

  \draw [->, dashed] (x04) to (x03);
  \draw [->, dashed] (x03) .. controls +(up:1) and +(up:1) .. (x01);
  \draw [->, dashed] (x01)  .. controls +(left:2) and +(left:1) .. (y2);
  \draw [->] (y4) to (x14);
  \draw [->] (x14) to (x13);
  \draw [->] (x12) to (x13);
  \draw [->] (x12) to (x11);
  \draw [->] (x11) to (y3);
  \draw [->] (x24) to (x23);
  \draw [->] (x22) to (x23);
  \draw [->] (y2) to (x21);
  \draw [->] (x34) to (x33);
  \draw [->] (x31) to (y1);
  \draw [->, dashed] (x44) to (x43);
  \draw [->, dashed] (x43) to (x42);
  \draw [->, dashed] (x42) to (x41);

  \draw [->] (x01) to (x12);
  \draw [->] (x13) to (x02);
  \draw [->] (x03) to (x14);
  \draw [->] (x04) to (y4);
  \draw [->] (y3) to (x21);
  \draw [->] (x11) to (x22);
  \draw [->] (x23) to (x12);
  \draw [->] (x13) to (x24);
  \draw [->] (x31) to (y2);
  \draw [->] (x32) to (x21);
  \draw [->] (x33) to (x22);
  \draw [->] (x23) to (x34);
  \draw [->] (y1) to (x41);
  \draw [->] (x31) to (x42);
  \draw [->] (x32) to (x43);
  \draw [->] (x33) to (x44);

\end{scope}

\draw[yshift=-5cm,xshift=1cm]
  node[below,text width=6cm] 
  {
  Figure 40b The quiver and the functions for $\Conf_3 \A_{Spin_{10}}$ after performing the sequence of mutations of $x_{ij}$ having maximums $1, 2$.
  };

\end{tikzpicture}
\end{center}

\begin{center}
\begin{tikzpicture}[scale=2.2]

\begin{scope}[xshift=-1.5cm]

  \node (x01) at (1,0) {$\tcfr{8}{}{2}$};
  \node (x02) at (2,0) {$\tcfr{7}{}{3}$};
  \node (x03) at (3,0) {$\dud{3}{7}$};
  \node (x04) at (4,0) {$\sqrt{\dud{5^*}{5}}$};
  \node (x11) at (1,-1) {${\tcfr{2,9}{7}{2}}$};
  \node (x12) at (2,-1) {${\tcfr{2, 8}{7}{3}}$};
  \node (x13) at (3,-1) {$\boldsymbol{\tcfr{2, 7}{7}{4}}$};
  \node (x14) at (4,-1) {$\sqrt{\tcfr{3, 7}{5}{5}}$};
  \node (x21) at (1,-2) {${\tcfr{1,9}{8}{2}}$};
  \node (x22) at (2,-2) {${\tcfr{1,9}{7}{3}}$};
  \node (x23) at (3,-2) {$\boldsymbol{\tcfr{1,8}{7}{4}}$};
  \node (x24) at (4,-2) {$\sqrt{\tcfr{2, 8}{5}{5^*}}$};
  \node (x31) at (1,-3) {$\boldsymbol{\tcfr{9}{9}{2}}$};
  \node (x32) at (2,-3) {$\boldsymbol{\tcfr{9}{8}{3}}$};
  \node (x33) at (3,-3) {$\boldsymbol{\tcfr{9}{7}{4}}$};
  \node (x34) at (4,-3) {$\sqrt{\tcfr{1, 9}{5}{5}}$};
  \node (x41) at (1,-4) {$\tcfr{}{9}{1}$};
  \node (x42) at (2,-4) {$\tcfr{}{8}{2}$};
  \node (x43) at (3,-4) {$\tcfr{}{7}{3}$};
  \node (x44) at (4,-4) {$\sqrt{\tcfr{}{5}{5^*}}$};

  \node (y1) at (0,-3) {$\tcfr{9}{}{1}$};
  \node (y2) at (0,-2) {$\dud{1}{9}$};
  \node (y3) at (0,-1) {$\dud{2}{8}$};
  \node (y4) at (5,-1) {$\sqrt{\tcfr{5^*}{}{5}}$};

  \draw [->] (x31) to (x41);
  \draw [->] (x31) to (x21);
  \draw [->] (x21) to (x11);
  \draw [->] (x11) to (x01);
  \draw [->] (x32) to (x42);
  \draw [->] (x32) to (x22);
  \draw [->] (x22) to (x12);
  \draw [->] (x12) to (x02);
  \draw [->] (x33) to (x43);
  \draw [->] (x03) to (x13);
  \draw [->] (x34) to (x44);
  \draw [->] (x24) to (x34);
  \draw [->] (x14) to (x24);
  \draw [->] (x04) to (x14);
  \draw [->, dashed] (x03) to (x04);
  \draw [->, dashed] (y3) .. controls +(up:2) and +(up:1) .. (x03);
  \draw [->, dashed] (y2) to (y3);

  \draw [->, dashed] (y4) .. controls +(up:2) and +(up:1) .. (x02);
  \draw [->, dashed] (x02) to (x01);
  \draw [->, dashed] (x01) .. controls +(left:2) and +(left:1) .. (y1);
  \draw [->] (x14) to (y4);
  \draw [->] (x13) to (x14);
  \draw [->] (x13) to (x12);
  \draw [->] (x12) to (x11);
  \draw [->] (x11) to (y3);
  \draw [->] (x23) to (x24);
  \draw [->] (x23) to (x22);
  \draw [->] (x22) to (x21);
  \draw [->] (x21) to (y2);
  \draw [->] (x33) to (x34);
  \draw [->] (y1) to (x31);
  \draw [->, dashed] (x43) to (x44);
  \draw [->, dashed] (x42) to (x43);
  \draw [->, dashed] (x41) to (x42);

  \draw [->] (x01) to (x12);
  \draw [->] (x02) to (x13);
  \draw [->] (x14) to (x03);
  \draw [->] (x04) to (y4);
  \draw [->] (y3) to (x21);
  \draw [->] (x11) to (x22);
  \draw [->] (x12) to (x23);
  \draw [->] (x24) to (x13);
  \draw [->] (y2) to (x31);
  \draw [->] (x21) to (x32);
  \draw [->] (x22) to (x33);
  \draw [->] (x34) to (x23);
  \draw [->] (x41) to (y1);
  \draw [->] (x42) to (x31);
  \draw [->] (x43) to (x32);
  \draw [->] (x44) to (x33);

\end{scope}

\draw[yshift=-5cm,xshift=1cm]
  node[below,text width=6cm] 
  {
  Figure 40c The quiver and the functions for $\Conf_3 \A_{Spin_{10}}$ after performing the sequence of mutations of $x_{ij}$ having maximums $1,  2, 3$.
  };

\end{tikzpicture}
\end{center}

Note that in Figure 40c, we have changed which of the doubled vertices shown. This is because mutation of vertices adjacent to the doubled vertices changes how they are connected to each other. It is not difficult to see that we end up with the functions and quiver as depicted.

In Figure 41, we depict the state of the quiver after performing the sequence of mutations of $x_{ij}$ having maximums $1, 2, 3$; $1, 2, 3, 1, 2$; and $1, 2, 3, 1, 2, 1$.

\begin{center}
\begin{tikzpicture}[scale=2.2]

\begin{scope}[xshift=-1.5cm]

  \node (x01) at (1,0) {$\tcfr{7}{}{3}$};
  \node (x02) at (2,0) {$\dud{2}{8}$};
  \node (x03) at (3,0) {$\dud{3}{7}$};
  \node (x04) at (4,0) {$\sqrt{\dud{5^*}{5}}$};
  \node (x11) at (1,-1) {$\boldsymbol{\tcfr{1,8}{8}{3}}$};
  \node (x12) at (2,-1) {$\boldsymbol{\tcfr{1,7}{8}{4}}$};
  \node (x13) at (3,-1) {${\tcfr{2, 7}{7}{4}}$};
  \node (x14) at (4,-1) {$\sqrt{\tcfr{3, 7}{5}{5}}$};
  \node (x21) at (1,-2) {$\boldsymbol{\tcfr{8}{9}{3}}$};
  \node (x22) at (2,-2) {$\boldsymbol{\tcfr{8}{8}{4}}$};
  \node (x23) at (3,-2) {${\tcfr{1,8}{7}{4}}$};
  \node (x24) at (4,-2) {$\sqrt{\tcfr{2, 8}{5}{5^*}}$};
  \node (x31) at (1,-3) {${\tcfr{9}{9}{2}}$};
  \node (x32) at (2,-3) {${\tcfr{9}{8}{3}}$};
  \node (x33) at (3,-3) {${\tcfr{9}{7}{4}}$};
  \node (x34) at (4,-3) {$\sqrt{\tcfr{1, 9}{5}{5}}$};
  \node (x41) at (1,-4) {$\tcfr{}{9}{1}$};
  \node (x42) at (2,-4) {$\tcfr{}{8}{2}$};
  \node (x43) at (3,-4) {$\tcfr{}{7}{3}$};
  \node (x44) at (4,-4) {$\sqrt{\tcfr{}{5}{5^*}}$};

  \node (y1) at (0,-3) {$\tcfr{9}{}{1}$};
  \node (y2) at (0,-2) {$\tcfr{8}{}{2}$};
  \node (y3) at (0,-1) {$\dud{1}{9}$};
  \node (y4) at (5,-1) {$\sqrt{\tcfr{5^*}{}{5}}$};

  \draw [->] (x11) to (x01);
  \draw [->] (x21) to (x11);
  \draw [->] (x21) to (x31);
  \draw [->] (x31) to (x41);
  \draw [->] (x02) to (x12);
  \draw [->] (x22) to (x32);
  \draw [->] (x32) to (x42);
  \draw [->] (x03) to (x13);
  \draw [->] (x13) to (x23);
  \draw [->] (x23) to (x33);
  \draw [->] (x33) to (x43);
  \draw [->] (x04) to (x14);
  \draw [->] (x14) to (x24);
  \draw [->] (x24) to (x34);
  \draw [->] (x34) to (x44);
  \draw [->, dashed] (y4) .. controls +(up:2) and +(up:1) .. (x01);
  \draw [->, dashed] (x01) .. controls +(left:2) and +(left:1) .. (y2);
  \draw [->, dashed] (y2) to (y1);

  \draw [->, dashed] (y3) .. controls +(up:2) and +(up:1) .. (x02);
  \draw [->, dashed] (x02) to (x03);
  \draw [->, dashed] (x03) to (x04);
  \draw [->] (x11) to (y3);
  \draw [->] (x12) to (x11);
  \draw [->] (x12) to (x13);
  \draw [->] (x13) to (x14);
  \draw [->] (x14) to (y4);
  \draw [->] (y2) to (x21);
  \draw [->] (x22) to (x23);
  \draw [->] (x23) to (x24);
  \draw [->] (y1) to (x31);
  \draw [->] (x31) to (x32);
  \draw [->] (x32) to (x33);
  \draw [->] (x33) to (x34);
  \draw [->, dashed] (x41) to (x42);
  \draw [->, dashed] (x42) to (x43);
  \draw [->, dashed] (x43) to (x44);

  \draw [->] (x01) to (x12);
  \draw [->] (x13) to (x02);
  \draw [->] (x14) to (x03);
  \draw [->] (y4) to (x04);
  \draw [->] (y3) to (x21);
  \draw [->] (x11) to (x22);
  \draw [->] (x23) to (x12);
  \draw [->] (x24) to (x13);
  \draw [->] (x31) to (y2);
  \draw [->] (x32) to (x21);
  \draw [->] (x33) to (x22);
  \draw [->] (x34) to (x23);
  \draw [->] (x41) to (y1);
  \draw [->] (x42) to (x31);
  \draw [->] (x43) to (x32);
  \draw [->] (x44) to (x33);

\end{scope}

\draw[yshift=-5cm,xshift=1cm]
  node[below,text width=6cm] 
  {
  Figure 41a The quiver and the functions for $\Conf_3 \A_{Spin_{10}}$ after performing the sequence of mutations of $x_{ij}$ having maximums $1,  2, 3, 1, 2$.
  };

\end{tikzpicture}
\end{center}

\begin{center}
\begin{tikzpicture}[scale=2.2]

\begin{scope}[xshift=-1.5cm]

  \node (x01) at (1,0) {$\dud{1}{9}$};
  \node (x02) at (2,0) {$\dud{2}{8}$};
  \node (x03) at (3,0) {$\dud{3}{7}$};
  \node (x04) at (4,0) {$\sqrt{\dud{5^*}{5}}$};
  \node (x11) at (1,-1) {$\boldsymbol{\tcfr{7}{9}{4}}$};
  \node (x12) at (2,-1) {${\tcfr{1,7}{8}{4}}$};
  \node (x13) at (3,-1) {${\tcfr{2, 7}{7}{4}}$};
  \node (x14) at (4,-1) {$\sqrt{\tcfr{3, 7}{5}{5}}$};
  \node (x21) at (1,-2) {${\tcfr{8}{9}{3}}$};
  \node (x22) at (2,-2) {${\tcfr{8}{8}{4}}$};
  \node (x23) at (3,-2) {${\tcfr{1,8}{7}{4}}$};
  \node (x24) at (4,-2) {$\sqrt{\tcfr{2, 8}{5}{5^*}}$};
  \node (x31) at (1,-3) {${\tcfr{9}{9}{2}}$};
  \node (x32) at (2,-3) {${\tcfr{9}{8}{3}}$};
  \node (x33) at (3,-3) {${\tcfr{9}{7}{4}}$};
  \node (x34) at (4,-3) {$\sqrt{\tcfr{1, 9}{5}{5}}$};
  \node (x41) at (1,-4) {$\tcfr{}{9}{1}$};
  \node (x42) at (2,-4) {$\tcfr{}{8}{2}$};
  \node (x43) at (3,-4) {$\tcfr{}{7}{3}$};
  \node (x44) at (4,-4) {$\sqrt{\tcfr{}{5}{5^*}}$};

  \node (y1) at (0,-3) {$\tcfr{9}{}{1}$};
  \node (y2) at (0,-2) {$\tcfr{8}{}{2}$};
  \node (y3) at (0,-1) {$\tcfr{7}{}{3}$};
  \node (y4) at (5,-1) {$\sqrt{\tcfr{5^*}{}{5}}$};

  \draw [->] (x01) to (x11);
  \draw [->] (x11) to (x21);
  \draw [->] (x21) to (x31);
  \draw [->] (x31) to (x41);
  \draw [->] (x02) to (x12);
  \draw [->] (x12) to (x22);
  \draw [->] (x22) to (x32);
  \draw [->] (x32) to (x42);
  \draw [->] (x03) to (x13);
  \draw [->] (x13) to (x23);
  \draw [->] (x23) to (x33);
  \draw [->] (x33) to (x43);
  \draw [->] (x04) to (x14);
  \draw [->] (x14) to (x24);
  \draw [->] (x24) to (x34);
  \draw [->] (x34) to (x44);
  \draw [->, dashed] (y4) .. controls +(up:2) and +(up:2) .. (y3);
  \draw [->, dashed] (y3) to (y2);
  \draw [->, dashed] (y2) to (y1);

  \draw [->, dashed] (x01) to (x02);
  \draw [->, dashed] (x02) to (x03);
  \draw [->, dashed] (x03) to (x04);
  \draw [->] (y3) to (x11);
  \draw [->] (x11) to (x12);
  \draw [->] (x12) to (x13);
  \draw [->] (x13) to (x14);
  \draw [->] (x14) to (y4);
  \draw [->] (y2) to (x21);
  \draw [->] (x21) to (x22);
  \draw [->] (x22) to (x23);
  \draw [->] (x23) to (x24);
  \draw [->] (y1) to (x31);
  \draw [->] (x31) to (x32);
  \draw [->] (x32) to (x33);
  \draw [->] (x33) to (x34);
  \draw [->, dashed] (x41) to (x42);
  \draw [->, dashed] (x42) to (x43);
  \draw [->, dashed] (x43) to (x44);

  \draw [->] (x12) to (x01);
  \draw [->] (x13) to (x02);
  \draw [->] (x14) to (x03);
  \draw [->] (y4) to (x04);
  \draw [->] (x21) to (y3);
  \draw [->] (x22) to (x11);
  \draw [->] (x23) to (x12);
  \draw [->] (x24) to (x13);
  \draw [->] (x31) to (y2);
  \draw [->] (x32) to (x21);
  \draw [->] (x33) to (x22);
  \draw [->] (x34) to (x23);
  \draw [->] (x41) to (y1);
  \draw [->] (x42) to (x31);
  \draw [->] (x43) to (x32);
  \draw [->] (x44) to (x33);

\end{scope}

\draw[yshift=-5cm,xshift=1cm]
  node[below,text width=6cm] 
  {
  Figure 41b The quiver and the functions for $\Conf_3 \A_{Spin_{10}}$ after performing the sequence of mutations of $x_{ij}$ having maximums $1, 2, 3, 1, 2, 1$.
  };

\end{tikzpicture}
\end{center}

From these diagrams the various quivers in the general case of $\Conf_3 \A_{Spin_{2n+2}}$ should be clear.

\begin{theorem}
If $\max(i,j)=k$, then $x_{ij}$ is mutated a total of $n-k$ times.

Recall that when $i \geq j$, we assign the function $\tcfr{n-i}{n+2+i-j}{j}$ to $x_{ij}$. Thus the function attached to $x_{ij}$ transforms as follows:
$$\tcfr{n-i}{n+2+i-j}{j} \rightarrow \tcfr{2n+1, n-i-1}{n+i-j+3}{j+1} \rightarrow \tcfr{2n, n-i-2}{n+i-j+4}{j+2} \rightarrow \dots$$ 
$$\rightarrow \tcfr{n+i+3, 1}{2n+1-j}{n-i+j-1} \rightarrow \tcfr{n+i+2}{2n+2-j}{n-i+j}=\tcfr{n-i}{j}{n+2+i-j}$$

When $i < j$ and $i \neq 0$, we assign the function $\tcfr{n-i, 2n+2+i-j}{n+2}{j}$ to $x_{ij}$. Thus the function attached to $x_{ij}$ transforms as follows:
$$\tcfr{n-i, 2n+2+i-j}{n+2}{j} \rightarrow \tcfr{n-i-1, 2n+1+i-j}{n+3}{j+1} \rightarrow \tcfr{n-i-2, 2n+i-j}{n+4}{j+2} \rightarrow \dots$$
$$\rightarrow \tcfr{j-i+1, n+i+3}{2n+1-j}{n-1} \rightarrow \tcfr{j-i, n+i+2}{2n+2-j}{n}=\tcfr{n-i, 2n+2+i-j}{j}{n+2}$$

\end{theorem}

\begin{proof}

The proof is identical to the case where $G=Sp_{2n}$. Note that we do note mutate any of the vertices that have square-roots. However, the functions with square-roots are involved in the mutations, as we mutate vertices adjacent to them, namely, the vertices $x_{i,n-1}$.

We then need to use that

\begin{equation}\label{products}
\begin{gathered}
\sqrt{\tcfr{n-i, n+2+i}{n+1^*}{n+1^*}} \cdot \sqrt{\tcfr{n-i, n+2+i}{n+1}{n+1}} = \tcfr{n-i, n+2+i}{n}{n+2}, \\
\sqrt{\tcfr{n-i, n+2+i}{n+1}{n+1^*}} \cdot \sqrt{\tcfr{n-i, n+2+i}{n+1^*}{n+1}} = \tcfr{n-i, n+2+i}{n}{n+2}. \\
\end{gathered}
\end{equation}
Both these identities are a consequence of the fact that
$$\Delta_{u_{in^*}\omega_{n^*}, \omega_{n^*}}(x) \cdot \Delta_{u_{in^*}\omega_n, \omega_n}(x)=\Delta_{u_{in^*}(\omega_n+\omega_{n^*}), \omega_n+\omega_{n^*}}(x).$$

Using these identities, we reduce all the mutation identities to those appearing in the cactus sequence.

\end{proof}

The above sequence of mutations takes us from one seed for the cluster algebra structure on $\Conf_3 \A_{Spin_{2n+2}}$ to another seed where the roles of the second and third principal flags have been reversed. Thus, we have realized the first of the transpositions necessary to construct all the $S_3$ symmetries of $\Conf_3 \A_{Spin_{2n+2}}$.

\subsubsection{The second transposition}

Let us now give the sequence of mutations that realizes that $S_3$ symmetry $(A,B,C) \rightarrow (C,B,A).$

The sequence of mutations is as in the case of $Sp_{2n}$, \eqref{13Sp}:
\begin{equation}
\begin{gathered}
x_{n-1,n}, x_{n-1,n^*}, x_{n-2,n-1}, x_{n-2,n}, x_{n-2,n^*}, x_{n-3,n-2}, x_{n-3,n-1}, x_{n-3,n}, x_{n-3,n^*}, \\
\dots, x_{1,2}, \dots, x_{1,n}, x_{1,n^*} \\
x_{n-1,n}, x_{n-1,n^*}, x_{n-2,n-1}, x_{n-2,n}, x_{n-2,n^*}, x_{n-3,n-2}, x_{n-3,n-1}, x_{n-3,n}, x_{n-3,n^*}, \\
\dots, x_{2,3}, \dots, x_{2,n}, \dots, x_{2,n^*}, \\
x_{n-1,n}, x_{n-1,n^*}, x_{n-2,n-1}, x_{n-2,n}, x_{n-2,n^*}, x_{n-3,n-2}, x_{n-3,n-1}, x_{n-3,n}, x_{n-3,n^*}, \\
\dots, x_{3,4}, \dots, x_{3,n}, \dots, x_{3,n^*}, \\
\dots \\
x_{n-1,n}, x_{n-1,n^*}, x_{n-2,n-1}, x_{n-2,n}, x_{n-2,n^*}, x_{n-3,n-2}, x_{n-3,n-1}, x_{n-3,n}, x_{n-3,n^*}, \\
x_{n-1,n}, x_{n-1,n^*}, x_{n-2,n-1}, x_{n-2,n}, x_{n-2,n^*}, \\
x_{n-1,n}, x_{n-1,n^*}, \\
\end{gathered}
\end{equation}

The sequence can be thought of as follows: We only mutate those $x_{ij}$ with $i < j$. At any step of the process, we mutate all $x_{ij}$ in the $k^{\textrm th}$ row (the $k^{\textrm th}$ row consists of $x_{ij}$ such that $i=k$) such that $i < j$ (we will consider that $n-1< n$ and $n-1<n^*$). It will not matter in which order we mutate these $x_{ij}$. The sequence of rows that we mutate is
$$n-1, n-2 \dots, 2,1 n-1, n-2, \dots, 2, n, \dots, 3, \dots, n-1, n-2, n-1.$$ 

As in the previous transposition, the evolution of the quiver for $\Conf_3 \A_{Spin_{2n+2}}$ is just as in the cases for $\Conf_3 \A_{Sp_{2n}}$ and $\Conf_3 \A_{Spin_{2n+1}}$, as pictured in Figures 15 and 16, with the only difference being that black and white vertices switch colors. 

In Figure 42, we depict how the quiver for $\Conf_3 \A_{Spin_{12}}$ changes after performing the sequence of mutations of $x_{ij}$ in rows $4$; $4,3$; $4,3,2$; and $4,3,2,1$.

\begin{center}

\end{center}

In Figure 43, we depict the state of the quiver after performing the sequence of mutations of $x_{ij}$ in rows $4,3,2,1$; $4,3,2,1,4,3,2$; $4,3,2,1,4,3,2,4,3$; and $4,3,2,1,4,3,2,4,3,4$.

\begin{center}
%
\end{center}

From these diagrams the various quivers in the general case of $\Conf_3 \A_{Spin_{2n+2}}$ should be clear. 

Recall the functions defined via Figures 17. They will appear when we perform the sequence of mutations above. In particular, we make use of functions of the form
$$(n-i, n+2+i; n+1^{(*)}, n+1^{(*)}; n-j, n+2+j).$$
These functions are invariants (unique up to scale) of the tensor products
$$[V_{2\omega_{n-i}} \otimes V_{4\omega_{n}} \otimes V_{2\omega_{n-j}}]^{Spin_{2n+2}},$$
$$[V_{2\omega_{n-i}} \otimes V_{4\omega_{n^*}} \otimes V_{2\omega_{n-j}}]^{Spin_{2n+2}},$$
or
$$[V_{2\omega_{n-i}} \otimes V_{2\omega_{n}+2\omega_{n^*}} \otimes V_{2\omega_{n-j}}]^{Spin_{2n+2}}.$$

These functions will have square-roots which are invariants (again, unique up to scale) of the tensor product
$$[V_{2\omega_{n-i}} \otimes V_{2\omega_{n}} \otimes V_{2\omega_{n-j}}]^{Spin_{2n+2}},$$
$$[V_{2\omega_{n-i}} \otimes V_{2\omega_{n^*}} \otimes V_{2\omega_{n-j}}]^{Spin_{2n+2}},$$
$$[V_{2\omega_{n-i}} \otimes V_{\omega_{n}+\omega_{n^*}} \otimes V_{2\omega_{n-j}}]^{Spin_{2n+2}}.$$
Thus
$$\sqrt{(n-i, n+2+i; n+1^{(*)}, n+1^{(*)}; n-j, n+2+j)}$$
is a well-defined function on $\Conf_3 \A_{Spin_{2n+2}}$.

Note that in the above sequence of mutations, $x_{ij}$ is mutated $i$ times if $i<j$. We can now state the main theorem of this section.

\begin{theorem}
If $i < j$, then $x_{ij}$ is mutated a total of $i$ times. Recall that when $i < j<n$, we assign the either the function $\tcfr{n-i, 2n+1+i-j}{n}{j}$ to $x_{ij}$. In these cases, $j < n$, the function attached to $x_{ij}$ transforms as follows:
$$\tcfr{n-i, 2n+2+i-j}{n+2}{j} \rightarrow \tcfr{n-i+1, 2n+1+i-j}{n, n+2}{j-1, n+3} \rightarrow $$
$$\tcfr{n-i+2, 2n+i-j}{n, n+2}{j-2, n+4} \rightarrow \dots \rightarrow \tcfr{n-1, 2n+3-j}{n, n+2}{j-i+1, n+1+i} $$
$$\rightarrow \tcfr{2n+2-j}{n}{j-i, n+i+2}=\tcfr{j}{n+2}{n-i, 2n+2+i-j}$$

The first transformation can be seen as the composite of two steps,
$$\tcfr{n-i, 2n+2+i-j}{n+2}{j} \rightarrow \tcfr{n-i, 2n+2+i-j}{n, n+2}{j, n+2} $$
$$\rightarrow \tcfr{n-i+1, 2n+1+i-j}{n, n+2}{j-1, n+3},$$
while the last transformation can also be seen as the composite of two steps, 
$$\tcfr{n-1, 2n+3-j}{n, n+2}{j-i+1, n+i+1} \rightarrow \tcfr{n, 2n+2-j}{n, n+2}{j-i, n+i+2} $$
$$\rightarrow \tcfr{2n+2-j}{n}{j-i, n+i+2}.$$
Then with each transformation, two of the parameters increase by one, and two decrease by one.

For $j=n$ or $n^*$, the mutation sequence is similar, but a bit more subtle, and depends on whether $n-i$ is odd or even. First note that the functions attached to $x_{in}$ and $x_{in^*}$ involve square-roots. Let us first consider $x_{in}$. We have the following cases:

\begin{itemize}
\item If $n$ is even and $i$ is odd, then the square of the function attached to $x_{in}$ transforms as follows:
$$\tcfr{n-i, n+i+2}{n+1}{n+1} \rightarrow \tcfr{n-i+1, n+i+1}{n+1^*, n+1^*}{n-1, n+3} \rightarrow $$
$$\tcfr{n-i+2, n+i}{n+1, n+1}{n-2, n+4} \rightarrow \dots \rightarrow \tcfr{n-1, n+3}{n+1, n+1}{n-i+1, n+i+1} $$
$$\rightarrow \tcfr{n+1^*}{n+1^*}{n-i, n+i+2}.$$

\item If $n$ is odd and $i$ is even, we have:
$$\tcfr{n-i, n+i+2}{n+1^*}{n+1} \rightarrow \tcfr{n-i+1, n+i+1}{n+1, n+1}{n-1, n+3} \rightarrow $$
$$\tcfr{n-i+2, n+i}{n+1^*, n+1^*}{n-2, n+4} \rightarrow \dots \rightarrow \tcfr{n-1, n+3}{n+1, n+1}{n-i+1, n+i+1} $$
$$\rightarrow \tcfr{n+1^*}{n+1^*}{n-i, n+i+2}.$$

\item If $n$ and $i$ are both even, we have:
$$\tcfr{n-i, n+i+2}{n+1^*}{n+1} \rightarrow \tcfr{n-i+1, n+i+1}{n+1, n+1^*}{n-1, n+3} \rightarrow $$
$$\tcfr{n-i+2, n+i}{n+1^*, n+1}{n-2, n+4} \rightarrow \dots \rightarrow \tcfr{n-1, n+3}{n+1, n+1^*}{n-i+1, n+i+1}$$
$$\rightarrow \tcfr{n+1^*}{n+1}{n-i, n+i+2}.$$

\item If $n$ and $i$ are both odd, we have:
$$\tcfr{n-i, n+i+2}{n+1}{n+1} \rightarrow \tcfr{n-i+1, n+i+1}{n+1^*, n+1}{n-1, n+3} \rightarrow $$
$$\tcfr{n-i+2, n+i}{n+1, n+1^*}{n-2, n+4} \rightarrow \dots \rightarrow \tcfr{n-1, n+3}{n+1, n+1^*}{n-i+1, n+i+1} $$
$$\rightarrow \tcfr{n+1}{n+1}{n-i, n+i+2}.$$
\end{itemize}

The case of $x_{in^*}$ switches all occurences of $n+1$ and $n+1^*$.

\end{theorem}

\begin{proof}

We have already described the quivers at the various stages of mutation. We must then check that the functions above satisfy the identities of the associated cluster transformations.

This is the first sequence of mutations where we have to mutate the vertices $x_{in}$ and $x_{in^*}$. The mutation of these vertices is the most difficult to handle.

We will need the following facts, which are analogous to the previous facts used for the case of $Spin_{2n+2}$. Let $N=2n+2$:
\begin{itemize}
\item Let $1 \leq a, b, c, d \leq N$, and $a+b+c+d=2N$. 
$$\tcfr{a, b}{n+2, n}{c, d}\tcfr{a+1, b-1}{n+2, n}{c-1, d+1}=$$
$$\tcfr{a, b}{n+2, n}{c-1, d+1}\tcfr{a+1, b-1}{n+2, n}{c, d}+\tcfr{a+1, b}{n+2, n}{c-1,d}\tcfr{a, b-1}{n+2, n}{c, d+1}.$$
\item If $a+c=n+2$ and $b+d=3n+2$,
$$\tcfr{a, b}{n+2, n}{c, d}=\tcfr{a}{n}{c}\tcfr{b}{n+2}{d}.$$
\item If $a+c=n$ and $b+d=3n+4$,
$$\tcfr{a, b}{n+2, n}{c, d}=\tcfr{a}{n+2}{c}\tcfr{b}{n}{d}.$$
\item If $a=n$,
$$\tcfr{n, b}{n+2, n}{c, d}=\dud{n}{n+2}\tcfr{b}{n}{c, d}.$$
Similarly, we have
$$\tcfr{a, b}{n+2, n}{c, n+2}=\tcfr{}{n}{n+2}\tcfr{a, b}{n+2}{c}.$$
\item We will need the duality identities of \eqref{dualities3}, and also the following duality identity:
$$\tcfr{a, b}{n+2, n}{c, d}=\tcfr{N-b, N-a}{n+2, n}{N-d, N-c}$$
\item We also need another group of identities, similar to \eqref{products}:
\begin{equation}\label{products2}
\begin{gathered}
\sqrt{\tcfr{a, N-a}{n+1, n+1^*}{b, N-b}} \cdot \sqrt{\tcfr{a, N-a}{n+1^*, n+1}{b, N-b}} \\
= \tcfr{a, N-a}{n, n+2}{b, N-b} \\
\sqrt{\tcfr{a, N-a}{n+1, n+1}{b, N-b}} \cdot \sqrt{\tcfr{a, N-a}{n+1^*, n+1^*}{b, N-b}} \\
= \tcfr{a, N-a}{n, n+2}{b, N-b} \\
\end{gathered}
\end{equation}

\end{itemize}

The first five identities can be proved by the same method as we used for $G=Sp_{2n}$ and $N=2n$. The last identity can be reduced to \eqref{products}.

All cluster mutations of $x_{ij}$, $j<n$ are obtained as before: they are either the first identity in the above list, or they are degenerations of this identity, and are obtained from the first one by applying one of the remaining identities.

The mutation of the vertices $x_{in}$ and $x_{in^*}$ requires some more work. We will also make use of functions of the form
$$\tcfr{a, b}{x, y}{c, d}$$
where $x, y=n+1$ or $n+1^*$, and $a+b+c+d=2N$. We again use the webs pictured in Figure 17*.

The most general identity takes different forms depending on the parity of $n$, $i$ and $j$. For example, we need to show:

$$\sqrt{\tcfr{n-i, n+2+i}{n+1^*, n+1}{n-j, n+2+j}} \cdot$$
$$\sqrt{\tcfr{n+1-i, n+1+i}{n+1, n+1^*}{n-1-j, n+3+j}} = $$
$$\sqrt{\tcfr{n-i, n+2+i}{n+1, n+1}{n-1-j, n+3+j}} \cdot$$
$$\sqrt{\tcfr{n+1-i, n+1+i}{n+1^*, n+1^*}{n-j, n+2+j}}+$$
$$\tcfr{n+1-i, n+2+i}{n+2, n}{n-1-j, n+2+j}.$$
Note that by \eqref{dualities2},
$$\tcfr{n-i, n+1+i}{n+2, n}{n-j, n+3+j}=\tcfr{n+1-i, n+2+i}{n+2, n}{n-1-j, n+2+j},$$
which explains the seeming asymmetry of the last term.

The general identity above follows directly from the following identities. Let $i, j < n$. Then 

\begin{itemize}
\item 
$$\tcfr{n-i, n+2+i}{n+2, n}{n-j, n+2+j}\tcfr{n+1-i, n+1+i}{n+2, n+1^*}{n-1-j, n+2+j}=$$
$$\tcfr{n+1-i, n+1+i}{n+2, n}{n-j, n+2+j}\tcfr{n-i, n+2+i}{n+2, n+1^*}{n-1-j, n+2+j}+$$
$$\tcfr{n+1-i, n+2+i}{n+2, n}{n-1-j, n+2-j}\tcfr{n-i, n+1+i}{n+2, n+1^*}{n-j, n+2+j},$$

\item 
$$\tcfr{n+1-i, n+1+i}{n+2, n+1^*}{n-1-j, n+2+j}=$$
$$\sqrt{2\tcfr{n+1-i, n+1+i}{n+1, n+1^*}{n-1-j, n+3+j}} \cdot$$
$$\sqrt{\tcfr{n+1-i, n+1+i}{n+1^*, n+1^*}{n-j, n+2+j}},$$

\item 
$$\tcfr{n-i, n+2+i}{n+2, n+1^*}{n-1-j, n+2+j}=$$
$$\sqrt{2\tcfr{n-i, n+2+i}{n+1^*, n+1^*}{n-1-j, n+3+j}} \cdot$$
$$\sqrt{\tcfr{n-i, n+2+i}{n+1, n+1^*}{n-j, n+2+j}},$$

\item 
$$\tcfr{n-i, n+1+i}{n+2, n+1^*}{n-j, n+2+j}=$$
$$\sqrt{2\tcfr{n-i, n+2+i}{n+1, n+1^*}{n-j, n+2+j}} \cdot$$
$$\sqrt{\tcfr{n+1-i, n+1+i}{n+1^*, n+1^*}{n-j, n+2+j}}.$$

\end{itemize}

These identities are proved as before. Used in combination with \eqref{products} and \eqref{products2}, we get the identity we seek.

The other mutations of black vertices will be degenerate specializations of the above general identity. For example, for $n \geq 3$, $n$ odd, the first mutation of $x_{n-1,n}$ is
$$\sqrt{\tcfr{1, N-1}{n+1^*}{n+1}\tcfr{2, N-2}{n+1, n+1}{n-1,n+3}}=$$
$$\sqrt{\tcfr{2, N-2}{n+1}{n+1}}\tcfr{1}{n+2}{n-1}+\sqrt{\tcfr{}{n+1}{n+1}}\tcfr{2, N-1}{n+2}{n-1}.$$ 
To derive this, we use the general identity
$$\sqrt{\tcfr{1, N-1}{n+1^*, n+1^*}{n+1, n+1^*}} \sqrt{\tcfr{2, N-2}{n+1, n+1}{n-1, n+3}} = $$
$$\sqrt{\tcfr{1, N-1}{n+2, n}{n-1, n+3}}\sqrt{\tcfr{2, N-2}{n+1, n+1^*}{n+1, n+1^*}}+$$
$$\tcfr{2, N-1}{n+2, n}{n-1, n+2}.$$
plus the facts
$$\tcfr{1, N-1}{n+1^*, n+1^*}{n+1, n+1^*}=\tcfr{1, N-1}{n+1^*}{n+1}\tcfr{}{n+1^*}{n+1^*},$$
$$\tcfr{1, N-1}{n+2, n}{n-1, n+3}=\tcfr{1}{n+2}{n-1}\tcfr{N-1}{n}{n+3}=\tcfr{1}{n+2}{n-1}^2,$$
$$\tcfr{2, N-2}{n+1, n+1^*}{n+1, n+1^*}=\tcfr{2, N-2}{n+1}{n+1}\tcfr{}{n+1^*}{n+1^*},$$
$$\tcfr{2, N-1}{n+2, n}{n-1, n+2}=\tcfr{2, N-1}{n+2}{n-1}\tcfr{}{n}{n+2},$$
$$\sqrt{\tcfr{}{n+1}{n+1}\tcfr{}{n+1^*}{n+1^*}}=\tcfr{}{n}{n+2}.$$

\end{proof}

\subsection{The sequence of mutations for a flip}

In this section, we will give a sequence of mutations that relates two of the clusters coming from different triangulations of the $4$-gon. Combined with the previous section, this allows us to connect by mutations all $72$ different clusters we have constructed for $\Conf_4 \A_{Spin_{2n+2}}$.

Given a configuration $(A,B,C,D) \in \Conf_4 \A_{Spin_{2n+2}}$, we will give a sequence of mutations that relates a cluster coming from the triangulation $ABC, ACD$ to a cluster coming from the triangulation $ABD, BCD$.

We will need to relabel the quiver with vertices $x_{ij}$, $y_k$, with $-n \leq i \leq n$, $j = 1, 2, \dots, n-1, n, n^*$ and $j = 1, 2, \dots, n-1, n, n^*, -1, -2, \dots, -(n-1), -n, -n^*$. The quiver we will start with is as in Figure 44, pictured for $Spin_8$. We don't picture the vertices $x_{in^*}$, which double the vertices $x_{in}$, or the vertices $y_{n^*}$ and $y_{-n^*}$, which double the vertices $y_{n}$ and $y_{-n}$.

Figure 44 Quiver for $\Conf_4 \A_{Spin_{8}}$ 

\begin{center}
\begin{tikzpicture}[scale=2.4]
  \foreach \x in {-3,-2,-1,0,1,2,3}
    \foreach \y in {1, 2}
      \node[] (x\x\y) at (\x,-\y) {\Large $\ontop{x_{\x\y}}{\bullet}$};
  \foreach \x in {-3,-2,-1,0,1,2,3}
    \foreach \y in {3}
      \node[] (x\x\y) at (\x,-\y) {\Large $\ontop{x_{\x\y}}{\bullet}$};
  \node (y-1) at (-0.5,0) {\Large $\ontop{y_{-1}}{\bullet}$};
  \node (y-2) at (-1.5,0) {\Large $\ontop{y_{-2}}{\bullet}$};
  \node (y-3) at (-2.5,-4) {\Large $\ontop{y_{-3}}{\bullet}$};
  \node (y1) at (0.5,0) {\Large $\ontop{y_1}{\bullet}$};
  \node (y2) at (1.5,0) {\Large $\ontop{y_2}{\bullet}$};
  \node (y3) at (2.5,-4) {\Large $\ontop{y_{3}}{\bullet}$};

  \draw [->] (x01) to (x11);
  \draw [->] (x11) to (x21);
  \draw [->] (x21) to (x31);
  \draw [->] (x02) to (x12);
  \draw [->] (x12) to (x22);
  \draw [->] (x22) to (x32);
  \draw [->] (x03) to (x13);
  \draw [->] (x13) to (x23);
  \draw [->] (x23) to (x33);
  \draw [->, dashed] (y1) to (y2);
  \draw [->, dashed] (y2) .. controls +(right:1) and +(up:1) .. (y3);

  \draw [->] (x03) to (x02);
  \draw [->] (x02) to (x01);

  \draw [->] (x13) to (x12);
  \draw [->] (x12) to (x11);
  \draw [->] (x11) to (y1);
  \draw [->] (y3) to (x23);
  \draw [->] (x23) to (x22);
  \draw [->] (x22) to (x21);
  \draw [->] (x21) to (y2);
  \draw [->, dashed] (x33) to (x32);
  \draw [->, dashed] (x32) to (x31);

  \draw [->] (y1) to (x01);
  \draw [->] (x11) to (x02);
  \draw [->] (x12) to (x03);
 \draw [->] (y2) to (x11);
  \draw [->] (x21) to (x12);
  \draw [->] (x22) to (x13);
  \draw [->] (x31) to (x22);
  \draw [->] (x32) to (x23);
  \draw [->] (x33) to (y3);

  \draw [->] (x01) to (x-11);
  \draw [->] (x-11) to (x-21);
  \draw [->] (x-21) to (x-31);
  \draw [->] (x02) to (x-12);
  \draw [->] (x-12) to (x-22);
  \draw [->] (x-22) to (x-32);
  \draw [->] (x03) to (x-13);
  \draw [->] (x-13) to (x-23);
  \draw [->] (x-23) to (x-33);
  \draw [->, dashed] (y-1) to (y-2);
  \draw [->, dashed] (y-2) .. controls +(left:1) and +(up:1) .. (y-3);

  \draw [->] (x-13) to (x-12);
  \draw [->] (x-12) to (x-11);
  \draw [->] (x-11) to (y-1);
  \draw [->] (y-3) to (x-23);
  \draw [->] (x-23) to (x-22);
  \draw [->] (x-22) to (x-21);
  \draw [->] (x-21) to (y-2);
  \draw [->, dashed] (x-33) to (x-32);
  \draw [->, dashed] (x-32) to (x-31);

  \draw [->] (y-1) to (x01);
  \draw [->] (x-11) to (x02);
  \draw [->] (x-12) to (x03);
  \draw [->] (y-2) to (x-11);
  \draw [->] (x-21) to (x-12);
  \draw [->] (x-22) to (x-13);
  \draw [->] (x-31) to (x-22);
  \draw [->] (x-32) to (x-23);
  \draw [->] (x-33) to (y-3);

\draw[yshift=-3.85cm]
  node[below,text width=6cm] 
  {
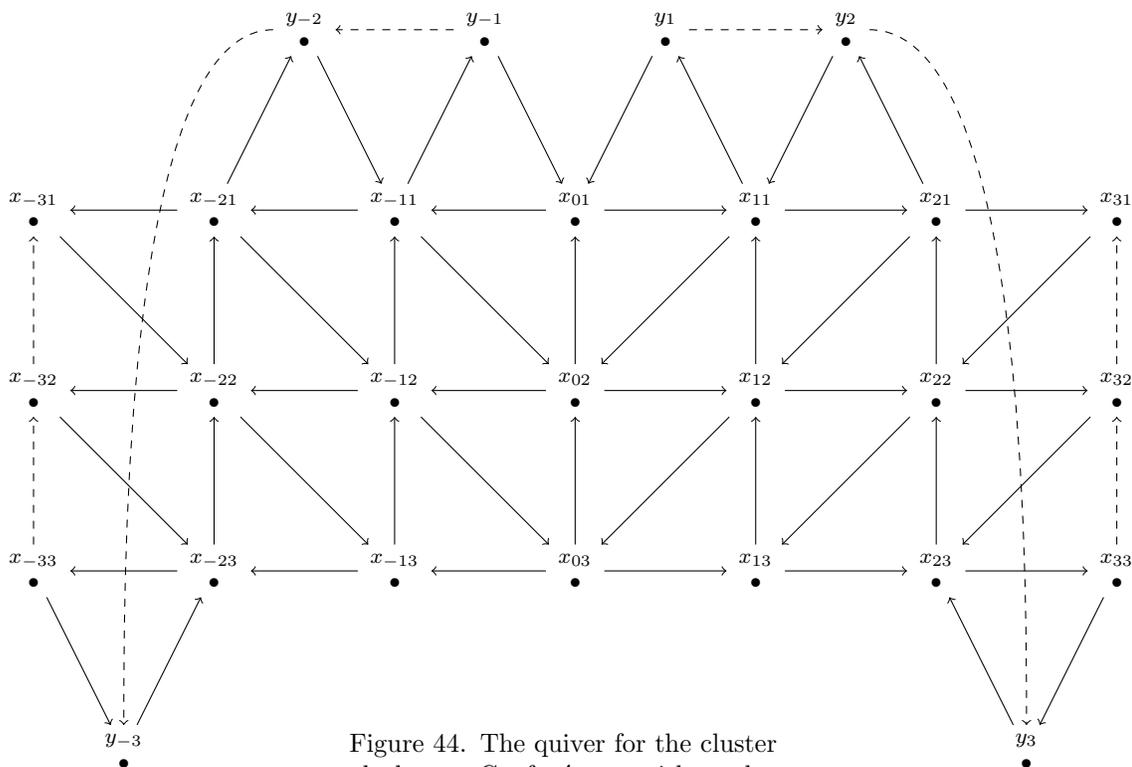
  Figure 44. The quiver for the cluster algebra on $\Conf_4 \A_{Spin_{8}}$ without doubled vertices. The associated functions are pictured in Figure 39.
  };

\end{tikzpicture}
\end{center}

Let $N=2n+2$. We will treat the case of $n$ even. The case where $n$ is odd is similar.

First make an assigment of functions to the edge vertices:

\begin{alignat*}{1}
\dur{k}{N-k} &\longleftrightarrow y_k, \textrm{ for } 0 < k \leq n-1; \\
\dur{n+1}{n+1} &\longleftrightarrow y_n; \\
\dur{n+1^*}{n+1^*} &\longleftrightarrow y_{n^*}; \\
\dld{|k|}{N-|k|} &\longleftrightarrow y_k, \textrm{ for } -(n-1) \leq k <0; \\
\dld{n+1}{n+1} &\longleftrightarrow y_{-n}, \\
\dld{n+1^*}{n+1^*} &\longleftrightarrow y_{-n^*}, \\
\dul{j}{N-j} &\longleftrightarrow x_{-n,j} \textrm{ for } 0 < j \leq n-1; \\
\dul{n+1}{n+1} &\longleftrightarrow x_{-n,n}; \\
\dul{n+1^*}{n+1^*} &\longleftrightarrow x_{-n,n^*}; \\
\ddr{N-j}{j} &\longleftrightarrow x_{nj}; \textrm{ for } 0 < j \leq n-1; \\
\ddr{n+1}{n+1} &\longleftrightarrow x_{nn}; \\
\ddr{n+1^*}{n+1^*} &\longleftrightarrow x_{nn^*}; \\
\dud{j}{N-j} &\longleftrightarrow x_{0j} \textrm{ for } 0 < j \leq n-1; \\
\dud{n+1}{n+1} &\longleftrightarrow x_{0n} \\
\dud{n+1^*}{n+1^*} &\longleftrightarrow x_{0n^*} \\
\end{alignat*}

The face functions in the triangle where $i>0$ are
\begin{alignat*}{1}
\tcfr{i+j}{N-j}{N-i} &\longleftrightarrow x_{ij}, \textrm{ for } 0<i<n, i+j \leq n; \\
\tcfr{n}{N-j}{j+i-n, N-i} &\longleftrightarrow x_{ij}, \textrm{ for } 0<i<n, i+j > n, j<n ; \\
\tcfr{n+1^*}{n+1}{i, N-i} &\longleftrightarrow x_{in}, \textrm{ for } 0<i<n; i \textrm{ odd }\\
\tcfr{n+1}{n+1}{i, N-i} &\longleftrightarrow x_{in}, \textrm{ for } 0<i<n; i \textrm{ even }\\
\tcfr{n+1}{n+1^*}{i, N-i} &\longleftrightarrow x_{in^*}, \textrm{ for } 0<i<n; i \textrm{ odd }\\
\tcfr{n+1^*}{n+1^*}{i, N-i} &\longleftrightarrow x_{in^*}, \textrm{ for } 0<i<n; i \textrm{ even }\\
\end{alignat*}
while the face functions in the triangle where $i<0$ are
\begin{alignat*}{1}
\tcfl{j}{|i|}{N-|i|-j} &\longleftrightarrow x_{ij}, \textrm{ for } -n<i<0, |i|+j \leq n; \\
\tcfl{j}{|i|, N+n-|i|-j}{n+2} &\longleftrightarrow x_{ij}, \textrm{ for } -n<i<0, |i|+j > n, j<n; \\
\tcfl{n+1}{|i|, N-|i|}{n+1^*} &\longleftrightarrow x_{ij}, \textrm{ for } -n<i<0; i \textrm{ odd }\\
\tcfl{n+1}{|i|, N-|i|}{n+1} &\longleftrightarrow x_{ij}, \textrm{ for } -n<i<0; i \textrm{ even }\\
\tcfl{n+1^*}{|i|, N-|i|}{n+1} &\longleftrightarrow x_{ij}, \textrm{ for } -n<i<0; i \textrm{ odd }\\
\tcfl{n+1^*}{|i|, N-|i|}{n+1^*} &\longleftrightarrow x_{ij}, \textrm{ for } -n<i<0; i \textrm{ even }\\
\end{alignat*}

\begin{rmk} Note that our labelling of the vertices is somewhat different from before. The vertices labelled $x_{ij}$ here correspond to the vertices labelled $x_{n-|i|, j}$ in the previous sections dealing with $\Conf_3 \A_{Spin_{2n+2}}$.
\end{rmk}

\begin{rmk} Note that to obtain the function attached to $x_{in^*}$ (respectively, $y_{n^*}$) from the function attached to $x_{in}$ (respectively, $y_{n}$), we simply switch every occurence of $n+1$ and $n+1^*$.
\end{rmk}

The functions above are defined by pulling back via the various natural maps 
$$p_1, p_2, p_3, p_4: \Conf_4 \A_{Spin_{2n+2}} \rightarrow \Conf_3 \A_{Spin_{2n+2}}$$
that map a configuration $(A,B,C,D)$ to $(B, C, D)$, $(A, C, D)$, $(A, B, D)$, $(A,B,C)$, respectively. Pulling back functions from $\Conf_3 \A_{Spin_{2n+2}}$ allows us to define functions on $\Conf_4 \A_{Spin_{2n+2}}$.

There is also a map
$$T: \Conf_4 \A  \rightarrow \Conf_4 \A$$
which sends
$$(A,B,C,D) \rightarrow (s_G \cdot D, A, B, C)$$
which allows us to define, for example
$$T^*\tcfu{n}{j}{n+1-i, N+i-j} =: \tcfr{j}{n+1-i, N+i-j}{n}.$$ The forgetful maps and twist maps, combined with the constructions below, will furnish all the functions necessary for the computation of the flip mutation sequence.

As in the cases of $G=Sp_{2n}$ and $Spin_{2n+1}$, we will have to use some functions which depend on all four flags. Let $N=2n+2$. Let $0 \leq a, b, c, d \leq N$ such that $a+b+c+d=4n+4=2N$ and $b+c \leq N$. Then we will define a function that we will call 
$$\qcfs{a}{b}{c}{d}.$$

It is a function on $\Conf_4 \A_{Spin_{2n+2}}$ which is pulled back from a function on $\Conf_4 \A_{SL_{N}}$. The function on $\Conf_4 \A_{SL_{N}}$ is given by an invariant in the space
$$[V_{\omega_a}  \otimes V_{\omega_b}  \otimes V_{\omega_c} \otimes V_{\omega_d}]^{SL_N}.$$
The function is given, as before, by the web in Figure 19.

We again use the symbol ``:'' because the function does not have cyclic symmetry. In other words,
$$T^*\qcfs{a}{b}{c}{d} \neq \qcfs{b}{c}{d}{a}.$$
Instead, we use the notation
$$T^*\qcfs{a}{b}{c}{d} =: \qcfb{b}{c}{d}{a}.$$
We can also define
$$(T^2)^*\qcfs{a}{b}{c}{d}=:\qcfs{c}{d}{a}{b}.$$

We now define a second type of function on  $\Conf_4 \A_{Spin_{2n+2}}$. If $a+b+c+d=2N+n$, we define the function
$$\qcfs{n+2}{a}{b}{c, d}.$$
It is pulled back from a function on $\Conf_4 \A_{SL_{2n+2}}$ and given by the invariant in the space
$$[V_{\omega_{n+2}}  \otimes V_{\omega_a}  \otimes V_{\omega_b} \otimes V_{\omega_c+\omega_d}]^{SL_N}$$
picked out by the web in Figure 33. Using the twist map $T$, we can also define the functions $\qcfb{a}{b}{c, d}{n+2}$, $\qcfs{b}{c, d}{n+2}{a}$, and $\qcfb{c, d}{n+2}{a}{b}$.

Using duality, there is also a function $\qcfs{n}{a}{b}{c, d}$ on $\Conf_4 \A_{Spin_{2n+2}}$ for $0 \leq a, b, c, d \leq N$, $a+b+c+d=3n+4=N+n+2$, and $c \leq d$. 

This function is pulled back from the function on $\Conf_4 \A_{SL_{2n+2}}$ given by an invariant in the space
$$[V_{\omega_n}  \otimes V_{\omega_a}  \otimes V_{\omega_b} \otimes V_{\omega_c+\omega_d}]^{SL_N}.$$
This vector space is generally multi-dimensional. To pick out the correct invariant, we use the same web as in Figure 20b.

Using the twist map $T$, we can also define the functions $\qcfb{a}{b}{c, d}{n}$, $\qcfs{b}{c, d}{n}{a}$, and $\qcfb{c, d}{n}{a}{b}$.

Finally, we will work with the function $\qcfs{x}{a, b}{y}{c, d}$ on $\Conf_4 \A_{Spin_{2n+2}}$, where $x, y = n, n+1, n+1^*$ or $n+2$. This is pulled back from a function on $\Conf_4 \A_{SL_{N}}$ via an appropriate map.

Let us first define the function $\qcfs{x}{a, b}{y}{c, d}$ on $\Conf_4 \A_{SL_{N}}$  where $x, y= n, n+1$ or $n+2$. It is given by an invariant in the space
$$[V_{\omega_x}  \otimes V_{\omega_a+\omega_b}  \otimes V_{\omega_y} \otimes V_{\omega_c+\omega_d}]^{SL_N}.$$
This vector space is generally multi-dimensional. To pick out the correct invariant, we use the web in Figure 34.

Now if the function $\qcfs{x}{a, b}{y}{c, d}$ on $\Conf_4 \A_{Spin_{2n+2}}$ has either $x$ or $y=n+1^*$, then we pull the function back from $\Conf_4 \A_{SL_{N}}$ using the map $\Conf_4 \A_{Spin_{2n+2}} \rightarrow \Conf_4 \A_{SL_{N}}$ that differs from the standard one by the outer automorphism of $Spin_{2n+2}$ on those flags that have an argument with the symbol $*$.

Note that
$$\qcfs{x}{a, b}{y}{c, d}=(T^2)^*\qcfs{y}{c, d}{x}{a, b}.$$

Note that when $a=0$, $b=N$, $c=0$ or $d=N$, we have
$$\qcfs{x}{0, b}{y}{c, d}=:\qcfs{x}{b}{y}{c, d},$$
$$\qcfs{x}{a, N}{y}{c, d}=:\qcfs{x}{a}{y}{c, d},$$
$$\qcfs{x}{a, b}{y}{0, d}=:\qcfs{x}{a, b}{y}{d},$$
$$\qcfs{x}{a, b}{y}{c, N}=:\qcfs{x}{a, b}{y}{c}.$$
If $a=0$ and $d=N$, we will have
$$\qcfs{x}{0, b}{y}{c, N}=\qcfs{x}{b}{y}{c},$$
where $\qcfs{x}{b}{y}{c} $ is as defined above. A similar equality holds when $b=N, c=0$. If $a=0$, $b=N$, $c=0$ and $d=N$, we will have that 
$$\qcfs{x}{0, N}{y}{0, N}=\qcfs{x}{0}{y}{0}.$$

Finally, note that if $a+b=c+d=N$, $x, y=n+1$ or $n+1^*$,then 
$$\sqrt{\qcfs{x}{a, b}{y}{c, d}}$$
is a well-defined function on $\Conf_4 \A_{Spin_{2n+2}}$. This is because the representations $V_{\omega_a}$ and $V_{\omega_b}$ of $SL_N$ give the same representations of $Spin_N$, and $V_{\omega_x}$ and $V_{\omega_y}$, as representations of $Spin_N$, have twice the weight of one of the spin representations. Thus, for example, when $n$ is odd and $a+c$ is odd, 
$$\sqrt{\qcfs{n+1}{a, b}{n+1^*}{c, d}}$$
is a well-defined function on $\Conf_4 \A_{Spin_{2n+2}}$.

Now we give the sequence of mutations realizing the flip of a triangulation. The sequence of mutations is a lift of the sequence of mutations for the cases of $Sp_{2n}$ and $Spin_{2n+1}$. The sequence of mutations leaves $x_{-n,j}, x_{nj}, y_k$ untouched as they are frozen variables. Hence we only mutate $x_{ij}$ for $-n \leq i \leq n$. We now describe the sequence of mutations. The sequence of mutations will have $3n-2$ stages. At the $r^{\textrm{th}}$ step, we mutate all vertices such that 
$$|i|+j \leq r,$$
$$j-|i| + 2n -2 \geq r,$$
$$|i|+j \equiv r \mod 2.$$

Note that the first inequality is empty for $r \geq 2n-1$, while the second inequality is empty for $r \leq n$. 

\begin{rmk} For the sake of the above inequalities, we $n+1^*=n+1$, so that whenever we mutate $x_{in}$ we will also mutate $x_{in^*}$. The vertices $x_{in}$ and $x_{in^*}$ will not have arrows between them in the quivers that are obtained in the stages of our mutation process, so they can be mutated in any order.
\end{rmk}

For example, for $Spin_8$, the sequence of mutations is

\begin{equation}
\begin{gathered}
x_{01}, \\
x_{-1,1}, x_{02}, x_{11},  \\
x_{-2,1}, x_{-1,2}, x_{01}, x_{03}, x_{03^*}, x_{12}, x_{21},  \\
x_{-2,2}, x_{-1,1}, x_{-1,3}, x_{-1,3^*}, x_{02}, x_{11}, x_{13}, x_{13^*}, x_{22},  \\
x_{-2,3}, x_{-2,3^*}, x_{-1,2}, x_{01}, x_{03}, x_{03^*}, x_{12}, x_{23}, x_{23^*}, \\
x_{-1,3}, x_{-1,3}, x_{02}, x_{13}, x_{13^*}, \\
x_{03}, x_{03^*}.
\end{gathered}
\end{equation}

In Figure 45, we depict how the quiver for $\Conf_4 \A_{Spin_{8}}$ changes after each of the seven stages of mutation.

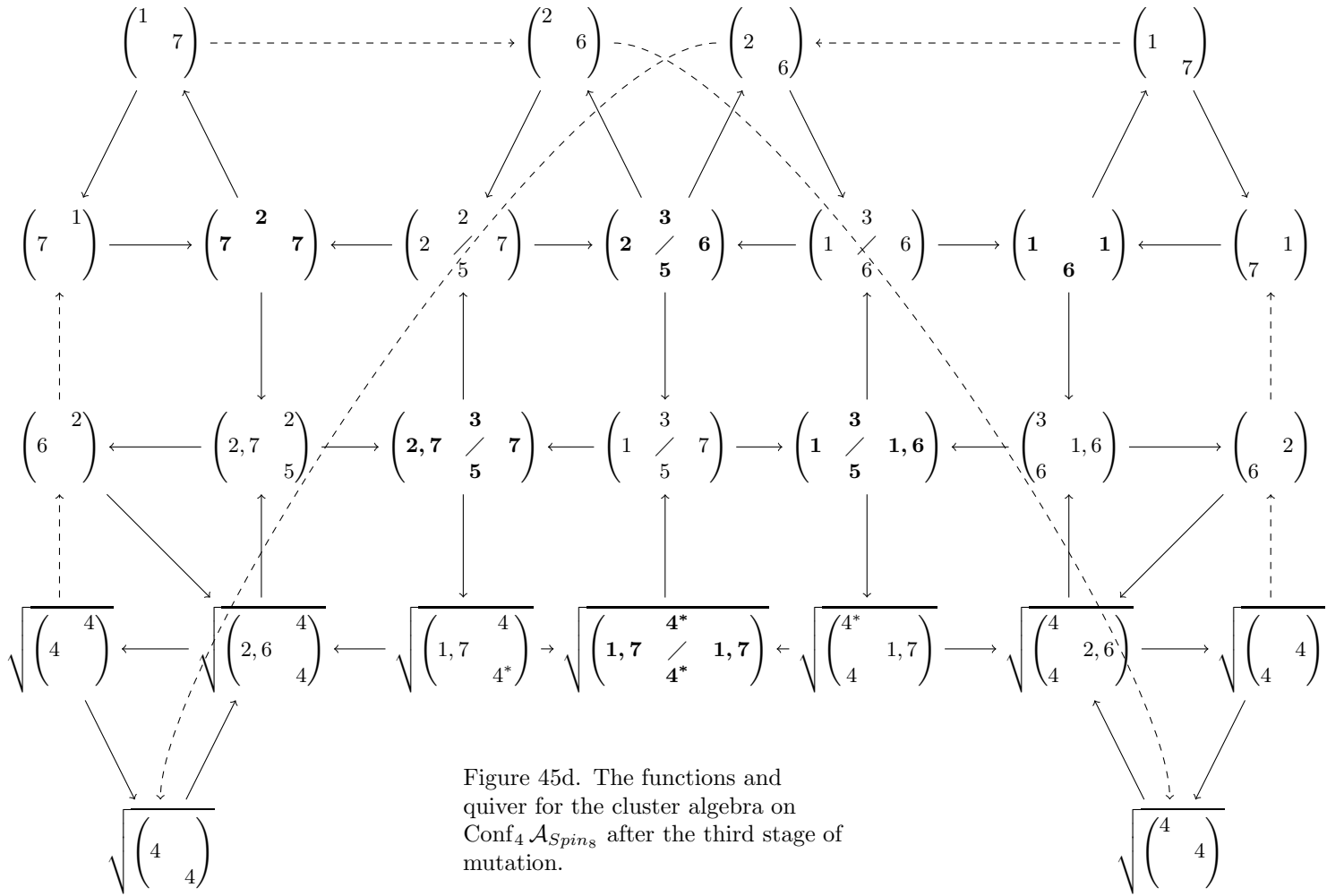
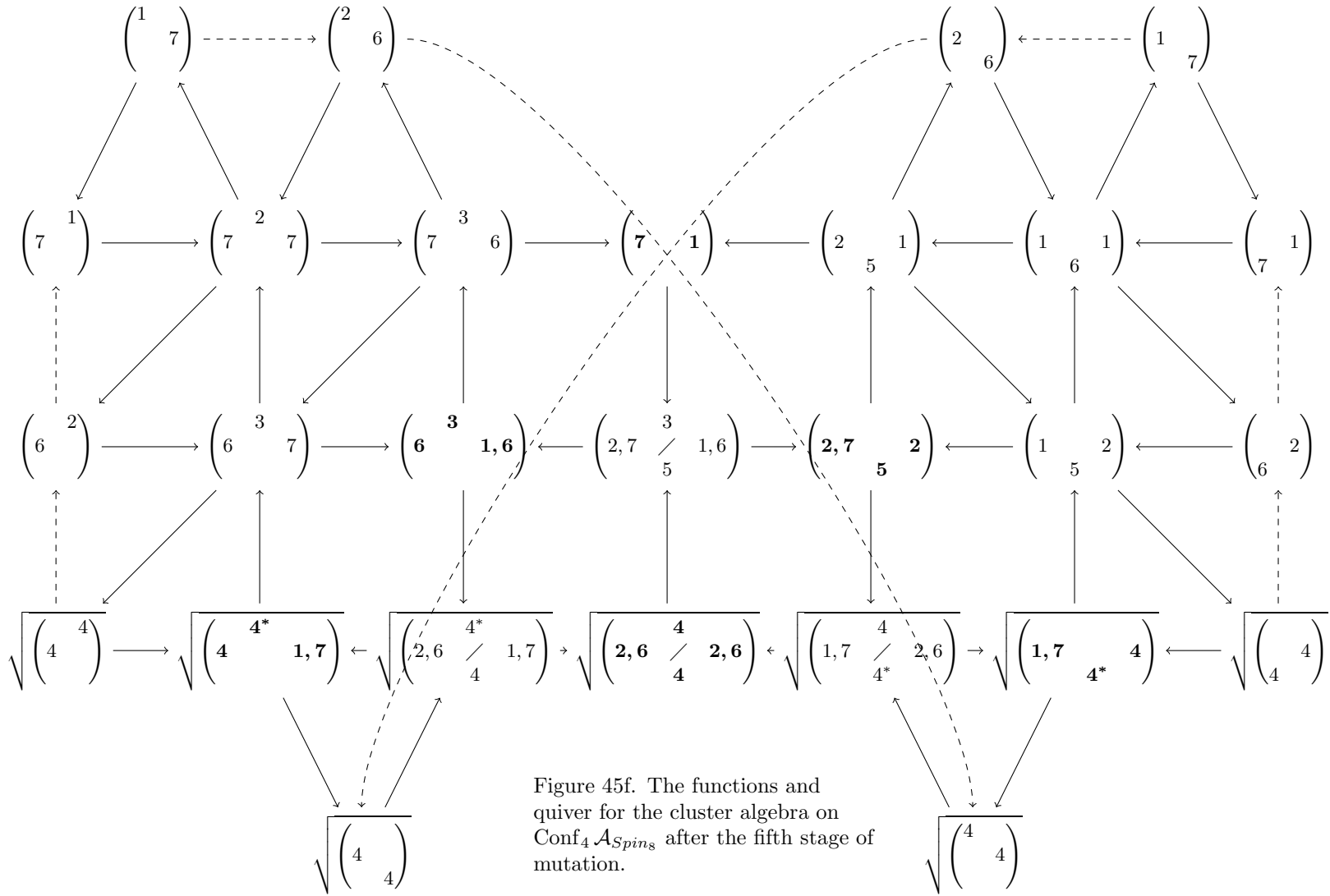
\begin{center}

\end{center}

The analogue $\Conf_4 \A_{Spin_{2n+2}}$ should be clear.

We now have the main theorem of this section:

\begin{theorem}
We first analyze the situation when $i > 0$. The vertex $x_{ij}$ is mutated a total of $n-i$ times. There are four cases.
\begin{itemize}
\item When $i+j < n$ and $i < j$, the function attached to $x_{ij}$ mutates in three stages, consisting of $n-i-j, i,$ and $j-i$ mutations, respectively:
\begin{enumerate}
\item $$\tcfr{i+j}{N-j}{N-i} \rightarrow \qcfs{i+j+1}{1}{N-j-1}{N-i-1} \rightarrow $$
$$\qcfs{i+j+2}{2}{N-j-2}{N-i-2} \rightarrow $$
$$\dots \rightarrow \qcfs{n}{n-i-j}{n+2+i}{n+2+j}$$
\item $$\qcfs{n}{n-i-j}{n+2+i}{n+2+j}=\qcfs{n}{n-i-j}{n+2+i}{0, n+2+j} \rightarrow $$
$$\qcfs{n}{n-i-j+1}{n+1+i}{1, n+1+j} \rightarrow$$
$$ \qcfs{n}{n-i-j+2}{n+i}{2, n+j} \rightarrow $$
$$\dots \rightarrow \qcfs{n}{n-j}{n+2}{i, n+2+j-i}$$
\item $$\qcfs{n}{n-j}{n+2}{i, n+2+j-i}=\qcfs{n}{n-j, N}{n+2}{i, n+2+j-i}$$
$$\rightarrow \qcfs{n}{n-j+1, N-1}{n+2}{i+1, n+1+j-i} \rightarrow $$
$$\qcfs{n}{n-j+2, N-2}{n+2}{i+2, n+j-i} \rightarrow $$
$$\dots \rightarrow [\qcfs{n}{n-i, N-j+i}{n+2}{j, n+2}] \textrm{ } \tcfd{n-i, N-j+i}{n+2}{j}$$
\end{enumerate}

\item When $i+j \geq n$, $i < j$, and $j \neq n$ the function attached to $x_{ij}$ mutates  in two stages, consisting of $n-j$ and $j-i$ mutations, respectively:
\begin{enumerate}
\item $$\tcfr{n}{N-j}{j+i-n, N-i} \rightarrow \qcfs{n}{1}{N-j-1}{j+i-n+1, N-i-1} \rightarrow $$
$$\qcfs{n}{2}{N-j-2}{j+i-n+2, N-i-2} \rightarrow$$
$$\dots  \rightarrow \qcfs{n}{n-j}{n+2}{i, n+2-i+j}$$
\item $$\qcfs{n}{n-j}{n+2}{i, n+2-i+j} [\qcfs{n}{n-j, N}{n+2}{i, n+2-i+j}] \rightarrow $$
$$\qcfs{n}{n-j+1, N-1}{n+2}{i+1, n+1-i+j} \rightarrow $$
$$\qcfs{n}{n-j+2, N-2}{n+2}{i+2, n-i+j} \rightarrow$$
$$\dots  \rightarrow [\qcfs{n}{n-i, N-j+i}{n+2}{j, n+2}]  \textrm{ } \tcfd{n-i, N-j+i}{n+2}{j} $$
\end{enumerate}

\item In the most interesting case, we have that if $j=n$ or $j=n^*$, the mutations happen in one stage consisting of $n-i$ mutations. If $j=n$ and $i$ odd, we have:
$$\sqrt{\qcfs{n+1^*}{0}{n+1}{i, N-i}} [\sqrt{\qcfs{n+1^*}{0, N}{n+1}{i, N-i}}] \rightarrow $$
$$\sqrt{\qcfs{n+1}{1, N-1}{n+1^*}{i+1, N-i-1}} \rightarrow $$
$$\sqrt{\qcfs{n+1^*}{2, N-2}{n+1}{i+2, N-i-2}} \rightarrow$$
$$\dots  \rightarrow [\sqrt{\qcfs{n+1^*}{n-i, N-n+i}{n+1}{n, n+2}}]  \textrm{ } \sqrt{\tcfd{n-i, N-n+i}{n+1}{n+1}}$$
Note in the last step we use
$$\sqrt{\qcfs{n+1^*}{n-i, N-n+i}{n+1}{n, n+2}}=\sqrt{\qcfs{n+1^*}{a, N-a}{n+1}{n+1, n+1^*}}.$$
If $j=n$, and $i$ is even, we have:
$$\sqrt{\qcfs{n+1}{0}{n+1}{i, N-i}} [\sqrt{\qcfs{n+1}{0, N}{n+1}{i, N-i}}] \rightarrow $$
$$\sqrt{\qcfs{n+1^*}{1, N-1}{n+1^*}{i+1, N-i-1}} \rightarrow $$
$$\sqrt{\qcfs{n+1}{2, N-2}{n+1}{i+2, N-i-2}} \rightarrow$$
$$\dots  \rightarrow [\sqrt{\qcfs{n+1^*}{n-i, N-n+i}{n+1^*}{n, n+2}}]  \textrm{ } \sqrt{\tcfd{n-i, N-n+i}{n+1^*}{n+1}}.$$
In the last step we use
$$\sqrt{\qcfs{n+1^*}{n-i, N-n+i}{n+1^*}{n, n+2}}=\sqrt{\qcfs{n+1^*}{a, N-a}{n+1^*}{n+1, n+1^*}}.$$
To obtain the formulas in the case where $j=n^*$, switch all occurences of $n+1$ and $n+1^*$.

\item When $i+j < n$ and $i \geq j$, the function attached to $x_{ij}$ mutates in two stages, consisting of $n-i-j$ and $j$ mutations, respectively:
\begin{enumerate}
\item $$\tcfr{i+j}{N-j}{N-i} \rightarrow \qcf{i+j+1}{1}{N-j-1}{N-i-1} \rightarrow $$
$$\qcf{i+j+2}{2}{N-j-2}{N-i-2} \rightarrow $$
$$\dots \rightarrow \qcf{n}{n-i-j}{n+2+i}{n+2+j}$$
\item $$\qcf{n}{n-i-j}{n+2+i}{n+2+j} \rightarrow \qcfs{n}{n-i-j+1}{n+1+i}{1, n+1+j} \rightarrow $$
$$\qcfs{n}{n-i-j+2}{n+i}{2, n+j} \rightarrow $$
$$\dots \rightarrow [\qcfs{n}{n-i}{n+2+i-j}{j, n+2}]  \textrm{ } \tcfd{n-i}{n+2+i-j}{j}$$
\end{enumerate}

\item When $i+j \geq n$ and $i \geq j$, the function attached to $x_{ij}$ mutates in one stage consisting of $n-i$ mutations:
$$\tcfr{n}{N-j}{j+i-n, N-i} \rightarrow \qcfs{n}{1}{N-j-1}{j+i-n+1, N-i-1} \rightarrow $$
$$\qcfs{n}{2}{N-j-2}{j+i-n+2, N-i-2} \rightarrow$$
$$\dots  \rightarrow [\qcfs{n}{n-i}{n+2+i-j}{j, n+2}] \textrm{ } \tcfd{n-i}{n+2+i-j}{j}$$

\end{itemize}

The mutation sequence when $i \leq 0$ is completely parallel. We include it in an appendix for reference.

In all these sequences, for each mutation, two parameters increase, and two decrease. Within a stage, the same parameters increase or decrease. The only exception is that sometimes after the last mutation, one removes the factor $\dur{n}{n+2}$, $\sqrt{\dur{n+1}{n+1}}$ or $\sqrt{\dur{n+1}{n+1}}$. The expressions in square brackets indicate the functions before removing these factors.

\end{theorem}

\begin{proof} The proof comes down to a handful of identities used in conjunction, as in previous proofs of this type. Except for some differences in the indices, the identities are essentially the same as in the cases where $G=Sp_{2n}$ or $Spin_{2n+1}$. Here are the identities we use:

\begin{itemize}
\item Let $0 \leq a, b, c, d \leq N$, and $a+b+c+d=2N$. 
$$\qcfs{a}{b}{c}{d}\qcfs{a+1}{b+1}{c-1}{d-1}=$$
$$\qcfs{a}{b+1}{c-1}{d}\qcfs{a+1}{b}{c}{d-1}+\qcfs{a+1}{b}{c-1}{d}\qcfs{a}{b+1}{c}{d-1}.$$

\item Let $0 \leq a, b, c, d \leq N$, and $a+b+c+d=N+n+2$. 
$$\qcfs{n}{a}{b}{c, d}\qcfs{n}{a+1}{b-1}{c+1, d-1}=$$
$$\qcfs{n}{a+1}{b-1}{c, d}\qcfs{n}{a}{b}{c+1, d-1}+\qcfs{n}{a+1}{b}{c, d-1}\qcfs{n}{a}{b-1}{c+1, d}.$$
There is also a dual identity when $a+b+c+d=2N+n$ that we use when $i < 0$:
$$\qcfs{n+2}{a}{b}{c, d}\qcfs{n+2}{a-1}{b+1}{c+1, d-1}=$$
$$\qcfs{n+2}{a}{b}{c+1, d-1}\qcfs{n+2}{a-1}{b+1}{c, d}+\qcfs{n+2}{a-1}{b}{c+1, d}\qcfs{n+2}{a}{b+1}{c, d-1}.$$

\item Let $0 \leq a, b, c, d \leq N$, and $a+b+c+d=4n+4$. 
$$\qcfs{n}{a, b}{n+2}{c, d}\qcfs{n}{a+1, b-1}{n+2}{c+1, d-1}=$$
$$\qcfs{n}{a+1, b-1}{n+2}{c, d}\qcfs{n}{a, b}{n+2}{c+1, d-1}+\qcfs{n}{a+1, b}{n+2}{c, d-1}\qcfs{n}{a, b-1}{n+2}{c+1, d}.$$

\item Let $0 \leq a, b, c, d \leq N$ such that $a+b+c+d=2N$, $a \leq b$ and $c \leq d$. If $a+d=b+c=N$.,
$$\qcfs{a}{b}{c}{d}=\dld{b}{c}\dur{a}{d}$$

\item Let $0 \leq a, b, c, d \leq N$ such that $a+b+c+d=2N$, $a \leq b$ and $c \leq d$. If $a$ or $b=n$ or $c$ or $d=n+2$,
$$\qcfs{n}{n, b}{n+2}{c, d}=\tcfu{n}{b}{c, d}\dld{n}{n+2},$$
$$\qcfs{n}{a, n}{n+2}{c, d}=\tcfu{n}{a}{c, d}\dld{n}{n+2},$$
$$\qcfs{n}{a, b}{n+2}{n+2, d}=\tcfd{a, b}{n+2}{d}\dur{n}{n+2},$$
$$\qcfs{n}{a, b}{n+2}{c, n+2}=\tcfd{a, b}{n+2}{c}\dur{n}{n+2}.$$

\item Let $0 \leq a, b, c, d \leq N$ such that $a+b+c+d=2N$, $a \leq b$ and $c \leq d$. When $a=0$, $b=N$, $c=0$ or $d=N$, we have
$$\qcfs{n}{0, b}{n+2}{c, d}=:\qcfs{n}{b}{n+2}{c, d},$$
$$\qcfs{n}{a, N}{n+2}{c, d}=:\qcfs{n}{a}{n+2}{c, d},$$
$$\qcfs{n}{a, b}{n+2}{0, d}=:\qcfs{n}{a, b}{n+2}{d},$$
$$\qcfs{n}{a, b}{n+2}{c, N}=:\qcfs{n}{a, b}{n+2}{c}.$$
If $a=0$ and $d=N$, we will have
$$\qcfs{n}{0, b}{n+2}{c, N}=\qcfs{n}{b}{n+2}{c}.$$
A similar equality holds when $b=N, c=0$. If $a=0$, $b=N$, $c=0$ and $d=N$, we will have that 
$$\qcfs{n}{0, N}{n+2}{0, N}=\qcfs{n}{0}{n+2}{0}.$$

\item The last set of identities is similar to \eqref{products} and \eqref{products2}. Which one we use depends on the parity of $a+b$:
\begin{equation}
\begin{gathered}
\sqrt{\qcfs{n+1^*}{a, N-a}{n+1}{b, N-b}} \cdot \sqrt{\qcfs{n+1}{a, N-a}{n+1^*}{b, N-b}} \\
= \qcfs{n}{a, N-a}{n+2}{b, N-b} \\
\sqrt{\qcfs{n+1}{a, N-a}{n+1}{b, N-b}} \cdot \sqrt{\qcfs{n+1^*}{a, N-a}{n+1^*}{b, N-b}} \\
= \qcfs{n}{a, N-a}{n+2}{b, N-b} \\
\end{gathered}
\end{equation}

\end{itemize}

The proof of the mutation identities is much like when $G=Sp_{2n}$. The first three sets of identities are the most important. They are variations on the octahedron recurrence. When $i+j < n$ and $i < j$, the three stages use the first, second and third set of identities, respectively. When $i+j \geq n$ and $i < j \neq n$, the two stages use the second and third set of identities, respectively. When $i+j < n$ and $i \geq j$, the two stages use the first and second set of identities, respectively. When $i+j \geq n$ and $i \geq j$, the one stage uses only the second set of identities.

The fourth through sixth identities are used to give degenerate versions of the previous three sets of identities. The last set of identities is used when mutating vertices adjacent to $x_{in}$ or $x_{in^*}$.

The main novelty occurs when mutating $x_{ij}$ for $j=n$ or $n^*$. We will handle the case when $j=n$. Here we will need to derive some new identities. The general mutation identity when $j=n$ has one of the following forms.

If $a+b$ is even, we have
$$\sqrt{\qcfs{n+1}{a, N-a}{n+1}{b, N-b}}\sqrt{\qcfs{n+1^*}{a+1, N-a-1}{n+1^*}{b+1, N-b-1}}=$$
$$\sqrt{\qcfs{n+1}{a+1, N-a-1}{n+1^*}{b, N-b}}\sqrt{\qcfs{n+1^*}{a, N-a}{n+1}{b+1, N-b-1}}+$$
$$\qcfs{n}{a+1, N-a}{n+2}{b, N-b-1}.$$

If $a+b$ is odd, we have
$$\sqrt{\qcfs{n+1^*}{a, N-a}{n+1}{b, N-b}}\sqrt{\qcfs{n+1}{a+1, N-a-1}{n+1^*}{b+1, N-b-1}}=$$
$$\sqrt{\qcfs{n+1^*}{a+1, N-a-1}{n+1^*}{b, N-b}}\sqrt{\qcfs{n+1}{a, N-a}{n+1}{b+1, N-b-1}}+$$
$$\qcfs{n}{a+1, N-a}{n+2}{b, N-b-1}.$$

We will treat the case where $a+b$ is even. The other case is parallel. The above identity in turn follows from the following identities:
\begin{itemize}
\item $$\qcfs{n}{a, N-a}{n+2}{b, N-b}\qcfs{n}{a+1, N-a}{n+1^*}{b+1, N-b-1}=$$
$$\qcfs{n}{a+1, N-a}{n+1^*}{b, N-b}\qcfs{n}{a, N-a}{n+2}{b+1, N-b-1}+$$
$$\qcfs{n}{a+1, N-a}{n+2}{b, N-b-1}\qcfs{n}{a, N-a}{n+1^*}{b+1, N-b}$$
\item $$\qcfs{n}{a+1, N-a}{n+1^*}{b+1, N-b-1}=$$
$$\sqrt{2\qcfs{n+1}{a, N-a}{n+1^*}{b+1, N-b-1}\qcfs{n+1^*}{a+1, N-a-1}{n+1^*}{b+1, N-b-1}}$$
\item $$\qcfs{n}{a+1, N-a}{n+1^*}{b, N-b}=$$
$$\sqrt{2\qcfs{n+1^*}{a, N-a}{n+1^*}{b, N-b}\qcfs{n+1}{a+1, N-a-1}{n+1^*}{b, N-b}}$$
\item $$\qcfs{n}{a, N-a}{n+1^*}{b+1, N-b}=$$
$$\sqrt{2\qcfs{n+1^*}{a, N-a}{n+1^*}{b+1, N-b-1}\qcfs{n+1}{a, N-a}{n+1^*}{b, N-b}}$$
\end{itemize}

Simply substitute each term on the left hand side of the last three identities with the corresponding term on the right-hand side into the first identity. Cancelling will give the general mutation identity. All other mutation identities for $j=n$ come from this one using degeneracies.



\end{proof}

\section{The space $\X_{G,S}$}

In this section, we explain how to derive the structure of a cluster $\X$-variety on $\Conf_m \B_G$. In \cite{FG2}, the authors explain how to construct from a cluster $\A$-variety the corresponding $\X$-variety. We have shown above that $\Conf_m \A_G$ is a cluster $\A$-variety when $G$ is a classical group. We would like to show the following:

\begin{theorem} $\Conf_m \B_G$ has the structure of a cluster $\X$-variety. This is the $\X$-variety which is attached, via the constructions of \cite{FG2}, to cluster structure that we have constructed on $\Conf_m \A_G$.
\end{theorem}

Let us recall the constructions of \cite{FG2} discussed in \label{cluster}. Suppose we have a cluster $\A$-variety with seed $\Sigma=(I, I_0,B, d)$. Then for every non-frozen index $i \in I$, there is a cluster variable $X_i$. There is a map from $p: \A_{\Sigma} \to \X_{\Sigma}$ given by 

$$p^*(X_i) = \prod_{j \in I}A_j^{B_{ij}}.$$

Let $\A$ be the cluster $\A$-variety $\Conf_3 \A_G$, and let $\Sigma$ be the initial seed we constructed. Then let us first compute the functions $p^*(X_i)$ and see that they descend to $Conf_3 \B_G$. Recall that all the cluster functions $A_j$ that we constructed on $\Conf_3 \A_G$ were invariants of tensor products:

$$A_j \in [V_{\lambda} \otimes V_{\mu} \otimes V_{\nu}]^G.$$

Now recall that $G/U$ has a left action of $H$, the Cartan subgroup. The functions on $G/U$ decompose as 
$$\bigoplus_{\lambda \in \Lambda_+} V_{\lambda}.$$
Moreover, $h \in H$ acts on the summand $V_{\lambda}$ by $\lambda(h)$.

Correspondingly, on $\Conf_3 \A_G$ there is an action of $H^3$, and $(h_1, h_2, h_3)$ acts on the summand 
$$[V_{\lambda} \otimes V_{\mu} \otimes V_{\nu}]^G$$
by 
$$\lambda(h_1)\mu(h_2)\nu(h_3).$$

From this action, it is easy to check case by case that the action of $H^3$ on $p^*(X_i)$ is trivial. In other words, the function $p^*(X_i)$ descends to the quotient of $\Conf_3 \A_G$ by $H^3$, which is precisely $\Conf_3 \B_G$.

Now we must check that the torus $\X_{\Sigma}$ is birational to $\Conf_3 \B_G$. From the above, we clearly have a map $p': \Conf_3 \B_G \to \X_{\Sigma}$, so that all the functions $X_i$ can be viewed as functions on $\Conf_3 \B_G$. We will show that they parameterize an open set in $\Conf_3 \B_G$.

To do this, we will adapt the results of \cite{FG4} and \cite{W} on parameterization of double Bruhat cells $G^{u,v}$, applied to the particular Bruhat cell $G^{w_0,e}$. Let us recall the setup. Recall that the functions on $\Conf_3 \A_G$ were associated to a reduced word composition for $w_0$. For the cases where $G= Sp_{2n}, Spin_{2n+1}, Spin_{2n+2}$, they were as follows:
$$w_{0,Sp_{2n}}=w_{0,Spin_{2n+1}}=(s_n s_{n-1} \cdots s_2 s_1)^n.$$
$$w_{0,Spin_{2n+2}}=(s_n s_{n^*} s_{n-1} \cdots s_2 s_1)^n.$$
The functions $A_{ij}$ attached to the vertices $x_{ij}$ for $i>0$ were associated to the subwords
$$u_{ij}=(s_n s_{n-1} \cdots s_2 s_1)^{i-1}s_n s_{n-1} \cdots s_j,$$
$$u_{ij}=(s_n s_{n^*} s_{n-1} \cdots s_2 s_1)^{i-1}s_n s_{n^*} s_{n-1} \cdots s_j.$$

Let $X_{ij}$ be the $\X$-function attached to the vertex $x_{ij}$ for $0 < i <n$ (the vertices $x_{ij}$ for $i=0$ or $n$ are frozen, so do not give variables on the $\X$-space). It is known that there is a parameterization of $\Conf_3 \B_G$ given by three flags
$$(B^+ ,u^-B^+, B^-),$$
where $u^-$ is determined up to the adjoint action of $H$. Let $b^-$ be an element of $B^-$. Then there is a natural projection 
$$\pi: B^- \rightarrow H=B/[B,B]$$
The choice of opposite flags $B^+$ and $B^-$ gives an inclusion
$$i: H \rightarrow B^-.$$
Then let
$$\rho(b^-):=i(\pi(b^-))^{-1}b^-.$$
This associates to each element of $B^-$ an element of $U^-$. We will be interested in $\rho(b^-)$ up to the adjoint action of $H$.

Then the co-ordinates $X_{ij}$ give a parameterization of $u^-$ by the following formula:
$$u^-=\rho(b^-),$$
where
$$b^-=F_n H_{\omega^{\vee}_n}(X_{n-1,n}^{-1})\dots F_2H_{\omega^{\vee}_2}(X_{n-1,2}^{-1})F_1H_{\omega^{\vee}_1}(X_{n-1,1}^{-1})\dots $$
$$F_n H_{\omega^{\vee}_n}(X_{1n}^{-1})\dots F_2H_{\omega^{\vee}_2}(X_{12}^{-1})F_1H_{\omega^{\vee}_1}(X_{11}^{-1}) F_n\dots F_3 F_2 F_1$$
when $G= Sp_{2n}$ or $Spin_{2n+1}$, and
$$b^-=F_n H_{\omega^{\vee}_n}(X_{n-1,n}^{-1}) F_{n^*} H_{\omega^{\vee}_{n^*}}(X_{n-1,n^*}^{-1}) \dots F_2H_{\omega^{\vee}_2}(X_{n-1,2}^{-1})F_1H_{\omega^{\vee}_1}(X_{n-1,1}^{-1})\dots $$
$$F_n H_{\omega^{\vee}_n}(X_{1n}^{-1})F_{n^*} H_{\omega^{\vee}_{n^*}}(X_{1n^*}^{-1})\dots F_2H_{\omega^{\vee}_2}(X_{12}^{-1})F_1H_{\omega^{\vee}_1}(X_{11}^{-1})F_nF_{n^*}\dots F_3 F_2 F_1$$
when $G=Spin_{2n+1}$. Let us explain the notation above. Here, $F_i$ are the usual generators of the $u^-$ associated to the simple roots. $\omega^{\vee}_i$ is the fundamental weight attached to the $i^{th}$ node of the Dynkin diagram for the simply connected form of $G^{\vee}$. The weights for $G^{\vee}$ are the coweights for the adjoint form of $G$, and $H_{\omega^{\vee}_i}$ is the cocharacter attached to this coweight.

Thus the functions $X_{ij}$ give a parameterization of $\Conf_3 \B_G$.

\begin{rmk} Replacing the the $B$-matrix by its negation does nothing to change the $\A$ space, but it replaces all the functions $X_{ij}$ on the $\X$-space by their inverses. It is a matter of convention how one chooses the signs of the $B$-matrix. If one uses the opposite convention to the one we have chosen, one would replace all the expressions $H_{\omega^{\vee}_j}(X_{ij}^{-1})$ above by $H_{\omega^{\vee}_j}(X_{ij})$, slightly simplying the above formulas.
\end{rmk}

\begin{rmk} A close reader will notice that our formulas differ slightly from those found in \cite{FG1}. We correct here an error in that paper. The formulas in \cite{FG1} give a factorization of an element in $B^-$ in terms of snakes, which are defined using {\emph intersections} of subspaces coming from a configuration of three flags. The resulting parameters $X$ (which correspond to our $X_{ij}$) in their formulas are then calculated using snakes. However, the $X$-coordinates in \cite{FG1} are defined in terms of {\emph projections} onto various subspaces. Thus they are {\emph dual} to the snake parameters. This explains the discrepancy above.
\end{rmk}

We then need to check that for any gluing of triangles to get a structure of an $\A$-space on $\Conf_4 \A_G$, the $\X$-coordinates on the edge gluing the two triangles parameterize gluings of configurations in $\Conf_3 \B_G$ to get a configuration in $\Conf_4 \B_G$. In other words, there is an equivalence
$$\Conf_4 \B_G \simeq \Conf_3 \B_G \times H \times \Conf_3 \B_G.$$
Let us examine the $\X$-coordinates on the $\X$-space corresponding to the $\A$-space $\Conf_4 \A_G$. We showed above that the face $\X$-coordinates parameterize the two copies of $\Conf_3 \B_G.$ Then we need to see that the edge $\X$ coordinates parameterize $H$, the space of gluings.

Explicitly, we have that $H$ acts by shearing the configuration of four flags in the following way:
$$h: (B^+, u^-B^+, B^-, u^+B^-) \rightarrow (B^+, u^-B^+, B^-, hu^+B^-).$$
It is then a simple matter to check that the edge $\X$-coordinates a torsor for $H$. For simplicity, let us consider the cluster structure on $\Conf_4 \A_G$ for $G=Sp_{2n}$ with the vertices labelled as in Figure 18: we have vertices $x_{ij}$, $y_k$, with $-n \leq i \leq n$, $1 \leq j \leq n$ and $1 \leq |k| \leq n$. Then the edge vertices are $x_{0j}$. An easy calculation shows that

\begin{prop} An element $h \in H$ acts on $x_{0j}$ by $\alpha_{j}(h)$, where $\alpha_j$ is the $j$-th simple root of $G'$.
\end{prop}

Note that because $G'$ is adjoint, the simple roots $\alpha_j$ span the weight lattice.

\begin{proof} The proof reduces to a computation. Suppose we have a configuration of flags
$$(B^-, u^+B^-, B^+, u^-B^+).$$
We can lift it to a configuration in $\Conf_4 \A_G$:
$$(U^-, u^+U^-, U^+, u^-U^+).$$
We can calculate the action of $H$ on $\Conf_4 \_G$ by lifting it to an action of $H$ on $\Conf_4 \A_G$:
$$h: (U^-, A_1, U^+, A_2) \rightarrow (U^- ,A_1, hU^+,hA_2).$$
Then we can easily calculate how $H$ acts on the functions on $\Conf_4 \A_G$. $\Conf_4 \A_G$ is glued from two copies of $\Conf_3 \A_G$, and the action on these two copies is as follows:
$$h: (U^-, A_1, U^+) \rightarrow (U^-, A_1, hU^+).$$
$$h: (U^-, U^+, A_2) \rightarrow (U^-, hU^+, hA_2)=(h^{-1}U^-, U^+, A_2).$$

Now the ring of functions on $\Conf_3 \A_G$ has a triple-grading by dominant weights; there is a grading by dominant weights for each flag. Suppose we have a function 
$$f \in [V_{\lambda} \otimes V_{\mu} \otimes V_{\nu}]^G.$$
Then the map
$$h: (U^-, A_1, U^+) \rightarrow (U^-, A_1, hU^+)$$
acts on this function by $\nu(h)$, i.e. $f(h\cdot x) = \nu(h) f(x)$, and the map
$$h: (U^-, U^+, A_2) \rightarrow (h^{-1}U^-, U^+, A_2)$$
acts on this function by $\lambda(w_0(h^{-1})),$ i.e. $f(h\cdot x) = \lambda(w_0(h^{-1})) f(x)$. Using the formula $p^*(X_i) = \prod_{j \in I}A_j^{B_{ij}},$ we get our result.

\end{proof}

Thus the edge coordinates on the $\X$ space give the usual ``shear'' coordinates. The usual cutting and gluing arguments allow us to conclude the following:

\begin{theorem} The spaces $\A_{G,S}$ and $\X_{G',S}$ together have the structure of a cluster ensemble.
\end{theorem}

\section{Applications}

Let $(\A,\X)$ form a cluster ensemble, attached to the seed data $\Sigma = (I,I_0,B,d)$. Then the results of \cite{FG2}, \cite{FG3} and \cite{FG5} give a deformation quantization $\X_q$ of the space $\X$, as well as representations of this quantum algebra on a Hilbert space. Moreover, on this Hilbert space, we get a natural projective unitary action of the {\emph cluster modular group} of $(\A,\X)$.

The results of this paper can be interpreted as saying that, first of all, $(\A_{G,S}, \X_{G',S})$ is a cluster ensemble whenever $G$ is a classical group; and second, that in these cases, the cluster modular group of $(\A_{G,S}, \X_{G',S})$ contains the mapping class group of $S$. Thus, we get a projective unitary representation of the mapping class group of $S$ coming from the higher Teichmuller spaces $\A_{G,S}$ and $\X_{G',S}$. This is the kind of data one expects to get from a modular functor. It remains an open question whether these projective unitary representations fit together to give a modular functor. This is was conjectured by Teschner in \cite{T}.

In this section, we sketch how to construct the space $\X_q$ and the Hilbert space it acts on.

Consider the seed $\Sigma = (I,I_0,B,d)$. The $B$-matrix encodes the Poisson structure on $\X$. Let $$\widehat{\varepsilon_{ij}}=b_{ij}d_j.$$
There is a canonical Poisson structure that is given in the torus chart $\X_{\Sigma}$ by
\[
\{X_i,X_j\} = \widehat{\varepsilon_{ij}}_{ij}X_iX_j.
\]
This at first glance seems to depend on which torus chart $\X_{\Sigma}$ we are using, but it turns out to be invariant under mutations.

In order to quantize the space $\X$, we first construct a $q$-deformation of each $\X_{\Sigma}$, which we will call $\X_{\Sigma,q}$. $\X_{\Sigma,q}$ will be a quantum torus. It is given by generators $X_i, i \in I$ and commutation relations

\begin{equation} 
q^{-\widehat \varepsilon_{ij}} X_i X_j = 
q^{-\widehat \varepsilon_{ji}} X_jX_i
\end{equation}

Thus each seed gives a quantum torus. We now need to glue together these quantum tori, i.e., find rational maps between the non-commutative algebras $\X_{\Sigma,q}$ that quantize the maps between the algebras $\X_{\Sigma}$.

Recall that if $\Sigma'$ is obtained from $\Sigma$ by mutation at $k \in I \setminus I_0$, we have a birational map $\mu_k: \X_\Sigma \to \X_{\Sigma'}$.  It is defined by
\begin{align}
\mu_k^*(X'_i) = \begin{cases}
X_i(1+X_k^{-1})^{-b_{ik}} & \textrm{ if } b_{ik}>0, i \neq k,  \\
X_i(1+X_k)^{-b_{ik}} & \textrm{ if } b_{ik}<0, i \neq k,  \\
X_k^{-1} & i = k.
\end{cases}
\end{align}

Now set $$q_k=q^{1/d_k}.$$ We define the quantum mutation $\mu^q_k: \X_{\Sigma,q} \to \X_{\Sigma',q}$ by

\begin{align}
\mu^{q*}_k(X'_i) = \begin{cases}
X_i \left((1+q^{-1}_kX_k^{-1})(1+q_k^{-3}X_k^{-1}) \ldots (1+q_k^{1-2|b_{ik}|}X_k^{-1})\right)^{-1} & \textrm{ if } b_{ik}>0, i \neq k,  \\
X_i (1+q_kX_k)(1+q^3_kX_k) \ldots (1+q_k^{2|b_{ik}|-1}X_k) & \textrm{ if } b_{ik}<0, i \neq k,  \\
X_k^{-1} & i = k.
\end{cases}
\end{align}

Setting $q=1$ recovers the space $\X$. The fact that $\mu^{q*}$ gives a map of algebras follows from decomposing it as a quantum torus isomorphism and conjugation by the quantum dilogarithm, \cite{FG2}.

Now consider the special case of the space  $\X_{G',S}$. Suppose that $S$ has $n$ boundary components. Then the monodromy around each boundary component is an element of the Borel subgroup $B' \subset G'$. There is a natural projection $B' \rightarrow H$. Therefore, we get a map 
$$\pi: \X_{G',S} \rightarrow H^n.$$
It turns out the the functions on $H^n$ are precisely the center of the algebra of functions on $\X_{G',S}$, i.e., they Poisson-commute with everything. The fibers of $\pi$ are the symplectic leaves of $\X_{G',S}$.

Every choice of a point in $c \in H^n$ prescribes values for the center of the algebra of functions on $\X_{G',S,q}$. The remaining functions essentially form a Heisenberg algebra, and hence have a unique unitary representation on a Hilbert space $\mathcal{H}_c$. This Hilbert space gives a projective unitary representation of the mapping class group of $S$.

\section{Appendix}

In this section, we include the formulas for the mutation sequence for the flip of $\Conf_4 \A_{G}$ when $i \leq 0$.

\subsection{$Sp_{2n}$}

The mutation sequence for the flip of $\Conf_4 \A_{Sp_{2n}}$ when $i \leq 0$:

\begin{itemize}
\item When $|i|+j < n$ and $|i| < j$, the function attached to $x_{ij}$ mutates in three stages, consisting of $n-|i|-j, |i|,$ and $j-|i|$ mutations, respectively:
\begin{enumerate}
\item $$\tcfl{j}{|i|}{N-|i|-j} = \qcfs{j}{|i|}{N-|i|-j}{N} \rightarrow \qcfs{j+1}{|i|+1}{N-|i|-j-1}{N-1} \rightarrow $$
$$\qcfs{j+2}{|i|+2}{N-|i|-j-2}{N-2} \rightarrow $$
$$\dots \rightarrow \qcfs{n-|i|}{n-j}{n}{n+|i|+j}$$
\item $$\qcfs{n-|i|}{n-j}{n}{n+|i|+j}=\qcfs{n-|i|}{n-j, N}{n}{n+|i|+j} \rightarrow $$
$$\qcfs{n-|i|+1}{n-j+1, N-1}{n}{n+|i|+j-1} \rightarrow$$
$$\qcfs{n-|i|+2}{n-j+2, N-2}{n}{n+|i|+j-2} \rightarrow $$
$$\dots \rightarrow \qcfs{n}{n-j+|i|, N-|i|}{n}{n+j}$$
\item $$\qcfs{n}{n-j+|i|, N-|i|}{n}{n+j}=\qcfs{n}{n-j+|i|, N-|i|}{n}{0, n+j} \rightarrow $$
$$\qcfs{n}{n-j+|i|+1, N-|i|-1}{n}{1, n+j-1} \rightarrow $$
$$\qcfs{n}{n-j+|i|+2, N-|i|-2}{n}{2, n+j-2} \rightarrow $$
$$\dots \rightarrow [\qcfs{n}{n, N-j}{n}{j-|i|, n+|i|}] \textrm{ } \tcfu{n}{N-j}{j-|i|, n+|i|}$$
\end{enumerate}
Here, if additionally $i=0$, we remove an additional factor of $\dur{n}{n}$ from the last term to get $\dlr{N-j}{j}$ after the last mutation.

\item When $|i|+j \geq n$ and $|i| < j$, the function attached to $x_{ij}$ mutates  in two stages, consisting of $n-j,$ and $j-|i|$ mutations, respectively:
\begin{enumerate}
\item $$\tcfl{j}{|i|, N+n-|i|-j}{n}=\qcf{j}{|i|, N+n-|i|-j}{n}{N} \rightarrow $$
$$\qcfs{j+1}{|i|+1, N+n-|i|-j-1}{n}{N-1} \rightarrow $$
$$\qcfs{j+2}{|i|+2, N+n-|i|-j-2}{n}{N-2} \rightarrow$$
$$\dots  \rightarrow \qcfs{n}{n+|i|-j, N-|i|}{n}{n+j}$$
\item $$\qcfs{n}{n+|i|-j, N-|i|}{n}{n+j}=\qcfs{n}{n+|i|-j, N-|i|}{n}{0, n+j} \rightarrow $$
$$\qcfs{n}{n+|i|-j+1, N-|i|-1}{n}{1, n+j-1} \rightarrow $$
$$\qcfs{n}{n+|i|-j+2, N-|i|-2}{n}{2, n+j-2} \rightarrow$$
$$\dots  \rightarrow [\qcfs{n}{n, N-j}{n}{j-|i|, n+|i|}]  \textrm{ } \qcfs{n}{N-j}{0}{j-|i|, n+|i|} $$
\end{enumerate}

\item When $|i|+j < n$ and $|i| \geq j$, the function attached to $x_{ij}$ mutates in two stages, consisting of $n-|i|-j,$ and $j$ mutations, respectively:
\begin{enumerate}
\item $$\tcfl{j}{|i|}{N-|i|-j} = \qcfs{j}{|i|}{N-|i|-j}{N} \rightarrow $$
$$\qcfs{j+1}{|i|+1}{N-|i|-j-1}{N-1} \rightarrow $$
$$\qcfs{j+2}{|i|+2}{N-|i|-j-2}{N-2} \rightarrow $$
$$\dots \rightarrow \qcfs{n-|i|}{n-j}{n}{n+|i|+j}$$
\item $$\qcfs{n-|i|}{n-j}{n}{n+|i|+j}=\qcfs{n-|i|}{n-j, N}{n}{n+|i|+j} \rightarrow $$
$$\qcfs{n-|i|+1}{n-j+1, N-1}{n}{n+|i|+j-1} \rightarrow$$
$$\qcfs{n-|i|+2}{n-j+2, N-2}{n}{n+|i|+j-2} \rightarrow $$
$$\dots \rightarrow [\qcfs{n-|i|+j}{n, N-j}{n}{n+|i|}]  \textrm{ } \tcfu{n-|i|+j}{N-j}{n+|i|}$$
\end{enumerate}

\item When $|i|+j \geq n$ and $|i| \geq j$, the function attached to $x_{ij}$ mutates in one stage consisting of $n-|i|$ mutations:
$$\tcfl{j}{|i|, N+n-|i|-j}{n}=\qcfs{j}{|i|, N+n-|i|-j}{n}{N} \rightarrow $$
$$\qcfs{j+1}{|i|+1, N+n-|i|-j-1}{n}{N-1} \rightarrow $$
$$\qcfs{j+2}{|i|+2, N+n-|i|-j-2}{n}{N-2} \rightarrow$$
$$\dots  \rightarrow [\qcfs{n-|i|+j}{n, N-j}{n}{n+|i|}] \textrm{ } \tcfu{n-|i|+j}{N-j}{n+|i|}$$

\end{itemize}

\subsection{$Spin_{2n+1}$}

The mutation sequence for the flip of $\Conf_4 \A_{Spin_{2n+1}}$ when $i \leq 0$:

\begin{itemize}
\item When $|i|+j < n$ and $|i| < j$, the function attached to $x_{ij}$ mutates in three stages, consisting of $n-|i|-j, |i|,$ and $j-|i|$ mutations, respectively:
\begin{enumerate}
\item $$\tcfl{j}{|i|}{N-|i|-j} = \qcfs{j}{|i|}{N-|i|-j}{N} \rightarrow \qcfs{j+1}{|i|+1}{N-|i|-j-1}{N-1} \rightarrow $$
$$\qcfs{j+2}{|i|+2}{N-|i|-j-2}{N-2} \rightarrow $$
$$\dots \rightarrow \qcfs{n-|i|}{n-j}{n+1}{n+1+|i|+j}$$
\item $$\qcfs{n-|i|}{n-j}{n+1}{n+1+|i|+j}=\qcfs{n-|i|}{n-j, N}{n+1}{n+1+|i|+j} \rightarrow $$
$$\qcfs{n-|i|+1}{n-j+1, N-1}{n+1}{n+|i|+j} \rightarrow$$
$$\qcfs{n-|i|+2}{n-j+2, N-2}{n+1}{n+|i|+j-1} \rightarrow $$
$$\dots \rightarrow \qcfs{n}{n-j+|i|, N-|i|}{n+1}{n+1+j}$$
\item $$\qcfs{n}{n-j+|i|, N-|i|}{n+1}{n+1+j}=\qcfs{n}{n-j+|i|, N-|i|}{n+1}{0, n+1+j} \rightarrow $$
$$\qcfs{n}{n-j+|i|+1, N-|i|-1}{n+1}{1, n+j} \rightarrow $$
$$\qcfs{n}{n-j+|i|+2, N-|i|-2}{n+1}{2, n+j-1} \rightarrow $$
$$\dots \rightarrow [\qcfs{n}{n, N-j}{n+1}{j-|i|, n+1+|i|}] \textrm{ } \tcfu{n}{N-j}{j-|i|, n+1+|i|}$$
\end{enumerate}
Here, if additionally $i=0$, we remove an additional factor of $\dur{n}{n+1}$ from the last term to get $\dlr{N-j}{j}$ after the last mutation.

\item When $|i|+j \geq n$, $|i| < j \neq n$, the function attached to $x_{ij}$ mutates  in two stages, consisting of $n-j,$ and $j-|i|$ mutations, respectively:
\begin{enumerate}
\item $$\tcfl{j}{|i|, N+n-|i|-j}{n+1}=\qcf{j}{|i|, N+n-|i|-j}{n+1}{N} \rightarrow $$
$$\qcfs{j+1}{|i|+1, N+n-|i|-j-1}{n+1}{N-1} \rightarrow $$
$$\qcfs{j+2}{|i|+2, N+n-|i|-j-2}{n+1}{N-2} \rightarrow$$
$$\dots  \rightarrow \qcfs{n}{n+|i|-j, N-|i|}{n+1}{n+1+j}$$
\item $$\qcfs{n}{n+|i|-j, N-|i|}{n+1}{n+1+j}=\qcfs{n}{n+|i|-j, N-|i|}{n+1}{0, n+1+j} \rightarrow $$
$$\qcfs{n}{n+|i|-j+1, N-|i|-1}{n+1}{1, n+j} \rightarrow $$
$$\qcfs{n}{n+|i|-j+2, N-|i|-2}{n+1}{2, n+j-1} \rightarrow$$
$$\dots  \rightarrow [\qcfs{n}{n, N-j}{n+1}{j-|i|, n+1+|i|}]  \textrm{ } \qcfs{n}{N-j}{0}{j-|i|, n+1+|i|} $$
\end{enumerate}

\item In the most interesting case, we have that if $j=n$, then mutations happen in one stage consisting of $n-|i|$ mutations:
$$\sqrt{\qcfs{n}{|i|, N-|i|}{n+1}{0}} [\sqrt{\qcfs{n}{|i|, N-|i|}{n+1}{0, N}}] \rightarrow $$
$$\sqrt{\qcfs{n}{|i|+1, N-|i|-1}{n+1}{1, N-1}} \rightarrow $$
$$\sqrt{\qcfs{n}{|i|+2, N-|i|-2}{n+1}{2, N-2}} \rightarrow$$
$$\dots  \rightarrow [\sqrt{\qcfs{n}{n, n+1}{n+1}{n-|i|, N-n+|i|}}]  \textrm{ } \sqrt{\tcfu{n}{n+1}{n-|i|, N-n+|i|}}$$

\item When $|i|+j < n$ and $|i| \geq j$, the function attached to $x_{ij}$ mutates in two stages, consisting of $n-|i|-j,$ and $j$ mutations, respectively:
\begin{enumerate}
\item $$\tcfl{j}{|i|}{N-|i|-j} = \qcfs{j}{|i|}{N-|i|-j}{N} \rightarrow $$
$$\qcfs{j+1}{|i|+1}{N-|i|-j-1}{N-1} \rightarrow $$
$$\qcfs{j+2}{|i|+2}{N-|i|-j-2}{N-2} \rightarrow $$
$$\dots \rightarrow \qcfs{n-|i|}{n-j}{n+1}{n+1+|i|+j}$$
\item $$\qcfs{n-|i|}{n-j}{n+1}{n+1+|i|+j}=\qcfs{n-|i|}{n-j, N}{n+1}{n+1+|i|+j} \rightarrow $$
$$\qcfs{n-|i|+1}{n-j+1, N-1}{n+1}{n+|i|+j} \rightarrow$$
$$\qcfs{n-|i|+2}{n-j+2, N-2}{n+1}{n+|i|+j-1} \rightarrow $$
$$\dots \rightarrow [\qcfs{n-|i|+j}{n, N-j}{n+1}{n+1+|i|}]  \textrm{ } \tcfu{n-|i|+j}{N-j}{n+1+|i|}$$
\end{enumerate}

\item When $|i|+j \geq n$ and $|i| \geq j$, the function attached to $x_{ij}$ mutates in one stage consisting of $n-|i|$ mutations:
$$\tcfl{j}{|i|, N+n-|i|-j}{n+1}=\qcfs{j}{|i|, N+n-|i|-j}{n+1}{N} \rightarrow $$
$$\qcfs{j+1}{|i|+1, N+n-|i|-j-1}{n+1}{N-1} \rightarrow $$
$$\qcfs{j+2}{|i|+2, N+n-|i|-j-2}{n+1}{N-2} \rightarrow$$
$$\dots  \rightarrow [\qcfs{n-|i|+j}{n, N-j}{n+1}{n+1+|i|}] \textrm{ } \tcfu{n-|i|+j}{N-j}{n+1+|i|}$$

\end{itemize}

\subsection{$Spin_{2n+2}$}

The mutation sequence for the flip of $\Conf_4 \A_{Spin_{2n+2}}$ when $i \leq 0$:

\begin{itemize}
\item When $|i|+j < n$ and $|i| < j$, the function attached to $x_{ij}$ mutates in three stages, consisting of $n-|i|-j, |i|,$ and $j-|i|$ mutations, respectively:
\begin{enumerate}
\item $$\tcfl{j}{|i|}{N-|i|-j} = \qcfs{j}{|i|}{N-|i|-j}{N} \rightarrow \qcfs{j+1}{|i|+1}{N-|i|-j-1}{N-1} \rightarrow $$
$$\qcfs{j+2}{|i|+2}{N-|i|-j-2}{N-2} \rightarrow $$
$$\dots \rightarrow \qcfs{n-|i|}{n-j}{n+2}{n+2+|i|+j}$$
\item $$\qcfs{n-|i|}{n-j}{n+2}{n+2+|i|+j}=\qcfs{n-|i|}{n-j, N}{n+2}{n+2+|i|+j} \rightarrow $$
$$\qcfs{n-|i|+1}{n-j+1, N-1}{n+2}{n+1+|i|+j} \rightarrow$$
$$\qcfs{n-|i|+2}{n-j+2, N-2}{n+2}{n+|i|+j} \rightarrow $$
$$\dots \rightarrow \qcfs{n}{n-j+|i|, N-|i|}{n+1}{n+1+j}$$
\item $$\qcfs{n}{n-j+|i|, N-|i|}{n+2}{n+2+j}=\qcfs{n}{n-j+|i|, N-|i|}{n+2}{0, n+2+j} \rightarrow $$
$$\qcfs{n}{n-j+|i|+1, N-|i|-1}{n+2}{1, n+1+j} \rightarrow $$
$$\qcfs{n}{n-j+|i|+2, N-|i|-2}{n+2}{2, n+j} \rightarrow $$
$$\dots \rightarrow [\qcfs{n}{n, N-j}{n+2}{j-|i|, n+2+|i|}] \textrm{ } \tcfu{n}{N-j}{j-|i|, n+2+|i|}$$
\end{enumerate}
Here, if additionally $i=0$, we remove an additional factor of $\dur{n}{n+2}$ from the last term to get $\dlr{N-j}{j}$ after the last mutation.

\item When $|i|+j \geq n$, $|i| < j \neq n$, the function attached to $x_{ij}$ mutates  in two stages, consisting of $n-j,$ and $j-|i|$ mutations, respectively:
\begin{enumerate}
\item $$\tcfl{j}{|i|, N+n-|i|-j}{n+2}=\qcf{j}{|i|, N+n-|i|-j}{n+2}{N} \rightarrow $$
$$\qcfs{j+1}{|i|+1, N+n-|i|-j-1}{n+2}{N-1} \rightarrow $$
$$\qcfs{j+2}{|i|+2, N+n-|i|-j-2}{n+2}{N-2} \rightarrow$$
$$\dots  \rightarrow \qcfs{n}{n+|i|-j, N-|i|}{n+2}{n+2+j}$$
\item $$\qcfs{n}{n+|i|-j, N-|i|}{n+2}{n+2+j}=\qcfs{n}{n+|i|-j, N-|i|}{n+2}{0, n+2+j} \rightarrow $$
$$\qcfs{n}{n+|i|-j+1, N-|i|-1}{n+2}{1, n+1+j} \rightarrow $$
$$\qcfs{n}{n+|i|-j+2, N-|i|-2}{n+2}{2, n+j} \rightarrow$$
$$\dots  \rightarrow [\qcfs{n}{n, N-j}{n+2}{j-|i|, n+2+|i|}]  \textrm{ } \qcfs{n}{N-j}{0}{j-|i|, n+2+|i|} $$
\end{enumerate}

\item In the most interesting case, we have that if $j=n$ or $j=n^*$, the mutations happen in one stage consisting of $n-|i|$ mutations.
If $j=n$ and $i$ odd, we have:
$$\sqrt{\qcfs{n+1}{|i|, N-|i|}{n+1^*}{0}} [\sqrt{\qcfs{n+1}{|i|, N-|i|}{n+1^*}{0, N}}] \rightarrow $$
$$\sqrt{\qcfs{n+1^*}{|i|+1, N-|i|-1}{n+1}{1, N-1}} \rightarrow $$
$$\sqrt{\qcfs{n+1}{|i|+2, N-|i|-2}{n+1^*}{2, N-2}} \rightarrow$$
$$\dots  \rightarrow [\sqrt{\qcfs{n+1}{n, n+2}{n+1^*}{n-|i|, N-n+|i|}}]  \textrm{ } \sqrt{\tcfu{n+1}{n+1}{n-|i|, N-n+|i|}}$$
Note in the last step we use
$$\sqrt{\qcfs{n+1}{n, n+2}{n+1^*}{a, N-a}}=\sqrt{\qcfs{n+1}{n+1, n+1^*}{n+1^*}{a, N-a}}$$
If $j=n$, and $i$ is even, we have:
$$\sqrt{\qcfs{n+1}{|i|, N-|i|}{n+1}{0}} [\sqrt{\qcfs{n+1}{|i|, N-|i|}{n+1}{0, N}}] \rightarrow $$
$$\sqrt{\qcfs{n+1^*}{|i|+1, N-|i|-1}{n+1^*}{1, N-1}} \rightarrow $$
$$\sqrt{\qcfs{n+1}{|i|+2, N-|i|-2}{n+1}{2, N-2}} \rightarrow$$
$$\dots  \rightarrow [\sqrt{\qcfs{n+1^*}{n, n+2}{n+1^*}{n-|i|, N-n+|i|}}]  \textrm{ } \sqrt{\tcfu{n+1^*}{n+1}{n-|i|, N-n+|i|}}$$
Note in the last step we use
$$\sqrt{\qcfs{n+1^*}{n, n+2}{n+1^*}{a, N-a}}=\sqrt{\qcfs{n+1^*}{n+1, n+1^*}{n+1^*}{a, N-a}}.$$
To obtain the formulas in the case where $j=n^*$, switch all occurences of $n+1$ and $n+1^*$.

\item When $|i|+j < n$ and $|i| \geq j$, the function attached to $x_{ij}$ mutates in two stages, consisting of $n-|i|-j,$ and $j$ mutations, respectively:
\begin{enumerate}
\item $$\tcfl{j}{|i|}{N-|i|-j} = \qcfs{j}{|i|}{N-|i|-j}{N} \rightarrow $$
$$\qcfs{j+1}{|i|+1}{N-|i|-j-1}{N-1} \rightarrow $$
$$\qcfs{j+2}{|i|+2}{N-|i|-j-2}{N-2} \rightarrow $$
$$\dots \rightarrow \qcfs{n-|i|}{n-j}{n+2}{n+2+|i|+j}$$
\item $$\qcfs{n-|i|}{n-j}{n+2}{n+2+|i|+j}=\qcfs{n-|i|}{n-j, N}{n+2}{n+2+|i|+j} \rightarrow $$
$$\qcfs{n-|i|+1}{n-j+1, N-1}{n+2}{n+1+|i|+j} \rightarrow$$
$$\qcfs{n-|i|+2}{n-j+2, N-2}{n+2}{n+|i|+j} \rightarrow $$
$$\dots \rightarrow [\qcfs{n-|i|+j}{n, N-j}{n+2}{n+2+|i|}]  \textrm{ } \tcfu{n-|i|+j}{N-j}{n+2+|i|}$$
\end{enumerate}

\item When $|i|+j \geq n$ and $|i| \geq j$, the function attached to $x_{ij}$ mutates in one stage consisting of $n-|i|$ mutations:
$$\tcfl{j}{|i|, N+n-|i|-j}{n+2}=\qcfs{j}{|i|, N+n-|i|-j}{n+2}{N} \rightarrow $$
$$\qcfs{j+1}{|i|+1, N+n-|i|-j-1}{n+2}{N-1} \rightarrow $$
$$\qcfs{j+2}{|i|+2, N+n-|i|-j-2}{n+2}{N-2} \rightarrow$$
$$\dots  \rightarrow [\qcfs{n-|i|+j}{n, N-j}{n+2}{n+2+|i|}] \textrm{ } \tcfu{n-|i|+j}{N-j}{n+2+|i|}$$

\end{itemize}

\end{document}